\documentclass[12pt]{amsart}
\usepackage[top=1.08in, left=1.28in, bottom=1.08in, right=1.28in]{geometry}
\usepackage{amsmath,amsfonts,amsbsy,amsgen,amscd,mathrsfs,amssymb,amsthm,mathtools,bbm,enumerate}
\usepackage[colorlinks=true,citecolor=teal!80!black,linkcolor=blue]{hyperref}
\usepackage[foot]{amsaddr}

\usepackage{hhline}
\usepackage{fouriernc}

\usepackage{pgfplots}
\usepackage{tikz-cd}
\usetikzlibrary{arrows,automata}

\numberwithin{equation}{section}
\newtheorem{theorem}{Theorem}[section]

\newtheorem{lemma}[theorem]{Lemma}
\newtheorem{proposition}[theorem]{Proposition}
\theoremstyle{definition}
\newtheorem{assumption}[theorem]{Assumption}
\newtheorem{definition}[theorem]{Definition}
\newtheorem{example}[theorem]{Example}
\newtheorem{notation}[theorem]{Notation}
\newtheorem{problem}[theorem]{Problem}

\newtheorem{remark}[theorem]{Remark}

\makeatletter\renewenvironment{proof}[1][\proofname] {\par\pushQED{\qed}\normalfont\topsep6\p@\@plus6\p@\relax\trivlist\item[\hskip\labelsep\bfseries#1\@addpunct{.}]\ignorespaces}{\popQED\endtrivlist}

\newcommand\al{\alpha}
\newcommand\be{\beta}
\newcommand\dd{\mathrm d}
\newcommand\De{\Delta}
\newcommand\de{\delta}
\newcommand\deq{\stackrel{\mathrm{distr.}}{=}}
\newcommand\eps{\varepsilon}

\newcommand\Ga{\Gamma}
\newcommand\ka{\kappa}
\newcommand\La{\Lambda}
\newcommand\la{\lambda}
\newcommand{\om}{\omega}
\newcommand\Om{\Omega}

\newcommand\si{\sigma}

\newcommand\ze{\zeta}
\newcommand\fo{\overrightarrow}

\renewcommand\bar{\overline}
\renewcommand\d{~\mathrm d}
\renewcommand\phi{\varphi}
\renewcommand\rho{\varrho}
\renewcommand\th{\vartheta}
\renewcommand\hat{\widehat}

\newcommand\bs{\boldsymbol}
\newcommand\mbb{\mathbb}
\newcommand\mbf{\mathbf}
\newcommand\mc{\mathcal}
\newcommand\mf{\mathfrak}
\newcommand\mr{\mathrm}

\newcommand\msf{\mathsf}

\begin{document}

\title[Trace Moments for
Schr\"odinger + Matrix Noise
and the Rigidity of MSAO]{Trace Moments for Schr\"odinger Operators with Matrix White Noise and the Rigidity of the Multivariate
Stochastic Airy Operator}
\author{Pierre Yves Gaudreau Lamarre}
\address{Department of Mathematics,
Syracuse University,
Syracuse, NY 13244}
\email{pgaudrea@syr.edu}
\maketitle

\begin{abstract}
We study the semigroups of random Schr\"odinger operators of the form $\hat Hf=-\frac12f''+(V+\xi)f$,
where $f:I\to\mbb F^r$ ($\mbb F=\mbb R,\mbb C,\mbb H$) are vector-valued functions on a possibly infinite interval $I\subset\mbb R$
that satisfy a mix of Robin and Dirichlet boundary conditions,
$V$ is a deterministic diagonal potential with power-law growth at infinity,
and $\xi$ is a matrix white noise.
Our main result consists of Feynman-Kac formulas for trace moments of
the form
$\mbf E\Big[\prod_{k=1}^n\mr{Tr}\big[\mr e^{-t_k\hat H}\big]\Big]$
($n\in\mbb N$, $t_k>0$).

One notable example covered by our main result consists of the multivariate stochastic
Airy operator (SAO) of Bloemendal and Vir\'ag \cite{BloemendalVirag2},
which characterizes the soft-edge eigenvalue
fluctuations of critical rank-$r$ spiked Wishart and GO/U/SE random matrices.
As a corollary of our main result, we prove that if
 $V$'s growth is at least linear
(this includes the multivariate SAO), then $\hat H$'s spectrum is number rigid in the sense
of Ghosh and Peres \cite{GhoshPeres}. Together with the rigidity of the scalar SAO \cite{Bufetov,RigiditySAO},
this completes the characterization of number rigidity in the soft-edge limits
of Gaussian $\be$-ensembles and their finite-rank spiked versions.
\end{abstract}

\setcounter{tocdepth}{1}
\tableofcontents

\section{Introduction}

\subsection{General Setup and Main Result}

In this paper, we are interested in the semigroup theory of a class of random
vector-valued Schr\"odinger operators defined as follows:
Let $\mbb F$ be the field of real numbers $\mbb R$, the field of complex numbers $\mbb C$, or
the ring of quaternions $\mbb H$.
Let $r>0$ be a fixed integer, and $I\subset\mbb R$ be a (possibly infinite) interval.
Consider the operator
\[Hf(i,x)=-\tfrac12f''(i,x)+V(i,x)f(i,x),\qquad f:I\to\mbb F^r,\]
where $\big(f(i,\cdot)\big)_{1\leq i\leq r}$ denote the components of a vector-valued function $f$,
the second derivative $f''$ acts on the variable $x\in I$ only, and $V(i,\cdot)$ are locally integrable deterministic
functions with sufficient growth when $I$ is infinite (see Assumption \ref{Assumption: Potential}).
Assume that $H$ acts on functions whose components $f(i,\cdot)$ satisfy a mix of Robin boundary conditions
$f'(i,c)+\al f(i,c)=0$ ($\al\in\mbb R$)
and Dirichlet boundary conditions $f(i,c)=0$ at the boundary points $c\in\partial I$; if there are any
(see Assumption \ref{Assumption: Boundary}).

We are interested in random perturbations of $H$ of the form
\begin{align}
\label{Equation: Operator}
\hat H=H+\xi,
\end{align}
where $\xi$ is a matrix white noise. That is, if we let $\cdot^*$ denote conjugation
in $\mbb F$, then $\xi$ is the singular matrix multiplication operator
\[\xi f(i,\cdot)=\sum_{j=1}^r\xi_{i,j}f(j,\cdot),\]
where $\xi_{i,j}=\xi_{j,i}^*$ for all $1\leq i,j\leq r$, and informally,
we view $(\xi_{i,i})_{1\leq i\leq r}$ as the derivatives of i.i.d. Brownian motions
in $\mbb R$ and $(\xi_{i,j})_{1\leq i<j\leq r}$ as the derivatives of i.i.d. Brownian motions in
$\mbb F$ (see Definition \ref{Definition: White Noise} for more details).

The aim of this paper is to provide a new avenue to study the
eigenvalues of $\hat H$ via probabilistic representations of
the semigroup $\mr e^{-t\hat H}$.
Our main result in this direction, {\bf Theorem \ref{Theorem: Trace Moment Formulas}}, consists of
Feynman-Kac formulas for the joint moments
\begin{align}
\label{Equation: Mixed Moments}
\mbf E\left[\mr{Tr}\big[\mr e^{-t_1\hat H}\big]\mr{Tr}\big[\mr e^{-t_2\hat H}\big]\cdots\mr{Tr}\big[\mr e^{-t_n\hat H}\big]\right],\qquad
n\in\mbb N,~t_1,\ldots,t_n\geq0.
\end{align}
Letting $\la_1(\hat H)\leq\la_2(\hat H)\leq\cdots$ denote the eigenvalues of $\hat H$, we recall that
\[\textstyle\mr{Tr}\big[\mr e^{-t\hat H}\big]=\sum_{k=1}^\infty\mr e^{-t\la_k(\hat H)}\qquad\text{for every }t>0.\]
Thus, a calculation of the quantities \eqref{Equation: Mixed Moments} provides a means of characterizing the distribution of $\hat H$'s
eigenvalue point process via its Laplace transform's joint moments.

A notable feature of our Feynman-Kac formulas
when $r>1$ is the following insight: In addition to the usual Brownian motion $\big(B(t)\big)_{t\geq0}$
(coming from the diffusion operator $f\mapsto-\tfrac12f''$), the dynamics are decorated by
a continuous-time random walk $\big(U(t)\big)_{t\geq0}$ on the complete graph with $r$ vertices, whose jumps times
$\tau_1\leq\tau_2\leq\cdots$ are conditioned to coincide with self-intersections of both $U$ and $B$. That is,
there must exist a way to partition the jump times
into pairs $\{\tau_{\ell_1},\tau_{\ell_2}\}$ such that
\begin{enumerate}
\item the jumps in $U$ that occur at times $\tau_{\ell_1}$ and $\tau_{\ell_2}$
are between the same two vertices in $\{1,\ldots,r\}$, and
\item $B(\tau_{\ell_1})=B(\tau_{\ell_2})$.
\end{enumerate}
Among other things, this induces a surprisingly rich combinatorial structure in the mixed moments
\eqref{Equation: Mixed Moments}, coming from counting paths of the random walk $U$ for which such pairings exist.
We refer to the statement of Theorem \ref{Theorem: Trace Moment Formulas} for the full details,
and to Section \ref{Section: Probabilistic OE} for a high-level explanation of this result and how it is proved.

The remainder of this introduction is organized as follows:
In Section \ref{Section: Motivation and Applications}, we discuss
the motivation for our work (mainly, understanding the multivariate stochastic Airy operators), as well as an application in the study of number rigidity of eigenvalue point processes.
In Section \ref{Section: Proof Methods and Past Results},
we provide a high-level explanation of our method of proof,
and we survey past results related to the main technical innovation in this paper.
In Section \ref{Section: Pathwise}, we formulate an open problem
regarding a pathwise Feynman-Kac formula for $\hat H$.
Then, in Section \ref{Section: Organization},
we discuss the organization of the rest of the paper.

\subsection{Motivation and Application}
\label{Section: Motivation and Applications}

\subsubsection{Motivation: Multivariate Stochastic Airy Operators}
\label{Section: mSAO}

Our main motivation for developing the semigroup theory of $\hat H$ comes from
the special case where $I=(0,\infty)$, $V(i,x)=\frac{rx}2$, and $\frac1{\sqrt2}\xi$ is a {\it standard}
matrix white noise (that is, the derivative of a matrix Brownian motion with GOE/GUE/GSE
increments; see Remark \ref{Remark: GOE/GUE/GSE Case} for more details).
In this special case, $2\hat H$ is the multivariate stochastic Airy operator
(SAO) introduced by Bloemendal and Vir\'ag in \cite{BloemendalVirag2}
(see also \cite{BloemendalVirag,RamirezRiderVirag} for previous results concerning the scalar case $r=1$).

The interest in the multivariate SAO lies in the fact that its spectrum describes the asymptotic
soft-edge scaling limits of critical rank-$r$ perturbations of Wishart and Gaussian
invariant ensembles; the so-called spiked models. More specifically, consider an
$r\times r$ generalized matrix of the form
\begin{align}
\label{Equation: W Matrix Boundary}
W=\sum_{i=1}^r(-\bar\al_i)u_iu_i^\dagger,\qquad \bar\al_i\in(-\infty,\infty],
\end{align}
where the vectors $u_i\in\mbb F^r$ form an orthonormal basis
and $\dagger$ denotes the conjugate transpose,
and let $2\hat H$ be the multivariate SAO with boundary conditions
\begin{align}
\label{Equation: Standard Basis Boundary}
\begin{cases}
f'(i,0)+\bar\al_if(i,0)=0&\text{if }\bar\al_i\in\mbb R,\\
f(i,0)=0&\text{if }\bar\al_i=-\infty.
\end{cases}
\end{align}
\begin{enumerate}[$\bullet$]
\item On the one hand, \cite[Theorem 1.2]{BloemendalVirag2} showed that
$2\hat H$'s spectrum describes the limiting fluctuations
(as $n\to\infty$)
of the top eigenvalues of $n\times n$ random matrices of the form $X_n+(W_n\oplus 0_{n-r})$,
where $X_n$ is a GOE, GUE, or GSE matrix (depending on
whether $\be=1,2,$ or $4$),
\begin{align}
\label{Equation: Critical Spike 1}
W_n\approx\sqrt{n}\,\mr{Id}_r-n^{1/6}W
\end{align}
is a critical additive perturbation of rank $r$,
and $0_k$ and $\mr{Id}_k$ denote respectively the $k\times k$ zero and identity matrices.
\item On the other hand, \cite[Theorem 1.3]{BloemendalVirag2} showed that
$2\hat H$'s spectrum also describes the limit fluctuations
(as $n,p\to\infty$) of the top eigenvalues of $p$-variate real, complex, or quaternion
(again depending on
whether $\be=1,2,$ or $4$) Wishart matrices
with sample size $n$ and a rank-$r$ spiked covariance of the form $W_{n,p}\oplus \mr{Id}_{p-r}$,
where
\begin{align}
\label{Equation: Critical Spike 2}
W_{n,p}\approx \mr{Id}_r+\sqrt{\frac{p}{n}} \left(\mr{Id}_r-\left(\frac{1}{\sqrt{n}}+\frac{1}{\sqrt{p}}\right)^{2/3}W\right).
\end{align}
\end{enumerate}

Since the celebrated works of Baik, Ben Arous, and P\'ech\'e
\cite{BaikBenArousPeche,Peche},
the extreme eigenvalue fluctuations of the above random matrix models
have been known to undergo a phase transition depending on the size of the perturbations
$W_n$ or $W_{n,p}$:
\begin{enumerate}[$\bullet$]
\item In the {\it subcritical} regime (i.e., $W_n<\sqrt{n}\mr{Id}_r$ or $W_{n,p}<(1+\sqrt{p/n})\mr{Id}_n$), the perturbations are so small that
the top eigenvalues behave as though $W_n=W_{n,p}=0$. In particular, the limit
fluctuations follow the Tracy-Widom distributions (equivalently, the joint eigenvalue distribution
of any multivariate SAO with only Dirichlet boundary conditions; i.e., $\bar\al_i=-\infty$ for all $1\leq i\leq r$).
\item In the {\it supercritical} regime (i.e., $W_n>\sqrt{n}\mr{Id}_r$ or $W_{n,p}>(1+\sqrt{p/n})\mr{Id}_n$), the perturbations are so large that
the top eigenvalues separate from the edge of the semicircle/Marchenko-Pastur
distributions, and thus become outliers.
\end{enumerate}
The perturbation sizes described in \eqref{Equation: Critical Spike 1}
and \eqref{Equation: Critical Spike 2} interpolate between these two settings,
and thus allow to describe {\it critical} perturbations. More specifically, the non-Dirichlet
boundary conditions in $2\hat H$ (i.e., the
indices $1\leq i\leq r$ such that $\bar\al_i>-\infty$ in \eqref{Equation: W Matrix Boundary}
and \eqref{Equation: Standard Basis Boundary}) correspond to critical spikes,
while the Dirichlet boundary conditions (i.e., $\bar\al_i=-\infty$) correspond to subcritical spikes.
In particular, the results of \cite{BloemendalVirag2} established the existence of a 
continuous family of deformations of the classical Tracy-Widom laws, which
are parametrized by the non-Dirichlet components of the boundary
condition \eqref{Equation: Standard Basis Boundary}.
We also refer to \cite{BaikBenArousPeche,Mo,Peche,Wang} for partial results in this direction,
and more generally to the introductions of \cite{BloemendalVirag,BloemendalVirag2}
for detailed expositions of these results
and their significance in the theories of random matrices and covariance estimation.
In this view, one of our main goals in this paper is to develop a new set of tools to study
the properties of limit laws that arise in critical spiked models.

\begin{remark}
The boundary conditions on the multivariate SAO used in \cite{BloemendalVirag2} are slightly different from those that we stated in \eqref{Equation: Standard Basis Boundary}; namely: 
\begin{align}
\label{Equation: Matrix Boundary Conditions}
\begin{cases}
u_i^\dagger\big(f'(0)-w_if(0)\big)=0&\text{if }w_i\in\mbb R,\\
u_i^\dagger f(0)=0&\text{if }w_i=\infty,
\end{cases}
\end{align}
for some $w_i\in(-\infty,\infty]$. The differences in weights
(i.e., $\bar\al_i=-w_i$) is merely a matter of
convention; we use $\bar\al_i$ to get a slightly neater
statement of the Feynman-Kac formula. In contrast to this, at first glance the
presence of $u_i^\dagger$ in \eqref{Equation: Matrix Boundary Conditions} appears
genuinely different from \eqref{Equation: Standard Basis Boundary}, as the boundary conditions
are imposed in the directions of the $u_i$ rather than the standard basis.

That said, we note that the potential $V(i,x)=\frac{rx}{2}$ does not
depend on $i$ (i.e., $V$ is a multiple of the identity matrix), and that $\frac1{\sqrt2}\xi$ in the case of the SAO is the derivative of a matrix Brownian motion
with GOE/GUE/GSE increments. Thus, since the GOE/GUE/GSE are invariant (in distribution) with respect to conjugation by the unitary matrix
with columns $u_i$, the eigenvalues of the SAO with boundary conditions \eqref{Equation: Matrix Boundary Conditions} have the same joint distribution
as the eigenvalues of the SAO with boundary conditions \eqref{Equation: Standard Basis Boundary},
provided we take $\bar\al_i=-w_i$. Given that
we are only interested in the eigenvalues of $\hat H$, in this paper we assume (without loss of generality)
that the SAO's boundary condition are diagonal (i.e., of the form \eqref{Equation: Standard Basis Boundary}).

\end{remark}

\subsubsection{Application: Number Rigidity}

A point process is number rigid if for a certain class of Borel sets $\mc K$,
the number of points inside $\mc K$ is uniquely determined by the configuration of points outside of $\mc K$.
(A common choice is to impose that $\mc K$ be bounded; in our paper we allow sets that are
merely bounded from above, as per Definition \ref{Definition: Number Rigidity}.)
Number rigidity was first introduced by Ghosh and Peres in \cite{GhoshPeres},
following-up on a number of studies of related notions of rigidity of point processes
(e.g., insertion and deletion singularity; notably \cite{HolroydSoo}).
Throughout the past decade, an extensive literature on the notion was developed,
wherein number rigidity has been proved (or disproved) for many point processes
using a variety of techniques.

Looking more specifically at the point processes that arise
as scaling limits of random matrices and/or the spectrum of random Schr\"odinger operators:
Number rigidity was proved for the Ginibre process in
\cite{GhoshPeres} and for the Sine$_{\be}$ processes in \cite{ChhaibiNajnudel,DHLM,Ghosh}.
In the scalar case $r=1$, the rigidity of the SAO
was proved in \cite{Bufetov,RigiditySAO}; more general one-dimensional random Schr\"odinger
operators (continuous and discrete) were considered in \cite{Prev2,Prev1}.

As an application of the main results of this paper, we prove in {\bf Theorem \ref{Theorem: Rigidity}}
that if $V(i,x)\geq \ka|x|-\nu$ for some $\ka,\nu>0$
(this is trivially true when $I$ is bounded), then $\hat H$'s eigenvalue point process is number rigid.
This of course includes the multivariate SAO, as in this case $x\in(0,\infty)$ and $V(i,x)=\frac{rx}{2}$.
Consequently, once combined with \cite{Bufetov,RigiditySAO},
Theorem \ref{Theorem: Rigidity} completely settles the question of number rigidity
in the point processes that describe the asymptotic soft-edge fluctuations of $\be$-ensembles
and their critical rank-$r$ spiked versions. More broadly, when combined with
\cite{Prev2,Prev1}, the present work shows that rigidity phenomena in random Schr\"odinger
operator eigenvalues also arise in general situations with matrix white noise.

Theorem \ref{Theorem: Rigidity} follows from
the standard methodology to prove number rigidity introduced in
\cite[Theorem 6.1]{GhoshPeres} and \cite[Proposition 7.1 and Lemma 7.2]{HolroydSoo}:
Control the fluctuations of sequences of linear statistics that converge to
1 uniformly on compact sets.
More specifically, following previous
works on the number rigidity of random Schr\"odinger
operators---i.e, \cite{Prev2,Prev1}, and most notably \cite{RigiditySAO}---we
prove rigidity using the following sufficient condition (see \cite[Section 4.1]{RigiditySAO} for a proof):

\begin{proposition}
\label{Proposition: Rigidity Sufficient Condition}
Suppose that
\begin{align}
\label{Equation: Rigidity Sufficient Condition 1}
\limsup_{t\to0}\mbf{Var}\Big[\mr{Tr}\big[\mr e^{-t\hat H}\big]\Big]<\infty
\end{align}
and that for every fixed $t_1\in(0,1]$, one has
\begin{align}
\label{Equation: Rigidity Sufficient Condition 2}
\lim_{t_2\to0}\mbf{Cov}\Big[\mr{Tr}\big[\mr e^{-t_1\hat H}\big],\mr{Tr}\big[\mr e^{-t_2\hat H}\big]\Big]=0.
\end{align}
Then, $\hat H$'s spectrum is number rigid.
\end{proposition}

The proof of \eqref{Equation: Rigidity Sufficient Condition 1} and \eqref{Equation: Rigidity Sufficient Condition 2}
under the hypothesis $V(i,x)\geq \ka|x|-\nu$
relies on a direct application of our Feynman-Kac formulas for \eqref{Equation: Mixed Moments}.

\subsubsection{Future Applications}

The main technical difficulties encountered in this paper
that are directly related to the generality of our setting
come from the cases where $I$ has a boundary and $\mbb F=\mbb H$
(see Remark \ref{Remark: From Scratch} for more details).
Both of these features are present in the SAOs. In particular, there is no additional
technical challenge that comes from considering generic one-dimensional operators,
which explains the generality of our setting beyond
just the multivariate SAO.
In this context, one potential source of future applications for our results
lies in the many recent advances in the study of random Schr\"odinger operators
with white noise potentials (one manifestation of the Anderson model); e.g., \cite{AllezChouk,ChoukvanZuijlen,DumazLabbe,DumazLabbe2,DumazLabbe3,DumazLabbe4,HsuLabbe,Labbe,MatsudavanZuijlen,Matsuda}. In these works, the potential energy and the functions on which
the operators act are assumed to be scalar-valued (i.e., $r=1$). The generality of our setting could
be used to begin exploring extensions of this theory to the vector-valued setting (i.e., $r>1$).

Going beyond spectral theory, the techniques developed in this paper could potentially be
used to study the moments of finite systems of stochastic PDEs on $I$
coupled by a matrix white noise. More specifically, if we define the function
$u(t;i,x)=\mr e^{-t\hat H}f(i,x)$, for a fixed $f:I\to\mbb F^r$, then $u$
formally solves the system of stochastic PDEs
\begin{align}
\label{Equation: SPDE}
\partial_t u(t;i,\cdot)=\frac12\partial_x^2u(t;i,\cdot)-V(i,\cdot)u(t;i,\cdot)-\sum_{j=1}^r\xi_{i,j}u(t;j,\cdot),\qquad 1\leq i\leq r,
\end{align}
with initial conditions $u(t;i,\cdot)=f(i,\cdot)$ and a mix of boundary conditions of the form
$\partial_xu(t;i,c)+\al u(t;i,c)=0$ or $u(t;i,c)=0$ for $c\in\partial I$.
When $r=1$ and $V=0$, \eqref{Equation: SPDE} reduces to the Stratonovich parabolic Anderson model
with time-independent white noise; e.g., \cite{ChenBook,Chen14,HuHuangNualartTindel,Mansmann}.
While the results of this paper only concern the case where
$f(i,\cdot)=\de_x(\cdot)$ (as we study the diagonal
of $\mr e^{-t\hat H}$'s integral kernel), treating general $f$'s would
amount to a very minor extension of our results.

Finally, we expect that our technology
could be used to obtain various asymptotic expansions of the moments
of $\mr{Tr}\big[\mr e^{-t\hat H}\big]$ or $\mr e^{-t\hat H}f(i,x)$ for some function $f$, either as $t\to\infty$ or $t\to0^+$. Indeed, despite the complexity
of the moment formulas obtained in Theorem \ref{Theorem: Trace Moment Formulas}, asymptotic expansions typically only involve a few dominant terms,
whose combinatorics should be tractable. In particular, in a forthcoming paper
\cite{GLPrep}, we use Theorem \ref{Theorem: Trace Moment Formulas}
to calculate the leading terms in the small-time (i.e., $t\to0^+$) asymptotic expansion of $\mr{Tr}\big[\mr e^{-t\mr{SAO}/2}\big]$; given the connection
between SAO spectra and critically-spiked matrix models
(see Section \ref{Section: mSAO}), this result will lead to new applications related to
the recovery of critical signals. A natural next step would be to investigate
the intermittency properties of \eqref{Equation: SPDE} using large-time
(i.e., $t\to\infty$) asymptotics of moment formulas similar to those in
Theorem \ref{Theorem: Trace Moment Formulas}.

\subsection{Proof Method and Past Results}
\label{Section: Proof Methods and Past Results}

In this section, we provide a high-level explanation of the novel difficulties that arise in
the proof of our main result, as well as the method that we use to solve them. We
also take this opportunity to provide references to past results similar to
our method of proof.

\begin{remark}
In order to simplify
the exposition, we only consider the case where $I=\mbb F=\mbb R$ in this outline, with the
understanding that our results do apply to more general domains and fields/rings. In fact, as we highlight
in Remark \ref{Remark: From Scratch}, the consideration of domains with boundaries
and quaternion-valued functions (both of which are mainly motivated
by the SAO) induces significant technical challenges.
\end{remark}

\subsubsection{The Classical Feynman-Kac Formula and its Vector-Valued Extension}

Consider a generic Schr\"odinger operator that acts on scalar-valued functions (i.e., $r=1$):
\begin{align}
\label{Equation: 1D Schrodinger on R}
\mc Hf(x)=-\tfrac12f''(x)+\mc V(x)f(x),\qquad f:\mbb R\to\mbb R.
\end{align}
In this classical setting, the Feynman-Kac formula has a well-known statement:

\begin{theorem}
\label{Theorem: Scalar F-K on R}
Let $\mc H$ be as in \eqref{Equation: 1D Schrodinger on R}.
If $\mc V$ is sufficiently well-behaved (see, e.g., \cite[Theorem 4.9]{Sznitman}
for a precise statement), then for every $t>0$, $\mr e^{-t\mc H}$ has the kernel
\[\mr e^{-t\mc H}(x,y)=\Pi_B(t;x,y)\mbf E\left[\mr e^{-\int_0^t\mc V(B^{x,y}_t(s))\d s}\right],\qquad t\geq0,~x,y\in\mbb R,\]
where $B$ is a standard Brownian motion on $\mbb R$,
\[\Pi_B(t;x,y)=\mbf P\big[B(t)\in\dd y~\big|~B(0)=x\big]\] is $B$'s transition kernel,
and $B^{x,y}_t$ denotes $B$ conditioned on
$\big\{B(0)=x\text{ and }B(t)=y\big\}$.
\end{theorem}

Consider now a generic vector-valued operator of the form
\begin{align}
\label{Equation: Vector-Valued Schrodinger on R}
\mc H f(i,x)=-\tfrac12f''(i,x)+\sum_{j=1}^r\mc V_{i,j}(x)f(j,x),\qquad f:\mbb R\to\mbb R^r,
\end{align}
where $\mc V:\mbb R\to\mbb R^{r\times r}$ takes values in self-adjoint $r\times r$ matrices.
In order to state the Feynman-Kac formula in this setting, we need to introduce the vector-valued
extension of the exponential integral $\mr e^{-\int_0^t\mc V(B(s))\d s}$, which leads to the consideration of ordered exponentials
(OEs):

\begin{definition}
Let $\big(F(t)\big)_{t\geq0}$ be a $r\times r$ matrix-valued function.
We define the forward OE of $F(t)$, denoted
$\left(\fo{\mc T}\left\{\mr e^{\int_0^tF(s)\d s}\right\}\right)_{t\geq0}$, as the $r\times r$
matrix-valued function that solves the differential equation
(assuming the solution exists and is unique)
\begin{align}
\label{Equation: Ordered Exponential Differential Equation}
\tfrac{\dd }{\dd t}\fo{\mc T}\left\{\mr e^{\int_0^tF(s)\d s}\right\}&=\fo{\mc T}\left\{\mr e^{\int_0^tF(s)\d s}\right\}F(t),
\qquad \fo{\mc T}\left\{\mr e^{\int_0^0F(s)\dd s}\right\}=\msf{Id}_r,
\end{align}
recalling that $\msf{Id}_r$ denotes the identity matrix.
\end{definition}

OEs are known in the literature under many different names
(e.g., path-ordered or time-ordered exponentials, product integrals, Volterra integrals),
which is a testament to their appearance in a wide variety of mathematical
problems. See \cite{DollardFriedman,GillJohansen,Slavik} for
a comprehensive survey of the theory of OEs and their applications.
Our interest in OEs lies in their appearance in the vector-valued extension of the classical Feynman-Kac
formula in Theorem \ref{Theorem: Scalar F-K on R}, as per the following result of G\"uneysu:

\begin{theorem}[{\cite[Theorem 1.8]{Guneysu}}]
\label{Theorem: Vector-Valued F-K on R}
Let $\mc H$ be as in \eqref{Equation: Vector-Valued Schrodinger on R}.
If $\mc V$ is sufficiently well-behaved (\cite[Definition 1.3]{Guneysu}), then for every $t>0$, $\mr e^{-t\mc H}$ has the ($r\times r$ matrix) kernel
\begin{align}
\label{Equation: Vector-Valued F-K on R}
\mr e^{-t\mc H}(x,y)=\Pi_B(t;x,y)\mbf E\left[\fo{\mc T}\left\{\mr e^{-\int_0^t\mc V(B^{x,y}_t(s))\d s}\right\}\right],\qquad t\geq0,~x,y\in\mbb R.
\end{align}
\end{theorem}

In this context, the problem that we contend with in this paper is as follows:

\begin{problem}
\label{Problem: Main}
Understand \eqref{Equation: Vector-Valued F-K on R}
when $\mc H=\hat H$ (i.e., $\mc V=V+\xi$).
\end{problem}

\subsubsection{Difficulty: Interpreting Feynman-Kac with a Matrix White Noise}

The main difficulty that we encounter in attacking Problem \ref{Problem: Main} is that $\xi$ is a Schwartz distribution,
which fails to satisfy the standard regularity assumptions required for the Feynman-Kac formulas stated in Theorems
\ref{Theorem: Scalar F-K on R} and \ref{Theorem: Vector-Valued F-K on R}. In particular,
since $\xi$ is not defined pointwise, it is not immediately obvious what the
expression $\fo{\mc T}\left\{\mr e^{-\int_0^t\xi(B(s))\d s}\right\}$ should mean from the point of view of the matrix ODE
in \eqref{Equation: Ordered Exponential Differential Equation}.

In the scalar case $r=1$, a rigorous interpretation of the exponential integral $\mr e^{-\int_0^t\xi(B(s))\d s}$ was provided
for the SAO in \cite{GaudreauLamarreShkolnikov,GorinShkolnikov}---with a later extension to more general operators in
\cite{GaudreauLamarreEJP}---as follows:
\begin{definition}
For any Borel set $\mc K\subset [0,\infty)$, we let $\big(L^x_{\mc K}(B)\big)_{x\in\mbb R}$ denote the continuous version
of the local time of $B$ on $\mc K$, that is,
\begin{align}
\label{Equation: Interior Local Time}
\int_{\mc K}f\big(B(s)\big)\d s=\int_{\mbb R} L^x_{\mc K}(B)f(x)\d x\qquad\text{for every measurable }f:\mbb R\to\mbb R.
\end{align}
In the case where $\mc K=[0,t)$, we use the convention $L^x_{[0,t)}(B)=L^x_t(B)$.
\end{definition}
If $W$ is the Brownian motion such that $\xi=W'$,
then a formal application of \eqref{Equation: Interior Local Time}
suggests that $\mr e^{-\int_0^t\xi(B(s))\d s}$ can be defined rigorously via
\begin{align}
\label{Equation: Difficulty when r=1}
\int_0^t\xi\big(B(s)\big)\d s=\int_0^tW'\big(B(s)\big)\d s=\int_{\mbb R} L_t^x(B)W'(x)\d x=\int_{\mbb R} L_t^x(B)\d W(x),
\end{align}
where the rightmost expression in \eqref{Equation: Difficulty when r=1} is interpreted as a stochastic integral.
However, \eqref{Equation: Difficulty when r=1} cannot work when $r>1$ due to the non-commutativity
of the matrix function $s\mapsto\xi\big(B(s)\big)$. To see this, we recall the series/limit representations
commonly used to calculate or approximate OEs
(see, e.g., \cite[Definition 2.4.1 and Theorems 2.4.3, 2.4.12, and 2.5.1]{Slavik}
for details on the assumptions on $F$ required for \eqref{Equation: Ordered Exponential Product}
and \eqref{Equation: Ordered Exponential Series} to hold):

\begin{theorem}[Informal]
Equivalent definitions of OEs include:
\begin{enumerate}[$\bullet$]
\item (Product Integral.)
For every $t>0$, one has
\begin{multline}
\label{Equation: Ordered Exponential Product}
\fo{\mc T}\left\{\mr e^{\int_0^tF(s)\d s}\right\}=\lim_{n\to\infty}\mr e^{\frac tnF(0)}\mr e^{\frac tnF(t/n)}\mr e^{\frac tnF(2t/n)}\cdots\mr e^{\frac tnF(t)}\\
=\lim_{n\to\infty}\big(\msf{Id}_r+\tfrac tnF(0)\big)
\big(\msf{Id}_r+\tfrac tnF(t/n)\big)
\big(\msf{Id}_r+\tfrac tnF(2t/n)\big)
\cdots
\big(\msf{Id}_r+\tfrac tnF(t)\big).
\end{multline}
\item (Dyson Series.)
For every $t>0$, one has
\begin{align}
\label{Equation: Ordered Exponential Series}
\fo{\mc T}\left\{\mr e^{\int_0^tF(s)\d s}\right\}&=\mr{Id}_r+\sum_{n=1}^\infty\int_0^t\cdots\int_0^{s_3}\int_0^{s_2} F(s_1)F(s_2)\cdots F(s_n)\d s_1\dd s_2\cdots\dd s_n.
\end{align}
\end{enumerate}
\end{theorem}

With this in hand, if we formally apply (e.g.) \eqref{Equation: Ordered Exponential Product} to write
\begin{align}
\label{Equation: Singular OE}
\fo{\mc T}\left\{\mr e^{-\int_0^t\xi(B(s))\dd s}\right\}=\lim_{n\to\infty}\mr e^{-\frac tn\xi(B(0))}\mr e^{-\frac tn\xi(B(t/n))}\mr e^{-\frac tn\xi(B(2t/n))}\cdots\mr e^{-\frac tn\xi(B(t))},
\end{align}
then it becomes clear that we cannot calculate this quantity if we only know
how often $B$'s path
visits each coordinate $x\in\mbb R$ (i.e., the local time $L^x_t(B)$):
The noncommutativity of $\xi\big(B(kt/n)\big)$ for different $k$'s means that we must know the order
in which each coordinate is visited by $B$ (i.e., the entire history of the path
$\big(B(s):0\leq s\leq t\big)$).

Finally, in the absence of an obvious candidate
for a pathwise interpretation of \eqref{Equation: Singular OE},
a natural strategy to define such an object is to use smooth approximations.
That is, introduce a sequence of smooth matrix noises $(\xi^\eps)_{\eps>0}$
such that $\xi^\eps\to\xi$ as $\eps\to0$ in the space of Schwartz distributions,
and then try to interpret the limit of the corresponding Feynman-Kac formulas.
However, the standard approximation theory of OEs
is not powerful enough to carry this out. 
To illustrate this,
one such standard estimate is as follows (see, e.g., \cite[Corollary 3.4.3]{Slavik}):
\begin{multline*}
\left\|\fo{\mc T}\left\{\mr e^{\int_0^t\mc V(B(s))\d s}\right\}-\fo{\mc T}\left\{\mr e^{\int_0^t\bar {\mc V}(B(s))\d s}\right\}\right\|_{\mr{op}}
\leq\mr e^{\int_0^t\|\mc V(B(s))\|_{\mr{op}}\d s}\left(\mr e^{\int_0^t\|\mc V(B(s))-\bar {\mc V}(B(s))\|_{\mr{op}}\d s}-1\right),
\end{multline*}
where we denote the operator norm by
\[\|F\|_{\mr{op}}=\sup_{v\in\mbb F^r}\frac{\|Fv\|_{\ell^2}}{\|v\|_{\ell^2}}.\]
For instance, a variation of this is used in \cite{Guneysu} to prove
Theorem \ref{Theorem: Vector-Valued F-K on R}
for certain locally-integrable potentials $\mc V$,
starting from the easier assumption that $\bar{\mc V}$
is continuous and bounded.
However, this strategy fails for white noise since $\|\xi(x)\|_{\mr{op}}=\infty$.

\subsubsection{Mixed Moments via a Probabilistic Representation of Ordered Exponentials}
\label{Section: Probabilistic OE}

The key to our ability to compute the joint moments \eqref{Equation: Mixed Moments} is to use
a method
to represent ordered exponentials that is probabilistic in nature. A simple
version of this representation is as follows:

\begin{definition}
\label{Definition: Uniform RW}
Let $\big(U(t)\big)_{t\geq0}$ be a uniform continuous-time random walk on
the set $\{1,\ldots,r\}$
with jump rate $(r-1)$. That is, if $N(t)$ is a Poisson process with rate $(r-1)$
and $M$ is a discrete Markov chain on $\{1,\ldots,r\}$ with transition matrix
\[P_{i,j}=\begin{cases}
0&\text{if }i=j,\\
\frac1{r-1}&\text{if }i\neq j,
\end{cases}\]
then we can construct $U$ as
\[U(t)=M\big(N(t)\big),\qquad t\geq0,\] assuming $M$ and $N$ are independent.
Next, we denote the jump times of the Poisson process $N$ as
$0=\tau_0\leq\tau_1\leq\tau_2\leq\cdots$, and the corresponding jumps in $U$ and $M$ as
\[J_k=\big(U(\tau_{k-1}),U(\tau_k)\big)=\big(M(k-1),M(k)\big)\qquad\text{for }k=1,2,3,\ldots.\]
Finally, we use $\Pi_U$ to denote $U$'s transition kernel,
and $U^{i,j}_t$ denotes the process $U$
conditioned on the event $\big\{U(0)=i\text{ and }U(t)=j\big\}$
(noting that this also induces a conditioning on $N,\tau_k$, and $J_k$).
\end{definition}

\begin{proposition}[Informal]
\label{Proposition: Probabilistic OE}
Let $\big(F(t)\big)_{t\geq0}$ be a continuous function that takes values in the space
of $r\times r$ self-adjoint matrices.
Define the process
\[\mf o_t(F,U)=
\begin{cases}
\mr e^{(r-1)t}&\text{if }N(t)=0\\
\displaystyle\mr e^{(r-1)t}\prod_{k=1}^{N(t)}F_{J_k}(\tau_k)&\text{otherwise}.
\end{cases}\]
For every $t\geq0$ and $1\leq i,j\leq r$, one has
\begin{align}
\label{Equation: Probabilistic OE}
\fo{\mc T}\left\{\mr e^{\int_0^tF(s)\d s}\right\}_{i,j}=\Pi_{U}(t;i,j)\mbf E\left[\mf o_t(F,U^{i,j}_t)\exp\left(\int_0^tF_{U^{i,j}_t(s),U^{i,j}_t(s)}(s)\d s\right)\right].
\end{align}
\end{proposition}

We postpone a discussion of Proposition \ref{Proposition: Probabilistic OE}'s appearance in past
results to Section \ref{Section: Past Results}, and focus for now on
explaining how it is used in the calculation of
the joint moments \eqref{Equation: Mixed Moments}:
If we combine \eqref{Equation: Probabilistic OE} with the
Feynman-Kac formula stated
in \eqref{Equation: Vector-Valued F-K on R} (without worrying about
whether the hypotheses of the two results match for now), then we obtain the
following informal restatement of Theorem \ref{Theorem: Vector-Valued F-K on R}:

\begin{theorem}[Informal]
\label{Theorem: Informal Vector-Valued Schrodinger on R}
If $\mc Hf(i,\cdot)=-\tfrac12f''(i,\cdot)+\sum_j\mc V_{i,j}f(j,\cdot)$ is as in \eqref{Equation: Vector-Valued Schrodinger on R},
then
\begin{multline}
\label{Equation: Informal Vector-Valued Kernel}
\mr e^{-t\mc H}(x,y)_{i,j}
=\Pi_B(t;x,y)\Pi_{U}(t;i,j)\\
\cdot\mbf E\left[\mf o_t(-\mc V\circ B^{x,y}_t,U^{i,j}_t)\exp\left(-\int_0^t\mc V_{U^{i,j}_t(s),U^{i,j}_t(s)}\big(B^{x,y}_t(s)\big)\d s\right)\right]
\end{multline}
for any $t\geq0$, $x,y\in\mbb R^d$, and $1\leq i,j\leq r$;
assuming that $B$ and $U$ are independent,
and where $f\circ g=f\big(g(t)\big)$ denotes function composition.
\end{theorem}

On the one hand, if $\xi$ is a matrix white noise, then the term
\[\int_0^t\xi_{U(s),U(s)}\big(B(s)\big)\d s\]
(which would be contained in the exponential in \eqref{Equation: Informal Vector-Valued Kernel} if
$\mc V$ contains $\xi$)
admits a straightforward pathwise interpretation similar to \eqref{Equation: Difficulty when r=1}:
Recalling the definition of local time in \eqref{Equation: Interior Local Time}
and letting $\xi_{i,i}=W_{i,i}'$ for some i.i.d. Brownian motions $(W_{i,i})_{1\leq i\leq r}$,
if we denote the sets $\mc K_t(i)=\{s\in[0,t):U(s)=i\}$ for all $t\geq0$ and $1\leq i\leq r$, then
\begin{align}
\label{Equation: Diagonal Noise Pathwise Interpretation}
\int_0^t\xi_{U(s),U(s)}\big(B(s)\big)\d s
=\sum_{i=1}^r\int_{\mc K_t(i)}\xi_{i,i}\big(B(s)\big)\d s
=\sum_{i=1}^r\int_{\mbb R}L_{\mc K_t(i)}^{x}(B)\d W_{i,i}(x).
\end{align}
Thus, the only difficulty in interpreting Theorem \ref{Theorem: Informal Vector-Valued Schrodinger on R}
when the potential $\mc V$ contains a matrix white noise comes from the product term
\begin{align}
\label{Equation: Difficulty when r>1}
\prod_{k=1}^{N(t)}\xi_{J_k}\big(B(\tau_k)\big),
\end{align}
which appears in the process $\mf o_t(-\mc V\circ B,U)$ if $\mc V=V+\xi$.
Indeed, since $\eqref{Equation: Difficulty when r>1}$
does not involve any kind of integral (either with respect to space or time), it is not clear that we can exploit
the interpretation of $\int f(x)\xi(x)\d x$ as a stochastic integral to make sense of
this product.

In this context, one of the main insights of this paper---which leads to our ability to calculate
the trace moments in \eqref{Equation: Mixed Moments}---is
that the expression \eqref{Equation: Difficulty when r>1}
admits a fairly tractable rigorous interpretation when we take an expectation
with respect to $\xi$. More specifically, if $\xi^\eps$ is
a smooth matrix noise that approximates $\xi$, then
we can use Isserlis' theorem (e.g., \cite[(1.7) and (1.8)]{MingoSpeicher}; see \cite{Isserlis} for Isserlis' original paper) to calculate
\begin{align}
\label{Equation: Difficulty when r>1 2}
\mbf E_{\xi^\eps}\left[\prod_{k=1}^{N(t)}\xi^\eps_{J_k}\big(B(\tau_k)\big)\right],
\end{align}
where $\mbf E_{\xi^\eps}$ denotes the expectation with respect to $\xi^\eps$ only,
conditional on $B$ and $U$ (which we assume are both independent of $\xi^\eps$):

\begin{definition}
\label{Definition: Pn}
Let $\mc P_0=\varnothing$ be the empty set, and
given an even integer $n\geq2$, let $\mc P_n$ denote the set of perfect pair matchings
of $\{1,\ldots,n\}$. That is, $\mc P_n$ is the set
\[\Big\{\big\{\{\ell_1,\ell_2\},\{\ell_3,\ell_4\},\ldots,\{\ell_{n-1},\ell_n\}\big\}:
\ell_i\in\{1,\ldots,n\}\text{ and }
\{\ell_i,\ell_{i+1}\}\cap\{\ell_j,\ell_{j+1}\}=\varnothing\text{ if }i\neq j\Big\}.\]
\end{definition}

Following the statement of Isserlis' theorem, \eqref{Equation: Difficulty when r>1 2}
can be simplified to
\begin{align}
\label{Equation: Difficulty when r>1 3}
\mbf 1_{\{N(t)\text{ is even}\}}
\sum_{p\in\mc P_{N(t)}}\prod_{\{\ell_1,\ell_2\}\in p}\mbf E\left[\xi^\eps_{J_{\ell_1}}\big(B(\tau_{\ell_1})\big)\xi^\eps_{J_{\ell_2}}\big(B(\tau_{\ell_2})\big)\right],
\end{align}
assuming that a sum over $p\in\mc P_0$ is equal to 1 by convention.
Informally, the correlation of the Gaussian white noise is given by a multiple of the
Dirac mass kernel; $\mbf E\big[\xi_{i,j}(x)\xi_{i,j}(y)^*\big]=c\de_0(x-y)$ for some $c>0$.
Thus, when $\eps\to0$, the covariance \eqref{Equation: Difficulty when r>1 3} will approximate
the expression
\begin{align}
\label{Equation: Difficulty when r>1 4}
\mbf 1_{\{N(t)\text{ is even}\}}
\sum_{p\in\mc P_{N(t)}}\prod_{\{\ell_1,\ell_2\}\in p}\mathbf 1_{\mf J(J_{\ell_1},J_{\ell_2})}\cdot c\de_0\big(B(\tau_{\ell_1})-B(\tau_{\ell_2})\big),
\end{align}
where the indicator $\mbf 1_{\mf J(J_{\ell_1},J_{\ell_2})}$
essentially ensures that the jumps $J_{\ell_1}$ and $J_{\ell_2}$ are between the same elements in $\{1,\ldots,r\}$
(the situation is a bit more complicated when $\mbb F=\mbb C,\mbb H$)---otherwise the covariance is zero by the independence
of the noises $(\xi_{i,j})_{1\leq i<j\leq r}$.
While \eqref{Equation: Difficulty when r>1 4} is ill-defined on its own,
we recall that in Theorem \ref{Theorem: Informal Vector-Valued Schrodinger on R},
the product \eqref{Equation: Difficulty when r>1} sits inside an integral/expectation with respect to the laws of $B$ and $U$.
In this context, we can interpret each nonvanishing summand in \eqref{Equation: Difficulty when r>1 4} as inducing a conditioning
on the probability-zero event
\begin{align}
\label{Equation: Probability Zero Events}
\big\{B(\tau_{\ell_1})=B(\tau_{\ell_2})\text{ for every }\{\ell_1,\ell_2\}\in p\big\}.
\end{align}
In rigorous terms, this can be defined by resampling the pairs of jump times $\{\tau_{\ell_1},\tau_{\ell_2}\}$
according to $B$'s self-intersection local time measure instead of the usual Poisson process measure;
see Definitions \ref{Definition: Regular Local Time} and \ref{Definition: Singular Process}
and Remark \ref{Remark: Self Intersection Uniform} for the
details.

In summary: While the approach developed in this paper does not
exhibit a pathwise interpretation of the Feynman-Kac
formula for vector-valued operators with matrix white noise
(see Section \ref{Section: Pathwise} for a discussion of this problem), Proposition
\ref{Proposition: Probabilistic OE}
nevertheless leads to a probabilistic interpretation
of any observable that relies
on the computation of the joint moments
\begin{align}
\label{Equation: Joint WN Kernel Moments}
\mbf E\big[\mr e^{-t_1\hat H}(x_1,y_1)_{i_1,j_1}\cdots \mr e^{-t_n\hat H}(x_n,y_n)_{i_n,j_n}\big].
\end{align}
The trace moments in \eqref{Equation: Mixed Moments} are one notable example of such observables,
since
\[\mr{Tr}\big[\mr e^{-t\hat H}\big]=\int_\mbb R\mr{Tr}\big[\mr e^{-t\hat H}(x,x)\big]\d x=\sum_{i=1}^r\int_\mbb R\mr e^{-t\hat H}(x,x)_{i,i}\d x.\]
Thus, the proof of our main result, Theorem \ref{Theorem: Trace Moment Formulas}, involves
the following two steps:
\begin{enumerate}[$\bullet$]
\item In {\bf Theorem \ref{Theorem: Regular FK}}, we prove a rigorous version of
Theorem \ref{Theorem: Informal Vector-Valued Schrodinger on R} that holds for
operators of the form $\hat H=-\frac12\De+V+\xi$, where
$\xi$ is a regular matrix noise (i.e., the components $\xi_{i,j}$ are Gaussian processes
with continuous sample paths).
\item Consider the operators $\hat H=-\frac12\De+V+\xi$ and $\hat H^\eps=-\frac12\De+V+\xi^\eps$
for $\eps>0$, where $\xi$ is a matrix white noise and $\xi^\eps$ are smooth approximations
such that $\xi^\eps\to\xi$ as $\eps\to0$. With Theorem \ref{Theorem: Regular FK},
we can compute the mixed moments
\begin{align}
\label{Equation: Approximate Mixed Moments}
\mbf E\left[\mr{Tr}\big[\mr e^{-t_1\hat H^\eps}\big]\mr{Tr}\big[\mr e^{-t_2\hat H^\eps}\big]\cdots\mr{Tr}\big[\mr e^{-t_n\hat H^\eps}\big]\right],\qquad
n\in\mbb N,~t_1,\ldots,t_n\geq0.
\end{align}
We then obtain the statement of {\bf Theorem \ref{Theorem: Trace Moment Formulas}}
by computing the $\eps\to0$ limit of \eqref{Equation: Approximate Mixed Moments}
using the process outlined in \eqref{Equation: Difficulty when r>1 2}--\eqref{Equation: Probability Zero Events},
and then showing that the expectations in \eqref{Equation: Approximate Mixed Moments}
also converge to the mixed moments in \eqref{Equation: Mixed Moments}.
\end{enumerate}

\begin{remark}
\label{Remark: From Scratch}
Contrary to what is suggested by the informal discussion in this section,
the proof of Theorem \ref{Theorem: Regular FK} is not simply
a matter of applying Proposition \ref{Proposition: Probabilistic OE} to already-known
Feynman-Kac formulas, such as Theorem
\ref{Theorem: Informal Vector-Valued Schrodinger on R}.
This is in part because (to the best of our knowledge) the standard semigroup theory
for vector-valued Schr\"odinger operators concerns domains with no boundary
and matrix potentials $\mc V$ that are complex-valued (e.g., \cite{Guneysu}, and the comprehensive survey
\cite{GuneysuBook}).
In contrast, a crucial
feature of Theorem \ref{Theorem: Regular FK} (which is explained by our specific interest in the multivariate SAO)
is that we allow operators that act on domains with
a boundary, and with potential functions that take values in matrices with quaternion entries.
Both of these requirements impose novel difficulties. Thus, some of the work
carried out in this paper---which partly explains its length---consists
of proving the Feynman-Kac formula in Theorem \ref{Theorem: Regular FK} "from scratch."
We point to 
Remark \ref{Remark: Boundary Terms Interpretation} for more details regarding the issue of boundaries, and to Remark \ref{Remark: Quaternion on R 1} for more details on the issue of quaternions.
\end{remark}

\subsubsection{Past Results}
\label{Section: Past Results}

In closing this subsection, we note that the idea of writing
OEs or similar objects using Poisson/Markov jump processes has appeared in various parts
of the mathematics and physics literatures. We begin with two very well-known
special cases of Proposition \ref{Proposition: Probabilistic OE}:

\begin{example}
Arguably the simplest example is
the special case $F(t)=\mc G+q(t)$, where $\mc G$ is the generator of the process $U$
in Definition \ref{Definition: Uniform RW} and $q(t)$ is diagonal
with entries denoted
\[q_{i,i}(t)=q(t,i),\qquad 1\leq i\leq r,~t\geq0.\]
In this case,
\eqref{Equation: Probabilistic OE}
simply reduces to the well-known discrete Feynman-Kac formula:
\[\fo{\mc T}\left\{\mr e^{\int_0^tF(s)\d s}\right\}_{i,j}=\Pi_U(t;i,j)\mbf E\left[\mr e^{\int_0^tq(s,U^{i,j}_t(s))\d s}\right].\]
\end{example}

\begin{example}
If $F(t)$ are the infinitesimal transition matrices
of a (possibly time inhomogeneous) Markov process $M$, then
\eqref{Equation: Ordered Exponential Differential Equation} is the
Kolmogorov forward equation for $M$'s transition semigroup, whereby
\begin{align}
\label{Equation: Change of Measure}
\fo{\mc T}\left\{\mr e^{\int_0^tF(s)\d s}\right\}_{i,j}=\mbf P\big[M(t)=j~\big|~M(0)=i\big]=\Pi_{M}(t;i,j).
\end{align}
In this case, \eqref{Equation: Probabilistic OE} is the known change of measure formula that
expresses $M$'s dynamics in terms of $U$; e.g., \cite[Appendix 1, Propositions 2.6 and 7.3]{KipnisLandim}.
\end{example}

More generally, the idea of using Proposition \ref{Proposition: Probabilistic OE}
in the context of vector-valued Schr\"odinger semigroups (thus producing
a vector-valued Feynman-Kac formula similar to
Theorem \ref{Theorem: Informal Vector-Valued Schrodinger on R}) is also not
new to this work.
For instance, \cite{OE1,OE2} introduced the same idea in the context of Pauli Hamiltonians,
which act on vector-valued functions due to the interaction of the particle's spin with an
external magnetic field. This idea has been extended in the mathematical physics literature
for various applications in the study of quantum systems subjected to magnetic fields, e.g.
\cite{Erdos2,Erdos1}.

Finally, we point out the work \cite{Dalangetal1}, which proves a probabilistic representation
for an impressive class of deterministic and stochastic PDEs (such as heat, wave, and telegraph)
wherein a Poisson process plays a fundamental role. In particular, the product of the
potential terms $V(t-\tau_i,X_{\tau_i})$ in \cite[(3.1)]{Dalangetal1} has striking similarities with
the term $\mf o_t(-\mc V\circ B^{x,y}_t,U^{i,j}_t)$ in \eqref{Equation: Informal Vector-Valued Kernel}.
The result in \cite{Dalangetal1} was later used in \cite{Dalangetal2} to study intermittency properties
of the stochastic wave equation.

\subsection{Pathwise Feynman-Kac Formula}
\label{Section: Pathwise}

As hinted at in the previous section, one notable problem left unsolved in this work is as follows:

\begin{problem}
\label{Problem: Pathwise}
Find a pathwise (with respect to $\xi$) probabilistic representation of
the random semigroup $(\mr e^{-t\hat H}:t>0)$.
\end{problem}

Given that $\xi=W'$ is most straightforward to understand once integrated against a
sufficiently regular test function, this would presumably rely on interpreting
the expectation of \eqref{Equation: Difficulty when r>1} with respect to $N$ and $B$
as an iterated stochastic integral against the joint densities of the points $B(\tau_k)$.
That said, this would also require understanding the distribution of the exponential
\[\exp\left(-\int_0^t\mc V_{U^{i,j}_t(s),U^{i,j}_t(s)}\big(B^{x,y}_t(s)\big)\d s\right)\]
in \eqref{Equation: Informal Vector-Valued Kernel} once we condition on $B(\tau_k)$
for $1\leq k\leq N(t)$.
We thus leave Problem \ref{Problem: Pathwise} to future investigations.

\subsection{Organization}
\label{Section: Organization}

The remainder of this paper is organized as follows:
In Sections \ref{Section: Main Results 1}--\ref{Section: Main Results 3}, we provide the statements of our main results.
More specifically:
\begin{enumerate}[$\bullet$]
\item Section \ref{Section: Main Results 1} contains results related to
the construction of $\hat H$ and its smooth approximations
as self-adjoint operators with compact resolvents, as well as the resulting construction of $\mr e^{-t\hat H}$ by means of a spectral expansion
(i.e., {\bf Proposition \ref{Proposition: Operator Definition}}).
\item In Section \ref{Section: Main Results 2}, we state our results regarding
pathwise Feynman-Kac formulas for $\hat H$'s smooth approximations
(i.e., {\bf Theorem \ref{Theorem: Regular FK}}).
\item In Section \ref{Section: Main Results 3}, we state our two main
results, namely the trace moment formula for \eqref{Equation: Mixed Moments} (i.e., {\bf Theorem \ref{Theorem: Trace Moment Formulas}})
and the number rigidity result (i.e., \bf{Theorem \ref{Theorem: Rigidity}}).
\end{enumerate}
Then, in Sections \ref{Section: Construction of Operators Proof}--\ref{Section: Covariance Estimates}
we prove these four results in the order they are stated.

\subsection{Acknowledgements}

The author gratefully acknowledges L\'aszl\'o Erd\H{o}s for
his generous comments on the content and presentation of a previous iteration
of some of the work contained in this paper (\href{https://arxiv.org/abs/2311.08564}{arXiv:2311.08564}), including pointing out the appearance of a version of
Proposition \ref{Proposition: Probabilistic OE} in \cite{OE1,OE2,Erdos2,Erdos1}.
The author also thanks Le Chen for a very insightful discussion on the connections between the present work and SPDEs,
in particular pointing out the works \cite{Dalangetal2,Dalangetal1}.

The author thanks the anonymous referees for their careful reading of the manuscript (in
particular the errors and typos they pointed out) and
their recommendations to improve the presentation.

\section{Main Results Part 1. Construction of Operators}
\label{Section: Main Results 1}

We now begin the process of providing precise statements for our results.
In this section, we describe the construction of the operator $\hat H$
and its smooth approximations
using quadratic forms. Among other things, this culminates in the ability to define $\mr e^{-t\hat H}$
via a spectral expansion; see \eqref{Equation: Semigroup via Spectral Expansion}.

\subsection{Vector-Valued Function Spaces}

Following-up on Section \ref{Section: Probabilistic OE}
(more specifically Theorem \ref{Theorem: Informal Vector-Valued Schrodinger on R}), we infer that the dynamics in the Feynman-Kac
formula for $\mr e^{-t\hat H}$
are a combination of a Brownian motion and a random walk on $\{1,\ldots,r\}$.
In order to provide a framework that is adapted to this, we henceforth reformulate vector-valued
functions $f:I\to\mbb F^r$ as $f:\{1,\ldots,r\}\times I\to\mbb F$.
This difference is only cosmetic, as both options can be viewed as a direct sum of the one-dimensional functions
$f(i,\cdot):I\to\mbb F$ over $1\leq i\leq r$. However, we settle on the latter to ensure that
the domain of the functions matches the state space of the combined random motions;
this makes for a much neater statement of our Feynman-Kac formulas
than Theorem \ref{Theorem: Informal Vector-Valued Schrodinger on R}, thus
leading us to the following definitions:

\begin{definition}
\label{Definition: Hilbert Spaces}
Let $I\subset\mbb R$ and $f,g:I\to\mbb F$.
We use $\langle\cdot,\cdot\rangle$ and $\|\cdot\|_2$ to denote
\[\langle f,g\rangle=
\displaystyle\int_I f(x)^* g(x)\d x\qquad\text{and}\qquad\|f\|_2=\sqrt{\langle f,f\rangle}.\]
We let $L^2(I,\mbb F)$ denote the associated Hilbert space.

Next, we denote $\mc A=\{1,\ldots,r\}\times I$, and we let $\mu$ denote
the measure on $\mc A$ obtained by taking the product of the counting measure
on $\{1,\ldots, r\}$ and the Lebesgue measure on $I$.
Given $f,g:\mc A\to\mbb F$, we denote
\[\langle f,g\rangle_\mu=\int_{\mc A}f(a)^*g(a)\d\mu(a)=\big\langle f(1,\cdot),g(1,\cdot)\big\rangle+\cdots+\big\langle f(r,\cdot),g(r,\cdot)\big\rangle;
\qquad
\textstyle\|f\|_\mu=\sqrt{\langle f,f\rangle_\mu},\]
and we let $L^2(\mc A,\mbb F)$ denote the associated Hilbert space; i.e.,
$L^2(\mc A,\mbb F)=L^2(I,\mbb F)^{\oplus r}$.

Finally, Let $C_0^\infty(I,\mbb F)$ be the set of $\mbb F$-valued functions that are smooth and
compactly supported on $I$'s closure.
Let $C_0^\infty(\mc A,\mbb F)=C_0^\infty(I,\mbb F)^{\oplus r}$.
\end{definition}

\begin{remark}
\label{Remark: Quaternion on R 1}
When $\mbb F=\mbb R$ or $\mbb C$,
$L^2(I,\mbb F)$ and $L^2(\mc A,\mbb F)$ are bona fide Hilbert spaces on the field $\mbb F$.
However, when $\mbb F=\mbb H$, calling these Hilbert spaces is an abuse of notation:
These spaces are instead
{\it quaternionic Hilbert spaces}, which are actually modules over the ring $\mbb H$;
see,
e.g., \cite[Section 4]{FarenickPidkowich}.
That said, we note that $L^2(I,\mbb H)$ and $L^2(\mc A,\mbb H)$ can be made into
standard Hilbert spaces if we view $\mbb H\cong\mbb R^4$ as a four-dimensional vector space on the field $\mbb R$,
and we replace $\langle\cdot,\cdot\rangle$ and $\langle\cdot,\cdot\rangle_\mu$ by the real inner products
\begin{align}
\label{Equation: Inner Product for H}
\prec f,g\succ=\Re\big(\langle f,g\rangle\big)\qquad\text{and}\qquad\prec f,g\succ_\mu=\Re\big(\langle f,g\rangle_\mu\big),
\end{align}
where $\Re$ denotes the real part.
At first glance, this might appear to be an inconsequential distinction, but this subtlety does
induce nontrivial concerns; most notably, two functions that are orthogonal with respect to
\eqref{Equation: Inner Product for H} need not be orthogonal with respect to $\langle\cdot,\cdot\rangle$ or $\langle\cdot,\cdot\rangle_\mu$.
While this does not change which numbers in $\mbb R$ are eigenvalues of $\hat H$, it does change the multiplicity of these eigenvalues,
which in turn changes the value of the trace $\mr{Tr}\big[\mr e^{-t\hat H}\big]$ by a constant (see \eqref{Equation: Trace Multiple of 4 1} and \eqref{Equation: Trace Multiple of 4 2}
for more details).

With this said, given that \cite{BloemendalVirag2} formulates its results using $\langle\cdot,\cdot\rangle_\mu$
(and that $\mr e^{-t\hat H}$'s kernel/semigroup is arguably easier to write and manipulate using quaternion multiplication rather than
$4\times 4$ matrix products over $\mbb R^4$), in this paper we choose to state (most of) our results using $\langle\cdot,\cdot\rangle_\mu$
as well. Nevertheless, we still need \eqref{Equation: Inner Product for H} in several places, 
since most of the standard functional-analytic
results that we use to construct $\hat H$ and $\mr e^{-t\hat H}$'s kernel
are stated under the assumption that the operator and kernel act on real/complex Hilbert spaces.
We refer to Sections \ref{Section: Quaternion-Operator Theory 1} and \ref{Section: Quaternion-Operator Theory 2} for
the full details of how this subtlety influences our constructions.
\end{remark}

\subsection{Domains and Boundary Conditions}

$\hat H$ is assumed to act on functions on the following domains:

\begin{assumption}\label{Assumption: Domain}
$I\subset\mbb R$ is one of the following three options: The full space $I=\mbb R$,
which we call {\bf Case 1};
the positive half line $I=(0,\infty)$, which we call {\bf Case 2}; or the bounded interval $I=(0,\vartheta)$ for some $\vartheta>0$,
which we call {\bf Case 3}.
\end{assumption}

In order to ensure that the operators that we consider are self-adjoint,
we must impose some appropriate boundary conditions when $I$
has a boundary:

\begin{assumption}\label{Assumption: Boundary}
In Case 2, we consider
functions $f:\mc A\to\mbb F$ that satisfy
\begin{align}
\label{Equation: Case 2 Boundary Conditions}
f'(i,0)+\bar\al_if(i,0)=0,\qquad1\leq i\leq r,
\end{align}
for a given vector $\bar\al=(\bar\al_1,\ldots,\bar\al_r)\in[-\infty,\infty)^r$,
with the convention that $\bar\al_i=-\infty$ corresponds to the Dirichlet
boundary condition $f(i,0)=0$.
In Case 3, we consider functions that satisfy
\[f'(i,0)+\bar\al_if(i,0)=0\qquad\text{and}\qquad-f'(i,\vartheta)+\bar\be_if(i,\vartheta)=0,\qquad1\leq i\leq r\]
for given vectors $\bar\al,\bar\be,\in[-\infty,\infty)^r$,
once again with the convention that $\bar\al_i=-\infty$ and $\bar\be_i=-\infty$
respectively correspond to $f(i,0)=0$ and $f(i,\vartheta)=0$.
\end{assumption}

\subsection{Deterministic and Random Potentials}

Next, we describe the potential functions that we consider in $\hat H$.
Most important for this are the following two definitions, which introduce the white
noise and its regular approximations:

\begin{notation}
When $\mbb F=\mbb C$, we use $\msf i$ to denote the imaginary unit,
and if $\mbb F=\mbb H$, then we use $\msf i$, $\msf j$, and $\msf k$
to denote the quaternion imaginary units.
\end{notation}

\begin{definition}
\label{Definition: Regular Noise}
Let $\rho:\mbb R\to\mbb R$ be a positive semidefinite function.
We say that the random function
$\mc N:I\to\mbb F$ is a regular noise in $\mbb F$ with covariance $\rho:\mbb R\to\mbb R$
if $\mc N$ is continuous with probability one, and we can write
\begin{align}
\label{Equation: x,y,z,w processes}
\mc N(x)=\begin{cases}
\mc X(x)&\text{if }\mbb F=\mbb R,\\
\frac1{\sqrt{2}}\big(\mc X(x)+\mc Y(x)\msf i\big)&\text{if }\mbb F=\mbb C,\\
\frac12\big(\mc X(x)+\mc Y(x)\msf i+\mc Z(x)\msf j+\mc W(x)\msf k\big)&\text{if }\mbb F=\mbb H,
\end{cases}
\end{align}
where $\mc X,\mc Y,\mc Z,\mc W:I\to\mbb R$ are i.i.d. continuous Gaussian processes
with mean zero and $\mbf E\big[\mc X(x)\mc X(y)\big]=\rho(x-y)$.
We define the quadratic form induced by $\mc N$ as
\[\mc N(f,g)=\langle f,\mc Ng\rangle,\qquad f,g\in C_0^\infty(I,\mbb F).\]
\end{definition}

\begin{definition}
\label{Definition: White Noise}
Let $W$ be a standard Brownian motion on $\mbb R$.
For every $f,g\in C_0^\infty(I,\mbb F)$, we define
$W'(f,g)\,``="\,\langle f,W'g\rangle$ using the formal integration by parts
\begin{align}
\label{Equation: White Noise 0}
W'(f,g)=\begin{cases}
-\langle f',Wg\rangle-\langle f,Wg'\rangle&\text{Cases 1 and 2},\\
f(\vartheta)^*W(\vartheta)g(\vartheta)-\langle f',Wg\rangle-\langle f,Wg'\rangle&\text{Case 3}.
\end{cases}
\end{align}
Then, we define the white noise in $\mbb F$ with variance $\si^2>0$ as the bilinear form
\begin{align}
\label{Equation: White Noise}
\mc N(f,g)=\begin{cases}
\si\,W_{1}'(f,g)&\text{if }\mbb F=\mbb R,\\
\frac\si{\sqrt{2}}\big(W_{1}'(f,g)+W_{\msf i}'(f,\msf ig)\big)&\text{if }\mbb F=\mbb C,\\
\frac\si2\big(W_{1}'(f,g)+W_{\msf i}'(f,\msf ig)+W_{\msf j}'(f,\msf jg)+W_{\msf k}'(f,\msf kg)\big)&\text{if }\mbb F=\mbb H,
\end{cases}
\end{align}
where $W_{1},W_{\msf i},W_{\msf j}$, and $W_{\msf k}$ are i.i.d. standard Brownian motions.
\end{definition}

With these definitions in hand, we now provide a precise description of $V+\xi$:

\begin{assumption}
\label{Assumption: Potential}
For every $1\leq i\leq r$, the function $V(i,\cdot):I\to\mbb R$ is bounded below and
locally integrable on $I$'s closure. Moreover, in Cases 1 and 2, there exists a constant
$\mf a>0$ such that
\[\lim_{|x|\to\infty}\frac{V(i,x)}{|x|^{\mf a}}=\infty\qquad\text{for all }1\leq i\leq r.\]
Let $\xi$ be the bilinear form on $C_0^\infty(\mc A,\mbb F)^2$ defined as
\begin{align}
\label{Equation: Xi quadratic form}
\xi(f,g)=\sum_{i,j=1}^r\xi_{i,j}\big(f(i,\cdot),g(j,\cdot)\big),
\end{align}
where
\begin{enumerate}
\item $\big(\xi_{i,i}:1\leq i\leq r\big)$ are i.i.d. copies of a regular or white noise in $\mbb R$.
\item $\big(\xi_{i,j}:1\leq i<j\leq r\big)$ are i.i.d. copies of a regular or white noise in $\mbb F$.
\item For $j<i$, we let $\xi_{j,i}=\xi_{i,j}^*$.
\item $\big(\xi_{i,i}:1\leq i\leq r\big)$ and $\big(\xi_{i,j}:1\leq i<j\leq r\big)$ are independent of each other.
\end{enumerate}
\end{assumption}

\begin{remark}
We note that, in Assumption \ref{Assumption: Potential}, we allow
the situation where the diagonal entries $\xi_{i,i}$ are white noises
and the off-diagonal entries $\xi_{i,j}$ are regular noises, and vice-versa.
This detail will become important once we outline the proof of our main
theorem in Section \ref{Section: Proof of Main Result Roadmap}.
\end{remark}

\begin{remark}
\label{Remark: GOE/GUE/GSE Case}
As advertised in the introduction, Assumptions \ref{Assumption: Domain},
\ref{Assumption: Boundary}, and \ref{Assumption: Potential} are partly designed
to include the multivariate SAO as a special case. That is, the multivariate SAO corresponds to
Case 2 with $V(i,\cdot)=\frac{rx}{2}$ and $\xi$ defined as follows:
\[\xi_{i,i}\text{ are white noises in $\mbb R$ with variance}=\begin{cases}
1&\text{if }\mbb F=\mbb R,\\
\frac12&\text{if }\mbb F=\mbb C,\\
\frac14&\text{if }\mbb F=\mbb H;
\end{cases}\]
and $\xi_{i,j}$ are white noises in $\mbb F$ with variance $\frac12$ when $i\neq j$.
\end{remark}

\subsection{Operators via Quadratic Forms}

We now provide a rigorous construction of $\hat H$ using quadratic forms.
We begin with the quadratic forms for $H=-\frac12\De+V$:

\begin{definition}
\label{Definition: Deterministic Quadratic Forms}
Let $\mr{AC}(I,\mbb F)$ denote the set of functions $f:I\to\mbb F$ that are locally absolutely
continuous on $I$'s closure. Given a function $h:I\to\mbb R$ that is bounded below and locally integrable,
let
\[\mr H^1_h(I,\mbb F)=\big\{f\in\mr{AC}(I,\mbb F):\|f\|_2,\|f'\|_2,\|h_+^{1/2}f\|_2<\infty\big\},\]
where $h_+(x)=\max\{0,h(x)\}$.
For each $1\leq i\leq r$, we define the bilinear form
$\mc E_i$ and the
form domain $D(\mc E_i)\subset\mr H^1_{V(i,\cdot)}(I,\mbb F)$ of the scalar-valued operator
\begin{align}
\label{Equation: One-Dimensional Operators}
H_if(x)=-\tfrac12f''(x)+V(i,x)f(x),\qquad f:I\to\mbb F
\end{align}
(with boundary $f'(0)-\bar\al_if(0)=0$ and $f'(\vartheta)-\bar\be_if(\vartheta)$ in Cases 2 and 3) as follows:
\begin{enumerate}
\item In Case 1, we let
\[\begin{cases}
D(\mc E_i)=\mr H^1_{V(i,\cdot)}(I,\mbb F),\\
\mc E_i(f,g)=\tfrac12\langle f',g'\rangle+\langle f,V(i,\cdot)g\rangle.
\end{cases}\]
\item In Case 2, if $\bar\al_i=-\infty$, we let
\[\begin{cases}
D(\mc E_i)=\left\{f\in\mr H^1_{V(i,\cdot)}(I,\mbb F):f(0)=0\right\},\\
\mc E_i(f,g)=\tfrac12\langle f',g'\rangle+\langle f,V(i,\cdot)g\rangle;
\end{cases}\]
and if $\bar\al_i\in\mbb R$, we let
\[\begin{cases}
D(\mc E_i)=\mr H^1_{V(i,\cdot)}(I,\mbb F),\\
\mc E_i(f,g)=\tfrac12\langle f',g'\rangle-\tfrac{\bar\al_i}2 f(0)^*g(0)+\langle f,V(i,\cdot)g\rangle.
\end{cases}\]
\item In Case 3, if $\bar\al_i=\bar\be_i=-\infty$, then
\[\begin{cases}
D(\mc E_i)=\left\{f\in\mr H^1_{V(i,\cdot)}(I,\mbb F):f(0)=f(\vartheta)=0\right\},\\
\mc E_i(f,g)=\tfrac12\langle f',g'\rangle+\langle f,V(i,\cdot)g\rangle;
\end{cases}\]
if $\bar\al_i,\bar\be_i\in\mbb R$, then
\[\begin{cases}
D(\mc E_i)=\mr H^1_{V(i,\cdot)}(I,\mbb F),\\
\mc E_i(f,g)=\tfrac12\langle f',g'\rangle-\tfrac{\bar\al_i}2 f(0)^*g(0)-\tfrac{\bar\be_i}2 f(\vartheta)^*g(\vartheta)+\langle f,V(i,\cdot)g\rangle;
\end{cases}\]
and if $\bar\al_i\in\mbb R$ and $\bar\be_i=-\infty$, then
\[\begin{cases}
D(\mc E_i)=\left\{f\in\mr H^1_{V(i,\cdot)}(I,\mbb F):f(\vartheta)=0\right\},\\
\mc E_i(f,g)=\tfrac12\langle f',g'\rangle-\tfrac{\bar\al_i}2 f(0)^*g(0)+\langle f,V(i,\cdot)g\rangle
\end{cases}\]
(if $\bar\al_i=-\infty$ and $\bar\be_i\in\mbb R$, then we define
$D(\mc E_i)$ and $\mc E_i$ in the same way as above, except that
we interchange the roles of $\bar\al_i$ and $\bar\be_i$, and $\th$ and $0$).
\end{enumerate}
Then, we define the bilinear form $\mc E$ and form domain $D(\mc E)=D(\mc E_1)\oplus\cdots\oplus D(\mc E_r)$
of the operator $H=H_1\oplus\cdots\oplus H_r$
as
\[\mc E(f,g)=\mc E_1\big(f(1,\cdot),g(1,\cdot)\big)+\cdots+\mc E_r\big(f(r,\cdot),g(r,\cdot)\big),\qquad f,g\in D(\mc E).\]
Finally, we define the following norm on $D(\mc E)$:
\[\|f\|^2_\star=\|f'\|_\mu^2+\|f\|_\mu^2+\|V_+^{1/2}f\|_\mu^2.\]
\end{definition}

The quadratic form associated with the noise $\xi$ is defined in \eqref{Equation: Xi quadratic form}.
That said, this form is a-priori only defined on smooth and compactly supported functions. If we want to define
the operator $\hat H=H+\xi$ via quadratic forms, then we have to ensure that the associated forms are well-defined
on the same domain. To this effect:

\begin{proposition}
\label{Proposition: Operator Form-Bound}
Let Assumptions \ref{Assumption: Domain}, \ref{Assumption: Boundary}, and \ref{Assumption: Potential} hold.
There exists a large enough $C>0$ so that
\begin{align}
\label{Equation: +1 Norm}
\|f\|_{+1}=\sqrt{\mc E(f,f)+(C+1)\|f\|_\mu}
\end{align}
is a norm on $D(\mc E)$, which is
equivalent to $\|f\|_\star$.
Moreover, almost surely,
for every $\ka\in(0,1)$ there exists a finite $\nu>0$ such that
\begin{align}
\label{Equation: Noise Form Bound}
|\xi(f,f)|\leq\ka\mc E(f,f)+\nu\|f\|_\mu^2\qquad\text{for every }f\in C_0^\infty(\mc A,\mbb F).
\end{align}
\end{proposition}

Thanks to Proposition \ref{Proposition: Operator Form-Bound},
almost surely,
there exists a unique continuous (with respect to $\|\cdot\|_\star$) extension of $\xi$ to a quadratic form on
$D(\mc E)$, which satisfies \eqref{Equation: Noise Form Bound} on all of $D(\mc E)$. Thus:

\begin{definition}
We define the quadratic form of $\hat H$ on the form domain $D(\mc E)$ as
\[\hat{\mc E}(f,f)=\mc E(f,f)+\xi(f,f),\qquad f\in D(\mc E).\]
\end{definition}

We now finally arrive at the main result in this section:

\begin{proposition}
\label{Proposition: Operator Definition}
Let Assumptions \ref{Assumption: Domain}, \ref{Assumption: Boundary}, and \ref{Assumption: Potential} hold. 
Almost surely, there exists a unique self-adjoint operator $\hat H$
with dense domain $D(\hat H)\subset L^2(\mc A,\mbb F)$ such that
\begin{enumerate}
\item $D(\hat H)\subset D(\mc E)$,
\item $\langle f,\hat Hf\rangle_\mu=\hat{\mc E}(f,f)$ for all $f\in D(\hat H)$, and
\item $\hat H$ has compact resolvent.
\end{enumerate}
In particular, we can define the eigenvalues $\la_1(\hat H)\leq\la_2(\hat H)\leq\cdots$
and corresponding eigenfunctions $\psi_k(\hat H)$ using the variational theorem:
For $k\in\mbb N$, one has
\begin{align}
\label{Equation: Variational}
\la_k(\hat H)=\inf_{f\in D(\mc E),~f\perp\psi_1(\hat H),\ldots,\psi_{k-1}(\hat H)}\frac{\hat{\mc E}(f,f)}{\|f\|_\mu^2},
\end{align}
and $f=\psi_k(\hat H)$ achieves the infimum.
The set $\big(\la_k(\hat H)\big)_{n\in\mbb N}$ is bounded below and has no accumulation point, and
the set $\big(\psi_k(\hat H)\big)_{n\in\mbb N}$ forms an orthonormal basis of $L^2(\mc A,\mbb F)$.
Thus, for every $t>0$,
\begin{align}
\label{Equation: Semigroup via Spectral Expansion}
\mr e^{-t\hat H}f=\sum_{k=1}^\infty\mr e^{-t\la_k(\hat H)}\psi_k(\hat H)\big\langle\psi_k(\hat H),f\big\rangle_\mu.
\end{align}
\end{proposition}

\begin{remark}
\label{Remark: Quaternion on R 2}
If $\mbb F=\mbb H$ in Proposition \ref{Proposition: Operator Definition},
then when we say that $\hat H$ is self-adjoint with compact resolvent, we mean so with $\hat H$
viewed as a operator on the bona fide Hilbert space over $\mbb R$ endowed with
the real-inner product in \eqref{Equation: Inner Product for H}. However, every
other claim uses $\langle\cdot,\cdot\rangle_\mu$. This includes the
orthogonality relation $\perp$ in the variational formula \eqref{Equation: Variational},
and the claim that the eigenfunctions form an orthonormal basis of $L^2(\mc A,\mbb H)$,
in the sense that each $f\in L^2(\mc A,\mbb H)$ has the unique series representation
\[f=\sum_{k=1}^\infty\psi_k(\hat H)\big\langle\psi_k(\hat H),f\big\rangle_\mu.\]
We further note that the order in the product $\psi_k(\hat H)\big\langle\psi_k(\hat H),f\big\rangle_\mu$ matters in $\mbb H$
due to the lack of commutativity, which explains our peculiar way of writing \eqref{Equation: Semigroup via Spectral Expansion}.
\end{remark}

\section{Main Results Part 2. Feynman-Kac Formula for Regular Noise}
\label{Section: Main Results 2}

In this section, we discuss our Feynman-Kac formula for the semigroup
\eqref{Equation: Semigroup via Spectral Expansion} in the case where $\xi$'s
components are regular noises.

\subsection{Stochastic Processes}

We begin by defining the stochastic processes that arise in
the scalar Feynman-Kac formula.

\begin{definition}
\label{Definition: 1D Stochastic Processes}
We use $B$ to denote a standard Brownian motion on $\mbb R$,
$X$ to denote a reflected standard Brownian motion on $(0,\infty)$,
and $Y$ to denote a reflected standard Brownian motion on $(0,\vartheta)$.
Then, we use $Z$ to denote the generic "reflected" Brownian motion on $I$,
namely,
\[Z=\begin{cases}
B&\text{Case 1,}\\
X&\text{Case 2,}\\
Y&\text{Case 3.}
\end{cases}\]

In Cases 2 and 3, 
given a boundary point $c\in\partial I$,
we let $\big(\mf L^c_{[s,t)}(Z):0\leq s<t\big)$
denote the continuous version (jointly in $s$ and $t$) of $Z$'s boundary local time. That is,
for every $0\leq s<t$, we have the almost-sure equalities 
\begin{align}
\label{Equation: Regular Boundary Local Time}
\mf L^0_{[s,t)}(Z)=\lim_{\eps\to0}\frac1{2\eps}\int_s^t\mbf 1_{\{0\leq Z(u)<\eps\}}\d u
\qquad\text{and}\qquad
\mf L^{\vartheta}_{[s,t)}(Y)=\lim_{\eps\to0}\frac1{2\eps}\int_s^t\mbf 1_{\{\vartheta-\eps<Y(u)\leq \vartheta\}}\d u
\end{align}
($\mf L^0_{[s,t)}(Z)$ is used in both Cases 2 and 3,
whereas $\mf L^{\vartheta}_{[s,t)}(Y)$ is only used in Case 3).
See, e.g., \cite[Chapter VI, Theorem 1.7 and Corollary 1.9]{RevuzYor}
for a proof that a continuous version of this process exists.
Finally, for any $t>0$, we denote $\mf L_t^c(Z)=\mf L_{[0,t)}^c(Z)$.
\end{definition}

Next, we introduce the processes in the vector-valued formula:

\begin{definition}
\label{Definition: Combined Stochastic Processes}
Let $U(t)$ and $\tau_k$ be as in Definition \ref{Definition: Uniform RW} (i.e., uniform continuous-time random walk on
the complete graph with vertices $\{1,\ldots,r\}$ with jump rate $(r-1)$).
We define the Markov process $A(t)=\big(U(t),Z(t)\big)$, assuming that $U$ and $Z$
are independent.
See Figure \ref{Figure: Generic Path of A} in Appendix \ref{Appendix: Illustrations} for an illustration
of $A$'s path, which takes values in $\mc A$.

In Cases 2 and 3,
for every $1\leq i\leq r$, $c\in\partial I$, and $0\leq s\leq t$, we define
\begin{align}
\label{Equation: Multivariate Boundary Local Time 1}
\mf L^{(i,c)}_{[s,t)}(A)=\sum_{k\geq1:~U(\tau_{k-1})=i}\mf L_{[\tau_{k-1},\tau_k)\cap[s,t)}^c(Z);
\end{align}
i.e., $Z$'s boundary local time at $c$ restricted to
the subset $\{u\in[s,t):U(u)=i\}$. We note
\[\mf L^c_{[s,t)}(Z)=\sum_{i=1}^r\mf L^{(i,c)}_{[s,t)}(A).\]
However, we are more interested in the
combination of these boundary local times weighted by the constants $\bar\al$ and $\bar\be$
that determine the boundary conditions
in Assumption
\ref{Assumption: Boundary},
thus leading us to define the boundary term
\begin{align}
\label{Equation: Multivariate Boundary Local Time 2}
\mf B_t(A)=\begin{cases}
0&\text{Case 1,}
\vspace{5pt}\\
\displaystyle\sum_{i=1}^r\bar\al_i\mf L^{(i,0)}_t(A)&\text{Case 2,}
\vspace{5pt}\\
\displaystyle\sum_{i=1}^r\bar\al_i\mf L^{(i,0)}_t(A)+\sum_{i=1}^r\bar\be_i\mf L^{(i,\vartheta)}_t(A)&\text{Case 3.}
\end{cases}
\end{align}
\end{definition}

Finally, we state a general notation for random paths conditioned on their start and end points,
which in the sequel will be applied to $U,Z,A,$ and other processes, as well as some notation
for conditional expectations:

\begin{definition}
\label{Definition: Combined Stochastic Processes 2}
Let $\mf X=\big(\mf X(t)\big)_{t\geq0}$ be any continuous-time random process
with c\`{a}dl\`{a}g sample paths.
We denote $\mf X$'s transition kernel by
\[\Pi_\mf X(t;x,y)=\begin{cases}
\mbf P[\mf X(t)=y|\mf X(0)=x]&\text{if $\mf X$ is discrete,}\\
\mbf P[\mf X(t)\in\dd y|\mf X(0)=x]&\text{if $\mf X$ is continuous.}
\end{cases}\]
For any states $x,y$ and any time $t\geq0$, we denote
the conditionings
\[\mf X^x=\big(\mf X\big|\mf X(0)=x\big)\qquad\text{and}\qquad \mf X^{x,y}_t=\big(\mf X\big|\mf X(0)=x\text{ and }\mf X(t)=y\big).\]
\end{definition}

\begin{definition}
Given a random variable/process $\mf X$ and a functional $F(\mf X)$
that may depend on other sources of randomness, we let
$\mbf E{_\mf X}\big[F(\mf X)\big]$ denote the expectation with respect to $\mf X$
only, conditional on any other source of randomness.
\end{definition}

\subsection{Regular Feynman-Kac Formula}

\begin{theorem}
\label{Theorem: Regular FK}
Let Assumptions \ref{Assumption: Domain}, \ref{Assumption: Boundary}, and \ref{Assumption: Potential} hold,
assuming further that each entry $\xi_{i,j}$ is a {\bf regular} noise
in the sense of Definition \ref{Definition: Regular Noise}.
To simplify our notation, let us write $V+\xi$'s diagonal part as
\[S(i,x)=V(i,x)+\xi_{i,i}(x),\qquad 1\leq i\leq r,~x\in I,\]
and the non-diagonal part as
\[Q_{i,j}(x)=\xi_{i,j}(x),\qquad 1\leq i\neq j\leq r,~x\in I.\]
Denote the functional
\begin{align}
\label{Equation: nt}
\mf n_t(Q,A)=\begin{cases}
\mr e^{(r-1)t}&\text{if }N(t)=0,\\
\displaystyle
\mr e^{(r-1)t}(-1)^{N(t)}\left(\prod_{k=1}^{N(t)}\xi_{J_k}\big(Z(\tau_k)\big)\right)
&\text{if }N(t)>0,
\end{cases}
\end{align}
where we recall that $N(t)$, $\tau_k$, and $J_k$ were introduced in
Definition \ref{Definition: Uniform RW}.
Given $t>0$, let $\hat K(t)$ be the random integral operator on $L^2(\mc A,\mbb F)$ with kernel
\begin{align}
\label{Equation: Regular FK Kernel}
\hat K(t;a,b)=\Pi_A(t;a,b)\mbf E_A\left[\mf n_t(Q,A^{a,b}_t)\,\mr e^{-\int_0^tS(A^{a,b}_t(s))\d s+\mf B_t(A^{a,b}_t)}\right],\qquad a,b\in\mc A,
\end{align}
where for $a=(i,x)$ and $b=(j,y)$, one has $\Pi_A(t;a,b)=\Pi_Z(t;x,y)\Pi_U(t;i,j)$
($\hat K(t)$ is random because it depends on the realization of $\xi$ via $S$ and $Q$).
In other words,
\[\hat K(t)f(a)=\int_{\mc A}\hat K(t;a,b)f(b)\d\mu(b),\qquad f\in L^2(\mc A,\mbb F),\]
where we recall the definition of $\mu$ in Definition \ref{Definition: Hilbert Spaces}.
Almost surely, for every $t>0$,
\begin{align}
\label{Equation: Regular FK}
\mr e^{-t\hat H}=\hat K(t)\qquad\text{and}\qquad
\mr{Tr}\big[\mr e^{-t\hat H}\big]=\int_\mc A\hat K(t;a,a)\d\mu(a)<\infty.
\end{align}
\end{theorem}

Before moving on, a few remarks:

\begin{remark}
In \eqref{Equation: Regular FK Kernel},
we use the conventions that $x-\infty=-\infty$ for all $x\in[-\infty,\infty)$, and that $\mr e^{-\infty}=0$.
Thus, when one of the boundary conditions is Dirichlet (in the convention of
Assumption \ref{Assumption: Boundary}, $\bar\al_i=-\infty$
and/or $\bar\be_i=-\infty$ for some $1\leq i\leq r$), its contribution to the
boundary term $\mf B_t(A)$ defined in \eqref{Equation: Multivariate Boundary Local Time 2} is interpreted as
\[\mr e^{-\infty\mf L^{(i,c)}_t(Z)}=\mbf 1_{\{\mf L^{(i,c)}_t(Z)=0\}}=\mbf 1_{\{Z(s)\neq c\text{ for every }s\in[0,t)\text{ such that }U(s)=i\}}.\]
\end{remark}

\begin{remark}
When $r=1$, the process $U(t)$ is just constant on $1$,
and $\mf n_t(Q,A)=1$. In this case, we can simply identify $\mc A=\{1\}\times I=I$,
and Theorem \ref{Theorem: Regular FK} reduces to
the one-dimensional Feynman-Kac formula for $Hf=-\frac12f''+Sf$,
with $S(x)=V(x)+\xi(x)$:
\begin{align}
\label{Equation: Classical Feynman-Kac}
\hat K(t;x,y)=
\begin{cases}
\Pi_B(t;x,y)\mbf E_A\left[\mr e^{-\int_0^tS(B^{x,y}_t(s))\d s}\right]&\text{Case 1,}
\vspace{5pt}\\
\Pi_X(t;x,y)\mbf E_A\left[\mr e^{-\int_0^tS(X^{x,y}_t(s))\d s+\bar\al\mf L^0_t(X^{x,y}_t)}\right]&\text{Case 2,}
\vspace{5pt}\\
\Pi_Y(t;x,y)\mbf E_A\left[\mr e^{-\int_0^tS(Y^{x,y}_t(s))\d s+\bar\al\mf L^0_t(Y^{x,y}_t)+\bar\be\mf L^{\vartheta}_t(Y^{x,y}_t)}\right]&\text{Case 3.}
\end{cases}
\end{align}
See, e.g., \cite{ChungZhao,Papanicolaou,Sznitman}
(more precise references can be found in Section \ref{Section: Generator Part 1}).
\end{remark}

\begin{remark}
\label{Remark: Boundary Terms Interpretation}
Looking back at \eqref{Equation: Regular Boundary Local Time},
we can informally write
\begin{align}
\label{Equation: Boundary Local Time as an Integral}
\mf L^c_{[s,t)}(Z)=\frac12\int_s^t\de_c\big(Z(u)\big)\d u,
\end{align}
where $\de_c$ denotes the delta Dirac mass at $c\in\{0,\vartheta\}$.
Thus, the
presence of the boundary terms (i.e., the multiples of $\mf L_t^c(Z)$)
in the scalar Feynman-Kac formula
\eqref{Equation: Classical Feynman-Kac} in Cases 2 and 3
can be explained by the following heuristic:
The process of going from the scalar full-space Feynman-Kac formula to the half-line or interval
involves
\begin{enumerate}
\item replacing $B$ by the appropriate reflected version $X$ or $Y$, and
\item replacing $V(x)$ by $V(x)-\frac{\bar\al}2\de_0(x)$ in Case 2
or $V(x)-\frac{\bar\al}2\de_0(x)-\frac{\bar\be}2\de_\th(x)$ in Case 3.
\end{enumerate}
(This heuristic can also be justified by examining the quadratic forms of the operators involved;
see Definition \ref{Definition: Deterministic Quadratic Forms}.) In the vector-valued case, we can
explain the presence of the boundary term $\mf B_t(Z)$ in \eqref{Equation: Regular FK Kernel}
by performing the same heuristic; the only difference is that we replace the entries of the matrix
potential $V+\xi=S+Q$ by
\begin{align}
\label{Equation: Vector-Valued Boundary Heuristic}
\begin{cases}
S(i,x)-\frac{\bar\al_i}2\de_0(x)&\text{if }i=j\\
Q_{i,j}(x)&\text{if }i\neq j
\end{cases}
\qquad\text{or}\qquad\begin{cases}
S(i,x)-\frac{\bar\al_i}2\de_0(x)-\frac{\bar\be_i}2\de_\th(x)&\text{if }i=j\\
Q_{i,j}(x)&\text{if }i\neq j
\end{cases}
\end{align}
in Cases 2 or 3 respectively. By Proposition \ref{Proposition: Probabilistic OE}, we understand that
the contribution of $V+\xi$'s diagonal to $\fo{\mc T}\left\{\mr e^{-\int_0^tV(Z(s))+\xi(Z(s))\d s}\right\}$ is in the exponential $\mr e^{-\int_0^tS(A(s))\d s}$; hence
\eqref{Equation: Vector-Valued Boundary Heuristic} very naturally leads to
$\mf B_t(Z)$.
In contrast, the contributions of the Dirac masses in \eqref{Equation: Vector-Valued Boundary Heuristic}
are arguably more difficult to parse/calculate if we use either
\eqref{Equation: Ordered Exponential Differential Equation},
\eqref{Equation: Ordered Exponential Product}, or \eqref{Equation: Ordered Exponential Series}
to write the operator exponentials.
\end{remark}

\section{Main Results Part 3. Trace Moments and Rigidity}
\label{Section: Main Results 3}

We now finish the statements of our main results with the
Feynman-Kac formulas for the trace moments \eqref{Equation: Mixed Moments},
as well as our rigidity result.

\subsection{Combinatorial and Stochastic Process Preliminaries}

We begin by introducing the combinatorial objects that
are used to characterize which jump sequences $J_1,J_2,J_3,\ldots$ in the process
$U$ contribute to the sum \eqref{Equation: Difficulty when r>1 4} (the latter of which
we recall arises from an application of Isserlis' theorem):

\begin{definition}
\label{Definition: Bn}
Given $n\in\mbb N$, let
$\mc B_n$ denote the set of binary sequences with $n$ steps,
namely, sequences of the form $m=(m_0,m_1,\ldots,m_n)\in\{0,1\}^{n+1}$.
Given $h,l\in\{0,1\}$, we let $\mc B^{h,l}_n$ denote the set of
sequences $m\in\mc B_n$ such that $m_0=h$ and $m_n=l$.
\end{definition}

\begin{definition}
\label{Definition: Combinatorial Constant}
Suppose we are given a realization of the process $U$
in Definition \ref{Definition: Combined Stochastic Processes}.
Let $t\geq0$ be fixed, and suppose that $N(t)\geq2$ is even.
Given any jump $J_k=(i,j)$,
we let $J_k^*=(j,i)$ denote that same jump in reverse order.
Given any matching $p\in\mc P_{N(t)}$
(recall Definition \ref{Definition: Pn}), we define the constant $\mf C_t(p,U)$ as follows:
\begin{enumerate}
\item If $\mbb F=\mbb R$, then
\[\mf C_t(p,U)
=\begin{cases}
1&\text{if either $J_{\ell_1}=J_{\ell_2}$ or $J_{\ell_1}=J_{\ell_2}^*$ for all $\{\ell_1,\ell_2\}\in p$},\\
0&\text{otherwise}.
\end{cases}\]
\item If $\mbb F=\mbb C$, then
\[\mf C_t(p,U)
=\begin{cases}
1&\text{if $J_{\ell_1}=J_{\ell_2}^*$ for all $\{\ell_1,\ell_2\}\in p$},\\
0&\text{otherwise}.
\end{cases}\]
\item If $\mbb F=\mbb H$, then
\[\mf C_t(p,U)
=\begin{cases}
\mf D_t(p,U)&\text{if either $J_{\ell_1}=J_{\ell_2}$ or $J_{\ell_1}=J_{\ell_2}^*$ for all $\{\ell_1,\ell_2\}\in p$},\\
0&\text{otherwise}.
\end{cases}\]
Here, we define $\mf D_t(p,U)$ as follows: A binary sequence $m\in\mc B^{0,0}_{N(t)}$
is said to respect $(p,J)$ if for every
pair $\{\ell_1,\ell_2\}\in p$, the following holds:
\begin{enumerate}[(3.1)]
\item if $J_{\ell_1}=J_{\ell_2}$, then $\big((m_{\ell_1-1},m_{\ell_1}),(m_{\ell_2-1},m_{\ell_2})\big)$
are equal to one of the following pairs: $\big((0,0),(1,1)\big)$, $\big((1,1),(0,0)\big)$,
$\big((0,1),(1,0)\big)$, or $\big((1,0),(0,1)\big)$.
\item if $J_{\ell_1}=J_{\ell_2}^*$, then $\big((m_{\ell_1-1},m_{\ell_1}),(m_{\ell_2-1},m_{\ell_2})\big)$
are equal to one of the following pairs: $\big((0,0),(0,0)\big)$, $\big((1,1),(1,1)\big)$,
$\big((0,1),(1,0)\big)$, or $\big((1,0),(0,1)\big)$.
\end{enumerate}
Lastly, we call any pair $\{\ell_1,\ell_2\}\in p$ such that $J_{\ell_1}=J_{\ell_2}$
and
\[\big((m_{\ell_1-1},m_{\ell_1}),(m_{\ell_2-1},m_{\ell_2})\big)=\big((0,1),(1,0)\big)\text{ or }\big((1,0),(0,1)\big)\]
a flip, and we let $f(m,p,J)$ denote the total number of flips that occur
in the combination of $m,p,$ and $J$.
With all this in hand, we finally define
\begin{align}
\label{Equation: D constant for H}
\mf D_t(p,U)=2^{-N(t)/2}\sum_{m\in\mc B^{0,0}_{N(t)},~m\text{ respects }(p,J)}(-1)^{f(m,p,J)}.
\end{align}
\end{enumerate}
\end{definition}

\begin{remark}
We refer to Figures \ref{Figure: Walk 1 and 4}, \ref{Figure: Walk 2}, \ref{Figure: Binary non-example},
\ref{Figure: Binary 1}, and \ref{Figure: Binary 2} in Appendix \ref{Appendix: Illustrations} for illustrations
of the objects related to the combinatorial constants in Definition \ref{Definition: Combinatorial Constant}.
Moreover, we note that for any $p\in\mc P_{N(t)}$, the number of binary paths that respect
$(p,J)$ is at most $2^{N(t)/2}$ (since every step $(m_{\ell_1-1},m_{\ell_1})$ determines
the other step $(m_{\ell_2-1},m_{\ell_2})$ that is matched to it via $\{\ell_1,\ell_2\}\in p$). Consequently, in all cases (i.e., for $\mbb F=\mbb R,\mbb C,\mbb H$), one has
\begin{align}
\label{Equation: Combinatorial constant bound}
|\mf C_t(p,U)|\leq1.
\end{align}
\end{remark}

Next, in order to formalize the singular conditioning on the event
\eqref{Equation: Probability Zero Events}, we introduce self-intersection
local time measures induced by a permutation, the latter of which require
a precise definition of local time processes:

\begin{definition}
\label{Definition: Regular Local Time}
Let $\big(L^x_{[s,t)}(Z):0\leq s<t,~x\in I\big)$ denote
the continuous (jointly in $s$, $t$, and $x$) version of $Z$'s local time process on the time intervals
$[s,t)$, that is,
\begin{align}
\label{Equation: Local Time Integral Definition}
\int_s^tf\big(Z(u)\big)\d u=\int_I L^x_{[s,t)}(Z)f(x)\d x\qquad\text{for all measurable }f:I\to\mbb R.
\end{align}
(see, e.g., \cite[5, Chapter VI, Corollary 1.6 and Theorem 1.7]{RevuzYor} for a proof of the existence of this continuous version). We use the convention $L_t^x(Z)=L_{[0,t)}^x(Z)$.

Next, we define the local time process $\big(L_{[s,t)}^a(A):0\leq s<t,~a\in\mc A\big)$ of
the multivariate process $A=(U,Z)$ as follows:
\begin{align}
\label{Equation: Multivariate Local Time}
L^{(i,x)}_{[s,t)}(A)=\sum_{k\geq1:~U(\tau_{k-1})=i}L_{[\tau_{k-1},\tau_k)\cap[s,t)}^x(Z);
\end{align}
\[\int_s^tf\big(A(u)\big)\d u=\sum_{k\geq1}\int_{[\tau_{k-1},\tau_k)\cap[s,t)}f\big(U(\tau_{k-1}),Z(u)\big)\d u=\int_{\mc A}L^a_{[s,t)}(A)f(a)\d\mu(a).\]
In particular, for every $0\leq s<t$ and $x\in I$, one has
\begin{align}
\label{Equation: Local Time Sum}
\sum_{i=1}^rL^{(i,x)}_{[s,t)}(A)=L^x_{[s,t)}(Z).
\end{align}

Finally, let $t>0$, let $n\geq2$ be an even integer, and let $q\in\mc P_n$.
Given a realization of $Z$, we let $\mr{si}_{t,n,q,Z}$ denote
the unique random Borel probability measure on $[0,t)^n$ such that for every
$s_1,t_1,\ldots,s_n,t_n\in[0,t)$ with $s_k<t_k$ for $k=1,\ldots,n$, one has
\begin{multline}
\label{Equation: SI Measures}
\mr{si}_{t,n,q,Z}\big([s_1,t_1)\times[s_2,t_2)\times\cdots\times[s_n,t_n)\big)\\
=\frac{1}{\|L_t(Z)\|_2^{2n}}\prod_{\{\ell_1,\ell_2\}\in q}\left\langle L_{[s_{\ell_1},t_{\ell_1})}(Z),L_{[s_{\ell_2},t_{\ell_2})}(Z)\right\rangle
\end{multline}
(this uniquely determines the measure since product rectangles are a determining class
for the Borel $\si$-algebra).
\end{definition}

\begin{remark}
\label{Remark: Self Intersection Uniform}
Informally, one can think of $\mr{si}_{t,n,q,Z}$ as the uniform probability measure
on the set of $n$-tuples of time coordinates in $[0,t)^n$ such that each pair of times matched
by $q$ coincide with a self-intersection of $Z$; namely:
\[\big\{(t_1,\ldots,t_n)\in[0,t)^n:Z(t_{\ell_1})=Z(t_{\ell_2})\text{ for all }\{\ell_1,\ell_2\}\in q\big\}.\]
(We point to \cite[Theorem 2.3.2 and Proposition 2.3.5]{ChenBook} for a proof that
$\mr{si}_{t,n,q,Z}$ is supported on that set.)
Recall that, conditional on $N(t)$, the jump times $\tau_1,\ldots,\tau_{N(t)}$
are the order statistics of i.i.d. uniform random variables in $[0,t)$. Thus,
$\mr{si}_{t,n,q,Z}$ serves as a natural means of formalizing the conditioning on the event \eqref{Equation: Probability Zero Events}
(assuming $q$ and the matching $p$ in \eqref{Equation: Probability Zero Events} are related to each other
in such a way that $p$ matches the same pairs of times $t_k$ as $q$ once they are arranged in nondecreasing order).
\end{remark}

Finally, given that products of Feynman-Kac kernels involve multiple independent copies
of the processes $U$ and $X$, we introduce some notation to help write such expressions
tidily:

\begin{definition}
Let $\mf X$ be a stochastic process with c\`{a}dl\`{a}g sample paths.
Given $n$-tuples of times $\bs t=(t_1,\ldots,t_n)$ and states $\bs x=(x_1,\ldots,x_n)$, we let
\begin{align}
\label{Equation: Concatenated Path}
\mf X^{\bs x}_{\bs t}
=\mf X^{(x_1,\ldots,x_n)}_{(t_1,\ldots,t_n)}
\end{align}
denote a c\`{a}dl\`{a}g concatenation of independent paths
\[\big(\mf X^{x_1}(s):s\in[0,t_1)\big),\big(\mf X^{x_2}(s):s\in[0,t_2)\big),\ldots,\big(\mf X^{x_n}(s):s\in[0,t_n)\big)\]
on the time-interval $[0,t_1+\cdots+t_k)$.
That is, to generate $\mf X^{\bs x}_{\bs t}(s)$,
we follow the steps:
\begin{enumerate}
\item firstly, we run a path of $\mf X$ started at $x_1$
on the time interval $[0,t_1)$; 
\item secondly, we run an independent path of $\mf X$ started at $x_2$
on the time interval $[t_1,t_1+t_2)$;
\item $\cdots$
\item lastly, we run an independent path of $\mf X$ started at $x_n$
on the time interval $[t_1+\cdots+t_{n-1},t_1+\cdots+t_n)$.
\end{enumerate}
Given $n$-tuples of times $\bs t=(t_1,\ldots,t_n)$
and states $\bs x=(x_1,\ldots,x_n)$, $\bs y=(y_1,\ldots,y_n)$, let
\begin{align}
\label{Equation: Concatenated Path Bridge}
\mf X^{\bs x,\bs y}_{\bs t}=\mf X^{(x_1,\ldots,x_n),(y_1,\ldots,y_n)}_{(t_1,\ldots,t_n)}
\end{align}
be the c\`{a}dl\`{a}g concatenation $\mf X^{\bs x}_{\bs t}$ with the additional endpoint conditioning
\[\mf X^{\bs x,\bs y}_{\bs t}=\left(\mf X^{\bs x}_{\bs t}~\bigg|~\lim_{s\to(t_1+\cdots+t_k)^-}\mf X^{\bs x}_{\bs t}(s)=y_k\text{ for all }1\leq k\leq n\right).\]
See Figure \ref{Figure: Concatenated Path of A} in Appendix \ref{Appendix: Illustrations} for an illustration
of one such c\`{a}dl\`{a}g concatenation in the case $\mf X=A$.
Finally, we use the following shorthand for a product of transition kernels:
\[\Pi_{\mf X}(\bs t;\bs x,\bs y)=\prod_{k=1}^n\Pi_{\mf X}(t_k;x_k,y_k).\]
\end{definition}

\subsection{Trace Moment Formulas}

We are now finally in a position to state our trace moment formulas
in the case where $\xi$ is a white noise. First,
we provide a precise statement of how the conditioning on the singular event \eqref{Equation: Probability Zero Events}
is achieved:

\begin{notation}
\label{Notation: Arrow means Ordered}
Let $t>0$ and $n\in\mbb N$ be fixed.
Given $\bs s=(s_1,\ldots,s_n)\in[0,t]^n$, let $(\fo{s_1},\ldots,\fo{s_n})$
denote $\bs s$'s coordinates
in increasing order, i.e., $\fo{s_1}\leq \fo{s_2}\leq\cdots\leq \fo{s_n}$.
\end{notation}

\begin{definition}
\label{Definition: Singular Process}
Suppose that we are given a realization of $Z$
on some time interval $[0,t)$. Given this, we generate the process $\hat U$ on $[0,t)$ as follows:
\begin{enumerate}[$\bullet$]
\item Given $Z$, sample $\hat N(t)$ according to the distribution
\[\mbf P\big[\hat N(t)=2m\big]=\left(\frac{(r-1)^2\|L_{t}(Z)\|^2_2}{2}\right)^m\frac{1}{m!}\mr e^{-\frac{(r-1)^2}2\|L_{t}(Z)\|_2^2},\qquad m=0,1,2,\ldots,\]
i.e., $\hat N(t)/2$ is Poisson with parameter $\frac{(r-1)^2}2\|L_{t}(Z)\|_2^2$.
\item Given $\hat N(t)$, sample $\hat q$ uniformly at random in $\mc P_{\hat N(t)}$.
\item Given $Z$, $\hat N(t)$, and $\hat q$,
draw $\bs T=(T_1,\ldots,T_{\hat N(t)})\in[0,t]^{\hat N(t)}$
according to the self-intersection measure $\mr{si}_{t,\hat N(t),\hat q,Z}$.
Then, using Notation \ref{Notation: Arrow means Ordered}, we define $\hat{\bs\tau}=(\hat\tau_1,\ldots,\hat\tau_{\hat N(t)})$ to be such that
 $\hat \tau_k=\fo{T_k}$ for all $k\geq1$.
In words, $\hat{\bs\tau}$ is the tuple $\bs T$ rearranged in nondecreasing order.
\item Given $\hat N(t)$, $\bs T$, and $\hat{\bs\tau}$, we let $\hat\pi$ be the random permutation
of $\{1,\ldots,\hat N(t)\}$
such that $\hat\tau_{\hat\pi(\ell)}=T_\ell$ for all $\ell\leq\hat N(t)$, and we let $\hat p\in\mc P_{\hat N(t)}$
be the random perfect matching such that $\{\ell_1,\ell_2\}\in\hat q$ if and only if
$\{\hat\pi(\ell_1),\hat\pi(\ell_2)\}\in\hat p$. In words, $\hat p$ ensures that each pair of times $T_k$
that were matched by $\hat q$ are still matched once arranged in nondecreasing order $\hat \tau_k$.
\item Sample the jumps $J_k$ in the same way as for $U$, i.e., uniform random walk on
the complete graph (without self-edges) on $\{1,\ldots,r\}$.
\item Combining all of the above, we then let
$\big(\hat U(s):1\leq s\leq t\big)$ be the c\`{a}dl\`{a}g
path with jump times $\hat{\bs\tau}$
and jumps $J$ (in particular, if $\hat N(t)=0$, then $\hat U(t)$ is constant).
\end{enumerate}
See Figure \ref{Figure: Condition on Jump Times and Values of X} in Appendix \ref{Appendix: Illustrations} for an illustration
of $(\hat U,Z)$
\end{definition}

Our trace moment formulas are as follows:

\begin{theorem}
\label{Theorem: Trace Moment Formulas}
Let Assumptions \ref{Assumption: Domain},
\ref{Assumption: Boundary}, and \ref{Assumption: Potential} hold,
assuming further that $\xi_{i,i}$ are white noises in $\mbb R$ with variance $\si^2>0$,
and that $\xi_{i,j}$ ($i\neq j$) are white noises in $\mbb F$ with variance $\upsilon^2>0$.
Define the functionals
\[\mf M_t(\hat p,\hat U)=
\begin{cases}
1&\text{if }\hat N(t)=0\\
\upsilon ^{\hat N(t)}\mf C_t(\hat p,\hat U)&\text{otherwise},
\end{cases}\]
and
\begin{align}
\label{Equation: Trace Moment Exponential Term}
\mf H_t(\hat U,Z)=
\frac{(r-1)^2}2\|L_t(Z)\|_2^2+\frac{\si^2}2\|L_t(\hat U,Z)\|_\mu^2-\int_0^{t}V\big(\hat U(s),Z(s)\big)\d s+\mf B_{t}(\hat U,Z).
\end{align}
For every $\bs t=(t_1,\ldots,t_n)\in(0,\infty)^n$, one has
\begin{multline}
\label{Equation: Trace Moment Formula}
\mbf E\left[\mr{Tr}\big[\mr e^{-t_1\hat H}\big]\cdots\mr{Tr}\big[\mr e^{-t_n\hat H}\big]\right]\\
=\int_{\mc A^n}
\Pi_Z(\bs t;\bs x,\bs x)\mbf E\left[\mf M_{|\bs t|}(\hat p,\hat U^{\bs i}_{\bs t})\mr e^{\mf H_{|\bs t|}(\hat U^{\bs i}_{\bs t},Z^{\bs x,\bs x}_{\bs t})}
\mbf 1_{\{\forall k:~\hat U^{\bs i}_{\bs t}(t_1+\cdots+t_{k-1}+s)\to i_k\text{ as }s\to t_k^-\}}\right]\d\mu^n(\bs i,\bs x),
\end{multline}
where we denote $\bs i=(i_1,\ldots,i_n)\in\{1,\ldots,r\}^n$ and $\bs x=(x_1,\ldots,x_n)\in I^n$.
\end{theorem}

\subsection{Rigidity}

We now state our result regarding the number rigidity of $\hat H$. For this purpose,
we recall the general formal definition of point processes on the real line:

\begin{definition}
Let $\mc K\subset\mbb R$ be a Borel set
and $\mc B_\mc K$ be the $\si$-algebra of all Borel subsets of $\mc K$.
Let $\mc N^\#_\mc K$ denote the set
of integer-valued measures $\mc M$ on $(\mc K,\mc B_\mc K)$ such that $\mc M(\bar{\mc K})<\infty$ for every bounded set $\bar{\mc K}\in\mc B_\mc K$.
We equip $\mc N^\#_\mc K$ with the topology generated by the maps
\[\mc M\mapsto\int f(x)\d\mc M(x),\qquad f:\mc K\to\mbb R\text{ is continuous and compactly supported,}\]
as well as the associated Borel $\si$-algebra.
We note that for every $\mc M\in\mc N^\#_\mc K$, there exists a sequence
$(x_k)_{k\in\mbb N}\subset \mc K$ without accumulation point such that
$\mc M=\sum_{k=1}^\infty\de_{x_k}$;
i.e., $\mc M(\bar{\mc K})$ counts the number of points in the sequence $x_k$
that are within the set $\bar{\mc K}$. Finally,
a point process on $\mc K$ is a random element
that takes values in $\mc N^\#_\mc K$.
\end{definition}

In this paper, we are only interested in the eigenvalue point process of $\hat H$:

\begin{definition}
Let Assumptions \ref{Assumption: Domain},
\ref{Assumption: Boundary}, and \ref{Assumption: Potential} hold.
We define
$\hat H$'s eigenvalue point process as
\[\hat{\mc M}=\sum_{k=1}^\infty\de_{\la_k(\hat H)}\in\mc N^\#_\mbb R.\]
For every Borel set $\mc K\subset\mbb R$, we denote the restriction
\[\hat{\mc M}|_{\mc K^c}:=\sum_{k=1}^\infty\mbf 1_{\{\la_k(\hat H)\in \mc K^c\}}\,\de_{\la_k(\hat H)}\in\mc N^\#_{\mc K^c}.\]
That is, the same random measure as $\hat{\mc M}$, but excluding the masses in $\mc K$.
\end{definition}

The definition of number rigidity used in this paper is as follows:

\begin{definition}
\label{Definition: Number Rigidity}
We say that a Borel set $\mc K\subset\mbb R$ is bounded above if
there exists some $c\in\mbb R$ such that $\mc K\subset(-\infty,c]$.
(Note that by Proposition \ref{Proposition: Operator Definition},
$\hat{\mc M}$ is such that, almost surely, $\mc{\hat M}(\mc K)<\infty$ for every
Borel set $\mc K\subset\mbb R$ that is bounded above.)
We say that $\hat{\mc M}$ is number rigid if
for every bounded above Borel set $\mc K\subset\mbb R$, there exists
a deterministic measurable function $F_\mc K:\mc N^\#_{\mc K^c}\to\mbb R$ such that
\begin{align}
\hat{\mc M}(\mc K)=F_{\mc K}\big(\hat{\mc M}|_{\mc K^c}\big)\qquad\text{almost surely};
\end{align}
in words, the number of eigenvalues in $\mc K$ is a deterministic function of the
configuration of eigenvalues outside $\mc K$.
\end{definition}

Our result regarding number rigidity is as follows:

\begin{theorem}
\label{Theorem: Rigidity}
Let Assumptions \ref{Assumption: Domain},
\ref{Assumption: Boundary}, and \ref{Assumption: Potential} hold,
assuming further that each entry of $\xi$ is a white noise.
\begin{enumerate}[$\bullet$]
\item In Cases 1 and 2, if there exists $\ka,\nu>0$ such that $V(i,x)\geq\ka|x|-\nu$ for all $1\leq i\leq r$ and $x\in I$,
then $\hat{\mc M}$ is number rigid. In particular, the eigenvalue point process of any
multivariate SAO (as defined in Remark \ref{Remark: GOE/GUE/GSE Case}) is number rigid.
\item In Case 3, $\hat{\mc M}$ is number rigid with no further condition on $V$.
\end{enumerate}
\end{theorem}

\section{Proof of Proposition \ref{Proposition: Operator Definition}: Construction of Operators}
\label{Section: Construction of Operators Proof}

Suppose that Assumptions \ref{Assumption: Domain}, \ref{Assumption: Boundary}, and \ref{Assumption: Potential} hold. 
In the argument that follows, we record (as remarks) a few properties of the operators $H$ and $\hat H$ that are not
part of the statement of Proposition \ref{Proposition: Operator Definition}, but which are nevertheless important
in the proofs of other results. Moreover, following-up on the Hilbert-space subtlety mentioned in Remark
\ref{Remark: Quaternion on R 1}, we deal with the real/complex and quaternion cases separately
throughout this section.

\subsection{The Cases $\mbb F=\mbb R,\mbb C$}
\label{Section: Operator Construction for R and C}

Assume first that $\mbb F=\mbb R,\mbb C$.
We begin by constructing the deterministic operator $H$.
For this purpose, for each $1\leq i\leq r$, recall the definition of
$\mc E_i$ and $D(\mc E_i)$ in Definition \ref{Definition: Deterministic Quadratic Forms}.
By \cite[Proposition 3.2]{GaudreauLamarreEJP}, the operators $H_i$ defined in \eqref{Equation: One-Dimensional Operators}
are the unique self-adjoint operators with dense domains $D(H_i)\subset D(\mc E_i)\subset L^2(I,\mbb F)$
such that $\langle f,H_ig\rangle=\mc E_i(f,g)$ for all $f,g\in D(H_i)$
(in \cite{GaudreauLamarreEJP}, the result is stated for operators acting on real-valued functions,
but the latter can be trivially extended to $\mbb C$ or $\mbb F$ by linearity).
Moreover, each $H_i$ has compact resolvent.

\begin{remark}
\label{Remark: Form Core for Hi}
For $1\leq i\leq r$, let $\mr{FC}_i$ denote the set of functions
$\phi\in C_0^\infty(I,\mbb F)$ that satisfy the following:
\begin{enumerate}
\item $\phi$ is supported inside $I$ in Case 1, as well as Cases 2 and 3
with Dirichlet boundary conditions (i.e., $\bar\al_i=-\infty$ in Case 2 and $\bar\al_i=\bar\be_i=-\infty$ in Case 3).
\item $\phi$ is supported inside $I$'s closure in Cases 2 and 3 with Robin boundary conditions (i.e., $\bar\al_i,\bar\be_i\in\mbb R$).
\item In Case 3, $\phi$ is supported in $[0,b)$ if $\bar\al_i\in\mbb R$ and $\bar\be_i=-\infty$,
and $\phi$ is supported in $(0,b]$ if $\bar\al_i=-\infty$ and $\bar\be_i\in\mbb R$.
\end{enumerate}
In \cite[Proposition 3.2]{GaudreauLamarreEJP}, it is also proved that $\mr{FC}_i\subset D(\mc E_i)$ is a form core for $\mc E_i$.
\end{remark}

\begin{remark}
\label{Remark: Hi is Trace Class}
In \cite[Theorem 5.4 and Lemma 5.13]{GaudreauLamarreEJP}, it is also proved that
\[\mr{Tr}\big[\mr e^{-tH_i}\big]=\sum_{k=1}^\infty\mr e^{-t\la_k(H_i)}<\infty.\]
\end{remark}

With this in hand, we can now define the operator
\begin{align}
\label{Equation: Diagonal Operator}
Hf(i,x)=-\tfrac12f''(i,x)+V(i,x)f(i,x),\qquad (i,x)\in\mc A
\end{align}
as the direct sum $H=H_1\oplus\cdots\oplus H_r$.
By \cite[Theorem XIII.85]{ReedSimon4}, $H$ is self-adjoint
on the domain $D(H)=D(H_1)\oplus\cdots\oplus D(H_r)$, which is dense in $L^2(\mc A,\mbb F)$.
Recalling the definition of $\mc E$ and $D(\mc E)$ at the end of Definition \ref{Definition: Deterministic Quadratic Forms},
this means that $D(H)\subset D(\mc E)$ and $\langle f,Hg\rangle_\mu=\mc E(f,g)$ for all $f,g\in D(H)$.
Finally, we note that since each $H_i$ has compact resolvent, the same is true for $H$:
Given that $H=H_1\oplus\cdots\oplus H_r$,
the eigenvalue-eigenfunction pairs $\big(\la_k(H),\psi_k(H)\big)_{k\in\mbb N}$ of $H$ are of the form
\begin{align}
\label{Equation: Oplus Eigenvalues}
\la_k(H)=\la_\ell(H_i)\qquad\text{and}\qquad \psi_k(H)=0\oplus\cdots\oplus0\oplus\psi_\ell(H_i)\oplus0\oplus\cdots\oplus 0,
\end{align}
for some $\ell\in\mbb N$.

\begin{remark}
\label{Remark: H is Trace Class}
If we combine Remark \ref{Remark: Hi is Trace Class} with \eqref{Equation: Oplus Eigenvalues}, then
we get that
\begin{align}
\label{Equation: H is Trace Class}
\mr{Tr}\big[\mr e^{-tH}\big]=\sum_{k=1}^\infty\mr e^{-t\la_k(H)}=\sum_{i=1}^r\sum_{k=1}^\infty\mr e^{-t\la_k(H_i)}<\infty.
\end{align}
\end{remark}

With the construction of $H$ completed, we are now ready to wrap up the proof of
Proposition \ref{Proposition: Operator Definition}. By Proposition \ref{Proposition: Operator Form-Bound},
almost surely, $\xi$ is an infinitesimally form-bounded perturbation of $\mc E$.
Thus, Proposition \ref{Proposition: Operator Definition}-(1) and -(2) follow directly from the
KLMN theorem (e.g., \cite[Theorem VIII.15]{ReedSimon1} and \cite[Theorem X.17]{ReedSimon2}), and the fact that $\hat H$ has compact
resolvent follows from the fact that $H$ has compact resolvent and standard variational estimates
(e.g., \cite[Theorem XIII.68]{ReedSimon4}).

\begin{remark}
\label{Remark: Form Core}
Another conclusion of the KLMN theorem is that any form core for $\mc E$
is also a form core for $\hat{\mc E}$ (e.g., \cite[Theorem 7.5.7]{SimonBook}). Thus,
following-up on Remark \ref{Remark: Form Core for Hi}, $\mr{FC}=\mr{FC}_1\oplus\cdots\oplus\mr{FC}_r$
is a form core for both $\mc E$ and $\hat{\mc E}$ (the latter almost surely).
\end{remark}

\subsection{The Case $\mbb F=\mbb H$}
\label{Section: Quaternion-Operator Theory 1}

Consider now the case $\mbb F=\mbb H$. The results in
\cite{ReedSimon4,ReedSimon2,ReedSimon1,SimonBook} that were cited in
Section \ref{Section: Operator Construction for R and C} require that
$H_i$, $H$, and $\hat H$ be defined on a standard real or complex
Hilbert space. For this purpose, we first assume
that $L^2(I,\mbb H)$ and $L^2(\mc A,\mbb H)$ are equipped with the
real inner products \eqref{Equation: Inner Product for H},
and that the bilinear forms for $H_i$ and $H$ are replaced by
\[\mc Q_i(f,g)=\Re\big(\mc E_i(f,g)\big)\qquad\text{and}\qquad\mc Q(f,g)=\Re\big(\mc E(f,g)\big);\]
the quadratic form
\[\hat{\mc E}(f,f)=\mc Q(f,f)+\Re\big(\xi(f,f)\big)=\mc E(f,f)+\xi(f,f)\]
does not need to be changed, as $\mc E(f,f)$ and $\xi(f,f)$ are always real.
The argument in Section \ref{Section: Operator Construction for R and C} then yields the conclusion of Proposition \ref{Proposition: Operator Definition},
the only difference being that, in the variational spectrum
\begin{align}
\label{Equation: Quaternion Variational}
\La_\ell(\hat H)=\inf_{f\in D(\mc E),~f\perp\Psi_1(\hat H),\ldots,\Psi_{\ell-1}(\hat H)}\frac{\hat{\mc E}(f,f)}{\|f\|_\mu^2},\qquad\ell=1,2,3,\ldots,
\end{align}
the eigenfunctions $\Psi_\ell(\hat H)$ form an orthonormal basis with respect to the standard
inner product $\prec\cdot,\cdot\succ_\mu$ instead of $\langle\cdot,\cdot\rangle_\mu$,
and we thus have
\begin{align}
\label{Equation: Semigroup via Spectral Expansion H}
\mr e^{-t\hat H}f=\sum_{\ell=1}^\infty\mr e^{-t\La_\ell(\hat H)}\prec\Psi_\ell(\hat H),f\succ_\mu\Psi_\ell(\hat H).
\end{align}

In order to recover the actual statement of Proposition \ref{Proposition: Operator Definition}
from this, we first note that each eigenvalue $\La\in\big(\La_\ell(\hat H)\big)_{\ell\in\mbb N}$ defined via \eqref{Equation: Quaternion Variational} has a multiplicity that is a multiple of $4$. This is because the corresponding eigenfunction $\Psi$
is part of a four-dimensional subspace spanned by the orthonormal (with respect to $\prec\cdot,\cdot\succ_\mu$; see \eqref{Equation: Orthogonality of l Psi}) basis $\{\Psi,\Psi\msf i,\Psi\msf j,\Psi\msf k\}$ (however, these basis elements
are not orthonormal with respect to $\langle\cdot,\cdot\rangle_\mu$). Up to reordering the
eigenfunctions $\big(\Psi_k(\hat H)\big)_{k\geq1}$ for repeated eigenvalues if necessary, we can assume that
for every $k\geq1$, one has
\begin{align*}
\Psi_{4(k-1)+2}(\hat H)=\Psi_{4(k-1)+1}(\hat H)\msf i,\\
\Psi_{4(k-1)+3}(\hat H)=\Psi_{4(k-1)+1}(\hat H)\msf j,\\
\Psi_{4(k-1)+4}(\hat H)=\Psi_{4(k-1)+1}(\hat H)\msf k.
\end{align*}
That is, the eigenfunctions are chosen in increments of  four consecutive
eigenfunctions within the same space of the form $\Psi,\Psi\msf i,\Psi\msf j,\Psi\msf k$,
in that particular order.

With this in hand,
we claim that the statement of Proposition \ref{Proposition: Operator Definition}
holds with the eigenvalues and eigenfunctions
\[\la_k(\hat H)=\La_{4(k-1)+1}(\hat H)\text{ and }\psi_k(\hat H)=\Psi_{4(k-1)+1}(\hat H)\qquad k=1,2,3,\ldots.\]
To see this, notice that for any $f,g\in L^2(\mc A,\mbb F)$, we can write
\[\langle f,g\rangle_\mu=\prec f,g\succ_\mu+\prec f\msf i,g\succ_\mu\msf i+\prec f\msf j,g\succ_\mu\msf j+\prec f\msf k,g\succ_\mu\msf k.\]
Indeed, for any $\ell\in\{1,\msf i,\msf j,\msf k\}$, one has
\begin{align}
\label{Equation: Orthogonality of l Psi}
\prec f\ell,g\succ_\mu=\Re\big(\langle f\ell,g\rangle_\mu\big)=\Re\big(\ell^*\langle f\msf ,g\rangle_\mu\big),
\end{align}
and by the properties of quaternion multiplication, the real part of
$\ell^*\langle f\msf ,g\rangle_\mu$ coincides with the scalar multiplying $\ell$
in $\langle f\msf ,g\rangle_\mu$. Given how we have ordered the $\Psi_{\ell}(\hat H)$
in the previous paragraph, this readily implies that the $\psi_k(\hat H)$ are an orthonormal
basis with respect to $\langle\cdot,\cdot\rangle_\mu$.
Moreover, rewriting \eqref{Equation: Semigroup via Spectral Expansion H}
using $\la_k(\hat H)$ and $\psi_k(\hat H)$ yields
\begin{multline*}
\mr e^{-t\hat H}f=\sum_{k=1}^\infty\mr e^{-t\la_k(\hat H)}
\left(\sum_{\ell\in\{1,\msf i,\msf j,\msf k\}}\prec\psi_k(\hat H)\ell,f\succ_\mu\psi_k(\hat H)\ell\right)\\
=\sum_{k=1}^\infty\mr e^{-t\la_k(\hat H)}
\psi_k(\hat H)\left(\sum_{\ell\in\{1,\msf i,\msf j,\msf k\}}\prec\psi_k(\hat H)\ell,f\succ_\mu\ell\right)
=\sum_{k=1}^\infty\mr e^{-t\la_k(\hat H)}\psi_k(\hat H)\big\langle\psi_k(\hat H),f\big\rangle_\mu
\end{multline*}
as desired.

\begin{remark}
Following-up on Remark \ref{Remark: Quaternion on R 1}, we see from
the analysis in this section that
the choice of inner product influences the trace of $\hat H$'s semigroup as follows:
\begin{align}
\label{Equation: Trace Multiple of 4 1}
\sum_{k=1}^\infty\mr e^{-t\La_k(\hat H)}=4\sum_{k=1}^\infty\mr e^{-t\la_k(\hat H)}=4\mr{Tr}\big[\mr e^{-t\hat H}\big].
\end{align}\end{remark}

\subsection{Proposition \ref{Proposition: Operator Form-Bound}}

At this point, the only part of Proposition \ref{Proposition: Operator Definition} that remains to
be proved is the technical estimate in Proposition \ref{Proposition: Operator Form-Bound}.
Given that the ideas used to prove the latter are also needed in several other results in the paper,
we postpone the proof of Proposition \ref{Proposition: Operator Form-Bound} to Section \ref{Section: Main Proposition 6}
below to make for a more streamlined treatment.

\section{Proof of Theorem \ref{Theorem: Regular FK}: Regular Feynman-Kac Formula}
\label{Section: Regular FK}

Throughout Section \ref{Section: Regular FK}, we assume without further mention that
Assumptions \ref{Assumption: Domain}, \ref{Assumption: Boundary}, and \ref{Assumption: Potential} hold,
and moreover that each entry $\xi_{i,j}$ is a regular noise.
This section is organized as follows: In Section \ref{Section: Roadmap for Regular FK and Smooth Trace Moments},
we provide an overview of the proof of Theorem \ref{Theorem: Regular FK}, wherein we state a number of
technical results without proof. Then, we provide the proofs of the technical
results in question in Sections \ref{Section: A Technical Lemma},
\ref{Section: Basic Properties of the Feynman-Kac Kernel}, and \ref{Section: Proof of Generator}.

\subsection{Proof Outline}
\label{Section: Roadmap for Regular FK and Smooth Trace Moments}

Recall the definition of $\hat K(t)$ in \eqref{Equation: Regular FK Kernel}.
The proof of Theorem \ref{Theorem: Regular FK} relies on the following two
technical results:

\begin{proposition}
\label{Proposition: Generic Kernel Properties}
Under the assumptions of Theorem \ref{Theorem: Regular FK},
almost surely,
$\big(\hat K(t)\big)_{t\geq0}$ is a strongly-continuous trace-class symmetric
semigroup on $L^2(\mc A,\mbb F)$; that is:
\begin{enumerate}
\item $\hat K(t;a,b)=\hat K(t;b,a)^*$ for every $t>0$ and $a,b\in\mc A$.
\item For every $s,t>0$, and $a,b\in\mc A$,
\[\int_{\mc A} \hat K(s;a,c)\hat K(t;c,b)\d\mu(c)=\hat K(s+t;a,b).\]
\item Letting $\|\cdot\|_{\mr{HS}}$ denote the Hilbert-Schmidt norm, for every $t>0$,
\[\|\hat K(t)\|_{\mr{HS}}^2=\int_{\mc A^2}|\hat K(t;a,b)|^2\d\mu(a)\dd\mu(b)=\int_{\mc A}\hat K(2t;a,a)\d\mu(a)<\infty.\]
\item For every $f\in L^2(\mc A,\mbb F)$, one has $\displaystyle\lim_{t\to0}\|\hat K(t)f-f\|_2=0.$
\end{enumerate}
\end{proposition}

\begin{proposition}
\label{Proposition: Generator Calculation}
Under the assumptions of Theorem \ref{Theorem: Regular FK},
almost surely,
\begin{align}
\label{Equation: Infinitesimal Generator Calculation}
\lim_{t\to0}\left\|t^{-1}\big(f-\hat K(t)f\big)-\hat Hf\right\|_2=0\qquad\text{for every $f\in D(\hat H)$}.
\end{align}
\end{proposition}

Before discussing these propositions, we use them to prove Theorem \ref{Theorem: Regular FK}:

\subsubsection{Cases $\mbb F=\mbb R,\mbb C$}
\label{Section: Generator for R C}

Suppose first that $\mbb F=\mbb R$ or $\mbb C$.
Consider an outcome in the probability-one event
where the conclusions of Propositions \ref{Proposition: Generic Kernel Properties}
and \ref{Proposition: Generator Calculation} hold.
By Proposition \ref{Proposition: Generic Kernel Properties}-(4),
there exists
a closed operator $G_K$ (called $\hat K(t)$'s generator) with dense domain
\begin{align}
\label{Equation: Generic Generator Domain}
D(G_K)=\left\{f\in L^2(\mc A,\mbb F):\lim_{t\to0}t^{-1}\big(f(a)-\hat K(t)f(a)\big)\text{ exists in }L^2(\mc A,\mbb F)\right\}
\end{align}
such that
\begin{align}
\label{Equation: Generic Generator}
\lim_{t\to0}\big\|t^{-1}\big(f-\hat K(t)f\big)-G_Kf\big\|_2=0\qquad\text{for every }f\in D(G_K)
\end{align}
(e.g., \cite[Chapter II, Theorem 1.4]{EngelNagel}). In other words, $\mr e^{-tG_K}=\hat K(t)$.
Our goal is to prove that $G_K=\hat H$.
This can be done with Proposition \ref{Proposition: Generator Calculation}:
First, we note that Proposition \ref{Proposition: Generic Kernel Properties}-(1) implies that $G_K$
is symmetric.
Then, given \eqref{Equation: Generic Generator Domain}
and \eqref{Equation: Generic Generator}, the limit \eqref{Equation: Infinitesimal Generator Calculation} implies that
$D(\hat H)\subset D(G_K)$ and that $G_K$ restricted to $D(\hat H)$ is equal to $\hat H$.
Since $\hat H$ is self-adjoint and $G_K$ is symmetric, $D(G_K)\subset D(G_K^*)\subset D(\hat H^*)=D(\hat H)$;
hence $\hat H=G_K$. This proves that $\mr e^{-t \hat H}=\hat K(t)$.

It now remains to prove the trace formula
$\mr{Tr}\big[\mr e^{-t \hat H}\big]=\int_{\mc A}\hat K(t;a,a)\d\mu(a)<\infty$.
For this purpose,
we note that
\[\mr{Tr}\big[\mr e^{-t\hat H}\big]=\sum_{k=1}^\infty\mr e^{-t\la_k(\hat H)}=\sum_{k=1}^\infty\big(\mr e^{-(t/2)\la_k(\hat H)}\big)^2
=\|\mr e^{-(t/2)\hat H}\|_{\mr{HS}}^2
=\|\hat K(t/2)\|_{\mr{HS}}^2,\]
where the last equality follows from $\mr e^{-t \hat H}=\hat K(t)$.
Then,
Proposition \ref{Proposition: Generic Kernel Properties}-(3) gives
\[\mr{Tr}\big[\mr e^{-t\hat H}\big]=\|\hat K(t/2)\|_{\mr{HS}}^2=\int_{\mc A^2}|K_{t/2}(a,b)|^2\d\mu(a)\dd\mu(b)
=\int_{\mc A} \hat K(t;a,a)\d\mu(a),\]
as desired.

\subsubsection{The Case $\mbb F=\mbb H$}
\label{Section: Quaternion-Operator Theory 2}

In similar fashion to the construction of $\hat H$,
several elements of the proof in the previous section require
the operators under consideration to be self-adjoint on a standard
Hilbert space on $\mbb R$ or $\mbb C$. The same argument can
apply to the case where $\mbb F=\mbb H$, provided we reframe
$\hat K(t)$ as an integral operator with respect to $\prec\cdot,\cdot\succ_\mu$, which acts on functions taking values in $\mbb R^4$.
For this purpose, let us expand the components of $\hat K(t)$ and any $f\in L^2(\mc A,\mbb H)$ in $\mbb R^4$ as follows:
\[\hat K(t;a,b)=\hat K_{1}(t;a,b)+\hat K_{\msf i}(t;a,b)\msf i+\hat K_{\msf j}(t;a,b)\msf j+\hat K_{\msf k}(t;a,b)\msf k,\qquad a,b\in\mc A\]
\[f(a)=f_{1}(a)+f_{2}(a)\msf i+f_{3}(a)\msf j+f_{4}(a)\msf k,\qquad a\in\mc A\]
where $f_{\ell}$ and $\hat K_{\msf l}(t)$ are real-valued for all $\ell\in\{1,2,3,4\}$ and $\msf l\in\{1,\msf i,\msf j,\msf k\}$.
Consider the $4\times 4$ $\mbb R$-matrix-valued kernel $J(t)$ defined as follows: 
\[J(t;a,b)=\left[\begin{array}{cccc}
\hat K_{1}(t;a,b)&-\hat K_{\msf i}(t;a,b)&-\hat K_{\msf j}(t;a,b)&-\hat K_{\msf k}(t;a,b)\\
\hat K_{\msf i}(t;a,b)&\hat K_{1}(t;a,b)&-\hat K_{\msf k}(t;a,b)&\hat K_{\msf j}(t;a,b)\\
\hat K_{\msf j}(t;a,b)&\hat K_{\msf k}(t;a,b)&\hat K_{1}(t;a,b)&-\hat K_{\msf i}(t;a,b)\\
\hat K_{\msf k}(t;a,b)&-\hat K_{\msf j}(t;a,b)&\hat K_{\msf i}(t;a,b)&\hat K_{1}(t;a,b)
\end{array}\right].\]
If we denote the entries of $g(a)=\hat K(t)f(a)=\int \hat K(t;a,b)f(b)\d\mu(b)$
in $\mbb R^4$ as $g_\ell(a)$, i.e.,
\[g(a)=g_{1}(a)+g_{2}(a)\msf i+g_{3}(a)\msf j+g_{4}(a)\msf k,\qquad a\in\mc A,\]
then by definition of quaternion product, we have that
\[g(a)_\ell=\int_{\mc A}\sum_{m=1}^4J(t;a,b)_{\ell,m}f(b)_m\d\mu(b)
=\prec J(t;a,\cdot)_{\ell,\cdot},f\succ_\mu\]
for $1\leq\ell\leq 4$, where we view $J(t;a,b)_{\ell,\cdot}$ as the
quaternion/vector in $\mbb R^4$ whose components are given by the $\ell^{\mr{th}}$ row of the
matrix $J(t;a,b)$.

With this in hand, we can now apply the same argument in Section
\ref{Section: Generator for R C} to the $4\times 4$ matrix kernel $J(t;a,b)_{\ell,m}$,
which yields that $J(t)$ coincides with the operator $\mr e^{-t\hat H}$ defined
via \eqref{Equation: Semigroup via Spectral Expansion H},
as well as the trace formula
\begin{align}
\label{Equation: Trace Formula in H-R4}
\sum_{k=1}^\infty\mr e^{-t\La_k(\hat H)}=\int_{\mc A}\mr{Tr}\big[J(t;a,a)\big]\d\mu(a)
\end{align}
(note that $\hat K(t;a,b)=\hat K(t;b,a)^*$ means that $J(t;a,b)_{\ell,m}=J(t;b,a)_{m,\ell}$).
If we pull this back to $L^2(\mc A,\mbb H)$ with $\langle\cdot,\cdot\rangle_\mu$,
this implies that $\hat K(t)=\mr e^{-t\hat H}$, with the semigroup now defined via \eqref{Equation: Semigroup via Spectral Expansion}.
The trace formula can then be obtained by combining \eqref{Equation: Trace Multiple of 4 1} and \eqref{Equation: Trace Formula in H-R4} with
the equalities
\begin{align}
\label{Equation: Trace Multiple of 4 2}
\int_{\mc A}\mr{Tr}\big[J(t;a,a)\big]\d\mu(a)=4\int_{\mc A}\hat K_{1}(t;a,a)\d\mu(a)=4\int_{\mc A}\hat K(t;a,a)\d\mu(a),
\end{align}
 where the last equality comes from $\hat K(t;a,a)=\hat K(t;a,a)^*$, i.e., Proposition \ref{Proposition: Generic Kernel Properties}-(1).

\subsubsection{Propositions \ref{Proposition: Generic Kernel Properties}
and \ref{Proposition: Generator Calculation}}
\label{Section: Generator Calculation Roadmap}

At this point, the only elements in the proof of Theorem \ref{Theorem: Regular FK}
that remain to be established are Propositions \ref{Proposition: Generic Kernel Properties}
and \ref{Proposition: Generator Calculation}.
The nontrivial elements of the proof of Proposition \ref{Proposition: Generic Kernel Properties} are mostly technical
and can be found in Section
\ref{Section: Basic Properties of the Feynman-Kac Kernel}.
The proof of Proposition \ref{Proposition: Generator Calculation},
which we provide in full in Section \ref{Section: Proof of Generator}, has the following structure:
Let $f\in D(\hat H)$.
For any $a\in\mc A$, we can write
\[t^{-1}\big(f(a)-\hat K(t)f(a)\big)=\mf G_0(t;a)+\mf G_1(t;a)+\mf G_2(t;a),\]
where we define the $\mf G_k$ by adding an indicator fixing the value of $N(t)$:
\begin{align}
\nonumber
\mf G_0(t;a)&=t^{-1}\left(f(a)-\mbf E_A\left[\mf n_t(Q,A^a)\mr e^{-\int_0^t S(A^a(s))\d s+\mf B_t(A^a)}f\big(A^a(t)\big)\mbf 1_{\{N(t)=0\}}\right]\right),\\
\nonumber
\mf G_1(t;a)&=t^{-1}\left(-\mbf E_A\left[\mf n_t(Q,A^a)\mr e^{-\int_0^t S(A^a(s))\d s+\mf B_t(A^a)}f\big(A^a(t)\big)\mbf 1_{\{N(t)=1\}}\right]\right),\\
\label{Equation: Generator 3.0}
\mf G_2(t;a)&=t^{-1}\left(-\mbf E_A\left[\mf n_t(Q,A^a)\mr e^{-\int_0^t S(A^a(s))\d s+\mf B_t(A^a)}f\big(A^a(t)\big)\mbf 1_{\{N(t)\geq2\}}\right]\right).
\end{align}
With this in hand, the proof of Proposition \ref{Proposition: Generator Calculation} is split into three steps:
\begin{align}
\label{Equation: Generator 1}
&\lim_{t\to0}\|\mf G_0(t;\cdot)-(H+\mr{diag}(\xi))f\|_2=0,\\
\label{Equation: Generator 2}
&\lim_{t\to0}\|\mf G_1(t;\cdot)-Qf\|_2=0,\\
\label{Equation: Generator 3}
&\lim_{t\to0}\|\mf G_2(t;\cdot)\|_2=0,
\end{align}
where we recall that $H=-\frac12\De+V$
is defined in \eqref{Equation: Diagonal Operator},
$\mr{diag}(\xi)$ denotes the direct sum $\xi_{1,1}\oplus\cdots\oplus\xi_{r,r}$,
and $Qf$ is defined as
\[Qf(i,x)=\sum_{j\neq i}\xi_{i,j}(x)f(j,x)=\sum_{j\neq i}\xi_{i,j}(x)f(j,x).\]
The rationale behind these three steps can be explained by the following remarks:

\begin{remark}
\label{Remark: Generator 1}
Regarding \eqref{Equation: Generator 1}, write
$a=(i,x)$. If $N(t)=0$, then the process $U^i$ is equal to $i$
in the time interval $[0,t]$. On this event, we notice that
\begin{align*}
&\bullet f\big(A^a(t)\big)=f\big(i,Z^x(t)\big);&&
\bullet\mf n_t(Q,A^a)=\mr e^{(r-1)t};\\
&\bullet\mf B_t(A^a)=\mf B^i_t(Z^x);&&
\bullet\int_0^tS\big(A^a(s)\big)\d s=\int_0^tS\big(i,Z^x(s)\big)\d s;
\end{align*}
where we define
\[\mf B^i_t(Z^x)=\begin{cases}
0&\text{Case 1,}
\vspace{5pt}\\
\displaystyle\bar\al_i\mf L^{0}_t(X^x)&\text{Case 2,}
\vspace{5pt}\\
\displaystyle\bar\al_i\mf L^{0}_t(Y^x)+\bar\be_i\mf L^{\vartheta}_t(Y^x)&\text{Case 3.}
\end{cases}\]
Therefore, we can write
\begin{multline*}
\mf G_0\big(t;(i,x)\big)=t^{-1}\left(f(i,x)-\mbf E_Z\left[\mr e^{(r-1)t}\mr e^{-\int_0^t S(i,Z^x(s))\d s
-\mf B^i_t(Z^x)}f\big(i,Z^x(t)\big)\right]\mbf P[N(t)=0]\right)\\
=t^{-1}\left(f(i,x)-\mbf E_Z\left[\mr e^{-\int_0^t S(i,Z^x(s))\d s
-\mf B^i_t(Z^x)}f\big(i,Z^x(t)\big)\right]\right),
\end{multline*}
where the last equality follows from the fact that $N(t)$ is Poisson with rate $(r-1)t$.
Recalling the definition of the one-dimensional operators $H_i$ in
\eqref{Equation: One-Dimensional Operators},
as well as the Feynman-Kac formula for $H_i+\xi_{i,i}$ in \eqref{Equation: Classical Feynman-Kac},
this can be further simplified to
\[\mf G_0\big(t;(i,x)\big)=t^{-1}\big(f(i,x)-\mr e^{-t(H_i+\xi_{i,i})}f(i,x)\big).\]
In addition to that, given that $H+\mr{diag}(\xi)=(H_1+\xi_{1,1})\oplus\cdots\oplus (H_r+\xi_{r,r})$, we note that
for every $1\leq i\leq r$,
\[\big(H+\mr{diag}(\xi)\big)f(i,x)=\big(H_i+\xi_{i,i}\big)f(i,x).\]
Therefore,
\begin{multline}
\label{Equation: Generator 1.1}
\|\mf G_0(t;\cdot)-(H+\mr{diag}(\xi))f\|_2^2=\sum_{i=1}^r\int_I\left|\mf G_0\big(t;(i,x)\big)-(H+\mr{diag}(\xi))f(i,x)\right|^2\d x
\\=\sum_{i=1}^r\int_I\left|t^{-1}\big(f(i,x)-\mr e^{-t(H_i+\xi_{i,i})}f(i,x)\big)-(H_i+\xi_{i,i})f(i,x)\right|^2\d x.
\end{multline}
The right-hand side of the above converges to zero thanks to the
one-dimensional Feynman-Kac formula \eqref{Equation: Classical Feynman-Kac};
see Section \ref{Section: Generator Part 1} for detailed references.
\end{remark}

\begin{remark}
\label{Remark: Generator 2}
Regarding \eqref{Equation: Generator 2},
denote the events
\[\mc J_j(t)=\{N(t)=1\}\cap\{U(\tau_1)=j\},\qquad t>0,~1\leq j\leq r.\]
With this in hand, we write $\mbf 1_{\{N(t)=1\}}=\mbf 1_{\mc J_1(t)}+\cdots+\mbf 1_{\mc J_r(t)}$,
noting that the $\mc J_j(t)$ are disjoint and $N(t)$ and $U(\tau_1)$ are independent.
On the event $\mc J_j(t)$, we
have the simplification
$\mf n_t(Q,A^{(i,x)})=-\mr e^{(r-1)t}\xi_{i,j}\big(Z^x(\tau_1)\big).$
Therefore, we can write
\[\mf G_2\big(t;(i,x)\big)=\sum_{j\neq i}\frac{\mr e^{(r-1)t}}t\mbf E_A\left[\xi_{i,j}\big(Z^x(\tau_1)\big)\mr e^{-\int_0^t S(A^{(i,x)}(s))\dd s+\mf B_t(A^{(i,x)})}f\big(j,Z^x(t)\big)\mbf 1_{\mc J_j(t)}\right].\]
Moreover, given that
\[\mbf P[\mc J_j(t)]=\mbf P[N(t)=1]\cdot\mbf P[U(\tau_1)=j]=(r-1)t\mr e^{-(r-1)t}\cdot\frac{1}{r-1}=\left(\frac{\mr e^{(r-1)t}}t\right)^{-1},\]
we can further simplify this expression to
\begin{align}
\label{Equation: Generator 2.1}
\mf G_1\big(t;(i,x)\big)=\sum_{j\neq i}\mbf E_A\left[\xi_{i,j}\big(Z(\tau_1)\big)\mr e^{-\int_0^t S(A^{(i,x)}(s))\dd s+\mf B_t(A^{(i,x)})}f\big(j,Z^x(t)\big)~\bigg|~\mc J_j(t)\right].
\end{align}
By continuity of $\xi_{i,j}$ and $Z$ and $0\leq\tau_1\leq t$, we expect that \eqref{Equation: Generator 2.1} should converge to
\[\sum_{j\neq i}\xi_{i,j}\big(Z^x(0)\big)f\big(j,Z^x(0)\big)=\sum_{j\neq i}\xi_{i,j}(x)f(j,x)=Qf(i,x)\]
as $t\to0$;
see Section \ref{Section: Generator Part 2} for the technical details.
\end{remark}

\begin{remark}
\label{Remark: Generator 3}
The proof of \eqref{Equation: Generator 3}
relies on the fact that $\mbf P[N(t)\geq2]=O(t^2)$ as $t\to0$,
which makes $\mf G_2$ vanish despite the $t^{-1}$ term
in \eqref{Equation: Generator 3.0}.
See Section \ref{Section: Generator Part 3} for details.
\end{remark}

\subsection{Growth Estimates of $S$ and $Q$}
\label{Section: A Technical Lemma}

We begin with a preliminary estimate on the growth
of $S$ and $Q$:

\begin{lemma}
\label{Lemma: Growth}
Let $\mf a>0$ be as in Assumption \ref{Assumption: Potential},
and let $\mf b>1/2$ be arbitrary.
The following two conditions hold almost surely:
\begin{enumerate}
\item For all $1\leq i\leq r$, $S(i,\cdot)$ is bounded below, locally integrable on $I$'s closure, and
in Cases 1 and 2,
\begin{align}
\label{Equation: S Growth Lower Bound}
\lim_{|x|\to\infty}\frac{S(i,x)}{|x|^{\mf a}}=\infty.
\end{align}
\item In Cases 1 and 2, for all $1\leq i<j\leq r$, 
\begin{align}
\label{Equation: Q Growth Upper Bound}
\lim_{|x|\to\infty}\frac{|\xi_{i,j}(x)|}{(\log|x|)^\mf b}=0.
\end{align}
\end{enumerate}
\end{lemma}
\begin{proof}
In Case 3, we need only prove that $S(i,x)=V(i,x)+\xi_{i,i}(x)$ are bounded below and locally integrable.
This follows from the fact that $V$ already satisfies these conditions (Assumption \ref{Assumption: Potential}),
and that $\xi_{i,i}$ are continuous real-valued functions on a bounded interval.

Consider then Cases 1 and 2. First, we note that by a Borel-Cantelli argument
(e.g., \cite[Corollary B.2]{GaudreauLamarreEJP}),
the fact that $\xi_{i,i}$ and $\xi_{i,j}=\xi_{j,i}^*$ are continuous and stationary
implies that there exists a finite random $\nu>0$ such that, almost surely,
\[|\xi_{i,i}(x)|,|\xi_{i,j}(x)|\leq\nu\sqrt{\log(2+|x|)},\qquad x\in I,\qquad 1\leq i\neq j\leq r.\]
This immediately implies \eqref{Equation: Q Growth Upper Bound} since $\mf b>1/2$.
To prove that
$S(i,\cdot)=V(i,\cdot)+\xi_{i,i}$ are bounded below and locally integrable and
\eqref{Equation: S Growth Lower Bound},
we simply notice that $V(i,\cdot)$ is already assumed to grow faster than $|x|^\mf a$ by Assumption \ref{Assumption: Potential}
(which itself grows faster than $\sqrt{\log|x|}$),
and that $\xi_{i,i}$ is continuous.
\end{proof}

\subsection{Proof of Proposition \ref{Proposition: Generic Kernel Properties}}
\label{Section: Basic Properties of the Feynman-Kac Kernel}

We split the proof into four steps, corresponding to Proposition \ref{Proposition: Generic Kernel Properties}-(1) to -(4).
Throughout the proof of Proposition \ref{Proposition: Generic Kernel Properties},
we assume without further mention that we are working on the probability-one event (with respect to
the randomness in $\xi$) where the conclusions of Lemma \ref{Lemma: Growth} hold.

\subsubsection{Proof of Proposition \ref{Proposition: Generic Kernel Properties}-(1)}

The symmetry property follows from applying a time reversal to
the path $A$, in which case a bridge from $a$ to $b$ becomes a bridge from
$b$ to $a$: On the one hand, the boundary local time term $\mf B_t(A)$ and the area $\int_0^t S\big(A(s)\big)\d s$
are invariant with respect to time reversal, hence these terms remain the same under
this operation. On the other hand, since a time reversal inverts the order of the jump times $\tau_k$
and the individual jumps $J_k$,
under this operation the product term $\mf n_t(Q,A)$ is simply conjugated (since $Q$ is
self-adjoint).

\subsubsection{Proof of Proposition \ref{Proposition: Generic Kernel Properties}-(2)}

Given $a,b,c\in\mc A$, if we condition the process $A^{a,b}_{s+t}$ on
the event $\{A^{a,b}_{s+t}(s)=c\}$,
then the path segments
\[\big(A^{a,b}_{s+t}(u):0\leq u\leq s\big)\qquad\text{and}\qquad\big(A^{a,b}_{s+t}(u):s\leq u\leq s+t\big)\]
are independent and have respective distributions $A^{a,c}_s$ and $A^{c,b}_{t}$.
Next, recalling that $A^{(a,c),(c,b)}_{(s,t)}$ denotes the path on the time interval $[0,s+t]$ obtained by concatenating
the independent paths of $A^{a,c}_s$ and $A^{c,b}_{t}$, we notice that
\[\int_0^s S\big(A^{a,c}_s(	u)\big)\d u+\int_0^t S\big(A^{c,b}_t(s)\big)\d u=\int_0^{s+t} S\big(A^{(a,c),(c,b)}_{(s,t)}(u)\big)\d u\]
and
\[\mf B_s(A^{a,c}_s)+\mf B_t(A^{c,b}_t)=\mf B_{s+t}(A^{(a,c),(c,b)}_{(s,t)}),\]
as well as
\[\mf n_s(Q,A^{a,c}_s)\mf n_t(Q,A^{c,b}_t)=\mf n_{s+t}(Q,A^{(a,c),(c,b)}_{(s,t)}).\]
With this, Proposition \ref{Proposition: Generic Kernel Properties}-(2) is an immediate consequence of the
law of total expectation, i.e., for any functional $F$, one has
\begin{align*}
&\int_\mc A\Pi_A(s;a,c)\Pi_A(t;c,b)\mbf E_A\left[F\left(A^{(a,c),(c,b)}_{(s,t)}\right)\right]\d\mu(c)\\
&=\Pi_A(s+t;a,b)\int_\mc A\frac{\Pi_A(s;a,c)\Pi_A(t;c,b)}{\Pi_A(s+t;a,b)}\mbf E_A\left[F\left(A^{a,b}_{s+t}\right)\Big|A^{a,b}_{s+t}(s)=c\right]\d\mu(c)\\
&=\Pi_A(s+t;a,b)\mbf E_A\Big[F\big(A^{a,b}_{s+t}\big)\Big].
\end{align*}

\subsubsection{Proof of Proposition \ref{Proposition: Generic Kernel Properties}-(3)}

We first prove that $\|\hat K(t)\|_{\mr{HS}}^2<\infty$.
We can write
\[\|\hat K(t)\|_{\mr{HS}}^2=\int_{\mc A^2}\Pi_A(t;a,b)^2\left|\mbf E_A\left[\mf n_t(Q,A^{a,b}_t)\mr e^{-\int_0^t S(A^{a,b}_t(s))\d s+\mf B_t(A^{a,b}_t)}\right]\right|^2\dd \mu(a)\dd\mu(b).\]
By Jensen's inequality,
\begin{align}
\label{Equation: HS 1}
\|\hat K(t)\|_{\mr{HS}}^2\leq\int_{\mc A^2}\Pi_A(t;a,b)^2\mbf E_A\left[|\mf n_t(Q,A^{a,b}_t)|\mr e^{-\int_0^t S(A^{a,b}_t(s))\d s+\mf B_t(A^{a,b}_t)}\right]^2\dd \mu(a)\dd\mu(b).
\end{align}
By arguing in the same way as in part (1) of this
lemma, but replacing $\mf n_t(Q,A)$ by its modulus, we
note that
\begin{multline*}
\Pi_A(t;a,b)\mbf E_A\left[|\mf n_t(Q,A^{a,b}_t)|\mr e^{-\int_0^t S(A^{a,b}_t(s))\d s+\mf B_t(A^{a,b}_t)}\right]\\
=\Pi_A(t;b,a)\mbf E_A\left[|\mf n_t(Q,A^{b,a}_t)|\mr e^{-\int_0^t S(A^{b,a}_t(s))\d s+\mf B_t(A^{b,a}_t)}\right].
\end{multline*}
If we combine this with a use of Tonelli's theorem in the $\dd\mu(b)$ integral in \eqref{Equation: HS 1},
the same argument as in part (2) of this lemma then implies that
\begin{align}
\label{Equation: HS 1.5}
\|\hat K(t)\|_{\mr{HS}}^2\leq\int_{\mc A}\Pi_A(2t;a,a)\mbf E_A\left[|\mf n_{2t}(Q,A^{a,a}_{2t})|\mr e^{-\int_0^{2t} S(A^{a,a}_{2t}(s))\d s+\mf B_{2t}(A^{a,a}_{2t})}\right]\d \mu(a).
\end{align}
In order to control the terms inside the expectation in \eqref{Equation: HS 1.5},
we use the following estimates, which we will use several more times in this section:

\begin{proposition}
\label{Proposition: F-K Integrability Bound}
In Cases 1 or 2
(i.e., $Z=B,X$), 
we let $W^0$ be a standard Brownian motion started at zero coupled
with $Z$ in such a way that
\begin{align}
\label{Equation: Brownian Initial Point Coupling}
Z(t)=\begin{cases}
B(0)+W^0(t)&\text{Case 1,}\\
|X(0)+W^0(t)|&\text{Case 2.}
\end{cases}
\end{align}
For any $t>0$ and $p\geq1$, there exists
constants $\ka,\nu>0$ and $0<\bar{\mf a}\leq 1$ such that
\begin{multline}
\label{Equation: F-K Integrability Bound}
\mbf E_{N}\left[|\mf n_t(Q,A)|^p\mr e^{-p\int_0^t S(A(s))\d s+p\mf B_t(A)}\right]\\
\leq
\begin{cases}
\displaystyle
\exp\left(-\ka t|B(0)|^{\bar{\mf a}}+\ka t\sup_{0\leq s\leq t}|W^0(s)|^{\bar{\mf a}}+\nu t\right)&\text{Case 1,}
\vspace{5pt}\\
\displaystyle
\exp\left(-\ka t|X(0)|^{\bar{\mf a}}+\ka t\sup_{0\leq s\leq t}|W^0(s)|^{\bar{\mf a}}+\nu t+p\bar\al_\mr m\mf L^0_t(X)\right)&\text{Case 2,}	
\vspace{5pt}\\
\displaystyle
\exp\left(\nu t+p\bar\al_\mr m\mf L^0_t(Y)+p\bar\be_\mr m\mf L^{\vartheta}_t(Y)\right)&\text{Case 3.}
\end{cases}
\end{multline}
where $\mbf E_{N}[\cdot]$ denotes the expectation with respect to
the unconditioned Poisson process $N=\big(N(t):t\geq0\big)$ only
(conditional on all other sources of randomness),
and
\[\bar\al_\mr m=\max_{1\leq i\leq r}\bar\al_i\qquad\text{and}\qquad\bar\be_\mr m=\max_{1\leq i\leq r}\bar\be_i.\]
\end{proposition}
\begin{proof}
First, we note that in Cases 2 and 3, it follows from
\eqref{Equation: Regular Boundary Local Time},
\eqref{Equation: Multivariate Boundary Local Time 1},
and \eqref{Equation: Multivariate Boundary Local Time 2}
that
\begin{align}
\label{Equation: F-K Integrability Bound 1}
p\mf B_t(A)\leq
\begin{cases}
p\bar\al_\mr m\mf L^0_t(X)&\text{Case 2},\\
p\bar\al_\mr m\mf L^0_t(Y)+p\bar\be_\mr m\mf L^{\vartheta}_t(Y)&\text{Case 3}.
\end{cases}
\end{align}
Since the terms on the right-hand side of \eqref{Equation: F-K Integrability Bound 1}
do not depend on $N$, they can be pulled out of the expectation $\mbf E_N[\cdot]$;
this explains the presence of the local time terms on the right-hand side of \eqref{Equation: F-K Integrability Bound}
in Cases 2 and 3.
Thus, it only remains to control the expression $\mbf E_{N}\left[|\mf n_t(Q,A)|^p\mr e^{-p\int_0^t S(A(s))\d s}\right]$.

Consider first Cases 1 and 2.
Let $\bar{\mf a}=\min\{1,\mf a\}$.
By \eqref{Equation: S Growth Lower Bound}, we get
that for every $\ka>0$, there exists some $\nu_0\in\mbb R$ such that
\[\min_{1\leq i\leq r}S(i,x)\geq\tfrac{2\ka}{p}|x|^{\bar{\mf a}}-\tfrac{\nu_0}{p}.\]
Thus,
\begin{align}
\label{Equation: F-K Integrability Bound 2}
\mbf E_{N}\left[|\mf n_t(Q,A)|^p\mr e^{-p\int_0^t S(A(s))\d s}\right]
\leq
\mbf E_{N}\big[|\mf n_t(Q,A)|^p\big]\mr e^{-2\ka\int_0^t|Z(s)|^{\bar{\mf a}}\d s+\nu_0t},
\end{align}
where only the term $|\mf n_t(Q,A)|^p$ remains in the expectation because
the other term no longer depends on $N$.
Next,
by \eqref{Equation: Q Growth Upper Bound},
we can find $\nu_1>0$ such that
\[\max_{1\leq i<j\leq r}|\xi_{i,j}(x)|^p\leq\frac{\ka}{r-1}|x|^{\bar{\mf a}}+\nu_1.\]
Denoting $\eta(x)=\frac{\ka}{r-1}|x|^{\bar{\mf a}}+\nu_1$ for simplicity, we then get
\begin{align}
\label{Equation: F-K Integrability Bound 3}
\mbf E_{N}\big[|\mf n_t(Q,A)|^p\big]\leq\mr e^{p(r-1)t}\mbf E_{N}\left[\prod_{k=1}^{N(t)}\eta\big(Z(\tau_k)\big)\right].
\end{align}
Conditional on $N(t)=n$, the points $(\tau_1,\ldots,\tau_n)$
are the order statistics of $n$ i.i.d. uniform points on $[0,t]$. Therefore,
since $\eta$ is real-valued and commutes,
we can rewrite \eqref{Equation: F-K Integrability Bound 3} as follows by conditioning
on $N(t)$:
\begin{align}
\nonumber
&\mbf E_{N}\big[|\mf n_t(Q,A)|^p\big]\\
\nonumber
&\leq\mr e^{p(r-1)t}\sum_{n=0}^\infty
\mbf E_{N}\left[\prod_{k=1}^{N(t)}\eta\big(Z(\tau_k)\big)\Bigg|N(t)=n\right]\mbf P[N(t)=n]\\
\nonumber
&=\mr e^{p(r-1)t}\sum_{n=0}^\infty\left(\frac1{t^n}\int_{[0,t]^n}\eta\big(Z(s_1)\big)\cdots\eta\big(Z(s_n)\big)\d\bs s\right)\frac{(r-1)^nt^n}{n!}\mr e^{-(r-1)t}\\
\nonumber
&=\mr e^{(p-1)(r-1)t}\sum_{n=0}^\infty\frac{(r-1)^n}{n!}\left(\int_0^t\eta\big(Z(s)\big)\d s\right)^n\\
&=\mr e^{(r-1)\int_0^t\eta(Z(s))\d s+(p-1)(r-1)t}=\mr e^{\ka\int_0^t|Z(s)|^{\bar{\mf a}}\d s+(p-1+\nu_1)(r-1)t}
\label{Equation: F-K Integrability Bound 4}
\end{align}
Thus,
if we apply \eqref{Equation: F-K Integrability Bound 4} to
\eqref{Equation: F-K Integrability Bound 2}, with $\nu=(p-1+\nu_1)(r-1)+\nu_0$, we get
\begin{align}
\label{Equation: F-K Integrability Bound 5}
\mbf E_{N}\left[|\mf n_t(Q,A)|^p\mr e^{-p\int_0^t S(A(s))\d s}\right]
\leq\mr e^{-\ka\int_0^t|Z(s)|^{\bar{\mf a}}\d s+\nu t}.
\end{align}
By \eqref{Equation: Brownian Initial Point Coupling} and the reverse triangle inequality (since $\bar{\mf a}\leq 1$),
\begin{align}
\label{Equation: Potential Power Growth Triangle Inequality}
\int_0^t|Z(s)|^{\bar{\mf a}}\d s\geq\int_0^t|Z(0)|^{\bar{\mf a}}-|W^0(s)|^{\bar{\mf a}}\d s\geq t|Z(0)|^{\bar{\mf a}}-t\sup_{0\leq s\leq t}|W^0(s)|^{\bar{\mf a}}.
\end{align}
Together with \eqref{Equation: F-K Integrability Bound 5},
this concludes the proof of \eqref{Equation: F-K Integrability Bound}
in Cases 1 and 2.

It now only remains to show that there exists some $\nu>0$
such that
\begin{align}
\label{Equation: F-K Integrability Bound 8}
\mbf E_{N}\left[|\mf n_t(Q,A)|^p\mr e^{-p\int_0^t S(A(s))\d s}\right]\leq\mr e^{\nu t}
\end{align}
in Case 3. On the one hand, the fact that $S$ is bounded below
(Lemma \ref{Lemma: Growth})
means that $\mr e^{-p\int_0^t S(A(s))\d s}\leq\mr e^{\nu_1t}$
for some $\nu_1>0$. Next, since $Q$ is continuous and $I$
is bounded in Case 3, there exists some constant $\nu_2>0$
such that
\[\mbf E_{N}\big[|\mf n_t(Q,A)|^p\big]
\leq\mbf E\left[\mr e^{p(r-1)t}\nu_2^{pN(t)}\right]
=\mr e^{(r-1)(\nu_2^p+p-1)t}.\] 
We then get \eqref{Equation: F-K Integrability Bound 8}
by taking $\nu=\nu_1+(r-1)(\nu_2^p+p-1)$.
\end{proof}

We now return to the task at hand, which is to
prove that \eqref{Equation: HS 1.5} is finite
using the bound in \eqref{Equation: F-K Integrability Bound}.
For this purpose, however, the conditioning on the process $A^{a,a}_{2t}$
is problematic, since the endpoint $U(2t)$
is not independent of $N$'s distribution. In order to get around this
we denote $a=(i,x)\in\mc A$, and note that for any nonnegative functional $F$,
we can write
\begin{multline*}
\Pi_A(2t;a,a)\mbf E_A\Big[F\big(A^{a,a}_{2t}\big)\Big]
=\Pi_Z(2t;x,x)\mbf E_A\Big[F\big(U^i,Z^{x,x}_{2t}\big)\mbf 1_{\{U(2t)=i\}}\Big]\\
\leq\Pi_Z(2t;x,x)\mbf E_A\Big[\big[F\big(U^i,Z^{x,x}_{2t}\big)\Big].
\end{multline*}
If we apply this to \eqref{Equation: HS 1.5},
together with the inequality
\begin{align}
\label{Equation: Transition Kernel Bound}
\sup_{x\in I}\Pi_Z(2t;x,x)<\infty
\end{align}
(e.g., \cite[(5.16)]{GaudreauLamarreEJP}),
then we get
\begin{align}
\label{Equation: HS 2}
\|\hat K(t)\|_{\mr{HS}}^2\leq C\sum_{i=1}^r\int_0^\infty
\mbf E_A\left[\big|\mf n_{2t}\big(Q,(U^i,Z^{x,x}_{2t})\big)\big|\mr e^{-\int_0^{2t} S(U^i(s),Z^{x,x}_{2t}(s))\d s+\mf B_{2t}(U^i,Z^{x,x}_{2t})}\right]\d x
\end{align}
for some finite $C>0$ that depends only on $t$.

In Cases 1 and 2,
if we let $W^0$ be defined as in the statement of Proposition
\ref{Proposition: F-K Integrability Bound}, then
under the conditioning $Z^{x,x}_{2t}$,
the process $W^0$ has law $W^{0,0}_{2t}$; hence
\[Z^{x,x}_{2t}(s)=
\begin{cases}
x+W^{0,0}_{2t}(s)&\text{Case 1},\\
\left|x+W^{0,0}_{2t}(s)\right|&\text{Case 2},
\end{cases}
\qquad\text{for }0\leq s\leq t\]
Therefore,
if we apply Proposition \ref{Proposition: F-K Integrability Bound} with $p=1$
to the expectation in \eqref{Equation: HS 2}
(more precisely, we first write $\mbf E_A[\cdot]=\mbf E_A\big[\mbf E_N[\cdot]\big]$
by the tower property---noting that $N$ is independent of the event $\{U(0)=i\}$---and then notice that the upper bound
in Proposition \ref{Proposition: F-K Integrability Bound}
does not depend on the process $U$),
then for every $\ka>0$, there is $\nu>0$ such that
\begin{align}
\label{Equation: HS 3}
\|\hat K(t)\|_{\mr{HS}}^2\leq\begin{cases}
\displaystyle Cr\mr e^{2\nu t}\int_{\mbb R}\mr e^{-2\ka t|x|^{\bar{\mf{a}}}}
\mbf E\left[\mr e^{2\ka t\sup_{0\leq s\leq 2t}|W^{0,0}_{2t}(s)|^{\bar{\mf{a}}}}\right]\d x&\text{Case 1,}
\vspace{5pt}\\
\displaystyle Cr\mr e^{2\nu t}\int_0^\infty\mr e^{-2\ka t|x|^{\bar{\mf{a}}}}
\mbf E\left[\mr e^{2\ka t\sup_{0\leq s\leq 2t}|W^{0,0}_{2t}(s)|^{\bar{\mf{a}}}+\bar\al_\mr m\mf L^0_{2t}(X^{x,x}_{2t})}\right]\d x&\text{Case 2,}
\vspace{5pt}\\
\displaystyle Cr\mr e^{2\nu t}\int_0^{\vartheta}\mbf E\left[\mr e^{\bar\al_\mr m\mf L^0_{2t}(Y^{x,x}_{2t})
+\bar\be_\mr m\mf L^{\vartheta}_{2t}(Y^{x,x}_{2t})}\right]\d x&\text{Case 3.}
\end{cases}
\end{align}
It now only remains to prove that \eqref{Equation: HS 3} is finite for all $t>0$.
By applying H\"older's inequality to the expectation in \eqref{Equation: HS 3},
the finiteness of $\|\hat K(t)\|_{\mr{HS}}$ is a consequence of the following statements:

\begin{lemma}
\label{Lemma: FK Integrability Bounds}
For every $\ka,t>0$ and $\bar{\mf{a}}>0$,
\begin{align}
\label{Equation: Log Integral Finite}
\int_{\mbb R}\mr e^{-2\ka t|x|^{\bar{\mf{a}}}}\d x<\infty.
\end{align}For every $t,\ka>0$ and $\bar{\mf{a}}\leq1$, one has
\begin{align}
\label{Equation: Exponential Moments of Brownian Maxima}
\mbf E\left[\mr e^{\ka\sup_{0\leq s\leq t}|W^0(s)|^{\bar{\mf{a}}}}\right]<\infty
\qquad\text{and}\qquad
\mbf E\left[\mr e^{\ka\sup_{0\leq s\leq t}|W^{0,0}_{t}(s)|^{\bar{\mf{a}}}}\right]<\infty.
\end{align}
In Cases 2 and 3,
for every $t,\nu>0$, one has
\begin{align}
\label{Equation: Exponential Moments of Boundary Local Time}
\sup_{x\in I}\mbf E\left[\mr e^{\nu\mf L^c_t(Z^x)}\right]<\infty
\qquad\text{and}\qquad
\sup_{x\in I}\mbf E\left[\mr e^{\nu\mf L^c_t(Z^{x,x}_t)}\right]<\infty,
\end{align}
where $c=0$ in Case 2 and $c\in\{0,\vartheta\}$ in Case 3.
\end{lemma}
\begin{proof}
\eqref{Equation: Log Integral Finite} is a trivial calculation.
Next, \eqref{Equation: Exponential Moments of Brownian Maxima}
follows from the fact that Brownian motion/bridge maxima have
Gaussian tails; see the reflection principle for $W^0$
and, e.g., \cite[Remark 3.1]{GruetShi} for $W^{0,0}_t$.
Finally, \eqref{Equation: Exponential Moments of Boundary Local Time}
is \cite[Lemmas 5.6 and 5.8]{GaudreauLamarreEJP}.
\end{proof}

We therefore conclude that $\|\hat K(t)\|_{\mr{HS}}<\infty$.
With this, the only claim in part (3) of this lemma that remains to be established is
\begin{align}
\label{Equation: HS 7}
\int_{\mc A^2}|\hat K(t;a,b)|^2\d\mu(a)\dd\mu(b)=\int_{\mc A}\hat K(2t;a,a)\d\mu(a).
\end{align}
To this effect, property (1) of this lemma implies that
\[\int_{\mc A^2}|\hat K(t;a,b)|^2\d\mu(a)\dd\mu(b)=\int_{\mc A^2}\hat K(t;b,a)\hat K(t;a,b)\d\mu(a)\dd\mu(b).\]
At this point, we obtain \eqref{Equation: HS 7} by part (2) of this lemma, so long as we
can apply Fubini's theorem to integrate out $\dd\mu(a)$ in the above expression.
We can do this thanks to our knowledge that $\|\hat K(t)\|_{\mr{HS}}^2<\infty$.

\subsubsection{Proof of Proposition \ref{Proposition: Generic Kernel Properties}-(4) Part 1: Reduction to Smooth and Compactly Supported Functions}

We now begin our proof of Proposition \ref{Proposition: Generic Kernel Properties}-(4).
We first argue that it suffices to prove the result
under the assumption that every $f(i,\cdot)$ is smooth and compactly supported
on $I$'s closure. 
For this, it is enough to show that there exists a constant $C>0$
such that
\begin{align}
\label{Equation: Strong Continuity 0}
\sup_{t\in(0,1)}\|\hat K(t)f\|_2\leq C\|f\|_2
\end{align}
for every $f\in L^2(\mc A,\mbb F)$. Indeed, with this we can write
\[\|\hat K(t)f-f\|_2\leq\|\hat K(t)\tilde f-\tilde f\|_2+(C+1)\|f-\tilde f\|_2\qquad\text{for all }t\in(0,1),\]
where $\tilde f$ is a smooth and compactly supported approximation of $f$.

We now establish \eqref{Equation: Strong Continuity 0}. Write
\begin{align*}
\|\hat K(t)f\|_2^2=\int_\mc A\left|\mbf E_A\left[\mf n_t(Q,A^a)\mr e^{-\int_0^t S(A^a(s))\d s+\mf B_t(A^a)}f\big(A^a(t)\big)\right]\right|^2\dd \mu(a).
\end{align*}
Taking the absolute value inside the expectation by Jensen's inequality, we get
\begin{align}
\label{Equation: Strong Continuity 1}
\|\hat K(t)f\|_2^2\leq\int_\mc A\mbf E_A\left[|\mf n_t(Q,A^a)|\mr e^{-\int_0^t S(A^a(s))\d s+\mf B_t(A^a)}\big|f\big(A^a(t)\big)\big|\right]^2\dd \mu(a).
\end{align}

We aim to apply Proposition \ref{Proposition: F-K Integrability Bound} to
the expectation in \eqref{Equation: Strong Continuity 1}. In this view, the
presence of $\big|f\big(A(t)\big)\big|$ in \eqref{Equation: Strong Continuity 1}
is problematic, since again the enpoint $U(t)$ depends on the point process $N$.
In order to get around this, we change the initial point $U(0)$ of
the process $U$ to its stationary measure: For any point $x\in I$
and nonnegative functional $F$, we note that we can write
\begin{multline*}
\mbf E_A\Big[F(A^{(i,x)})\,\big|f\big(A^{(i,x)}(t)\big)\big|\Big]
\leq\sum_{j=1}^r\mbf E_A\Big[F(A^{(j,x)})\,\big|f\big(A^{(j,x)}(t)\big)\big|\Big]\\
=r\,\mbf E_A\Big[F(A^{(u,x)})\,\big|f\big(A^{(u,x)}(t)\big)\big|\Big],
\end{multline*}
where $u$ is a uniform random variable on $\{1,\ldots,r\}$
that is independent of $Z$, $N$, and the jumps $J_k$;
and $A^{(u,x)}$ denotes $A=(U,Z)$
conditioned on $\{Z(0)=x\}$ and $\{U(0)=u\}$. Since the
uniform distribution is the
stationary measure of $U$, under the conditioning $\{U(0)=u\}$,
the endpoint $U(t)$ is now independent of the Poisson process $N$.
Thus, we can now apply the tower property, whereby
\begin{align}
\label{Equation: Strong Continuity 2}
\mbf E_A\Big[F(A^{(i,x)})\,\big|f\big(A^{(i,x)}(t)\big)\big|\Big]\leq
r\,\mbf E_A\Big[\mbf E_N\big[F(A^{(u,x)})\big]\,\big|f\big(A^{(u,x)}(t)\big)\big|\Big].
\end{align}
Finally, in the context of this type of expectation sitting inside a $\dd\mu(a)$ integral
with a square,
\eqref{Equation: Strong Continuity 2} yields
\begin{multline}
\label{Equation: Strong Continuity 3}
\int_{\mc A}\mbf E_A\Big[F(A^a)\,\big|f\big(A^a(t)\big)\big|\Big]^2\d\mu(a)
=\sum_{i=1}^r\int_0^\infty\mbf E_A\Big[F(A^{(i,x)})\,\big|f\big(A^{(i,x)}(t)\big)\big|\Big]^2\d x\\
\leq r^3\int_0^\infty\mbf E_A\Big[\mbf E_N\big[F(A^{(u,x)})\big]\,\big|f\big(A^{(u,x)}(t)\big)\big|\Big]^2\d x.
\end{multline}

If we use \eqref{Equation: Strong Continuity 3} to apply 
Proposition \ref{Proposition: F-K Integrability Bound} (with $p=1$) to
\eqref{Equation: Strong Continuity 1}, then we
get
\begin{multline}
\label{Equation: Strong Continuity 4}
\|\hat K(t)f\|_2^2\leq
r^3\mr e^{2\nu t}\\
\cdot
\begin{cases}
\displaystyle
\int_{\mbb R}\mr e^{-2\ka t|x|^{\bar{\mf{a}}}}\mbf E\left[\mr e^{\ka t\sup_{0\leq s\leq t}\left|W^0(s)\right|^{\bar{\mf{a}}}}\big|f\big(u,B^x(t)\big)\big|\right]^2\dd x&\text{Case 1},
\vspace{5pt}\\
\displaystyle
\int_0^\infty\mr e^{-2\ka t|x|^{\bar{\mf{a}}}}\mbf E\left[\mr e^{\ka t\sup_{0\leq s\leq t}\left|W^0(s)\right|^{\bar{\mf{a}}}+\bar\al_\mr m\mf L^0_t(X^x)}\big|f\big(u,X^x(t)\big)\big|\right]^2\dd x&\text{Case 2},
\vspace{5pt}\\
\displaystyle
\int_0^{\vartheta}\mbf E\left[\mr e^{\bar\al_\mr m\mf L^0_t(Y^x)+\bar\be_\mr m\mf L^{\vartheta}_t(Y^x)}\big|f\big(u,Y^x(t)\big)\big|\right]^2\dd x&\text{Case 3,}
\end{cases}
\end{multline}
where $u$ is uniform on $\{1,\ldots r\}$ and independent of $Z^x$ and $W^0$.
Note that for any $t>0$ and $x\in\mbb R$, one has $\mr e^{-2\ka t|x|^{\bar{\mf{a}}}}\leq 1$.
Thus, by an application of H\"older's inequality
(i.e., $\mbf E\left[\mf X\,\big|f\big(u,Z^x(t)\big)\big|\right]^2\leq\mbf E[\mf X^2]\mbf E\big[\big|f\big(u,Z^x(t)\big)\big|^2\big]$
for any random variable $\mf X$) as well as \eqref{Equation: Exponential Moments of Brownian Maxima}
and \eqref{Equation: Exponential Moments of Boundary Local Time}
(for the latter, notice that $0\leq\mf L^c_{t}(Z)\leq\mf L^c_{1}(Z)$
for all $t\in(0,1)$ by \eqref{Equation: Regular Boundary Local Time}), it suffices to prove
that there exists some constant $C>0$ such that
\begin{align}
\label{Equation: Strong Continuity 5}
\sup_{t\in(0,1)}\int_I\mbf E\left[\big|f\big(u,Z^x(t)\big)\big|^2\right]\d x\leq C\|f\|_2^2.
\end{align}
Toward this end, the fact that $u$ is uniform and independent of $Z$ implies that
\begin{align}
\label{Equation: L2 Integral with uniform starting point in I}
\int_I\mbf E\left[\big|f\big(u,Z^x(t)\big)\big|^2\right]\d x
=\frac1r\sum_{i=1}^r\int_{I^2}\Pi_Z(t;x,y)|f(i,y)|^2\d x\dd y=\frac{\|f\|_2^2}r,
\end{align}
where the last equality follows from the symmetry of $Z$'s transition kernel and Tonelli's theorem
(i.e., integrating out the $x$ variable first).
Hence \eqref{Equation: Strong Continuity 0} holds.

\subsubsection{Proof of Proposition \ref{Proposition: Generic Kernel Properties}-(4) Part 2: Pointwise Convergence}

Let us then assume that each $f(i,\cdot)$ is continuous and compactly supported.
We begin by proving the pointwise convergence
\begin{align}
\label{Equation: Strong Continuity 7}
\lim_{t\to0}\hat K(t)f(a)=f(a)\qquad\text{for every }a\in\mc A.
\end{align}
For this, by an application of the tower property, we can write
\begin{align}
\label{Equation: Strong Continuity 8}
\hat K(t)f(a)=\mbf E_A\left[\mbf E_N\left[\mf n_t(Q,A^a)\mr e^{-\int_0^t S(A^a(s))\d s+\mf B_t(A^a)}f\big(A^a(t)\big)\right]\right].
\end{align}
Since $f$ and $Z$ are continuous and $U$ is piecewise constant, for every $a\in\mc A$,
\begin{align}
\label{Equation: Strong Continuity 9}
\lim_{t\to0}\mf n_t(Q,A^a)\mr e^{-\int_0^t S(A^a(s))\d s+\mf B_t(A^a)}f\big(A^a(t)\big)=f(a)
\qquad\text{almost surely.}
\end{align}

With
\eqref{Equation: Strong Continuity 9} in hand, we first argue that the same limit holds
when we take the expectation with respect to $N$. For this purpose, we combine
\eqref{Equation: F-K Integrability Bound 1} with the facts that
that $f$ is bounded, that there exists some constant $\nu>0$ such that $S(i,x)\geq-\nu$
(by Lemma \ref{Lemma: Growth}), and some $\ka,\nu_1>0$ such that
\[\max_{1\leq i<j\leq r}|\xi_{i,j}(x)|\leq\eta(x)=\ka\big(\log(1+|x|\big)^{\mf b}\big)+\nu_1\]
(by \eqref{Equation: Q Growth Upper Bound}) to bound
\begin{multline*}
\sup_{t\in(0,1)}\left|\mf n_t(Q,A^a)\mr e^{-\int_0^t S(A^a(s))\d s+\mf B_t(A^a)}f\big(A^a(t)\big)\right|\\
\leq
\begin{cases}
\displaystyle\max\left\{1,\eta\left(\sup_{0\leq s\leq 1}|B^x(s)|\right)^{N(t)}\right\}\mr e^{\nu t}\|f\|_\infty&\text{Case 1,}
\vspace{5pt}\\
\displaystyle\max\left\{1,\eta\left(\sup_{0\leq s\leq 1}X^x(s)\right)^{N(t)}\right\}\mr e^{\nu t+\bar\al_\mr m\mf L_t^0(X^x)}\|f\|_\infty&\text{Case 2,}
\vspace{5pt}\\
\displaystyle\max\left\{1,\left(\sup_{1\leq i,j\leq r,~y\in(0,\vartheta)}|\xi_{i,j}(y)|\right)^{N(t)}\right\}\mr e^{\nu t+\bar\al_\mr m\mf L_t^0(Y^x)+\bar\be_\mr m\mf L_t^{\vartheta}(Y^x)}\|f\|_\infty&\text{Case 3,}
\end{cases}
\end{multline*}
The expectation of this upper bound in $N$ is finite almost surely (conditional on $Z$)
for every $a\in\mc A$,
and thus by the dominated convergence theorem we have that
\begin{align}
\label{Equation: Strong Continuity 10}
\lim_{t\to0}\mbf E_N\left[\mf n_t(Q,A^a)\mr e^{-\int_0^t S(A^a(s))\d s+\mf B_t(A^a)}f\big(A^a(t)\big)\right]=f(a)
\qquad\text{almost surely.}
\end{align}

Finally, in order to conclude \eqref{Equation: Strong Continuity 7},
it only remains to prove that the limit \eqref{Equation: Strong Continuity 10}
persists once we take the expectation $\mbf E_A$ in
\eqref{Equation: Strong Continuity 8}. Toward this end,
we note that, by Jensen's inequality and the boundedness of $f$, one has
\begin{multline}
\label{Equation: Strong Continuity 11}
\left|\mbf E_N\left[\mf n_t(Q,A^a)\mr e^{-\int_0^t S(A^a(s))\d s+\mf B_t(A^a)}f\big(A^a(t)\big)\right]\right|\\
\leq\mbf E_N\left[|\mf n_t(Q,A^a)|\mr e^{-\int_0^t S(A^a(s))\d s+\mf B_t(A^a)}\right]\|f\|_\infty.
\end{multline}
An application of Proposition \ref{Proposition: F-K Integrability Bound} with $p=1$ to \eqref{Equation: Strong Continuity 11} then yields
\begin{multline}
\label{Equation: Strong Continuity 12}
\sup_{t\in(0,1)}\left|\mbf E_N\left[\mf n_t(Q,A^a)\mr e^{-\int_0^t S(A^a(s))\d s+\mf B_t(A^a)}f\big(A^a(t)\big)\right]\right|\\
\leq
\begin{cases}
\mr e^{\ka\sup_{0\leq s\leq 1}|W^0(s)|^{\bar{\mf{a}}}+\nu}\|f\|_\infty&\text{Case 1,}\\
\mr e^{\ka\sup_{0\leq s\leq 1}|W^0(s)|^{\bar{\mf{a}}}+\nu+\bar\al_\mr m\mf L^0_1(X^x)}\|f\|_\infty&\text{Case 2,}\\
\mr e^{\nu+\bar\al_\mr m\mf L^0_1(Y^x)+\bar\be_\mr m\mf L^{\vartheta}_1(Y^x)}\|f\|_\infty&\text{Case 3.}
\end{cases}
\end{multline}
By combining \eqref{Equation: Exponential Moments of Brownian Maxima} and
\eqref{Equation: Exponential Moments of Boundary Local Time} with an application of
H\"older's inequality, the expectation $\mbf E_A$ of the dominating function
on the right-hand side of \eqref{Equation: Strong Continuity 12}
is finite for any choice of $a\in\mc A$. This proves \eqref{Equation: Strong Continuity 7} by dominated convergence.

\subsubsection{Proof of Proposition \ref{Proposition: Generic Kernel Properties}-(4) Part 3: Uniform Integrability}
\label{Section: Vitali Convergence}

In order to conclude that $\|\hat K(t)f-f\|_2\to0$ as $t\to0$
from the pointwise convergence \eqref{Equation: Strong Continuity 7},
we use the Vitali convergence theorem (e.g., \cite[Theorem 2.24]{FonsecaGiovanni}):
Convergence in $L^2(\mc A,\mbb F)$ follows from
pointwise convergence if
for every $\eps>0$, there exists a $\de>0$ such that
\begin{align}
\label{Equation: Vitali 1}
\sup_{t\in(0,1)}\int_{\mc K}|\hat K(t)f(a)|^2\d\mu(a)<\eps\qquad\text{for all }\mc K\subset\mc A\text{ with }\mu(\mc K)<\de,
\end{align}
and there also exists a $\mc K\subset\mc A$ with $\mu(\mc K)<\infty$ and
\begin{align}
\label{Equation: Vitali 2}
\sup_{t\in(0,1)}\int_{\mc A\setminus\mc K}|\hat K(t)f(a)|^2\d\mu(a)<\eps.
\end{align}
Toward this end, if we reapply the bounds in
\eqref{Equation: Strong Continuity 3}--\eqref{Equation: Strong Continuity 5}
to $|\hat K(t)f(a)|^2$,
then we see that the proof of \eqref{Equation: Vitali 1}
and \eqref{Equation: Vitali 2} can be reduced to showing that
for every fixed $1\leq i\leq r$, the following holds:
For every $\eps>0$, there exists a $\de>0$ such that
\begin{align}
\label{Equation: Vitali 1 simplified}
\sup_{t\in(0,1)}\int_{\mc K}\mbf E\left[\big|f\big(i,Z^x(t)\big)\big|^2\right]\d x<\eps\qquad\text{for all }\mc K\subset I\text{ with }|\mc K|<\de,
\end{align}
and there also exists a $\mc K\subset I$ with $|\mc K|<\infty$ and
\begin{align}
\label{Equation: Vitali 2 simplified}
\sup_{t\in(0,1)}\int_{I\setminus\mc K}\mbf E\left[\big|f\big(i,Z^x(t)\big)\big|^2\right]\d x<\eps
\end{align}
(here, $|\mc K|$ denotes the Lebesgue measure).

Toward this end, we note that \eqref{Equation: Vitali 1 simplified} is trivial in all cases since $f$ is bounded:
\begin{align}
\label{Equation: Vitali Simplified Proof 1}
\sup_{t\in(0,1)}\int_{\mc K}\mbf E\left[\big|f\big(i,Z^x(t)\big)\big|^2\right]\d x\leq\|f\|_\infty^2\,|\mc K|;
\end{align}
we can then simply choose $\de=\eps/\|f\|_\infty^2$.
Similarly, \eqref{Equation: Vitali 2 simplified} is trivial in Case 3 since in that situation
$I$ is bounded. Thus, it only remains to prove \eqref{Equation: Vitali 2 simplified}
in Cases 1 and 2:
If we let $\mc K_\ka=I\cap[-\ka,\ka]$ for $\ka>0$, then
\begin{multline}
\label{Equation: Vitali Simplified Proof 2}
\int_{I\setminus\mc K_\ka}\mbf E\left[\big|f\big(i,Z^x(t)\big)\big|^2\right]\d x
=\int_{(I\setminus\mc K_\ka)\times I}\Pi_Z(t;x,y)|f(i,y)|^2\d x\dd y\\
=\int_0^\infty\mbf P\big[|Z^y(t)|>\ka\big]|f(i,y)|^2\d y.
\end{multline}
For every fixed $y>0$, in Cases 1 and 2 we have that
\begin{align}
\label{Equation: Vitali Simplified Proof 3}
\lim_{\ka\to\infty}\sup_{t\in(0,1)}\mbf P\big[|Z^y(t)|>\ka\big]=0.
\end{align}
Since $f\in L^2(\mc A,\mbb F)$, \eqref{Equation: Vitali 2 simplified}
then follows from dominated convergence.
This then concludes the proof of Proposition \ref{Proposition: Generic Kernel Properties}-(4),
and thus also of Proposition \ref{Proposition: Generic Kernel Properties}.

\subsection{Proof of Proposition \ref{Proposition: Generator Calculation}}
\label{Section: Proof of Generator}

Following-up on the roadmap provided in Section \ref{Section: Generator Calculation Roadmap},
the proof of Proposition \ref{Proposition: Generator Calculation} relies on three distinct
steps, namely:  \eqref{Equation: Generator 1}, \eqref{Equation: Generator 2}, and \eqref{Equation: Generator 3}.
The proofs of these steps were outlined in Remarks \ref{Remark: Generator 1}, \ref{Remark: Generator 2} and \ref{Remark: Generator 3}.
We now carry out the details of these calculations.
Similarly to the proof of Proposition \ref{Proposition: Generic Kernel Properties},
we assume throughout Section \ref{Section: Proof of Generator} that
we are restricted to a probability-one event wherein Lemma \ref{Lemma: Growth} holds.

\subsubsection{Proof of Proposition \ref{Proposition: Generator Calculation} Part 1}
\label{Section: Generator Part 1}

We begin with \eqref{Equation: Generator 1}.
Given \eqref{Equation: Generator 1.1}, it only remains to verify
that the operators $H_i+S(i,0)$ satisfy the hypotheses of the
one-dimensional Feynman-Kac formula \eqref{Equation: Classical Feynman-Kac} for each $1\leq i\leq r$.
For this purpose, we note that
a formal version of \eqref{Equation: Classical Feynman-Kac}
can be obtained by combining \cite[Theorem 4.9]{Sznitman}
with
a straightforward modification of \cite[(3.3'), (3.4), Theorem 3.4 (b), and Lemmas 4.6 and 4.7]{Papanicolaou} for
the Robin/mixed boundary condition, and \cite[Chapter 3-(34) and Theorem 3.27]{ChungZhao} for the Dirichlet
boundary condition. For convenience, we point to \cite[Theorem 5.4]{GaudreauLamarreEJP} for a unified statement that
specifically contains all cases considered in this paper: According to that result,
in order for the Feynman-Kac formula
for $H_i+\xi_{i,i}$ to hold,
it suffices to check that $S(i,\cdot)=V(i,\cdot)+\xi_{i,i}$
can be written as a difference $S(i,\cdot)=S(i,\cdot)_+-S(i,\cdot)_-$,
where $S(i,\cdot)_+$ and $S(i,\cdot)_-$ are nonnegative,
$S(i,\cdot)_-$ is in the Kato class (e.g., \cite[(5.2)]{GaudreauLamarreEJP}),
and $S(i,\cdot)_+$ restricted to any compact set is in the Kato class.
This follows from Lemma \ref{Lemma: Growth},
since $S(i,\cdot)_-$ is bounded and any one-dimensional locally integrable function
(in particular, $S(i,\cdot)_+$) is in the Kato class
when restricted to a compact set.

\subsubsection{Proof of Proposition \ref{Proposition: Generator Calculation} Part 2}
\label{Section: Generator Part 2}

Next, we prove \eqref{Equation: Generator 2}.
Our first step is to establish the pointwise convergence
\begin{align}
\label{Equation: Generator 2.2}
\lim_{t\to0}\big|\mf G_1\big(t;(i,x)\big)-Qf(i,x)\big|=0\qquad\text{for every }(i,x)\in\mc A.
\end{align}
Thanks to \eqref{Equation: Generator 2.1},
by the triangle and Jensen's inequalities,
\begin{multline}
\label{Equation: Generator 2.3}
\big|\mf G_1\big(t;(i,x)\big)-Qf(i,x)\big|\\
\leq
\sum_{j\neq i}\mbf E_A\left[\left|\xi_{i,j}\big(Z^x(\tau_1)\big)\mr e^{-\int_0^t S(A^{(i,x)}(s))\dd s+\mf B_t(A^{(i,x)})}f\big(j,Z^x(t)\big)-\xi_{i,j}(x)f(j,x)\right|~\bigg|~\mc J_j(t)\right].
\end{multline}
Conditional on $\mc J_j(t)$, $\tau_1$ is uniform on the interval $[0,t]$. Thus,
if we let $u_t$ denote a uniform random variable on $[0,t]$ independent of $Z$,
and we denote
\[\mf Y_{i,j}(t;x)=\xi_{i,j}\big(Z^x(u_t)\big)\mr e^{-\int_0^{u_t} S(i,Z^x(s))\d s
-\int_{u_t}^t S(j,Z^x(s))\d s\\
+\mf Z_{i,j}(u_t,Z^x)}f\big(j,Z^x(t)\big),\]
where
\[\mf Z_{i,j}(u_t,Z^x)=\begin{cases}
0&\text{Case 1,}\\
\bar\al_i\mf L^0_{[0,u_t)}(X^x)+\bar\al_j\mf L^0_{[u_t,t]}(X^x)&\text{Case 2,}\\
\bar\al_i\mf L^0_{[0,u_t)}(Y^x)+\bar\be_i\mf L^{\vartheta}_{[0,u_t)}(Y^x)
+\bar\al_j\mf L^0_{[u_t,t]}(Y^x)+\bar\be_j\mf L^{\vartheta}_{[u_t,t]}(Y^x)&\text{Case 3,}\\
\end{cases}\]
then we obtain from \eqref{Equation: Generator 2.3} that
\begin{align}
\label{Equation: Generator 2.4}
\big|\mf G_1\big(t;(i,x)\big)-Qf(i,x)\big|\leq\sum_{j\neq i}
\mbf E_A\big[\big|\mf Y_{i,j}(t;x)-\xi_{i,j}(x)f(j,x)\big|\big].
\end{align}
Since we only consider regular noises
(in the sense of Definition \ref{Definition: Regular Noise}) in Theorem \ref{Theorem: Regular FK}, the $\xi_{i,j}$
are continuous. Moreover, since $f\in D(\hat H)\subset D(\mc E)$
(by Proposition \ref{Proposition: Operator Definition})
and $D(\mc E)$
only contains direct sums of locally absolutely continuous functions (see Definition \ref{Definition: Deterministic Quadratic Forms}),
each $f(j,\cdot)$ is continuous.
Thus, $\mf Y_{i,j}(t;x)\to \xi_{i,j}(x)f(j,x)$ as $t\to0$ almost surely
for any fixed $i,j$, and $x$. At this point, \eqref{Equation: Generator 2.2}
follows if we show that we can apply the dominated convergence theorem
to the expectation $\mbf E_A$ in \eqref{Equation: Generator 2.4}.
For this purpose, we remark that since $f\in D(H)$, both $\|f\|_\mu$ and $\|f'\|_\mu$ are finite;
hence $f$ is bounded (see, e.g., \cite[Fact 3.1]{BloemendalVirag2}
and \cite[Lemma 3.1]{GaudreauLamarreEJP}).
If we combine this with the facts that $\xi_{i,j}(x)\leq\eta(x)$ with $\eta(x)=\ka\big(\log(1+|x|)\big)^{\mf b}+\nu$
for some $\ka,\nu>0$
in Cases 1 and 2 (by \eqref{Equation: Q Growth Upper Bound}),
that there exists some $c>0$ such that $S(i,x)\geq-c$
for all $1\leq i\leq r$ (by Lemma \ref{Lemma: Growth}),
then we can bound
\begin{multline}
\label{Equation: Generator 2.5}
\sup_{t\in(0,1)}\big|\mf Y_{i,j}(t;x)\big|\\
\leq
\begin{cases}
\eta\left(\sup_{t\in(0,1)}|B^x(t)|\right)
\mr e^{c}\|f\|_\infty&\text{Case 1,}
\vspace{5pt}\\
\eta\left(\sup_{t\in(0,1)}X^x(t)\right)
\mr e^{c+\bar\al_\mr m\mf L^0_{1}(X^x)}\|f\|_\infty&\text{Case 2,}
\vspace{5pt}\\
\left(\sup_{1\leq i,j\leq r,~y\in(0,\vartheta)}|\xi_{i,j}(y)|\right)
\mr e^{c+\bar\al_\mr m\mf L^0_{1}(Y^x)+\bar\be_\mr m\mf L^{\vartheta}_{1}(Y^x)}\|f\|_\infty&\text{Case 3.}
\end{cases}
\end{multline}
By \eqref{Equation: Exponential Moments of Boundary Local Time} and the fact that Brownian motion suprema
have Gaussian tails, this
has a finite expectation for any $x$; thus \eqref{Equation: Generator 2.2}
follows by dominated convergence.

We now conclude the proof of \eqref{Equation: Generator 2}
by showing that \eqref{Equation: Generator 2.2} can be improved
to convergence in $L^2(\mc A,\mbb F)$. For this purpose, we once
again use the Vitali convergence theorem. That is, we must prove
for every $\eps>0$, there exists a $\de>0$ such that
\begin{align}
\label{Equation: Vitali Generator 1}
\sup_{t\in(0,1)}\int_{\mc K}\big|\mf G_1\big(t;(i,x)\big)\big|^2\d\mu(a)<\eps\qquad\text{for all }\mc K\subset\mc A\text{ with }\mu(\mc K)<\de,
\end{align}
and there also exists a $\mc K\subset\mc A$ with $\mu(\mc K)<\infty$ and
\begin{align}
\label{Equation: Vitali Generator 2}
\sup_{t\in(0,1)}\int_{\mc A\setminus\mc K}\big|\mf G_1\big(t;(i,x)\big)\big|^2\d\mu(a)<\eps.
\end{align}
For this purpose, we note that by Jensen's inequality,
\[\big|\mf G_1\big(t;(i,x)\big)\big|^2
\leq
(r-1)\sum_{j\neq i}\mbf E\big[\big|\mf Y_{i,j}(t;x)\big|\big]^2;\]
then the same argument used to obtain \eqref{Equation: Generator 2.5}
(except that we do not bound $f$ by its infinity norm)
yields that for every $t\in(0,1)$,
\begin{multline}
\label{Equation: Generator 2.6}
\big|\mf G_1\big(t;(i,x)\big)\big|^2
\leq(r-1)\mr e^{2c}\sum_{j\neq i}\cdots\\
\begin{cases}
\mbf E\left[\eta\left(\sup_{t\in(0,1)}|B^x(t)|\right)\big|f\big(j,B^x(t)\big)\big|\right]^2&\text{Case 1,}
\vspace{5pt}\\
\mbf E\left[\eta\left(\sup_{t\in(0,1)}X^x(t)\right)
\mr e^{\bar\al_\mr m\mf L^0_{1}(X^x)}\big|f\big(j,X^x(t)\big)\big|\right]^2&\text{Case 2,}
\vspace{5pt}\\
\left(\sup_{1\leq i,j\leq r,~y\in(0,\vartheta)}|\xi_{i,j}(y)|\right)^2
\mbf E\left[\mr e^{\bar\al_\mr m\mf L^0_{1}(Y^x)+\bar\be_\mr m\mf L^{\vartheta}_{1}(Y^x)}\big|f\big(j,Z^x(t)\big)\big|\right]^2&\text{Case 3.}
\end{cases}
\end{multline}
By
H\"older's inequality, the second line of \eqref{Equation: Generator 2.6} is bounded above by
\[\begin{cases}
\mbf E\left[\eta\left(\sup_{t\in(0,1)}|B^x(t)|\right)^2\right]\mbf E\big[\big|f\big(j,B^x(t)\big)\big|^2\big]&\text{Case 1,}
\vspace{5pt}\\
\mbf E\left[\eta\left(\sup_{t\in(0,1)}X^x(t)\right)^{2p}\right]^{1/p}
\mbf E\left[\mr e^{2q\bar\al_\mr m\mf L^0_{1}(X^x)}\right]^{1/q}\mbf E\big[\big|f\big(j,X^x(t)\big)\big|^2\big]&\text{Case 2,}
\vspace{5pt}\\
C\mbf E\left[\mr e^{2\bar\al_\mr m\mf L^0_{1}(Y^x)+2\bar\be_\mr m\mf L^{\vartheta}_{1}(Y^x)}\right]\mbf E\big[\big|f\big(j,Z^x(t)\big)\big|^2\big]&\text{Case 3,}
\end{cases}\]
for any $p,q>1$ such that $1/p+1/q=1$, where $C$ is a constant independent of $t$.

By \eqref{Equation: Exponential Moments of Boundary Local Time},
this means that in Case 3 the uniform integrability follows from
 \eqref{Equation: Vitali 1 simplified} and \eqref{Equation: Vitali 2 simplified}.
We now deal with Cases 1 and 2:
Recall that $\eta(x)=\ka\big(\log(1+|x|)\big)^{\mf b}+\nu$. If we combine the coupling
\eqref{Equation: Brownian Initial Point Coupling} with the triangle and Jensen's inequalities,
and $\log(1+|x|+|y|)^{\mf b}\leq C\log(1+|x|)^\mf b+C|y|^\mf b$ for some constant $C>0$,
we can find some $\bar\ka,\bar\nu>0$ large enough so that
\[\textstyle\mbf E\left[\eta\left(\sup_{t\in(0,1)}|B^x(t)|\right)^2\right],\mbf E\left[\eta\left(\sup_{t\in(0,1)}X^x(t)\right)^{2p}\right]^{1/p}
\leq\bar\ka\big(\log(1+|x|)\big)^{2\mf b}+\bar\nu.\]
Here, we used the fact that
\[\textstyle\mbf E\left[\sup_{t\in(0,1)}|W^0(s)|^{\mf c}\right]<\infty\qquad\text{for all }\mf c>0,\]
and thus this expectation can be absorbed in the constant $\bar \nu$.
By \eqref{Equation: Exponential Moments of Boundary Local Time},
\[\mbf E\left[\mr e^{2q\bar\al_\mr m\mf L^0_{1}(X^x)}\right]^{1/q}<\infty.\]
Thus, by \eqref{Equation: Vitali 1 simplified} and \eqref{Equation: Vitali 2 simplified},
the uniform integrability in \eqref{Equation: Vitali Generator 1} and \eqref{Equation: Vitali Generator 2}
in Cases 1 and 2 follows from this:
For every $\eps>0$, there is some $\de>0$ such that
\begin{align}
\label{Equation: Vitali 1 simplified log}
\sup_{t\in(0,1)}\int_{\mc K}\big(\log(1+|x|)\big)^{2\mf b}\mbf E\left[\big|f\big(i,Z^x(t)\big)\big|^2\right]\d x<\eps\qquad\text{for all }\mc K\subset I\text{ with }|\mc K|<\de,
\end{align}
and there also exists a $\mc K\subset I$ with $|\mc K|<\infty$ and
\begin{align}
\label{Equation: Vitali 2 simplified log}
\sup_{t\in(0,1)}\int_{I\setminus\mc K}\big(\log(1+|x|)\big)^{2\mf b}\mbf E\left[\big|f\big(i,Z^x(t)\big)\big|^2\right]\d x<\eps.
\end{align}

Toward this end, we note that for any set $\mc K\subset I$,
and functions $g,h$, we can rewrite
\begin{align}
\nonumber
\int_{\mc K}g(x)\mbf E\left[h\big(Z^x(t)\big)\right]\d x
&=\int_{I\times I}g(x)\mbf 1_{\{x\in\mc K\}}\Pi_Z(t;x,y)h(y)\d x\dd y\\
\label{Equation: 2 BM Change Direction Trick}
&=\int_{I}\mbf E\Big[g\big(y+Z^0(t)\big)\mbf 1_{\{y+Z^0(t)\in\mc K\}}\Big]h(y)\d y.
\end{align}
Applying this to $g(x)=\big(\log(1+|x|)\big)^{2\mf b}$ and $h(y)=|f(i,y)|^2$, we then get
\begin{align*}
&\int_{\mc K}\big(\log(1+|x|)\big)^{2\mf b}\mbf E\left[\big|f\big(i,Z^x(t)\big)\big|^2\right]\d x\\
&=\int_{I}\mbf E\left[\big(\log(1+|y+Z^0(t)|)\big)^{2\mf b}\mbf 1_{\{y+Z^0(t)\in\mc K\}}\right]|f(i,y)|^2\dd y\\
&\leq C\int_{I}\mbf E\left[\big(\log(1+|y|\big)^{2\mf b}+|Z^0(t)|^{2\mf b}\big)\mbf 1_{\{y+Z^0(t)\in\mc K\}}\right]|f(i,y)|^2\dd y
\end{align*}
for some $C>0$ independent of $t$.
Then, an application of H\"older's inequality yields
\begin{multline}
\label{Equation: Vitali Simplified Log Split}
\int_{\mc K}\big(\log(1+|x|)\big)^{2\mf b}\mbf E\left[\big|f\big(i,Z^x(t)\big)\big|^2\right]\d x
\leq C\int_{I}\mbf P\big[Z^y(t)\in\mc K\big]\log(1+|y|\big)^{2\mf b}|f(i,y)|^2\dd y\\
+C\int_{I}\mbf E\left[|Z^0(t)|^{4\mf b}\right]^{1/2}\mbf P\big[Z^y(t)\in\mc K\big]^{1/2}|f(i,y)|^2\dd y
\end{multline}
for some $C>0$ independent of $t$.
Regarding the second line in \eqref{Equation: Vitali Simplified Log Split},
we note that
\[\sup_{t\in(0,1)}\mbf E\left[|Z^0(t)|^{4\mf b}\right]^{1/2}<\infty,\]
and then, by Jensen's inequality and the fact that $\int|f(i,y)|^2\d y<\infty$,
there exists some $C>0$ independent of $t$
such that
\[\int_{I}\mbf P\big[Z^y(t)\in\mc K\big]^{1/2}|f(i,y)|^2\dd y
\leq C\left(\int_{I}\mbf P\big[Z^y(t)\in\mc K\big]|f(i,y)|^2\dd y\right)^{1/2}.\]
Thus, \eqref{Equation: Vitali Simplified Log Split} becomes
\begin{align*}
&\int_{\mc K}\big(\log(1+|x|)\big)^{2\mf b}\mbf E\left[\big|f\big(i,Z^x(t)\big)\big|^2\right]\d x\\
&\leq C\int_{I}\mbf P\big[Z^y(t)\in\mc K\big]\log(1+|y|\big)^{2\mf b}|f(i,y)|^2\dd y
+C\left(\int_{I}\mbf P\big[Z^y(t)\in\mc K\big]|f(i,y)|^2\dd y\right)^{1/2}\\
&=C\int_{\mc K}\mbf E\left[\log(1+|Z^x(t)|\big)^{2\mf b}\big|f\big(i,Z^x(t)\big)\big|^2\right]\d x
+C\left(\int_{\mc K}\mbf E\left[\big|f\big(i,Z^x(t)\big)\big|^2\right]\d x\right)^{1/2}
\end{align*}
for some $C>0$ independent of $t$,
where we have used \eqref{Equation: 2 BM Change Direction Trick}
in reverse (with $g=1$) in the third line.
At this point, if we show that $\log(1+|y|\big)^{2\mf b}|f(i,y)|^2$
is both integrable and bounded, then \eqref{Equation: Vitali 1 simplified log}
and \eqref{Equation: Vitali 2 simplified log}
follows from the same arguments used in \eqref{Equation: Vitali Simplified Proof 1}--\eqref{Equation: Vitali Simplified Proof 3}
(replace $|f(i,y)|^2$ therein by $\log(1+|y|\big)^{2\mf b}|f(i,y)|^2$).

Toward this end, since $f\in D(\hat H)$ we know that
$\|S_+^{1/2}f\|_2<\infty$ (see Proposition \ref{Proposition: Operator Definition}
and the definition of $D(\mc E)$ therein). Thus,
by \eqref{Equation: S Growth Lower Bound} and \eqref{Equation: Q Growth Upper Bound},
this implies that
\begin{align}
\label{Equation: Vitali Simplified Log Split 2}
\int_{I}\log(1+|y|\big)^{\mf c}|f(i,y)|^2\d y<\infty\qquad\text{for every }\mf c>0
\end{align}
(in particular for $\mf c=2\mf b$).
Regarding boundedness,
up to expanding $f(i,y)$ into a sum of its real and imaginary parts, we may
assume without loss of generality that $f(i,y)$ is real-valued.
Under this assumption, by the fundamental theorem of calculus,
\begin{align*}
\log(1+|x|\big)^{2\mf b}f(i,x)^2\leq\left|\int_x^{\infty}\tfrac{\dd}{\dd y}\big(\log(1+|y|\big)^{2\mf b}f(i,y)^2\big)\d y\right|,\qquad x\in I.
\end{align*}
Using the product rule and Jensen's inequality, this implies
\[\log(1+|x|\big)^{2\mf b}f(i,x)^2\leq2\int_I\big(\log (1+|y|)\big)^{2\mf b}\big|f(i,y)f'(i,y)\big|+\frac{\mf b\log(1+|y|)^{2\mf b-1}f(i,y)^2}{1+|y|}\d y;\]
thus it only remains to check that the above integral is finite. On the one hand, since $\frac1{1+|y|}\leq 1$,
we get that
\[\int_I\frac{\mf b\log(1+|y|)^{2\mf b-1}f(i,y)^2}{1+|y|}\d y<\infty\]
by \eqref{Equation: Vitali Simplified Log Split 2}. On the other hand, if we use the inequality $|xy|\leq x^2+y^2$,
we get
\[
\int_I\big(\log (1+|y|)\big)^{2\mf b}\big|f(i,y)f'(i,y)\big|\d y
\leq
\int_I\big(\log (1+|y|)\big)^{4\mf b}f(i,y)^2+f'(i,y)^2\d y;\]
this is finite by a combination of \eqref{Equation: Vitali Simplified Log Split 2}
and $\|f'\|_\mu<\infty$. With this in hand, the proof of \eqref{Equation: Generator 2} is now complete.

\subsubsection{Proof of Proposition \ref{Proposition: Generator Calculation} Part 3}
\label{Section: Generator Part 3}

We now finish the proof of Proposition \ref{Proposition: Generator Calculation}---and therefore of Theorem \ref{Theorem: Regular FK}---by
establishing \eqref{Equation: Generator 3}.
Recalling the definition of $\mf G_2$ in \eqref{Equation: Generator 3.0},
by Jensen's and H\"older's inequalities,
\begin{multline}
\label{Equation: Generator 3.1}
\big|\mf G_2\big(t;(i,x)\big)\big|^2
\leq t^{-2}\mbf E_A\left[|\mf n_t(Q,A^{(i,x)})|^4\mr e^{-4\int_0^t S(A^{(i,x)}(s))\d s-4\mf B_t(A^{(i,x)})}\right]^{1/2}\\
\cdot\mbf P\big[N(t)\geq2\big]^{1/2}\mbf E\Big[\big|f\big(A^{(i,x)}(t)\big)\big|^2\mbf 1_{\{N(t)\geq2\}}\Big].
\end{multline}
Then, by \eqref{Equation: Strong Continuity 2} (but replacing $f$ by $|f|^2$), we get that
\begin{align}
\label{Equation: Generator 3.2}
\mbf E\Big[\big|f\big(A^{(i,x)}(t)\big)\big|^2\mbf 1_{\{N(t)\geq2\}}\Big]\leq r\,\mbf P\big[N(t)\geq2\big]\mbf E\Big[\big|f\big(u,Z^x(t)\big)\big|^2\Big],
\end{align}
where $u$ is uniform on $\{1,\ldots,r\}$ and independent of $Z$.
If we now combine Proposition \ref{Proposition: F-K Integrability Bound} (in the case $p=4$)
with \eqref{Equation: Exponential Moments of Brownian Maxima}, \eqref{Equation: Exponential Moments of Boundary Local Time},
and the fact that $\mr e^{-t\ka|x|^{\bar{\mf a}}}\leq 1$, then we get
\[\sup_{t\in(0,1),~(i,x)\in\mc A}\mbf E_A\left[|\mf n_t(Q,A^{(i,x)})|^4\mr e^{-4\int_0^t S(A^{(i,x)}(s))\d s-4\mf B_t(A^{(i,x)})}\right]^{1/2}<\infty.\]
If we combine this with \eqref{Equation: Generator 3.1} and \eqref{Equation: Generator 3.2},
given that $\mbf P[N(t)\geq2]=O(t^2)$ as $t\to0$, then we conclude that there exists a constant $C>0$ independent of $t$ such that
\begin{align}
\label{Equation: Generator 3.3}
\|\mf G_2(t;\cdot)\|_2^2\leq Ct
\int_{\mc A}\mbf E\Big[\big|f\big(u,Z^x(t)\big)\big|^2\Big]\d\mu(i,x).
\end{align}
We then obtain \eqref{Equation: Generator 3} from \eqref{Equation: L2 Integral with uniform starting point in I}.

\section{Proof of Theorem \ref{Theorem: Trace Moment Formulas}: Trace Moment Formulas}
\label{Section: Kernel Moment Limits}

We assume throughout Section \ref{Section: Kernel Moment Limits} that
Assumptions \ref{Assumption: Domain},
\ref{Assumption: Boundary}, and \ref{Assumption: Potential} hold,
that $\xi_{i,i}$ are white noises with variance $\si^2>0$,
and that $\xi_{i,j}$ ($i\neq j$) are white noises with variance $\upsilon^2>0$.
This section is organized as follows: We begin with an outline in Section \ref{Section: Proof of Main Result Roadmap};
then the technical results appearing in that outline are proved in Sections \ref{Section: Main Proposition 1}--\ref{Section: Main Proposition 7}.

\subsection{Proof Outline}
\label{Section: Proof of Main Result Roadmap}

\subsubsection{Step 0. Smooth Approximations}

Following Section \ref{Section: Probabilistic OE}, we aim to approximate $\hat H$
with a sequence of operators whose noises are regular, and then apply the Feynman-Kac
formula in Theorem \ref{Theorem: Regular FK}
to compute trace moments of these approximations. For this purpose, we introduce the following smooth approximations of
$\xi$'s entries:

\begin{notation}
Given $\phi,\psi:I\to\mbb R$, we let $\phi\star\psi$ denote their convolution.
That is,
\[\phi\star\psi(x)=\int_I \phi(y)\psi(x-y)\d y,\qquad x\in I.\]
More generally, if $f:\mc A\to\mbb F$ and $\psi:I\to\mbb R$,
then we let $f\star\psi:\mc A\to\mbb F$ denote the convolution of $\psi$ with $f$'s components.
That is, for every $a=(i,x)\in\mc A$, one has
\[f\star \psi(a)=f\star \psi(i,x)=\int_If(i,y)\psi(y-x)\d y.\]
\end{notation}

\begin{definition}
\label{Definition: Noise Coupling}
Suppose that $(W_{i,j})_{1\leq i\leq j\leq r}$ denote the independent Brownian motions
in $\mbb R$ or $\mbb F$ that are used to define the white noises $\xi_{i,j}$ as
\[\xi_{i,j}(f,f)=\begin{cases}
\si W_{i,i}'(f,f)&\text{if }i=j,\\
\upsilon W_{i,j}'(f,f)&\text{if }i<j,
\end{cases}\]
in the sense of \eqref{Equation: White Noise}. More specifically, $W_{i,i}$ are just
standard Brownian motions in $\mbb R$, and for $i<j$,
we can decompose $W_{i,j}$ as the sum
\begin{align}
\label{Equation: Noise Coupling Component Decomposition}
W_{i,j}(x)=\begin{cases}
W_{i,j;1}(x)&\text{if }\mbb F=\mbb R,\\
\frac1{\sqrt{2}}\big(W_{i,j;1}(x)+W_{i,j;\msf i}(x)\msf i\big)&\text{if }\mbb F=\mbb C,\\
\frac12\big(W_{i,j;1}(x)+W_{i,j;\msf i}(x)\msf i+W_{i,j;\msf j}(x)\msf j+W_{i,j;\msf k}(x)\msf k\big)&\text{if }\mbb F=\mbb H,
\end{cases}
\end{align}
where $W_{i,j;1},W_{i,j;\msf i},W_{i,j;\msf j},W_{i,j;\msf k}$ are i.i.d. standard Brownian motions in $\mbb R$.

Let $\bar\rho:\mbb R\to\mbb R$ be a nonnegative, smooth, compactly supported,
and even probability density function (i.e., $\bar\rho(-x)=\bar\rho(x)$ for all $x\in\mbb R$ and $\int_{\mbb R}\bar\rho(x)\d x=1$). Let $\rho=\bar\rho\star\bar\rho$. For every $\eps>0$,
define the rescaled function $\bar\rho_\eps(x)=\eps^{-1}\bar\rho(x/\eps)$, and for $\ze,\eta>0$, let
\begin{align}
\label{Equation: Twofold Convolution}
\rho_{\ze}=\bar\rho_\ze\star\bar\rho_\ze\qquad\text{and}\qquad\rho_{\ze,\eta}=\bar\rho_\ze\star\bar\rho_\eta.
\end{align}
For $\eps=0$, we use the convention that $\bar\rho_0=\rho_{0,0}=\de_0$
is the Dirac delta distribution. 

For every $1\leq i\leq j\leq r$ and $\eps\geq0$, we define the quadratic forms
\[\xi^\eps_{i,j}(f,g)=\begin{cases}
\si (W_{i,j}\star\bar\rho_\eps)'(f,g)&\text{if }i=j\\
\upsilon (W_{i,j}\star\bar\rho_\eps)'(f,g)&\text{if }i< j
\end{cases}.\]
In particular, $\xi^0_{i,j}=\xi_{i,j}$, and for $\eps>0$, the process $\xi^\eps_{i,j}$ is a regular noise
with covariance $\si^2\rho_\eps$ in $\mbb R$ or covariance $\upsilon^2\rho_\eps$ in $\mbb F$
in the sense of Definition \ref{Definition: Regular Noise},
depending on whether $i=j$
 (see, e.g., \cite[Remark 3.7]{GaudreauLamarreEJP}).

Finally, for any $\eps,\ze\geq0$, we define the matrix noise
\[\xi^{\eps,\ze}(f,g)=\sum_{i=1}^r\xi^{\eps}_{i,i}\big(f(i,\cdot),g(i,\cdot)\big)+\sum_{i\neq j}\xi^{\ze}_{i,j}\big(f(i,\cdot),g(j,\cdot)\big),\qquad\eps,\ze\geq0;\]
in particular, when $\eps$ and $\ze$ are positive, the diagonal terms $\xi^\eps_{i,i}$ are i.i.d. regular noises in $\mbb R$ with covariance $\si^2\rho_{\eps}$,
and the off-diagonal terms $\xi^\ze_{i,j}$ are i.i.d. regular noises in $\mbb F$ with covariance $\upsilon^2\rho_{\ze}$. In keeping with some of the notation introduced in Theorem \ref{Theorem: Regular FK},
for $\ze>0$ we denote the off-diagonal part of $\xi^{\eps,\ze}$ as
\begin{align}
\label{Equation: Q zeta}
Q^\ze=\sum_{i\neq j}\xi^\ze_{i,j}.
\end{align}
\end{definition}

We can now define the smooth approximations of $\hat H$, which are all on the same probability
space by virtue of $\xi^{\eps,\ze}$ all being defined from the same Brownian motions $W_{i,j}$:
For any $\eps,\ze\geq0$, let
\begin{align}
\label{Equation: Double Approximation Operator}
\hat H^{\eps,\ze}=H+\xi^{\eps,\ze}
\end{align}
be the operator constructed in Proposition \ref{Proposition: Operator Definition}; since the noises $\xi_{i,i}^\eps$ and $\xi_{i,j}^\ze$ are all regular or white---depending on whether or not $\eps,\ze=0$---the
assumptions of Proposition \ref{Proposition: Operator Definition} are satisfied for any choice of $\eps,\ze\geq0$. In particular, we note
that $\hat H=\hat H^{0,0}.$

In the remainder of Section \ref{Section: Proof of Main Result Roadmap}
(namely, Sections \ref{Section: Main Result Roadmap Step 1} and \ref{Section: Main Result Roadmap Step 2}), we state a number of technical
results related to the limits of $\hat H^{\eps,\ze}$ as $\eps\to0$ and
$\ze\to0$, and then use the latter to prove Theorem \ref{Theorem: Trace Moment Formulas}.
Before getting on with this, however, we make a remark on the specific form
of our approximation scheme:
At first glance, it might appear more natural to simply consider the sequence of
operators $\hat H^\eps=\hat H^{\eps,\eps}$ and then take a single limit of the
latter as $\eps\to0$---as suggested by the informal outline of proof in Section \ref{Section: Probabilistic OE}.
The reason why we consider the two separate parameters $\eps$ and $\ze$ is entirely technical, and has to do with the convergence
of the kernels
\[\hat K^{\eps,\ze}(t;a,b)=\mr e^{-t\hat H^{\eps,\ze}}(a,b),\qquad a,b\in\mc A.\]
Namely, it is easier to
analyze the limits of this object if we first send $\eps\to0$ and then after that send $\ze\to0$
(specifically, the sequence of limits \eqref{Equation: Trace Moments Limit 1}--\eqref{Equation: Trace Moments Limit 2.2} below). 
A key factor in this is that, as mentoned in
\eqref{Equation: Diagonal Noise Pathwise Interpretation}, it is possible to provide a pathwise interpretation of 
the kernel of $\mr e^{-t\hat H^{0,\ze}}$ when $\ze>0$:

\begin{definition}
Recall the Brownian motions $W_{i,i}$ introduced in Definition \ref{Definition: Noise Coupling};
i.e., $\xi_{i,i}=\si W_{i,i}'$.
Given $a=(i,x)\in\mc A$, we denote
\begin{align}
\label{Equation: Diagonal Noise Alternate Notation}
W^0_{\mr d}(a)=W^0_{\mr d}(i,x)=\si W_{i,i}(x)
\end{align}
and for any
continuous $f:\mc A\to\mbb R$ with compact support, we let
\[\int_{\mc A}f(a)\d W^0_{\mr d}(a)=\si \sum_{i=1}^r\int_{I}f(i,x)\d W_{i,i}(x),\]
where the summands on the right-hand side are interpreted as pathwise stochastic integrals
(see Remark \ref{Remark: Pathwise} for more details).
Then, let $S^0$ denote the functional
\begin{align}
\label{Equation: Diagonal Potential Alternate Notation}
S^0(f)=\int_{\mc A} V(a)f(a)\d\mu(a)+\int_{\mc A} f(a) \d W^0_{\mr d}(a),
\end{align}
and let $Q^\ze$ be as in \eqref{Equation: Q zeta}.
For any $\ze>0$, we define the kernel
\begin{align}
\label{Equation: Kernel at epsilon zero}
\hat{K}^{0,\ze}(t;a,b)=\Pi_A(t;a,b)\mbf E_{A}\left[\mf n_t(Q^\ze,A^{a,b}_t)\,\mr e^{-S^0(L_t(A^{a,b}_t))+\mf B_t(A^{a,b}_t)}\right],\qquad a,b\in\mc A,
\end{align}
where we recall that the local time $L_t(A)$ is defined in Definition \ref{Definition: Regular Local Time}.
\end{definition}

\begin{remark}
\label{Remark: Pathwise}
If $f:I\to\mbb R$ is continuously differentiable, then a pathwise stochastic integral
can be defined by means of the integration by parts formula
\[\int_{I}f(x)\d W_{i,i}(x)=-\int_{I}f'(x)W_{i,i}(x)\d x+\text{boundary terms}.\]
However, in \eqref{Equation: Kernel at epsilon zero} we evaluate these stochastic integrals
in the local time process of $A^{a,b}_t$, which is only locally H\"older continuous
with exponent strictly less than $1/2$. Thus, we need a slightly more sophisticated notion of pathwise stochastic integration.
For the purposes of this paper, the only two properties that we need this stochastic integral to satisfy
are as follows:
\begin{enumerate}[(1)]
\item If we let $(\Om_1,\mc F_1,\mbf P_1)$ denote the probability space on which
$W_{1,1},\ldots,W_{r,r}$ are defined, and $(\Om_2,\mc F_2,\mbf P_2)$ be the probability space on which $(A^{a,b}_t:a,b\in\mc A)$ (and its local times) are defined, then for every fixed $1\leq i\leq r$ and $t>0$ the map
\[\big((a,b),\om_1,\om_2\big)\mapsto\int_{I}L^{(i,y)}\big(A^{a,b}_t(\om_2)\big)\d W_{i,i}(\om_1;y)\in\mbb R\]
is measurable with respect to the product $\mu\times\mbf P_1\times\mbf P_2$ and the Borel $\si$-algebra on $\mbb R$,
where $W_{i,i}(\om_1;\cdot)$ denotes the realization of the Brownian motion $W_{i,i}$ for the sample $\om_1\in\Om_1$,
and likewise for $A^{a,b}_t(\om_2)$.
Among other things, this ensures that the map
\[\big((a,b),\om_1,\om_2\big)\mapsto\Pi_A(t;a,b)\mf n_t(Q^\ze,A^{a,b}_t)\,\mr e^{-S^0_{\om_1}(L_t(A^{a,b}_t(\om_2)))+\mf B_t(A^{a,b}_t(\om_2))}\in\mbb R\]
is measurable, where $S^0_{\om_1}(f)$ denotes the realization of \eqref{Equation: Diagonal Potential Alternate Notation}
associated to $W_{i,i}(\om_1;\cdot)$.
Therefore, the random kernel in \eqref{Equation: Kernel at epsilon zero} is a well-defined
measurable function with respect to $\mu\times\mbf P_1$
by Fubini's theorem.
\item For any continuous and compactly supported functions $f_1,\ldots,f_n:I\to\mbb R$
and indices $i_1,\ldots,i_n\in\{1,\ldots,r\}$, the map
\[\om_1\mapsto\left(\int_{I}f_1(x)\d W_{i_1,i_1}(\om_1;x),\ldots,\int_{I}f_n(x)\d W_{i_n,i_n}(\om_1;x)\right)\in\mbb R^n\]
is a multivariate Gaussian random vector with mean zero and covariance
\[\mbf E\left[\int_{I}f_k(x)\d W_{i_k,i_k}(x)\int_{I}f_\ell(x)\d W_{i_\ell,i_\ell}(x)\right]=\mbf 1_{\{i_k=i_\ell\}}\langle f_k,f_\ell\rangle.\]
Thus, the contribution of the $W_{i,i}$'s to any joint moments $\mbf E\big[\prod_{k=1}^n\mr{Tr}\big[\mr e^{-t_k\hat H^{0,\ze}}\big]\big]$
can be calculated from \eqref{Equation: Kernel at epsilon zero} using Fubini's theorem: For any fixed continuous
and compactly supported functions $f_1,\ldots,f_n:\mc A\to\mbb R$
(such as $A^{a_k,a_k}_{t_k}(\om_2)$'s local times for any fixed $\om_2\in\Om_2$
and $a_k\in\mc A$), the expectation
\[\mbf E\left[\exp\left(-\si\sum_{k=1}^n\sum_{i=1}^r\int_{I}f_k(i,x)\d W_{i,i}(x)\right)\right]=\exp\left(\frac{\si^2}2\sum_{k=1}^n\|f_k\|_\mu^2\right)\]
amounts to a straightforward Gaussian moment generating function calculation.
\end{enumerate}
For this, any measurable pathwise stochastic integral that is almost-surely equal to the
stochastic integral defined as an $L^2(\Om_1)$-limit (using the It\^{o} isometry) for
any fixed continuous deterministic integrand suffices.
For instance, we can use Karandikar's simple dyadic construction in \cite{Karandikar}
(see \cite[Section 3.2.1 and Appendix A]{GaudreauLamarreEJP} for the details of how
Karandikar's pathwise integral satisfies the two conditions above), instead of more sophisticated
constructs such as Terry Lyons' rough paths \cite{Lyons}.
\end{remark}

We now carry on with the proof of Theorem \ref{Theorem: Trace Moment Formulas}.

\subsubsection{Step 1. Kernel Moment Limits}
\label{Section: Main Result Roadmap Step 1}

Thanks to Theorem \ref{Theorem: Regular FK} (as well as \eqref{Equation: Kernel at epsilon zero}),
the trace moments of $\hat H^{\eps,\ze}$ for $\eps\geq0$ and $\ze>0$ can be computed explicitly
as follows:

\begin{proposition}
\label{Proposition: Smooth Trace Moments}
Let $n\in\mbb N$,
$\bs t=(t_1,\ldots,t_n)\in(0,\infty)^n$, $\bs\eps=(\eps_1,\ldots,\eps_n)\in[0,\infty)^n,$ and $\bs\ze=(\ze_1,\ldots,\ze_n)\in(0,\infty)^n$
be fixed. Let $|\bs t|=t_1+\cdots+t_n$ denote the $\ell_1$ norm of $\bs t$,
and for every $0\leq s\leq|\bs t|$, let
\begin{align}
\label{Equation: zeta s}
\ze(s)=\sum_{k=1}^n\ze_k\mbf 1_{\{s\in[t_0+\cdots+t_{k-1},t_0+\cdots+t_k)\}},
\end{align}
with the usual convention that $t_0=0$. In words, $\ze(s)=\ze_k$ if $s$ lies in the $k^{\mr{th}}$
sub-interval (of length $t_k$) in $[0,|\bs t|)$.
Define the functionals
\begin{align}
\label{Equation: Vector Valued Self-Intersection 1}
\mf m^{\bs\ze}_{\bs t}(A)=\mr e^{(r-1)|\bs t|}\mbf 1_{\{N(|\bs t|)\text{ is even}\}}\sum_{p\in\mc P_{N(|\bs t|)}}
\mf C_{|\bs t|}(p,U)\prod_{\{\ell_1,\ell_2\}\in p}\upsilon^2\rho_{\ze(\tau_{\ell_1}),\ze(\tau_{\ell_2})}\big(Z(\tau_{\ell_1})-Z(\tau_{\ell_2})\big),
\end{align}
(with the convention that the sum over $p$ is equal to one if $N(|\bs t|)=0$)
where we recall that $\mc P_n$ and $\mf C_t(p,U)$ are defined in Definitions \ref{Definition: Pn} and \ref{Definition: Combinatorial Constant}
and that $\rho_{\ze,\eta}$ is defined in \eqref{Equation: Twofold Convolution};
and
\begin{align}
\label{Equation: Vector Valued Self-Intersection 2}
\mf s^{\bs\eps}_{\bs t}(A)=\frac{\si^2}{2}\left\|\sum_{k=1}^nL_{[t_0+\cdots+t_{k-1},t_0+\cdots+t_k)}(A)\star\bar\rho_{\eps_k}\right\|_\mu^2,
\end{align}
where we recall that $L^a_{[s,t)}(A)$ is defined in Definition \ref{Definition: Regular Local Time}.
It holds that
\begin{multline}
\label{Equation: Smooth Mixed Moment}
\mbf E\left[\prod_{k=1}^n\int_{\mc A}\hat K^{\eps_k,\ze_k}(t_k;a,a)\d\mu(a)\right]\\
=\int_{\mc A^n}\Pi_A(\bs t;\bs a,\bs a)\mbf E\left[\mf m^{\bs\ze}_{\bs t}(A^{\bs a,\bs a}_{\bs t})\mr e^{-\int_0^{|\bs t|}V(A^{\bs a,\bs a}_{\bs t}(s))\d s+\mf s^{\bs\eps}_{\bs t}(A^{\bs a,\bs a}_{\bs t})+\mf B_{|\bs t|}(A^{\bs a,\bs a}_{\bs t})}\right]\d\mu^n(\bs a).
\end{multline}
\end{proposition}

With this in hand, we can prove that approximate trace moments are finite (in fact, uniformly integrable)
and converge to the expression stated in Theorem \ref{Theorem: Trace Moment Formulas}:

\begin{proposition}
\label{Proposition: Uniform Integrability}
For every $\bs t=(t_1,\ldots,t_n)\in(0,\infty)^n$,
\[\sup_{\bs\eps\in[0,\infty)^n,~\bs\ze\in(0,\infty)^n}\mbf E\left[\left|\prod_{k=1}^n\int_{\mc A}\hat K^{\eps_k,\ze_k}(t_k;a,a)\d\mu(a)\right|^2\right]<\infty.\]
\end{proposition}

\begin{proposition}
\label{Proposition: Trace Moments Limit 1}
For every $t,\ze>0$,
\begin{align}
\label{Equation: Trace Moments Limit 1}
\lim_{\eps\to0}\mbf E\left[\left(\int_{\mc A} \hat K^{\eps,\ze}(t;a,a)-\hat K^{0,\ze}(t;a,a)\d\mu(a)\right)^2\right]=0.
\end{align}
\end{proposition}

\begin{proposition}
\label{Proposition: Trace Moments Limit}
For every $\bs t=(t_1,\ldots,t_n)\in(0,\infty)^n$,
one has
\begin{align}
\label{Equation: Trace Moments Limit 2.1}
\lim_{\ze_1,\ldots,\ze_n\to0}\mbf E\left[\prod_{k=1}^n\int_{\mc A} \hat K^{0,\ze_k}(t_k;a,a)\d\mu(a)\right]=\text{right-hand-side of }\eqref{Equation: Trace Moment Formula},
\end{align}
and for every $t>0$, one has
\begin{align}
\label{Equation: Trace Moments Limit 2.2}
\lim_{\ze_1,\ze_2\to0}\mbf E\left[\left(\int_{\mc A} \hat K^{0,\ze_1}(t;a,a)-\hat K^{0,\ze_2}(t;a,a)\d\mu(a)\right)^2\right]=0.
\end{align}
\end{proposition}

These four propositions are respectively proved in Sections
\ref{Section: Main Proposition 1}, \ref{Section: Main Proposition 2}, \ref{Section: Main Proposition 3},
and \ref{Section: Main Proposition 4}.

\subsubsection{Step 2. Operator Limits}
\label{Section: Main Result Roadmap Step 2}

With Propositions \ref{Proposition: Uniform Integrability}--\ref{Proposition: Trace Moments Limit}
in hand, the proof of Theorem \ref{Theorem: Trace Moment Formulas} now relies on ensuring that
the limits in \eqref{Equation: Trace Moments Limit 2.1} coincide with the mixed moments
of $\hat H=\hat H^{0,0}$ in \eqref{Equation: Mixed Moments}.
Two key technical ingredients in this process are the following eigenvalue boundedness and convergence results:

\begin{proposition}
\label{Proposition: Uniform Eigenvalue Bounds}
For every $0<\ka<1$, there exists a finite random variable $\nu>0$ such that,
almost surely,
\[(1-\ka)\la_k(H)-\nu\leq\la_k(\hat H^{\eps,\ze})\leq(1+\ka)\la_k(H)+\nu\]
for every $k\geq1$ and $\eps,\ze\in[0,1)$.
\end{proposition}

\begin{proposition}
\label{Proposition: Eigenvalue Convergence}
Let $\ze>0$ be fixed. Almost surely, every vanishing sequence $(\eps_m)_{m\in\mbb N}$
has a subsequence $(\eps_{m_\ell})_{\ell\in\mbb N}$ along which
\begin{align}
\label{Equation: Eigenvalue Convergence 1}
\lim_{\ell\to\infty}\la_k(\hat H^{\eps_{m_\ell},\ze})=\la_k(\hat H^{0,\ze})\qquad\text{for every $k\geq1$.}
\end{align}
Moreover, almost surely,
every vanishing sequence $(\ze_m)_{m\in\mbb N}$
has a subsequence $(\ze_{m_\ell})_{\ell\in\mbb N}$ along which
\begin{align}
\label{Equation: Eigenvalue Convergence 2}
\lim_{\ell\to\infty}\la_k(\hat H^{0,\ze_{m_\ell}})=\la_k(\hat H)\qquad\text{for every $k\geq1$.}
\end{align}
\end{proposition}

These results (along with Proposition \ref{Proposition: Operator Form-Bound}) are proved in
Sections \ref{Section: Main Proposition 5}, \ref{Section: Main Proposition 6}, and \ref{Section: Main Proposition 7}.

\begin{remark}
Unlike the kernel convergence results in
Propositions \ref{Proposition: Uniform Integrability}--\ref{Proposition: Trace Moments Limit},
there is no advantage to considering the limits $\eps\to0$ and $\ze\to0$ separately in
Proposition \ref{Proposition: Eigenvalue Convergence}. This is only done for consistency
with the kernel convergence results.
\end{remark}

\subsubsection{Step 3. Proof of Theorem \ref{Theorem: Trace Moment Formulas}}

We are now in a position to wrap up the proof of Theorem \ref{Theorem: Trace Moment Formulas}.
First, we establish that for every $t,\ze>0$, almost surely,
\begin{align}
\label{Equation: Trace Formula for 0,zeta}
\mr{Tr}\big[\mr e^{-t\hat H^{0,\ze}}\big]=\int_{\mc A}\hat K^{0,\ze}(t;a,a)\d\mu(a).
\end{align}
As $L^2$ convergence implies convergence in probability,
thanks to \eqref{Equation: Trace Moments Limit 1},
there exists a vanishing sequence $(\eps_m)_{m\in\mbb N}\subset(0,1)$ such that
\begin{align}
\label{Equation: Trace Formula for 0,zeta 1}
\int_{\mc A}\hat K^{0,\ze}(t;a,a)\d\mu(a)=\lim_{m\to\infty}\int_{\mc A}\hat K^{\eps_m,\ze}(t;a,a)\d\mu(a)
\end{align}
almost surely. By combining Theorem \ref{Theorem: Regular FK}
with Propositions \ref{Proposition: Uniform Eigenvalue Bounds} and \ref{Proposition: Eigenvalue Convergence},
we can restrict this convergence to a (possibly) smaller probability-one event on which the following two conditions hold: On the one hand, by Theorem \ref{Theorem: Regular FK},
\begin{align}
\label{Equation: Trace Formula for 0,zeta 2}
\int_{\mc A}\hat K^{\eps_m,\ze}(t;a,a)\d\mu(a)=\mr{Tr}\big[\mr e^{-t\hat H^{\eps_m,\ze}}\big]\qquad\text{for every }m\in\mbb N;
\end{align}
on the other hand there exists a subsequence $(\eps_{m_\ell})_{\ell\in\mbb N}$ (which may depend on the
specific outcome selected in the probability-one event)
along which
\begin{align}
\label{Equation: Trace Formula for 0,zeta 3}
\mr{Tr}\big[\mr e^{-t\hat H^{0,\ze}}\big]=\lim_{\ell\to\infty}\mr{Tr}\big[\mr e^{-t\hat H^{\eps_{m_\ell},\ze}}\big].
\end{align}
More specifically, in order to get \eqref{Equation: Trace Formula for 0,zeta 3},
we combine the pointwise convergence of eigenvalues in \eqref{Equation: Eigenvalue Convergence 1}
along a subsequence
with an application of dominated convergence in the series
\[\mr{Tr}\big[\mr e^{-t\hat H^{\eps,\ze}}\big]=\sum_{k=1}^\infty\mr e^{-t\la_k(\hat H^{\eps,\ze})}\] using 
the bounds in Proposition
\ref{Proposition: Uniform Eigenvalue Bounds} and \eqref{Equation: H is Trace Class}.
A combination of \eqref{Equation: Trace Formula for 0,zeta 2}
with the limits \eqref{Equation: Trace Formula for 0,zeta 1}
and \eqref{Equation: Trace Formula for 0,zeta 3} yields that \eqref{Equation: Trace Formula for 0,zeta} holds on a probability-one
event, as desired.

We now obtain the statement of Theorem \ref{Theorem: Trace Moment Formulas}:
Let $t_1,\ldots,t_n>0$ be fixed. By \eqref{Equation: Trace Moments Limit 2.2}, there exists a vanishing sequence
$(\ze_m)_{m\in\mbb N}$ and some square-integrable random variables $\mr T(t_k)$ such that
\begin{align}
\label{Equation: Trace Formula for 0,0 1}
\lim_{m\to\infty}\int_{\mc A}\hat K^{0,\ze_m}(t_k;a,a)\d\mu(a)=\mr T(t_k)
\end{align}
almost surely for all $1\leq k\leq n$.
By \eqref{Equation: Trace Moments Limit 2.1} and Proposition \ref{Proposition: Uniform Integrability}
(the latter of which implies the uniform integrability of $\prod_{k=1}^n\int_{\mc A} \hat K^{0,\ze}(t_k;a,a)\d\mu(a)$ over $\ze\in(0,1)$),
\[\lim_{m\to\infty}\mbf E\left[\prod_{k=1}^n\int_{\mc A} \hat K^{0,\ze_m}(t_k;a,a)\d\mu(a)\right]=\mbf E\big[\mr T(t_1)\cdots\mr T(t_n)\big]=\text{right-hand-side of }\eqref{Equation: Trace Moment Formula}.\]
It now only remains to show that
\begin{align}
\label{Equation: Trace Formula for 0,0 2}
\mr T(t_k)=\mr{Tr}\big[\mr e^{-t_k\hat H}\big]\qquad\text{almost surely for all }1\leq k\leq n.
\end{align}
Toward this end,
by \eqref{Equation: Eigenvalue Convergence 2} and
\eqref{Equation: H is Trace Class}/Proposition \ref{Proposition: Uniform Eigenvalue Bounds}
(the latter two of which once again allows to use dominated convergence), we can restrict to
a probability-one event on which there always exists a subsequence $(\ze_{m_\ell})_{\ell\in\mbb N}$ along which
\[\lim_{\ell\to\infty}\mr{Tr}\big[\mr e^{-t_k\hat H^{0,\ze_{m_\ell}}}\big]=\mr{Tr}\big[\mr e^{-t_k\hat H}\big]\qquad\text{for all }1\leq k\leq n.\]
If we combine this limit with \eqref{Equation: Trace Formula for 0,zeta} and \eqref{Equation: Trace Formula for 0,0 1},
then we obtain \eqref{Equation: Trace Formula for 0,0 2},
thus concluding the proof of Theorem \ref{Theorem: Trace Moment Formulas}.

\subsection{Proof of Proposition \ref{Proposition: Smooth Trace Moments}}
\label{Section: Main Proposition 1}

Recall the notations $Q^\ze$ ($\ze>0$) and $S^0$
introduced respectively in \eqref{Equation: Q zeta} and \eqref{Equation: Diagonal Potential Alternate Notation},
and for each $\eps>0$, define also the functional
\begin{align}
W^\eps_{\mr d}(a)=W^\eps_{\mr d}(i,x)=\si W^\eps_{i,i}(x);
\qquad
\label{Equation: Diagonal Potential Alternate Notation Epsilon}
S^\eps(f)=\int_{\mc A} \Big(V(a)+\big(W^\eps_{\mr d}(a)\big)' \Big)f(a)\d\mu(a),
\end{align}
By definition of the kernels $\hat K^{\ze,\eps}$ for $\eps\geq0$ and $\ze>0$
in \eqref{Equation: Regular FK Kernel} and \eqref{Equation: Kernel at epsilon zero}, one has
\begin{multline}
\label{Equation: Smooth Trace Moments 1}
\mbf E\left[\prod_{k=1}^n\int_{\mc A}\hat K^{\eps_k,\ze_k}(t_k;a,a)\d\mu(a)\right]\\
=\mbf E\left[\prod_{k=1}^n\int_{\mc A}\Pi_A(t_k;a_k,a_k)\mbf E_A\left[\mf n_t(Q^{\ze_k},A^{a_k,a_k}_{t_k})\,\mr e^{-S^{\eps_k}(L_{t_k}(A^{a_k,a_k}_{t_k}))+\mf B_t(A^{a_k,a_k}_{t_k})}\right]\d\mu(a_k)\right],
\end{multline}
where $A$ is independent of $\xi$ (i.e., the randomness in $Q^{\ze_k}$ and $S^{\eps_k}$).
In order to calculate this expectation, we aim to apply Fubini's theorem three times:
Firstly to take the product over $k$ inside the $\dd\mu(a)$
integrals (and thus combine the latter into a single integral with respect to $\bs a=(a_1,\ldots,a_n)\in\mc A^n$);
secondly to combine the product of expectations $\mbf E_A[\cdot]$
into a single expectation with respect to a concatenated path $A^{\bs a,\bs a}_{\bs t}$;
and thirdly to bring the expectation with respect to $\xi$
inside the integral and expectation with respect to $\mc A^n$ and $A$.
Thus, if we define the functionals
\begin{multline}
\label{Equation: Multivariate nt functional}
\mf n_{\bs t}(Q^{\bs\ze},A)=\mr e^{(r-1)|\bs t|}(-1)^{N(|\bs t|)}\prod_{k=1}^n\left(\prod_{\ell=N(t_0+\cdots+t_{k-1})+1}^{N(t_0+\cdots+t_k)}\xi^{\ze_k}_{J_\ell}\big(Z(\tau_\ell)\big)\right)\\
=\mr e^{(r-1)|\bs t|}(-1)^{N(|\bs t|)}\prod_{\ell=1}^{N(|\bs t|)}\xi^{\ze(\tau_\ell)}_{J_\ell}\big(Z(\tau_\ell)\big),
\end{multline}
where we recall the notation of $\ze(s)$ in \eqref{Equation: zeta s}, and
\begin{align}
\label{Equation: Diagonal Potential Alternate Notation Multivariate Epsilon}
S^{\bs \eps}\big(L_{\bs t}(A^{\bs a,\bs a}_{\bs t})\big)=\sum_{i=1}^kS^{\eps_k}\left(L_{[t_0+\cdots+t_{k-1},t_0,\ldots,t_k)}(A^{\bs a,\bs a}_{\bs t})\right)
\end{align}
(i.e., the respective product and sum $\prod_{k=1}^n\mf n_t(Q^{\ze_k},A^{a_k,a_k}_{t_k})$ and $\sum_{k=1}^nS^{\eps_k}(L_{t_k}(A^{a_k,a_k}_{t_k}))$
where the independent processes $A^{a_k,a_k}_{t_k}$ in each expectation $\mbf E_A[\cdot]$ in \eqref{Equation: Smooth Trace Moments 1} have been combined into a single
concatenated process $A^{\bs a,\bs a}_{\bs t}$), then
an application of Fubini's theorem as described above allows to rewrite \eqref{Equation: Smooth Trace Moments 1} as follows:
\begin{align}
\label{Equation: Smooth Trace Moments 2}
\int_{\mc A^n}\Pi_A(\bs t;\bs a,\bs a)\mbf E_A\left[\mbf E_\xi\left[\mf n_{\bs t}(Q^{\bs\ze},A^{\bs a,\bs a}_{\bs t})\,\mr e^{-S^{\bs \eps}(L_{\bs t}(A^{\bs a,\bs a}_{\bs t}))+\mf B_{|\bs t|}(A^{\bs a,\bs a}_{\bs t})}\right]\right]\d\mu^n(\bs a).
\end{align}
Once this is established, the proof of Proposition \ref{Proposition: Smooth Trace Moments}
reduces to verifying that \eqref{Equation: Smooth Trace Moments 2} 
is equal to the right-hand side of \eqref{Equation: Smooth Mixed Moment}. Before doing so,
however, we verify that the application of Fubini's theorem is legitimate.

\subsubsection{Step 1. Justify Fubini's Theorem}

In order to be able to apply Fubini's theorem and obtain \eqref{Equation: Smooth Trace Moments 2}, it suffices to check
that the product
\[\prod_{k=1}^n\Pi_A(t_k;a_k,a_k)\mf n_t(Q^{\ze_k},A^{a_k,a_k}_{t_k})\,\mr e^{-S^{\eps_k}(L_{t_k}(A^{a_k,a_k}_{t_k}))+\mf B_t(A^{a_k,a_k}_{t_k})}\]
is integrable with respect to $\mbf E_\xi\otimes\mbf E_{A}\otimes\dd\mu^n(\bs a)$, which
by Tonelli's theorem reduces to
\begin{align}
\label{Equation: Fubini's Theorem}
\int_{\mc A^n}\Pi_A(\bs t;\bs a,\bs a)\mbf E_A\left[\mbf E_\xi\left[\left|\mf n_{\bs t}(Q^{\bs\ze},A^{\bs a,\bs a}_{\bs t})\right|\mr e^{-S^{\bs \eps}(L_{\bs t}(A^{\bs a,\bs a}_{\bs t}))+\mf B_{|\bs t|}(A^{\bs a,\bs a}_{\bs t})}\right]\right]\d\mu^n(\bs a)<\infty.
\end{align}
Since the diagonal and off-diagonal entries of $\xi$ are independent
and the boundary term $\mf B_{|\bs t|}$ does not depend on $\xi$ we can write this as
\begin{align}
\label{Equation: Fubini's Theorem 1}
\int_{\mc A^n}\Pi_A(\bs t;\bs a,\bs a)\mbf E_A\left[\mbf E_\xi\left[\left|\mf n_{\bs t}(Q^{\bs\ze},A^{\bs a,\bs a}_{\bs t})\right|\right]\mbf E_\xi\left[\mr e^{-S^{\bs \eps}(L_{\bs t}(A^{\bs a,\bs a}_{\bs t}))}\right]\mr e^{\mf B_{|\bs t|}(A^{\bs a,\bs a}_{\bs t})}\right]\d\mu^n(\bs a)<\infty.
\end{align}

We control the terms appearing in the expectation $\mbf E_A[\cdot]$
in \eqref{Equation: Fubini's Theorem 1} one at a time:
If we denote $\bs a$'s components by $a_k=(i_k,x_k)$
for $1\leq k\leq n$
and we let $\bs x=(x_1,\ldots,x_n)$, then
by \eqref{Equation: F-K Integrability Bound 1} we have that
\begin{align}
\label{Equation: UI Boundary Bound}
\mf B_{|\bs t|}(A^{\bs a,\bs a}_{\bs t})\leq
\bar{\mf B_{|\bs t|}}(Z^{\bs x,\bs x}_{\bs t})=\begin{cases}
0&\text{Case 1},\\
\bar\al_\mr m\mf L^0_{|\bs t|}(X^{\bs x,\bs x}_{\bs t})&\text{Case 2},\\
\bar\al_\mr m\mf L^0_{|\bs t|}(Y^{\bs x,\bs x}_{\bs t})+\bar\be_\mr m\mf L^{\vartheta}_{|\bs t|}(Y^{\bs x,\bs x}_{\bs t})&\text{Case 3}.
\end{cases}
\end{align}

Next, by \eqref{Equation: Multivariate nt functional} and H\"older's inequality,
\[\mbf E_\xi\left[\left|\mf n_{\bs t}(Q^{\bs\ze},A^{\bs a,\bs a}_{\bs t})\right|\right]\leq\mr e^{(r-1)|\bs t|}\prod_{\ell=1}^{N(|\bs t|)}\mbf E_\xi\left[\left|\xi^{\ze(\tau_\ell)}_{J_\ell}\big(Z(\tau_\ell)\big)\right|^{N(|\bs t|)}\right]^{1/N(|\bs t|)}.\]
Note that the components of $\xi^{\ze}_{i,j}$ (or $\xi^{\ze}_{i,j}$ itself if $\mbb F=\mbb R$)
are i.i.d. centered Gaussian processes on $\mbb R$ with stationary covariance $\upsilon^2\rho^\ze=\upsilon^2\ze^{-1}\rho(\cdot/\ze)$. Therefore, if we denote the minimum $\ze_{\mr m}=\min\{\ze_1,\ldots,\ze_n\}$,
then there exists some deterministic constant $\theta>0$
(which depends only on $|\bs t|$ and $\mbb F$) such that
\begin{multline}
\label{Equation: Fubini nt bound}
\mbf E_Q\big[|\mf n_{|\bs t|}(Q,A)|\big]\leq\mr e^{(r-1)|\bs t|}\prod_{\ell=1}^{N(|\bs t|)}\mbf E_Q\Big[\big|\xi_{J_\ell}^{\ze(\tau_\ell)}(0)\big|^{N(|\bs t|)}\Big]^{1/N(|\bs t|)}\\
\leq\theta\big(2\upsilon^2\ze_{\mr m}^{-1}\rho(0)\big)^{N(|\bs t|)/2}\Ga\left(\frac{N(|\bs t|)+1}{2}\right),
\end{multline}
where $\Ga$ is the Gamma function.

Lastly,
by definition of $S^\eps$ in \eqref{Equation: Diagonal Potential Alternate Notation} and
\eqref{Equation: Diagonal Potential Alternate Notation Epsilon}, we can alternatively write
\[S^\eps\big(L_{[s,t)}(A)\big)=\int_s^tV\big(A(u)\big)\d u+\int_{\mc A}L_{[s,t)}(A)\star\bar\rho_\eps(a)\d W_{\mr d}(a).\]
Thus, by definition of $S^{\bs\eps}$ in \eqref{Equation: Diagonal Potential Alternate Notation Multivariate Epsilon}
and of $W_{\mr d}$ in \eqref{Equation: Diagonal Noise Alternate Notation}, we get that
\[S^{\bs \eps}\big(L_{\bs t}(A^{\bs a,\bs a}_{\bs t})\big)=\int_0^{|\bs t|}V\big(A^{\bs a,\bs a}_{\bs t}(s)\big)\d s+\int_{\mc A}\sum_{k=1}^nL_{[t_0+\cdots+t_{k-1},t_0+\cdots+t_k)}(A)\star\bar\rho_{\eps_k}(a)\d W_{\mr d}(a)\]
is Gaussian with mean $\int_0^{|\bs t|}V\big(A^{\bs a,\bs a}_{\bs t}(s)\big)\d s$ and variance
\[\si^2\left\|\sum_{k=1}^nL_{[t_0+\cdots+t_{k-1},t_0+\cdots+t_k)}(A)\star\bar\rho_{\eps_k}\right\|_\mu^2.\]
Thus, by definition of $\mf s^{\bs\eps}_{\bs t}$ in \eqref{Equation: Vector Valued Self-Intersection 2} and a moment generating function calculation,
\begin{align}
\label{Equation: Smooth Moments: Gaussian MGF Part}
\mbf E_\xi\left[\mr e^{-S^{\bs \eps}(L_{\bs t}(A^{\bs a,\bs a}_{\bs t}))}\right]=\mr e^{-\int_0^{|\bs t|}V(A^{\bs a,\bs a}_{\bs t}(s))\d s+\mf s^{\bs\eps}_{\bs t}(A^{\bs a,\bs a}_{\bs t})}.
\end{align}
By Assumption \ref{Assumption: Potential}, there exists some $\bar{\mf a}\in(0,1]$ and $\ka,\nu>0$
such that $V(i,x)\geq\ka|x|^{\bar{\mf a}}-\nu$ for all $x\in I$ (in Case 1 and 2, take $\bar{\mf a}=\min\{1,\mf a\}$, where
$\mf a>0$ in as in Assumption \ref{Assumption: Potential}; in Case 3, this follows from the fact
that $V$ is bounded below). Therefore,
\begin{align}
\label{Equation: Fubini V bound}
\mr e^{-\int_0^{|\bs t|}V(A^{\bs a,\bs a}_{\bs t}(s))\d s}\leq\mr e^{-\int_0^{|\bs t|} \ka|Z^{\bs x,\bs x}_{\bs t}(s)|^{\bar{\mf a}}\d s+\nu|\bs t|}.
\end{align}
By Jensen's inequality,
\begin{align}
\label{Equation: s_t bound 1}
\mf s^{\bs\eps}_{\bs t}(A^{\bs a,\bs a}_{\bs t})\leq\frac{n\si^2}{2}\sum_{k=1}^n\|L_{[t_0+\cdots+t_{k-1},t_0+\cdots+t_k)}(A^{\bs a,\bs a}_{\bs t})\star\bar\rho_{\eps_k}\|_\mu^2.
\end{align}
Then, since $\rho_{\eps}$ is a probability density function for all $\eps>0$, it follows from applying Young's convolution inequality to \eqref{Equation: s_t bound 1} that
\begin{align}
\label{Equation: s_t bound 2}
\mf s^{\bs\eps}_{\bs t}(A^{\bs a,\bs a}_{\bs t})\leq\frac{n\si^2}{2}\sum_{k=1}^n\|L_{[t_0+\cdots+t_{k-1},t_0+\cdots+t_k)}(A^{\bs a,\bs a}_{\bs t})\|_\mu^2.
\end{align}
For any nonnegative function $f\in L^2(\mc A,\mbb R)$,
if we denote $\bar f(x)=\sum_{i=1}^rf(i,x)$, then
\begin{align}
\label{Equation: s_t bound 3.0}
\|f\|_\mu^2
=\sum_{i=1}^r\int_I f(i,x)^2\d x
\leq\int_I \left(\sum_{i=1}^rf(i,x)\right)^2\d x
=\|\bar f\|_2^2.
\end{align}
Thus, we get from \eqref{Equation: s_t bound 2} that
\begin{align}
\label{Equation: s_t bound 3}
\mf s^{\bs\eps}_{\bs t}(A^{\bs a,\bs a}_{\bs t})\leq\frac{n\si^2}{2}\sum_{k=1}^n\|L_{[t_0+\cdots+t_{k-1},t_0+\cdots+t_k)}(Z^{\bs x,\bs x}_{\bs t})\|_2^2.
\end{align}

At this point, if we put \eqref{Equation: UI Boundary Bound},
\eqref{Equation: Fubini nt bound}, \eqref{Equation: Fubini V bound},
and \eqref{Equation: s_t bound 3} into \eqref{Equation: Fubini's Theorem 1},
then justifying \eqref{Equation: Smooth Trace Moments 2} reduces
to proving the following claim:
\begin{multline}
\label{Equation: Fubini's Theorem 2}
\mbf E\left[\big(2\upsilon^2\ze_{\mr m}^{-1}\rho(0)\big)^{N(|\bs t|)/2}\Ga\left(\frac{N(|\bs t|)+1}{2}\right)\right]
\int_{I^n}\Pi_Z(\bs t;\bs x,\bs x)\\
\cdot\mbf E\left[\mr e^{-\int_0^{|\bs t|} \ka|Z^{\bs x,\bs x}_{\bs t}(s)|^{\bar{\mf a}}\d s+\bar{\mf B_{|\bs t|}}(Z^{\bs x,\bs x}_{\bs t})}
\exp\left(\frac{n\si^2}{2}\sum_{k=1}^n\|L_{[t_0+\cdots+t_{k-1},t_0+\cdots+t_k)}(Z^{\bs x,\bs x}_{\bs t})\|_2^2\right)\right]\d\bs x<\infty.
\end{multline}
Since $N$ is a Poisson process, $\mbf E\left[\big(2\upsilon^2\ze_{\mr m}^{-1}\rho(0)\big)^{N(|\bs t|)/2}\Ga\left(\frac{N(|\bs t|)+1}{2}\right)\right]$ is finite.
Thus, by an application of H\"older's inequality, \eqref{Equation: Fubini's Theorem 2} reduces to
\begin{multline}
\label{Equation: Fubini's Theorem 3}
\sup_{\bs x\in I^n}\mbf E\left[\exp\left(\frac{3n\si^2}2\sum_{k=1}^n\|L_{[t_0+\cdots+t_{k-1},t_0+\cdots+t_k)}(Z^{\bs x,\bs x}_{\bs t})\|_2^2\right)\right]^{1/3}\mbf E\left[\mr e^{3\bar{\mf B_{|\bs t|}}(Z^{\bs x,\bs x}_{\bs t})}\right]^{1/3}\\
\cdot
\int_{I^n}\Pi_Z(\bs t;\bs x,\bs x)\mbf E\left[\mr e^{-3\int_0^{|\bs t|} \ka|Z^{\bs x,\bs x}_{\bs t}(s)|^{\bar{\mf a}}\d s}\right]^{1/3}\d\bs x<\infty.
\end{multline}
If we combine \eqref{Equation: Potential Power Growth Triangle Inequality}, \eqref{Equation: Transition Kernel Bound}, and \eqref{Equation: Log Integral Finite},
then we get that the second line in \eqref{Equation: Fubini's Theorem 3} is finite.
As for the expectations on the first line of \eqref{Equation: Fubini's Theorem 3},
we note that
\begin{align}
\label{Equation: Concatenated intersection local time is piecewise independent}
L_{[t_0+\cdots+t_i,t_0+\cdots+t_{i+1})}(Z^{\bs x,\bs x}_{\bs t})\deq L_{t_{i+1}}(Z^{x_{i+1},x_{i+1}}_{t_{i+1}})
\qquad\text{for }0\leq i\leq n-1,
\end{align}
and that these variables are independent for different $i$'s. Moreover,
\begin{align}
\label{Equation: Concatenated boundary local time is piecewise independent 1}
&\mf L^c_{|\bs t|}(Z^{\bs x,\bs x}_{\bs t})=\sum_{i=0}^{n-1}\mf L^c_{[t_0+\cdots+t_i,t_0+\cdots+t_{i+1})}(Z^{\bs x,\bs x}_{\bs t}),\\
\label{Equation: Concatenated boundary local time is piecewise independent 2}
&\mf L^c_{[t_0+\cdots+t_i,t_0+\cdots+t_{i+1})}(Z^{\bs x,\bs x}_{\bs t})\deq\mf L^c_{t_{i+1}}(Z^{x_{i+1},x_{i+1}}_{t_{i+1}})
\qquad\text{for }0\leq i\leq n-1,
\end{align}
with the random variables on the second line being independent for different $i$'s.
Thus, we get that the first line of \eqref{Equation: Fubini's Theorem 3}
is finite by combining \eqref{Equation: Exponential Moments of Boundary Local Time}
with
\begin{align}
\label{Equation: Finiteness of Bridge SILT Exponential Moments}
\sup_{x\in I}\mbf E\left[\mr e^{\nu\|L_t(Z_t^{x,x})\|_2^2}\right]<\infty\qquad\text{for every }\nu,t>0,
\end{align}
the latter of which is proved in \cite[Lemma 5.11]{GaudreauLamarreEJP}.

This concludes the proof that we can use Fubini's theorem in the calculation
of \eqref{Equation: Smooth Trace Moments 1}; we thus proceed to the computation of
the right-hand side of \eqref{Equation: Smooth Mixed Moment}
using \eqref{Equation: Smooth Trace Moments 2}.

\subsubsection{Step 2. Smooth Mixed Moments Calculation}

If we use the independence of the diagonal and off-diagonal noises in $\xi$ and the fact that $\mf B_{|\bs t|}$
does not depend on $\xi$, then we can rewrite \eqref{Equation: Smooth Trace Moments 2} as
\begin{align}
\label{Equation: Smooth Trace Moments 3}
\int_{\mc A^n}\Pi_A(\bs t;\bs a,\bs a)\mbf E_A\left[\mbf E_\xi\left[\mf n_{\bs t}(Q^{\bs\ze},A^{\bs a,\bs a}_{\bs t})\right]\mbf E_\xi\left[\mr e^{-S^{\bs \eps}(L_{\bs t}(A^{\bs a,\bs a}_{\bs t}))}\right]\mr e^{\mf B_{|\bs t|}(A^{\bs a,\bs a}_{\bs t})}\right]\d\mu^n(\bs a).
\end{align}
If we plug \eqref{Equation: Smooth Moments: Gaussian MGF Part} into this and compare the result with
the right-hand side of \eqref{Equation: Smooth Mixed Moment}, then we see that the only
claim that remains to be proved to get the statement of Proposition \ref{Proposition: Smooth Trace Moments}
is the following:
\begin{multline}
\label{Equation: Isserlis}
\mbf E_\xi\left[\mf n_{\bs t}(Q^{\bs\ze},A)\right]=
\mf m^{\bs\ze}_{\bs t}(A)\\
=\mr e^{(r-1)|\bs t|}\mbf 1_{\{N(|\bs t|)\text{ is even}\}}\sum_{p\in\mc P_{N(t)}}
\mf C_{|\bs t|}(p,U)\prod_{\{\ell_1,\ell_2\}\in p}\rho_{\ze(\tau_{\ell_1}),\ze(\tau_{\ell_2})}\big(Z(\tau_{\ell_1})-Z(\tau_{\ell_2})\big),
\end{multline}
where the second line is just a repeat of \eqref{Equation: Vector Valued Self-Intersection 1} for convenience.
When $\mbb F=\mbb R,\mbb C$, this is a
straightforward consequence of Isserlis' theorem \cite{Isserlis}.
The argument is somewhat similar for $\mbb F=\mbb H$, but
rather more involved since the product in $\mbb H$ is not commutative.
We thus finish the proof of Proposition \ref{Proposition: Smooth Trace Moments}
on a case-by-case basis:

\subsubsection{Step 2 Part 1. Proof of \eqref{Equation: Isserlis} for $\mbb F=\mbb R,\mbb C$}

For $\mbb F=\mbb R$ or $\mbb C$, we invoke the statement
of Isserlis' theorem from \cite[(1.7) and (1.8)]{MingoSpeicher}, whereby
\[\mbf E_\xi\left[\prod_{\ell=1}^{N(|\bs t|)}\xi^{\ze(\tau_\ell)}_{J_\ell}\big(Z(\tau_\ell)\big)\right]
=\mbf 1_{\{N(|\bs t|)\text{ is even}\}}
\sum_{p\in\mc P_{N(t)}}\prod_{\{\ell_1,\ell_2\}\in p}\mbf E_\xi\left[\xi^{\ze(\tau_{\ell_1})}_{J_{\ell_1}}\big(Z(\tau_{\ell_1})\big)\xi^{\ze(\tau_{\ell_1})}_{J_{\ell_2}}\big(Z(\tau_{\ell_2})\big)\right]\]
By independence,
\[\mbf E\big[\xi^\ze_{i,j}(x)\xi^\eta_{k,\ell}(y)\big]=0\]
whenever $(i,j)\neq(k,\ell)$ and $(i,j)\neq(\ell,k)$. Otherwise, the analysis depends on $\mbb F$:

If $\mbb F=\mbb R$,
then it is easy to see (e.g., the calculation in \cite[(3.3)]{GaudreauLamarreEJP},
where the values of $\eps$ in $\Xi_\eps'$ are not equal) that for any $\ze,\eta>0$ and $x,y\in I$,
\begin{align}
\label{Equation: Isserlis on R}
\mbf E\left[\xi^{\ze}_{i,j}(x)\xi^{\eta}_{i,j}(y)\right]=\mbf E\left[\xi^{\ze}_{i,j}(x)\xi^{\eta}_{j,i}(y)\right]=\upsilon^2\bar\rho_{\ze}\star\bar\rho_{\eta}
=\upsilon^2\rho_{\ze,\eta}(x-y).
\end{align}
\eqref{Equation: Isserlis} for $\mbb R$ then follows directly by the definition of $\mf C_t(p,U)$
in Definition \ref{Definition: Combinatorial Constant}-(1).

If $\mbb F=\mbb C$,
then looking back at the decomposition \eqref{Equation: Noise Coupling Component Decomposition},
we can write
\[\xi^\ze_{i,j}(x)=\tfrac1{\sqrt2}\left(\xi^\ze_{i,j;1}(x)+\xi^\ze_{i,j;\msf i}(x)\msf i\right),\qquad \ze>0,~x\in I,\]
where $\xi^\ze_{i,j;\msf l}=(W_{i,j;\msf l}\star\bar\rho_\ze)'$ for $\msf l=1,\msf i$.
Thus, if we reuse the calculation done in \eqref{Equation: Isserlis on R},
we get that, in the complex case,
\begin{align}
\nonumber
&\mbf E\left[\xi^{\ze}_{i,j}(x)\xi^{\eta}_{i,j}(y)\right]\\
\nonumber
&=\tfrac12\mbf E\left[\xi^{\ze}_{i,j;1}(x)\xi^{\eta}_{i,j;1}(y)+\big(\xi^{\ze}_{i,j;1}(x)\xi^{\eta}_{i,j;\msf i}(y)+\xi^{\ze}_{i,j;\msf i}(x)\xi^{\eta}_{i,j;1}(y)\big)\msf i-\xi^{\ze}_{i,j;\msf i}(x)\xi^{\eta}_{i,j;\msf i}(y)\right]\\
\label{Equation: Isserlis Complex 1}
&=\tfrac{\upsilon^2}2\rho_{\ze,\eta}(x-y)+0-\tfrac{\upsilon^2}2\rho_{\ze,\eta}(x-y)=0.
\end{align}
and similarly
\begin{align}
\nonumber
&\mbf E\left[\xi^{\ze}_{i,j}(x)\xi^{\eta}_{j,i}(y)\right]=\mbf E\left[\xi^{\ze}_{i,j}(x)\xi^{\eta}_{i,j}(y)^*\right]\\
\nonumber
&=\tfrac12\mbf E\left[\xi^{\ze}_{i,j;1}(x)\xi^{\eta}_{i,j;1}(y)+\big(-\xi^{\ze}_{i,j;1}(x)\xi^{\eta}_{i,j;\msf i}(y)+\xi^{\ze}_{i,j;\msf i}(x)\xi^{\eta}_{i,j;1}(y)\big)\msf i+\xi^{\ze}_{i,j;\msf i}(x)\xi^{\eta}_{i,j;\msf i}(y)\right]\\
\label{Equation: Isserlis Complex 2}
&=\tfrac{\upsilon^2}2\rho_{\ze,\eta}(x-y)+0+\tfrac{\upsilon^2}2\rho_{\ze,\eta}(x-y)=\upsilon^2\rho_{\ze,\eta}(x-y).
\end{align}
where in both calculations we have used the fact that $W_{i,j;1}$ and $W_{i,j;\msf i}$
are independent.
\eqref{Equation: Isserlis} for $\mbb C$ then follows by the definition of $\mf C_t(p,U)$
in Definition \ref{Definition: Combinatorial Constant}-(2).

\subsubsection{Step 2 Part 1. Proof of \eqref{Equation: Isserlis} for $\mbb F=\mbb H$}

In the case of quaternions, Isserlis' theorem does not
apply directly due to the lack of commutativity.
To get around this, we proceed as follows:
Let $q=a+b\msf i+c\msf j+d\msf k$ be a quaternion.
We recall that there exists an injective homomorphism from
$\mbb F$ to the space of $2\times 2$ complex matrices via
the representation
\begin{align}
\label{Equation: Injective Homomorphism}
q=\left[\begin{array}{cc}
a+b\msf i&c+d\msf i\\
-(c-d\msf i)&a-b\msf i
\end{array}\right]=\left[\begin{array}{cc}
q^{(0,0)}&q^{(0,1)}\\
q^{(1,0)}&q^{(1,1)}
\end{array}\right].
\end{align}
(Note that we index the entries $q^{(h,l)}$ of $q$
using binary pairs $(h,l)\in\mc B_1$, recalling
the notations $\mc B_n$ and $\mc B_n^{h,l}$ in Definition \ref{Definition: Bn}.)
Moreover, if $q$ is random, then
$\mbf E[q]=\mbf E[a]+\mbf E[b]\msf i+\mbf E[c]\msf j+\mbf E[d]\msf k$
implies by linearity of expectation that
\[\mbf E[q]=\left[\begin{array}{cc}
\mbf E[a]+\mbf E[b]\msf i&\mbf E[c]+\mbf E[d]\msf i\\
-\mbf E[c]+\mbf E[d]\msf i&\mbf E[a]-\mbf E[b]\msf i
\end{array}\right]=\left[\begin{array}{cc}
\mbf E[a+b\msf i]&\mbf E[c+d\msf i]\\
\mbf E[-c+d\msf i]&\mbf E[a-b\msf i]
\end{array}\right];\]
in other words, $\mbf E[q]^{(h,l)}=\mbf E[q^{(h,l)}]$ for all $0\leq h,l\leq 1$.
Therefore, by a matrix product,
\[\mbf E_\xi\left[\prod_{\ell=1}^{N(|\bs t|)}\xi^{\ze(\tau_\ell)}_{J_\ell}\big(Z(\tau_\ell)\big)\right]^{(h,l)}
=\sum_{\substack{m\in\mc B^{h,l}_{N(t)}}}
\mbf E_\xi\left[\prod_{\ell=1}^{N(|\bs t|)}\xi^{\ze(\tau_\ell)}_{J_\ell}\big(Z(\tau_\ell)\big)^{(m_{k-1},m_k)}\right],
\qquad 0\leq h,l\leq 1.\]
Since the entries of the matrices \eqref{Equation: Injective Homomorphism}
are complex-valued, we may now apply the classical statement of
Isserlis' theorem (e.g., \cite[(1.8)]{MingoSpeicher}),
whereby
\begin{multline}
\label{Equation: OE Expectation Part 3}
\mbf E_\xi\left[\prod_{\ell=1}^{N(|\bs t|)}\xi^{\ze(\tau_\ell)}_{J_\ell}\big(Z(\tau_\ell)\big)\right]^{(h,l)}
=
\mbf 1_{\{N(t)\text{ is even}\}}\sum_{m\in\mc B^{h,l}_{N(t)}}
\sum_{p\in\mc P_{N(t)}}\Bigg\{\\
\prod_{\{\ell_1,\ell_2\}\in p}\mbf E\left[\xi^{\ze(\tau_{\ell_1})}_{J_\ell}\big(Z(\tau_{\ell_1})\big)^{(m_{\ell_1-1},m_{\ell_1})}
\xi^{\ze(\tau_{\ell_2})}_{J_\ell}\big(Z(\tau_{\ell_2})\big)^{(m_{\ell_2-1},m_{\ell_2})}\right]\Bigg\}.
\end{multline}
In order to calculate the above expectations we use the following:

\begin{lemma}
\label{Lemma: Quaternion Isserlis}
Let $\ze,\eta>0$ and $x,y\in I$ be arbitrary.
On the one hand,
\begin{align}
\label{Equation: Quaternion Isserlis 0}
\mbf E\left[\xi^\ze_{J_{\ell_1}}(x)^{(m_{\ell_1-1},m_{\ell_1})}
\xi^\eta_{J_{\ell_2}}(y)^{(m_{\ell_2-1},m_{\ell_2})}\right]=0
\end{align}
whenever $J_{\ell_1}\not\in\{J_{\ell_2},J_{\ell_2}^*\}$.
On the other hand:
\begin{enumerate}

\item If $J_{\ell_1}=J_{\ell_2}$, then
\begin{multline*}
\mbf E\left[\xi^\ze_{J_{\ell_1}}(x)^{(m_{\ell_1-1},m_{\ell_1})}
\xi^\eta_{J_{\ell_2}}(y)^{(m_{\ell_2-1},m_{\ell_2})}\right]\\
=\begin{cases}
\frac{\rho_{\ze,\eta}(x-y)}2&\text{if $\big((m_{\ell_1-1},m_{\ell_1}),(m_{\ell_2-1},m_{\ell_2})\big)=\big((0,0),(1,1)\big)$ or $\big((1,1),(0,0)\big)$},\\
-\frac{\rho_{\ze,\eta}(x-y)}2&\text{if $\big((m_{\ell_1-1},m_{\ell_1}),(m_{\ell_2-1},m_{\ell_2})\big)=\big((0,1),(1,0)\big)$ or $\big((1,0),(0,1)\big)$},\\
0&\text{otherwise}.
\end{cases}
\end{multline*}

\item If $J_{\ell_1}=J_{\ell_2}^*$, then
\[\mbf E\left[\xi^\ze_{J_{\ell_1}}(x)^{(m_{\ell_1-1},m_{\ell_1})}
\xi^\eta_{J_{\ell_2}}(y)^{(m_{\ell_2-1},m_{\ell_2})}\right]=\frac{\rho_{\ze,\eta}(x-y)}2\]
if $\big((m_{\ell_1-1},m_{\ell_1}),(m_{\ell_2-1},m_{\ell_2})\big)$ is equal to $\big((0,0),(0,0)\big)$, $\big((1,1),(1,1)\big)$, $\big((0,1),(1,0)\big)$, or $\big((1,0),(0,1)\big)$;
and
\[\mbf E\left[\xi^\ze_{J_{\ell_1}}(x)^{(m_{\ell_1-1},m_{\ell_1})}
\xi^\eta_{J_{\ell_2}}(y)^{(m_{\ell_2-1},m_{\ell_2})}\right]=0\]
otherwise.
\end{enumerate}
\end{lemma}
\begin{proof}
Referring back to \eqref{Equation: Noise Coupling Component Decomposition},
in the case $\mbb F=\mbb H$
we can write
\[\xi^\ze_{i,j}(x)=\tfrac12\left(\xi^\ze_{i,j;1}(x)+\xi^\ze_{i,j;\msf i}(x)\msf i+\xi^\ze_{i,j;\msf j}(x)\msf j+\xi^\ze_{i,j;\msf k}(x)\msf k\right),\qquad \ze>0,~x\in I,\]
where $\xi^\ze_{i,j;\msf l}=(W_{i,j;\msf l}\star\bar\rho_\ze)'$ for $\msf l=1,\msf i,\msf j,\msf k$.
Therefore, for any $1\leq i<j\leq r$ and $x\in I$,
\[\xi^\ze_{i,j}(x)=\frac12\left[\begin{array}{cc}
\xi^\ze_{i,j;1}(x)+\xi^\ze_{i,j;\msf i}(x)\msf i&\xi^\ze_{i,j;\msf j}(x)+\xi^\ze_{i,j;\msf k}(x)\msf i
\vspace{5pt}\\
-\big(\xi^\ze_{i,j;\msf j}(x)-\xi^\ze_{i,j;\msf k}(x)\msf i\big)&\xi^\ze_{i,j;1}(x)-\xi^\ze_{i,j;\msf i}(x)\msf i
\end{array}\right]\]
and
\[\xi^\ze_{j,i}(x)=\xi^\ze_{i,j}(x)^*=\frac12\left[\begin{array}{cc}
\xi^\ze_{i,j;1}(x)-\xi^\ze_{i,j;\msf i}(x)\msf i&-\big(\xi^\ze_{i,j;\msf j}(x)+\xi^\ze_{i,j;\msf k}(x)\msf i\big)
\vspace{5pt}\\
\xi^\ze_{i,j;\msf j}(x)-\xi^\ze_{i,j;\msf k}(x)\msf i&\xi^\ze_{i,j;1}(x)+\xi^\ze_{i,j;\msf i}(x)\msf i
\end{array}\right].\]
With this in mind, \eqref{Equation: Quaternion Isserlis 0} follows immediately
by independence.
As for the calculations stated as items (1) and (2) of this lemma, we use
a simple matrix product along with the remarks that
\[\mbf E\left[\left(\xi^\ze_{i,j;1}(x)\pm\xi^\ze_{i,j;\msf i}(x)\msf i\right)\left(\xi^\eta_{i,j;\msf j}(x)\pm\xi^\eta_{i,j;\msf k}(x)\right)\right]=0\]
by independence, and
\begin{multline*}
\mbf E\left[\left(\xi^\ze_{i,j;1}(x)\pm\xi^\ze_{i,j;\msf i}(x)\msf i\right)\left(\xi^\eta_{i,j;1}(x)\pm\xi^\eta_{i,j;\msf i}(x)\msf i\right)\right]\\=
\mbf E\left[\left(\xi^\ze_{i,j;\msf j}(x)\pm\xi^\ze_{i,j;\msf k}(x)\msf i\right)\left(\xi^\eta_{i,j;\msf j}(x)\pm\xi^\eta_{i,j;\msf k}(x)\msf i\right)\right]=0,
\end{multline*}
and
\begin{multline*}
\tfrac14\mbf E\left[\left(\xi^\ze_{i,j;1}(x)+\xi^\ze_{i,j;\msf i}(x)\msf i\right)\left(\xi^\ze_{i,j;1}(x)-\xi^\ze_{i,j;\msf i}(x)\msf i\right)\right]\\
=\tfrac14\mbf E\left[\left(\xi^\ze_{i,j;\msf j}(x)+\xi^\ze_{i,j;\msf k}(x)\msf i\right)\left(\xi^\ze_{i,j;\msf j}(x)-\xi^\ze_{i,j;\msf k}(x)\msf i\right)\right]=\frac{\rho^{\ze,\eta}(x-y)}2
\end{multline*}
by replicating the calculations in \eqref{Equation: Isserlis Complex 1} and \eqref{Equation: Isserlis Complex 2}.
\end{proof}

We may now wrap up the proof of \eqref{Equation: Isserlis}:
Firstly, we claim that
\begin{align}
\label{Equation: OE Expectation Part 4}
\mbf E_\xi\left[\prod_{\ell=1}^{N(|\bs t|)}\xi^{\ze(\tau_\ell)}_{J_\ell}\big(Z(\tau_\ell)\big)\right]^{(0,1)}=\mbf E_\xi\left[\prod_{\ell=1}^{N(|\bs t|)}\xi^{\ze(\tau_\ell)}_{J_\ell}\big(Z(\tau_\ell)\big)\right]^{(1,0)}=0.
\end{align}
Indeed, by Lemma \ref{Lemma: Quaternion Isserlis}, the only binary sequences $m\in\mc B_{N(t)}^{h,l}$
that can contribute to \eqref{Equation: OE Expectation Part 3} are those made up of
some permutation of pairs of steps in the following set:
\[\Big\{\big\{(0,0),(1,1)\big\},\big\{(0,1),(1,0)\big\},\big\{(0,0),(0,0)\big\},\big\{(1,1),(1,1)\big\}\Big\}.\]
In particular, there is only one type of pair that induces steps that go from $0$ to $1$ or vice versa---
namely, $\big\{(0,1),(1,0)\big\}$---and each such pair induces one $(0,1)$ step and one $(1,0)$ step.
Therefore, any $m\in\mc B_{N(t)}^{h,l}$ that contributes to \eqref{Equation: OE Expectation Part 3} must have
an equal number of $(1,0)$ and $(0,1)$ steps, meaning that $h=m_0=m_{N(t)}=l$.
Hence \eqref{Equation: OE Expectation Part 4} holds.

Secondly, we note that, by \eqref{Equation: Injective Homomorphism},
\[\mbf E_\xi\left[\prod_{\ell=1}^{N(|\bs t|)}\xi^{\ze(\tau_\ell)}_{J_\ell}\big(Z(\tau_\ell)\big)\right]^{(0,0)}=\left(\mbf E_\xi\left[\prod_{\ell=1}^{N(|\bs t|)}\xi^{\ze(\tau_\ell)}_{J_\ell}\big(Z(\tau_\ell)\big)\right]^{(1,1)}\right)^*.\]
By combining \eqref{Equation: OE Expectation Part 3} and Lemma \ref{Lemma: Quaternion Isserlis},
it is clear that the above entries have no complex part; hence we conclude that
\[\mbf E_\xi\left[\prod_{\ell=1}^{N(|\bs t|)}\xi^{\ze(\tau_\ell)}_{J_\ell}\big(Z(\tau_\ell)\big)\right]=\mbf E_\xi\left[\prod_{\ell=1}^{N(|\bs t|)}\xi^{\ze(\tau_\ell)}_{J_\ell}\big(Z(\tau_\ell)\big)\right]^{(0,0)},\]
noting furthermore that the above is just a real number.
With this in hand,
we obtain \eqref{Equation: Isserlis}
 in the case $\mbb F=\mbb H$
by combining
\eqref{Equation: OE Expectation Part 3}
with the fact that the constant $\mf D_t(p,U)$ in
Definition \ref{Definition: Combinatorial Constant}-(3)
is constructed by a direct application of Lemma \ref{Lemma: Quaternion Isserlis}.
More specifically, the binary sequences that respect $(p,J)$ are exactly
those that give a nonvanishing contribution to \eqref{Equation: OE Expectation Part 3}
(as per the covariance calculations in Lemma \ref{Lemma: Quaternion Isserlis}-(1) and -(2)),
and the flips correspond to the terms in Lemma \ref{Lemma: Quaternion Isserlis}-(1)
that have a negative sign multiplying the covariance $-\frac{\rho^{\ze,\eta}(x-y)}{2}$.

\subsection{Proof of Proposition \ref{Proposition: Uniform Integrability}}
\label{Section: Main Proposition 2}

Let $\bs t=(t_1,\ldots,t_n)\in(0,\infty)^n$ be fixed.
According to \eqref{Equation: Regular FK} (for the case $\eps,\ze>0$)
and \eqref{Equation: Trace Formula for 0,zeta} (for the case $\eps=0$ and $\ze>0$),
\[\int_{\mc A}\hat K^{\eps,\ze}(t;a,a)\d\mu(a)=\mr{Tr}\big[\mr e^{-t\hat H^{\eps,\ze}}\big]\]
for all $t>0$, $\eps\in[0,1)$ and $\ze\in(0,1)$. In particular, since the trace of $\mr e^{-t\hat H^{\eps,\ze}}$ is always positive, we can write
\[\mbf E\left[\left|\prod_{k=1}^n\int_{\mc A}\hat K^{\eps_k,\ze_k}(t_k;a,a)\d\mu(a)\right|^2\right]
=\mbf E\left[\left(\prod_{k=1}^n\int_{\mc A}\hat K^{\eps_k,\ze_k}(t_k;a,a)\d\mu(a)\right)^2\right].\]
Therefore, up to relabelling the $\eps_k$'s, $\ze_k$'s and $t_k$'s to account for repeated indices in the squared
term above,
it suffices to prove the following:
\[\sup_{\bs\eps\in[0,\infty)^n,~\bs\ze\in(0,\infty)^n}\mbf E\left[\prod_{k=1}^n\int_{\mc A}\hat K^{\eps_k,\ze_k}(t_k;a,a)\d\mu(a)\right]<\infty.\]
By Proposition \ref{Proposition: Smooth Trace Moments}, this means that it is enough to prove
\begin{multline}
\label{Equation: Trace Moments UI 1}
\int_{\mc A^n}\Pi_A(\bs t;\bs a,\bs a)\\
\cdot\sup_{\bs\eps\in[0,\infty)^n,~\bs\ze\in(0,\infty)^n}\mbf E\left[\mf m^{\bs\ze}_{\bs t}(A^{\bs a,\bs a}_{\bs t})\mr e^{-\int_0^{|\bs t|}V(A^{\bs a,\bs a}_{\bs t}(s))\d s+\mf s^{\bs\eps}_{\bs t}(A^{\bs a,\bs a}_{\bs t})+\mf B_{|\bs t|}(A^{\bs a,\bs a}_{\bs t})}\right]\d\mu^n(\bs a)<\infty.
\end{multline}

Recall that we use $\bs x=(x_1,\ldots,x_n)$ and $\bs i=(i_1,\ldots,i_n)$ to denote the components $a_k=(i_k,x_k)$ of $\bs a$.
Thus, if we combine \eqref{Equation: UI Boundary Bound}, \eqref{Equation: Fubini V bound}, and
\eqref{Equation: s_t bound 3} with the fact that $\Pi_A(\bs t;\bs a,\bs a)=\Pi_Z(\bs t;\bs x,\bs x)\Pi_U(\bs t;\bs i,\bs i)$,
then \eqref{Equation: Trace Moments UI 1}
reduces to
\begin{multline}
\label{Equation: Trace Moments UI 2}
\int_{\mc A^n}\Pi_Z(\bs t;\bs x,\bs x)\sup_{\bs\ze\in(0,\infty)^n}\mbf E\Bigg[\Pi_U(\bs t;\bs i,\bs i)\mbf E_U\left[\mf m^{\bs\ze}_{\bs t}(A^{\bs a,\bs a}_{\bs t})\right]\mr e^{-\int_0^{|\bs t|} \ka|Z^{\bs x,\bs x}_{\bs t}(s)|^{\bar{\mf a}}\d s+\bar{\mf B_{|\bs t|}}(Z^{\bs x,\bs x}_{\bs t})}\\
\cdot\exp\left(\frac{n\si^2}{2}\sum_{k=1}^n\|L_{[t_0+\cdots+t_{k-1},t_0+\cdots+t_k)}(Z^{\bs x,\bs x}_{\bs t})\|_2^2\right)\Bigg]\d\mu^n(\bs i,\bs x)<\infty.
\end{multline}
Thanks to \eqref{Equation: Combinatorial constant bound}, \eqref{Equation: Vector Valued Self-Intersection 1},
and the fact that $\bar\rho$ (and thus $\rho_{\ze,\eta}$ for any choice of $\ze,\eta>0$) is nonnegative, we have that 
$\mf m^{\bs\ze}_{\bs t}(A)\leq\mr e^{(r-1)|\bs t|}\bar{\mf m}^{\bs\ze}_{\bs t}(A),$
where
\begin{align}
\label{Equation: Vector Valued Self-Intersection 1 Without C}
\bar{\mf m}^{\bs\ze}_{\bs t}(A)=\mbf 1_{\{N(|\bs t|)\text{ is even}\}}\sum_{p\in\mc P_{N(|\bs t|)}}
\prod_{\{\ell_1,\ell_2\}\in p}\upsilon^2\rho_{\ze(\tau_{\ell_1}),\ze(\tau_{\ell_2})}\big(Z(\tau_{\ell_1})-Z(\tau_{\ell_2})\big),
\end{align}
Calculating the expectation of this quantity with respect to $U$ (i.e., the jump process $N$) is easier without the endpoint conditionings on
$U^{\bs i,\bs i}_{\bs t}$; for this purpose, note that
\begin{multline}
\label{Equation: Trace Moments UI 3}
\Pi_U(\bs t;\bs i,\bs i)\mbf E_U\left[\mf m^{\bs\ze}_{\bs t}(A^{\bs a,\bs a}_{\bs t})\right]\leq
\mr e^{(r-1)|\bs t|}\Pi_U(\bs t;\bs i,\bs i)\mbf E_U\left[\bar{\mf m}^{\bs\ze}_{\bs t}(U^{\bs i,\bs i}_{\bs t},Z^{\bs x,\bs x}_{\bs t})\right]\\
=\mr e^{(r-1)|\bs t|}\mbf E_U\left[\bar{\mf m}^{\bs\ze}_{\bs t}(U^{\bs i}_{\bs t},Z^{\bs x,\bs x}_{\bs t})\mbf 1_{\{\forall k:~U^{\bs i}_{\bs t}(s)\to i_k\text{ as }s\to(t_0+\cdots+t_k)^-\}}\right]
\leq\mr e^{(r-1)|\bs t|}\mbf E_U\left[\bar{\mf m}^{\bs\ze}_{\bs t}(U^{\bs i}_{\bs t},Z^{\bs x,\bs x}_{\bs t})\right].
\end{multline}
Once we remove the endpoint conditioning on $\bs i$, the jump times in $U^{\bs i}_{\bs t}$ are given by a Poisson process
with rate $(r-1)$. In particular, conditional on $N(|\bs t|)=m$ for some $m\in\mbb N$, the times
$\tau_1\leq\tau_2\leq\cdots\leq\tau_{m}$ are the order statistics of $m$ i.i.d. uniform variables on $[0,|\bs t|)$.
Thus, by applying the law of total probability to
\eqref{Equation: Vector Valued Self-Intersection 1 Without C} (conditioning on $N(|\bs t|)=2m$ and then on the
values of $\tau_1,\ldots,\tau_{2m}$), we can directly calculate
\begin{multline}
\label{Equation: Trace Moments UI 4}\mbf E_{U}\bigg[\bar{\mf m}_{|\bs t|}^{\bs\ze}(U^{\bs i}_{\bs t},Z^{\bs x,\bs x}_{\bs t})\bigg]=\sum_{m=0}^\infty
\int_{[0,|\bs t|)_<^{2m}}\sum_{p\in\mc P_{2m}}\Bigg\{\cdots\\
\prod_{\{\ell_1,\ell_2\}\in p}\upsilon^2\rho_{\ze(s_{\ell_1}),\ze(s_{\ell_2})}\big(Z^{\bs x,\bs x}_{\bs t}(s_{\ell_1})-Z^{\bs x,\bs x}_{\bs t}(s_{\ell_2})\big)\Bigg\}
\,\frac{(2m)!}{|\bs t|^{2m}}\d\bs s\cdot\frac{\big((r-1)|\bs t|\big)^{2m}}{(2m)!}\mr e^{-(r-1)|\bs t|},
\end{multline}
where $[0,|\bs t|)_<^{2m}$ denotes the simplex of components $\bs s=(s_1,\ldots,s_{2m})$
in the interval $[0,|\bs t|)^{2m}$
such that $s_1<s_2<\cdots<s_{2m}$ (which has area $\frac{|\bs t|^{2m}}{(2m)!}$).
Since the function
\[\bs s\mapsto\sum_{p\in\mc P_{2m}}
\prod_{\{\ell_1,\ell_2\}\in p}\upsilon^2\rho_{\ze(s_{\ell_1}),\ze(s_{\ell_2})}\big(Z^{\bs x,\bs x}_{\bs t}(s_{\ell_1})-Z^{\bs x,\bs x}_{\bs t}(s_{\ell_2})\big)\]
is symmetric on $[0,|\bs t|)^{2m}$, we can extend the product of the integral in \eqref{Equation: Trace Moments UI 4}
to the whole interval---up to removing the term $(2m)!$---whereby
\begin{multline}
\label{Equation: Trace Moments UI 5}
\mbf E_{U}\bigg[\bar{\mf m}_{|\bs t|}^{\bs\ze}(U^{\bs i}_{\bs t},Z^{\bs x,\bs x}_{\bs t})\bigg]=\sum_{m=0}^\infty
\int_{[0,|\bs t|)^{2m}}\sum_{p\in\mc P_{2m}}\Bigg\{\cdots\\
\prod_{\{\ell_1,\ell_2\}\in p}\upsilon^2\rho_{\ze(s_{\ell_1}),\ze(s_{\ell_2})}\big(Z^{\bs x,\bs x}_{\bs t}(s_{\ell_1})-Z^{\bs x,\bs x}_{\bs t}(s_{\ell_2})\big)\Bigg\}
\,\frac{1}{|\bs t|^{2m}}\d\bs s\cdot\frac{\big((r-1)|\bs t|\big)^{2m}}{(2m)!}\mr e^{-(r-1)|\bs t|}.
\end{multline}
For any fixed $p\in\mc P_{2m}$, it follows from Young's convolution inequality that
\begin{align}
\nonumber
&\int_{[0,|\bs t|)^{2m}}\prod_{\{\ell_1,\ell_2\}\in p}\upsilon^2\rho_{\ze(s_{\ell_1}),\ze(s_{\ell_2})}\big(Z^{\bs x,\bs x}_{\bs t}(s_{\ell_1})-Z^{\bs x,\bs x}_{\bs t}(s_{\ell_2})\big)\d\bs s\\
\nonumber
&=\prod_{\{\ell_1,\ell_2\}\in p}\upsilon^2\int_{[0,|\bs t|)^2}\rho_{\ze(s_{\ell_1}),\ze(s_{\ell_2})}\big(Z^{\bs x,\bs x}_{\bs t}(s_1)-Z^{\bs x,\bs x}_{\bs t}(s_2)\big)\d\bs s\\
\nonumber
&=\prod_{\{\ell_1,\ell_2\}\in p}\upsilon^2\int_{I^2}L^y_{|\bs t|}(Z^{\bs x,\bs x}_{\bs t})\rho_{\ze(s_{\ell_1}),\ze(s_{\ell_2})}(y-z)L^z_{|\bs t|}(Z^{\bs x,\bs x}_{\bs t})\d y\dd z\\
\label{Equation: Young's for zeta}
&\leq\upsilon^2\|L_{|\bs t|}(Z^{\bs x,\bs x}_{\bs t})\|_2^{2m}.
\end{align}
With this in hand, if we recall
that $|\mc P_{2m}|=(2m-1)!!$ and that $\frac{(2m-1)!!}{(2m)!}=\frac{1}{2^mm!}$, then we obtain from \eqref{Equation: Trace Moments UI 5} that
\begin{align}
\nonumber
&\sup_{\bs\ze\in(0,\infty)^n}\mbf E_{U}\bigg[\bar{\mf m}_{|\bs t|}^{\bs\ze}(U^{\bs i}_{\bs t},Z^{\bs x,\bs x}_{\bs t})\bigg]
\leq\sum_{m=0}^\infty\frac{1}{m!}\left(\frac{(r-1)^2\upsilon^2\|L_{|\bs t|}(Z^{\bs x,\bs x}_{\bs t})\|_2^2}{2}\right)^m\mr e^{-(r-1)|\bs t|}\\
\nonumber
&=\mr e^{(r-1)^2\upsilon^2\|L_{|\bs t|}(Z^{\bs x,\bs x}_{\bs t})\|_2^2/2-(r-1)|\bs t|}\\
\label{Equation: Trace Moments UI 6}
&\leq\exp\left(\frac{n(r-1)^2\upsilon^2}{2}\sum_{k=1}^n\|L_{[t_0+\cdots+t_{k-1},t_0+\cdots+t_k)}(Z^{\bs x,\bs x}_{\bs t})\|_2^2-(r-1)|\bs t|\right),
\end{align}
where the last line follows from an application of Jensen's inequality.
Therefore, if we combine \eqref{Equation: Trace Moments UI 6} with \eqref{Equation: Trace Moments UI 3},
then we get that \eqref{Equation: Trace Moments UI 2}---and thus the proof of Proposition \ref{Proposition: Uniform Integrability}---reduces to
\begin{multline}
\label{Equation: Trace Moments UI 7}
\int_{I^n}\Pi_Z(\bs t;\bs x,\bs x)\mbf E\Bigg[\mr e^{-\int_0^{|\bs t|} \ka|Z^{\bs x,\bs x}_{\bs t}(s)|^{\bar{\mf a}}\d s+\bar{\mf B_{|\bs t|}}(Z^{\bs x,\bs x}_{\bs t})}\\
\cdot\exp\left(\frac{n\si^2+n(r-1)^2\upsilon^2}{2}\sum_{k=1}^n\|L_{[t_0+\cdots+t_{k-1},t_0+\cdots+t_k)}(Z^{\bs x,\bs x}_{\bs t})\|_2^2\right)\Bigg]\d\bs x<\infty.
\end{multline}
This follows from the same argument used to prove \eqref{Equation: Fubini's Theorem 2}.

\subsection{Proof of Proposition \ref{Proposition: Trace Moments Limit 1}}
\label{Section: Main Proposition 3}

By Fubini's theorem, we can write
\begin{multline}
\label{Equation: Trace Moment Limit 1 Fundamental Formula 1}
\mbf E\left[\left(\int_{\mc A} \hat K^{\eps,\ze}(t;a,a)-\hat K^{0,\ze}(t;a,a)\d\mu(a)\right)^2\right]
=\int_{\mc A}\mbf E\left[\hat K^{\eps,\ze}(t;a,a)\hat K^{\eps,\ze}(t;b,b)\right]\\
-2\mbf E\left[\hat K^{\eps,\ze}(t;a,a)\hat K^{0,\ze}(t;b,b)\right]+\mbf E\left[\hat K^{0,\ze}(t;a,a)\hat K^{0,\ze}(t;b,b)\right]\d\mu^2(a,b).
\end{multline}
Thus, in order to show that \eqref{Equation: Trace Moment Limit 1 Fundamental Formula 1} vanishes
as $\eps\to0$, it suffices to prove the following:
\begin{multline*}
\lim_{(\eps_1,\eps_2)\to(0,0)}\int_{\mc A^2}\mbf E\left[\hat K^{\eps_1,\ze}(t;a_1,a_1)\hat K^{\eps_2,\ze}(t;a_2,a_2)\right]\d\mu^2(\bs a)\\
=\int_{\mc A}\mbf E\left[\hat K^{0,\ze}(t;a_1,a_1)\hat K^{0,\ze}(t;a_2,a_2)\right]\d\mu^2(\bs a).
\end{multline*}
If we let $\bs\eps=(\eps_1,\eps_2)$, $\bs\ze=(\ze,\ze)$, and $\bs t=(t,t)$,
by Proposition \ref{Proposition: Smooth Trace Moments} this translates to
\begin{multline}
\label{Equation: Trace Moment Limit 1 Fundamental Formula 2}
\lim_{\bs\eps\to(0,0)}\int_{\mc A^2}\Pi_A(\bs t;\bs a,\bs a)\mbf E\left[\mf m^{\bs\ze}_{\bs t}(A^{\bs a,\bs a}_{\bs t})\mr e^{-\int_0^{|\bs t|}V(A^{\bs a,\bs a}_{\bs t}(s))\d s+\frac12\mf s^{\bs\eps}_{\bs t}(A^{\bs a,\bs a}_{\bs t})+\mf B_{|\bs t|}(A^{\bs a,\bs a}_{\bs t})}\right]\d\mu^2(\bs a)\\
=\int_{\mc A^2}\Pi_A(\bs t;\bs a,\bs a)\mbf E\left[\mf m^{\bs\ze}_{\bs t}(A^{\bs a,\bs a}_{\bs t})\mr e^{-\int_0^{|\bs t|}V(A^{\bs a,\bs a}_{\bs t}(s))\d s+\frac12\mf s^{(0,0)}_{\bs t}(A^{\bs a,\bs a}_{\bs t})+\mf B_{|\bs t|}(A^{\bs a,\bs a}_{\bs t})}\right]\d\mu^2(\bs a).
\end{multline}
The only term that depends on $\bs\eps$ above is
\begin{align}
\label{Equation: Vector Valued Self-Intersection for Trace Moments Limit 1}
\mf s^{\bs\eps}_{\bs t}(A^{\bs a,\bs a}_{\bs t})=\left\|L_{t}(A^{\bs a,\bs a}_{\bs t})\star\bar\rho_{\eps_1}+L_{[t,2t)}(A^{\bs a,\bs a}_{\bs t})\star\bar\rho_{\eps_2}\right\|_\mu^2.
\end{align}
Thus, the proof of \eqref{Equation: Trace Moment Limit 1 Fundamental Formula 2} can be done
in three steps. That is, first we examine the almost-sure limit of \eqref{Equation: Vector Valued Self-Intersection for Trace Moments Limit 1},
and then ensure that this limit persists at the level of the expectation and the integral over $\mc A^2$:

\subsubsection{Step 1. Convergence Inside the Expectations}

For any $f\in L^2(\mc A,\mbb F)$,
one has $\|f\star\bar\rho_\eps-f\|_\mu\to0$ as $\eps\to0$ (e.g., \cite[Lemma 2.2.2]{ChenBook}).
Since the local time process is almost-surely continuous and compactly supported (see Definition \ref{Definition: Regular Local Time}),
this means by \eqref{Equation: Vector Valued Self-Intersection for Trace Moments Limit 1}
that $\mf s^{\bs\eps}_{\bs t}(A^{\bs a,\bs a}_{\bs t})\to\mf s^{(0,0)}_{\bs t}(A^{\bs a,\bs a}_{\bs t})$ as $\bs\eps\to(0,0)$ almost surely; hence
\begin{multline}
\label{Equation: Trace Moment Limit 1 Fundamental Formula 3}
\lim_{\bs\eps\to(0,0)}\mf m^{\bs\ze}_{\bs t}(A^{\bs a,\bs a}_{\bs t})\mr e^{-\int_0^{|\bs t|}V(A^{\bs a,\bs a}_{\bs t}(s))\d s+\frac12\mf s^{\bs\eps}_{\bs t}(A^{\bs a,\bs a}_{\bs t})+\mf B_{|\bs t|}(A^{\bs a,\bs a}_{\bs t})}\\
=\mf m^{\bs\ze}_{\bs t}(A^{\bs a,\bs a}_{\bs t})\mr e^{-\int_0^{|\bs t|}V(A^{\bs a,\bs a}_{\bs t}(s))\d s+\frac12\mf s^{(0,0)}_{\bs t}(A^{\bs a,\bs a}_{\bs t})+\mf B_{|\bs t|}(A^{\bs a,\bs a}_{\bs t})}
\end{multline}
almost surely.

\subsubsection{Step 2. Convergence Inside the Integrals}

We now aim to prove
\begin{multline}
\label{Equation: Trace Moment Limit 1 Fundamental Formula 4}
\lim_{\bs\eps\to(0,0)}\mbf E\left[\mf m^{\bs\ze}_{\bs t}(A^{\bs a,\bs a}_{\bs t})\mr e^{-\int_0^{|\bs t|}V(A^{\bs a,\bs a}_{\bs t}(s))\d s+\frac12\mf s^{\bs\eps}_{\bs t}(A^{\bs a,\bs a}_{\bs t})+\mf B_{|\bs t|}(A^{\bs a,\bs a}_{\bs t})}\right]\\
=\mbf E\left[\mf m^{\bs\ze}_{\bs t}(A^{\bs a,\bs a}_{\bs t})\mr e^{-\int_0^{|\bs t|}V(A^{\bs a,\bs a}_{\bs t}(s))\d s+\frac12\mf s^{(0,0)}_{\bs t}(A^{\bs a,\bs a}_{\bs t})+\mf B_{|\bs t|}(A^{\bs a,\bs a}_{\bs t})}\right].
\end{multline}
By applying the dominated convergence theorem to \eqref{Equation: Trace Moment Limit 1 Fundamental Formula 3},
it suffices to dominate the left-hand side of \eqref{Equation: Trace Moment Limit 1 Fundamental Formula 4}
with an integrable random variable.
For this purpose, we invoke the fact that $V\geq-\nu$ for some constant $\nu>0$
(by Assumption \ref{Assumption: Potential}) together with \eqref{Equation: UI Boundary Bound},
\eqref{Equation: s_t bound 3}, and \eqref{Equation: Vector Valued Self-Intersection 1 Without C} to get
\begin{multline*}
\sup_{\bs\eps\in[0,\infty)^2}\mf m^{\bs\ze}_{\bs t}(A^{\bs a,\bs a}_{\bs t})\mr e^{-\int_0^{|\bs t|}V(A^{\bs a,\bs a}_{\bs t}(s))\d s+\frac12\mf s^{\bs\eps}_{\bs t}(A^{\bs a,\bs a}_{\bs t})+\mf B_{|\bs t|}(A^{\bs a,\bs a}_{\bs t})}\\
\leq\bar{\mf m}^{\bs\ze}_{\bs t}(A^{\bs a,\bs a}_{\bs t})\mr e^{(\nu+r-1)|\bs t|+\si^2(\|L_t(Z^{\bs x,\bs x}_{\bs t})\|_2^2+\|L_{[t,2t)}(Z^{\bs x,\bs x}_{\bs t})\|_2^2)+\bar{\mf B}_{|\bs t|}(Z^{\bs x,\bs x}_{\bs t})}.
\end{multline*}
In order to prove that the expectation of the dominating random variable on the second line of the above display is finite,
we first apply
\eqref{Equation: Trace Moments UI 3} and \eqref{Equation: Trace Moments UI 6} to
control the expectation of $\bar{\mf m}^{\bs\ze}_{\bs t}(A^{\bs a,\bs a}_{\bs t})$ with respect to $U$,
and then apply the same argument used to prove \eqref{Equation: Fubini's Theorem 2}
to deal with the remaining terms involving $Z$.

\subsubsection{Step 3. Convergence of the Integrals}

Thanks to
\eqref{Equation: Trace Moment Limit 1 Fundamental Formula 4} and
another application of the dominated convergence theorem,
the proof of \eqref{Equation: Trace Moment Limit 1 Fundamental Formula 2}
(and therefore Proposition \ref{Proposition: Trace Moments Limit 1})
is reduced to finding a $\mu^2$-integrable function that dominates
the integrands on the left-hand side of \eqref{Equation: Trace Moment Limit 1 Fundamental Formula 2}
for all $\eps_1,\eps_2\geq0$.
This follows from \eqref{Equation: Trace Moments UI 1}.

\subsection{Proof of Proposition \ref{Proposition: Trace Moments Limit}}
\label{Section: Main Proposition 4}

By Fubini's theorem, we can write the limit in \eqref{Equation: Trace Moments Limit 2.2} as follows:
\begin{align}
\label{Equation: Trace Moments Limit 2.3}
\lim_{\ze_1,\ze_2\to0}\sum_{1\leq u,v\leq 2}(-1)^{u+v}\mbf E\left[\int_{\mc A} \hat K^{0,\ze_u}(t;a,a)\d\mu(a)\int_{\mc A} \hat K^{0,\ze_v}(t;a,a)\d\mu(a)\right].
\end{align}
since \eqref{Equation: Trace Moments Limit 2.1} in the case $n=2$
implies that the summands in \eqref{Equation: Trace Moments Limit 2.3} cancel out,
the proof of Proposition \ref{Proposition: Trace Moments Limit} is reduced to establishing \eqref{Equation: Trace Moments Limit 2.1}.

By \eqref{Equation: Vector Valued Self-Intersection 2}, we note that we can simplify
\[\mf s^{(0,\ldots,0)}_{\bs t}(A)=\frac{\si^2}{2}\|L_{|\bs t|}(A)\|_\mu^2.\]
Thus, by Proposition \ref{Proposition: Smooth Trace Moments}, for every $\bs\ze=(\ze_1,\ldots,\ze_n)\in(0,\infty)^n$, we have
\begin{multline}
\label{Equation: Trace Moment Limit 1}
\mbf E\left[\prod_{k=1}^n\int_{\mc A}\hat K^{0,\ze_k}(t_k;a,a)\d\mu(a)\right]\\
=\int_{\mc A^n}\Pi_A(\bs t;\bs a,\bs a)\mbf E\left[\mf m^{\bs\ze}_{\bs t}(A^{\bs a,\bs a}_{\bs t})\mr e^{-\int_0^{|\bs t|}V(A^{\bs a,\bs a}_{\bs t}(s))\d s+\frac{\si^2}{2}\|L_{|\bs t|}(A^{\bs a,\bs a}_{\bs t})\|_\mu^2+\mf B_{|\bs t|}(A^{\bs a,\bs a}_{\bs t})}\right]\d\mu^n(\bs a).
\end{multline}
In order to recover the right-hand side of \eqref{Equation: Trace Moment Formula} from the limit
of this expression as $\bs\ze\to0$,
our aim is to write the right-hand side of \eqref{Equation: Trace Moment Limit 1}
as nested sums/integrals involving all the possible realizations of $U$ (i.e., the jumps and
jump times). That said, the fact that we are conditioning on the endpoints of $U$
in $A^{\bs a,\bs a}_{\bs t}$ changes the distribution of $N$'s jump times.
In order to get around this difficulty and simplify the analysis, we
recall the following: If we use $\bs x=(x_1,\ldots,x_n)$ and $\bs i=(i_1,\ldots,i_n)$ to denote the components $a_k=(i_k,x_k)$ of $\bs a$,
then we can factor $\Pi_A(\bs t;\bs a,\bs a)=\Pi_Z(\bs t;\bs x,\bs x)\Pi_U(\bs t;\bs i,\bs i)$, and thus rewrite the right-hand side of
\eqref{Equation: Trace Moment Limit 1} by absorbing the conditioning into an indicator as follows:
\begin{multline}
\label{Equation: Trace Moment Limit 2}
\int_{\mc A^n}\Pi_Z(\bs t;\bs x,\bs x)\mbf E\Bigg[\mf m^{\bs\ze}_{\bs t}(U^{\bs i}_{\bs t},Z^{\bs x,\bs x}_{\bs t})\mr e^{-\int_0^{|\bs t|}V(U^{\bs i}_{\bs t}(s),Z^{\bs x,\bs x}_{\bs t}(s))\d s+\frac{\si^2}{2}\|L_{|\bs t|}(U^{\bs i}_{\bs t},Z^{\bs x,\bs x}_{\bs t})\|_\mu^2+\mf B_{|\bs t|}(U^{\bs i}_{\bs t},Z^{\bs x,\bs x}_{\bs t})}\\
\cdot \mbf 1_{\{\forall k:~U^{\bs i}_{\bs t}(t_1+\cdots+t_{k-1}+s)\to i_k\text{ as }s\to t_k^-\}}\Bigg]\d\mu^n(\bs i,\bs x).
\end{multline}

In order to understand the limiting behavior of this object as $\bs\ze\to0$, we now point out that drawing a realization of $U^{\bs i}_{\bs t}$ requires
the following steps:
\begin{enumerate}[$\bullet$]
\item Sample $N(|\bs t|)$ according to a Poisson variable with parameter $(r-1)|\bs t|$:
\[\mbf P\big[N(|\bs t|)=m\big]=\frac{\big((r-1)|\bs t|\big)^m}{m!}\mr e^{-(r-1)|\bs t|},\qquad m=0,1,2,3,\ldots\]
\item Given $N(|\bs t|)=m$, sample $(\tau_1,\ldots,\tau_{m})$ by ordering i.i.d. uniform
variables on $[0,|\bs t|)$. That is, using the symbol $\fo{s_i}$ from
Notation \ref{Notation: Arrow means Ordered},
\[\mbf P\big[(\tau_1,\ldots,\tau_m)\in\dd(\fo{s_1},\ldots,\fo{s_m})\big]
=\frac1{|\bs t|^m}\dd\bs s,\qquad \bs s=(s_1,\ldots,s_m)\in[0,|\bs t|)^m.\]
\item
Given $N(|\bs t|)=m$ and $\tau_i=\fo{s_i}$ for $\bs s=[0,|\bs t|)^m$,
we let $\bs z(|\bs t|,m,\bs s)=(z_1,\ldots,z_n)$ denote the partition of
$m$ such that
\begin{align}
\label{Equation: Tuples of Jumps 0}
z_k=\big|\big\{1\leq\ell\leq m:\fo{s_\ell}\in[t_0+\cdots+t_{k-1},t_0+\cdots+t_k)\big\}\big|,\qquad t_0=0,~1\leq k\leq n.
\end{align}
That is, $z_k$ is the number of jumps $\tau_\ell=\fo{s_\ell}$ that land inside the time interval of the
$k^{\mr{th}}$ segment of the concatenated path $U^{\bs i}_{\bs t}$.
Finally, let $\mc W^{\bs i}_{\bs z(|\bs t|,m,\bs s)}$ denote the set of all tuples of the form
\begin{align}
\label{Equation: Tuples of Jumps}
\bs u=(u^{(1)}_0,\ldots,u^{(1)}_{z_1},u^{(2)}_0,\ldots,u^{(2)}_{z_2},\ldots,u^{(n)}_0,\ldots,u^{(n)}_{z_n}),
\end{align}
where for every $1\leq k\leq n$, the sequence
\[u^{(k)}_0,\ldots,u^{(k)}_{z_k}\]
form a random walk path on the complete graph with vertices $\{1,\ldots,r\}$
(without loops) with $z_k$ jumps that starts at $u^{(k)}_0=i_k$.
We note that
\[|\mc W^{\bs i}_{\bs z(|\bs t|,m,\bs s)}|=\prod_{k=1}^n(r-1)^{z_k}=(r-1)^m,\]
and that given $N(|\bs t|)=m$ and $\tau_i=\fo{s_i}$, the jumps $(J_k:1\leq k\leq m)$ in the concatenated path $U^{\bs i}_{\bs t}$
are sampled uniformly at random within $\mc W^{\bs i}_{\bs z(|\bs t|,m,\bs s)}$.
\end{enumerate}
Based on all of this, given some $m\geq0$, $\bs s\in[0,|\bs t|)^m$,
and $\bs u\in\mc W^{\bs i}_{\bs z(|\bs t|,m,\bs s)}$, we let
$\mc U_{(\bs s,\bs u)}$ denote the realization of $U^{\bs i}_{\bs t}$ with $m$ jumps,
whose jump times are given by $\fo{s_i}$, and whose jumps are
given by $\bs u$. Moreover, to alleviate notation, we henceforth denote
\begin{multline}
\label{Equation: Trace Moment Limit Functional F}
F(U^{\bs i}_{\bs t},Z^{\bs x,\bs x}_{\bs t})=\mr e^{-\int_0^{|\bs t|}V(U^{\bs i}_{\bs t}(s),Z^{\bs x,\bs x}_{\bs t}(s))\d s+\frac{\si^2}{2}\|L_{|\bs t|}(U^{\bs i}_{\bs t},Z^{\bs x,\bs x}_{\bs t})\|_\mu^2+\mf B_{|\bs t|}(U^{\bs i}_{\bs t},Z^{\bs x,\bs x}_{\bs t})}\\
\cdot \mbf 1_{\{\forall k:~U^{\bs i}_{\bs t}(t_1+\cdots+t_{k-1}+s)\to i_k\text{ as }s\to t_k^-\}}.
\end{multline}
By the law of total expectation
(conditioning with respect to $N(|\bs t|)$, $\tau_k$, and $J_k$)
and the fact that $U$ and $Z$ are independent,
we may write \eqref{Equation: Trace Moment Limit 2} as follows:
\begin{multline}
\label{Equation: Trace Moment Limit 3}
\int_{\mc A^n}\Pi_Z(\bs t;\bs x,\bs x)\mbf E\Bigg[\sum_{m=0}^\infty
\Bigg(\frac{1}{|\bs t|^{2m}}\int_{[0,|\bs t|)^{2m}}
\frac1{(r-1)^{2m}}
\sum_{\bs u\in\mc W^{\bs i}_{\bs z(|\bs t|,2m,\bs s)}}\Bigg\{
\mf m^{\bs\ze}_{\bs t}\big(\mc U_{(\bs s,\bs u)},Z^{\bs x,\bs x}_{\bs t}\big)\\
\cdot F\big(\mc U_{(\bs s,\bs u)},Z^{\bs x,\bs x}_{\bs t}\big)\Bigg\}\d\bs s\Bigg)
\frac{\mr e^{-(r-1)|\bs t|}\big((r-1)|\bs t|\big)^{2m}}{2m!}\Bigg]\d\mu^n(\bs i,\bs x);
\end{multline}
we note that we have only included the case where $N(|\bs t|)$ is an even number $2m$
in the above because of the indicator function in $\mf m^{\bs\ze}_{\bs t}$ (i.e., \eqref{Equation: Vector Valued Self-Intersection 1}).

Given a fixed choice of $m$, $\bs s$, $\bs u$, and $\mc U_{(\bs s,\bs u)}$ in \eqref{Equation: Trace Moment Limit 3},
it follows from \eqref{Equation: Vector Valued Self-Intersection 1} that
\[\mf m^{\bs\ze}_{\bs t}\big(\mc U_{(\bs s,\bs u)},Z\big)=\mr e^{(r-1)|\bs t|}\sum_{p\in\mc P_{2m}}
\mf C_{|\bs t|}\big(p,\mc U_{(\bs s,\bs u)}\big)\prod_{\{\ell_1,\ell_2\}\in p}\upsilon^2\rho_{\ze(\fo{s_{\ell_1}}),\ze(\fo{s_{\ell_2}})}\big(Z(\fo{s_{\ell_1}})-Z(\fo{s_{\ell_2}})\big).\]
As explained in \eqref{Equation: Difficulty when r>1 2}--\eqref{Equation: Probability Zero Events}, we aim to interpret
the limit of $\rho_{\ze,\eta}$ as $\ze,\eta\to0$ as the "density function" of $Z$'s self-intersection
probability measures. That said, the fact that we have the ordered points $\fo{s_{\ell_i}}$ in the above
display instead of $s_{\ell_i}$ is problematic for this purpose. We thus introduce a change of variables
to circumvent this:
Given any permutation $\pi$ of the set $\{1,\ldots,2m\}$
and $p\in\mc P_{2m}$,
we let $\pi(p)\in\mc P_{2m}$ denote the perfect matching such that
$\{\ell_1,\ell_2\}\in p$ if and only if $\{\pi(\ell_1),\pi(\ell_2)\}\in\pi(p)$.
Then, given $\bs s=(s_1,\ldots,s_{2m})$, we let $\pi_{\bs s}$ denote the specific permutation
of $\{1,\ldots,2m\}$ such that for any $1\leq \ell\leq 2m$, one has
$\fo{s_{\pi_{\bs s}(\ell)}}=s_\ell$, or equivalently
$\fo{s_\ell}=s_{\pi^{-1}_{\bs s}(\ell)}$.
Then, by the change of variables $q=\pi_{\bs s}^{-1}(p)$, we can rewrite 
\[\mf m^{\bs\ze}_{\bs t}\big(\mc U_{(\bs s,\bs u)},Z\big)=\mr e^{(r-1)|\bs t|}\sum_{q\in\mc P_{2m}}
\mf C_{|\bs t|}\big(\pi_{\bs s}(q),\mc U_{(\bs s,\bs u)}\big)\prod_{\{\ell_1,\ell_2\}\in q}\upsilon^2\rho_{\ze(s_{\ell_1}),\ze(s_{\ell_2})}\big(Z(s_{\ell_1})-Z(s_{\ell_2})\big).\]
At this point, if we insert this expression into \eqref{Equation: Trace Moment Limit 3}
(and take this opportunity to cancel out both appearances of $|\bs t|^{2m}$ and $\mr e^{(r-1)|\bs t|}$ therein,
as well as move $\upsilon^{2m}$ in front of the constant $\mf C_{|\bs t|}\big(\pi_{\bs s}(q),\mc U_{(\bs s,\bs u)}\big)$),
then we obtain the following expression for the mixed moment \eqref{Equation: Trace Moment Limit 1}:
\begin{multline}
\label{Equation: Trace Moment Limit 4}
\int_{\mc A^n}\Pi_Z(\bs t;\bs x,\bs x)\mbf E\Bigg[\sum_{m=0}^\infty
\Bigg(\sum_{q\in\mc P_{2m}}\int_{[0,|\bs t|)^{2m}}
\prod_{\{\ell_1,\ell_2\}\in q}\rho_{\ze(s_{\ell_1}),\ze(s_{\ell_2})}\big(Z(s_{\ell_1})-Z(s_{\ell_2})\big)\\
\cdot\frac1{(r-1)^{2m}}
\sum_{\bs u\in\mc W^{\bs i}_{\bs z(|\bs t|,2m,\bs s)}}
\upsilon^{2m}\mf C_{|\bs t|}\big(\pi_{\bs s}(q),\mc U_{(\bs s,\bs u)}\big)
F\big(\mc U_{(\bs s,\bs u)},Z^{\bs x,\bs x}_{\bs t}\big)\d\bs s\Bigg)
\frac{(r-1)^{2m}}{2m!}\Bigg]\d\mu^n(\bs i,\bs x).
\end{multline}

We are now in a position to establish the convergence of $\rho_{\ze,\eta}$ to the self-intersection measures.
For this purpose, we begin with the following definition:

\begin{definition}
\label{Definition: Approximate SI Local Time}
Let
\[\msf L^{\bs\ze}_{\bs t,2m,q,Z}=\int_{[0,|\bs t|)^{2m}}\prod_{\{\ell_1,\ell_2\}\in q}\rho_{\ze(s_{\ell_1}),\ze(s_{\ell_2})}\big(Z(s_{\ell_1})-Z(s_{\ell_2})\big)\d\bs s,\]
and let $\mr{si}^{\bs\ze}_{\bs t,2m,q,Z}$ be the random probability measure on $[0,|\bs t|)^{2m}$ with the density function
\[\bs s\mapsto\frac1{\msf L^{\bs\ze}_{\bs t,2m,q,Z}}\prod_{\{\ell_1,\ell_2\}\in q}\rho_{\ze(s_{\ell_1}),\ze(s_{\ell_2})}\big(Z(s_{\ell_1})-Z(s_{\ell_2})\big),\qquad\bs s\in[0,|\bs t|)^{2m}.\]
\end{definition}

If we sample the random vector
$\bs T^{\bs\ze}=(T^{\bs\ze}_1,\ldots,T^{\bs\ze}_{2m})$
according to $\mr{si}^{\bs\ze}_{\bs t,2m,q,Z^{\bs x,\bs x}_{\bs t}}$,
then we note that the $\dd\bs s$ integral inside
\eqref{Equation: Trace Moment Limit 4} can be written as the expectation
\begin{align}
\label{Equation: Trace Moment Limit 5}
\msf L^{\bs\ze}_{\bs t,2m,q,Z}\mbf E_{\bs T^{\bs\ze}}\left[\frac1{(r-1)^{2m}}
\sum_{\bs u\in\mc W^{\bs i}_{\bs z(|\bs t|,2m,\bs T^{\bs\ze})}}
\upsilon^{2m}\mf C_{|\bs t|}\big(\pi_{\bs T^{\bs\ze}}(q),\mc U_{(\bs T^{\bs\ze},\bs u)}\big)
F\big(\mc U_{(\bs T^{\bs\ze},\bs u)},Z^{\bs x,\bs x}_{\bs t}\big)\right].
\end{align}
Our claim is that this expression converges to
\begin{align}
\label{Equation: Trace Moment Limit 6}
\|L_{|\bs t|}(Z^{\bs x,\bs x}_{\bs t})\|_2^{2m}\mbf E_{\bs T}\left[\frac{1}{(r-1)^{2m}}
\sum_{\bs u\in\mc W^{\bs i}_{\bs z(|\bs t|,2m,\bs T)}}
\upsilon^{2m}\mf C_{|\bs t|}\big(\pi_{\bs T}(q),\mc U_{(\bs T,\bs u)}\big)
F\big(\mc U_{(\bs T,\bs u)},Z^{\bs x,\bs x}_{\bs t}\big)\right],
\end{align}
where $\bs T=(T_1,\ldots,T_{2m})$
is now sampled according to $\mr{si}_{|\bs t|,2m,q,Z^{\bs x,\bs x}_{\bs t}}$
(recall the definition of the latter in \eqref{Equation: SI Measures}).
If this can be established, then we get that the limit of \eqref{Equation: Trace Moment Limit 4}
as $\bs\ze\to0$ can be written as follows (in order to obtain the expression below,
we also use the facts that $|\mc P_{2m}|=(2m-1)!!$ and the simplification $\frac{(2m-1)!!}{(2m)!}=\frac{1}{2^mm!}$):
\begin{multline}
\label{Equation: Trace Moment Limit Final}
\int_{\mc A^n}\Pi_Z(\bs t;\bs x,\bs x)\mbf E\Bigg[\sum_{m=0}^\infty
\Bigg(\frac{1}{|\mc P_{2m}|}\sum_{q\in\mc P_{2m}}\mbf E_{\bs T}\Bigg[\frac1{(r-1)^{2m}}\\
\cdot\sum_{\bs u\in\mc W^{\bs i}_{\bs z(|\bs t|,2m,\bs T)}}
\upsilon^{2m}\mf C_{|\bs t|}\big(\pi_{\bs T}(q),\mc U_{(\bs T,\bs u)}\big)
F\big(\mc U_{(\bs T,\bs u)},Z^{\bs x,\bs x}_{\bs t}\big)\Bigg]\Bigg)\\
\cdot\left(\frac{(r-1)^2\|L_{|\bs t|}(Z^{\bs x,\bs x}_{\bs t})\|_2^2}{2}\right)^{m}\frac1{m!}\Bigg]\d\mu^n(\bs i,\bs x).
\end{multline}
Once established, this convergence concludes the proof of \eqref{Equation: Trace Moments Limit 2.1}.
Indeed, we now recognize that in \eqref{Equation: Trace Moment Limit Final}, the sums over $m$
and $q\in\mc P_{2m}$ correspond to sampling $\hat N(|\bs t|)$ and $\hat q$ as described in Definition \ref{Definition: Singular Process},
up to introducing a renormalization due to
\[\sum_{m=0}^\infty\left(\frac{(r-1)^2\|L_{|\bs t|}(Z^{\bs x,\bs x}_{\bs t})\|_2^2}{2}\right)^{m}\frac1{m!}=\mr e^{\frac{(r-1)^2}2\|L_{|\bs t|}(Z)\|_2^2},\]
which is necessary in order to turn $m\mapsto\left(\frac{(r-1)^2\|L_{|\bs t|}(Z^{\bs x,\bs x}_{\bs t})\|_2^2}{2}\right)^{m}\frac1{m!}$ into a probability measure
(this renormalization also explains the presence of $\frac{(r-1)^2}2\|L_{|\bs t|}(Z)\|_2^2$ in
\eqref{Equation: Trace Moment Exponential Term}).
This then also implies that
$\bs T$ is sampled according to $\mr{si}_{t,\hat N(t),\hat q,Z}$,
that $\hat p$ can be taken as $\pi_{\bs T}(\hat q)$ with $\hat q$ sampled uniformly at random in $\mc P_{\hat N(t)}$,
and that $\mc U_{(\bs T,\bs u)}$ generates a realization of $\hat U^{\bs i}_{\bs t}$, all of which are consistent with Definition \ref{Definition: Singular Process}.

Thus, in summary, the proof of Proposition \ref{Proposition: Trace Moments Limit} will be complete if we
establish that \eqref{Equation: Trace Moment Limit 5} converges to \eqref{Equation: Trace Moment Limit 6}
almost surely.
For this purpose, we use the following:

\begin{lemma}
\label{Lemma: Weak Convergence of SI Measures}
Let $\bs x\in I^n$, $\bs t\in[0,\infty)^n$ be fixed.
Almost surely,
\begin{align}
\label{Equation: Weak Convergence of SI Measures Normalization}
\lim_{\bs\ze\to0}\msf L^{\bs\ze}_{\bs t,2m,q,Z^{\bs x,\bs x}_{\bs t}}=\|L_{|\bs t|}(Z^{\bs x,\bs x}_{\bs t})\|_2^{2m}
\end{align}
and
\[\lim_{\bs\ze\to0}\mr{si}^{\bs\ze}_{\bs t,2m,q,Z^{\bs x,\bs x}_{\bs t}}=\mr{si}_{|\bs t|,2m,q,Z^{\bs x,\bs x}_{\bs t}}\qquad\text{weakly}\]
for every $m\in\mbb N$ and $q\in\mc P_{2m}$.
\end{lemma}

\begin{lemma}
\label{Lemma: Zero Sets of SI Measures}
Let $\bs x\in I^n$, $\bs t\in[0,\infty)^n$ be fixed. Define the set
\begin{multline*}
\mr{B}_{\bs t,2m}=\big\{(s_1,\ldots,s_{2m})\in[0,|\bs t|)^{2m}:s_\ell=t_0+\cdots+t_k\\
\text{ for at least one pair }1\leq \ell\leq 2m,~1\leq k\leq n-1\big\}.
\end{multline*}
Almost surely, $\mr{si}_{|\bs t|,2m,q,Z^{\bs x,\bs x}_{\bs t}}\big(\mr{B}_{\bs t,2m}\big)=0$ for all $m\in\mbb N$ and $q\in\mc P_{2m}$.
\end{lemma}

\begin{lemma}
\label{Lemma: Continuity of SI Measures Integrands}
Let $\bs x\in I^n$, $\bs t\in[0,\infty)^n$ be fixed.
Almost surely, the function
\[\bs s\mapsto\sum_{\bs u\in\mc W^{\bs i}_{\bs z(|\bs t|,2m,\bs s)}}
\mf C_{|\bs t|}\big(\pi_{\bs s}(q),\mc U_{(\bs s,\bs u)}\big)
F\big(\mc U_{(\bs s,\bs u)},Z^{\bs x,\bs x}_{\bs t}\big)\]
is bounded and continuous on $[0,|\bs t|)^{2m}\setminus\mr{B}_{\bs t,2m}$
for all $m\in\mbb N$ and $q\in\mc P_{2m}$.
\end{lemma}

With these three lemmas in hand, the convergence of
\eqref{Equation: Trace Moment Limit 5} to \eqref{Equation: Trace Moment Limit 6}
is a consequence of the following: If probability measures $(P_n)_{n\in\mbb N}$ converge weakly to
$P$ as $n\to\infty$ and $h$ is a bounded function whose discontinuities have probability zero
with respect to $P$,
then $\mbf E_{P_n}[h]\to\mbf E_P[h]$
(this can be established by combining the continuous mapping theorem \cite[Theorem 2.7]{Billingsley},
the Skorokhod representation theorem \cite[Theorem 6.7]{Billingsley}, and the dominated convergence theorem).
We now finish this section with the proofs of Lemmas \ref{Lemma: Weak Convergence of SI Measures},
\ref{Lemma: Zero Sets of SI Measures}, and \ref{Lemma: Continuity of SI Measures Integrands}:

\begin{proof}[Proof of Lemma \ref{Lemma: Weak Convergence of SI Measures}]
We begin with \eqref{Equation: Weak Convergence of SI Measures Normalization}.
The analysis of the expression
\[\msf L^{\bs\ze}_{\bs t,2m,q,Z}=\int_{[0,|\bs t|)^{2m}}\prod_{\{\ell_1,\ell_2\}\in q}\rho_{\ze(s_{\ell_1}),\ze(s_{\ell_2})}\big(Z(s_{\ell_1})-Z(s_{\ell_2})\big)\d\bs s\]
is made difficult by the fact that the parameters $\ze(s)$ change with $s$. In order to get around this,
for any $1\leq k_1,\ldots,k_{2m}\leq n$, we denote
\[\mc K_{k_\ell}=\big[t_0+\cdots+t_{k_\ell-1},t_0+\cdots+t_{k_\ell}\big),\qquad
\mc C_{k_1,\ldots,k_{2m}}=\mc K_{k_1}\times\cdots\times\mc K_{k_{2m}},\]
with the convention $t_0=0$,
and we write
\[[0,|\bs t|)^{2m}=\bigcup_{1\leq k_1,\ldots,k_{2m}\leq n}\mc C_{k_1,\ldots,k_{2m}},\]
noting that the union is disjoint. With this convention, if $\bs s=(s_1,\ldots,s_{2m})\in\mc C_{k_1,\ldots,k_{2m}}$,
then we have that $\ze(s_\ell)=\ze_{k_\ell}$ for all $1\leq\ell\leq2m$ thanks to \eqref{Equation: zeta s}.
Therefore, by definition of local time in \eqref{Equation: Local Time Integral Definition}, we can write
\begin{align}
\nonumber
\msf L^{\bs\ze}_{\bs t,2m,q,Z}
&=\sum_{1\leq k_1,\ldots,k_{2m}\leq n}\int_{\mc C_{k_1,\ldots,k_{2m}}}\prod_{\{\ell_1,\ell_2\}\in q}\rho_{\ze_{k_{\ell_1}}\ze_{k_{\ell_2}}}\big(Z(s_{\ell_1})-Z(s_{\ell_2})\big)\d\bs s\\
\label{Equation: SI Weak Convergence 1}
&=\sum_{1\leq k_1,\ldots,k_{2m}\leq n}\prod_{\{\ell_1,\ell_2\}\in q}\int_{I^2}
L_{\mc K_{k_{\ell_1}}}^x(Z)\rho_{\ze_{k_{\ell_1}}\ze_{k_{\ell_2}}}(x-y)L_{\mc K_{k_{\ell_1}}}^y(Z)\d x\dd y.
\end{align}
Recall that $\rho_{\ze,\eta}=\bar\rho_\ze\star\bar\rho_{\eta}$ is a convolution of mollifiers
(Definition \ref{Definition: Noise Coupling}). Thus,
on the probability-one event where the local time is continuous (and thus, in particular,
square-integrable, see Definition \ref{Definition: Regular Local Time}),
\eqref{Equation: SI Weak Convergence 1} converges to
\begin{align}
\label{Equation: SI Weak Convergence 2}
\sum_{1\leq k_1,\ldots,k_{2m}\leq n}\prod_{\{\ell_1,\ell_2\}\in q}
\left\langle L_{\mc K_{k_{\ell_1}}}(Z),L_{\mc K_{k_{\ell_2}}}(Z)\right\rangle
\end{align}
as $\bs\ze\to0$, since convolution against a mollifier is an approximate identity
with respect to any $L^p$ norm (e.g., \cite[Lemma 2.2.2]{ChenBook}).
We then obtain \eqref{Equation: Weak Convergence of SI Measures Normalization}
by noting that if we compute the sum in \eqref{Equation: SI Weak Convergence 2} one index
$k_\ell$ at a time (i.e., keeping all other indices fixed), then the bilinearity of the inner product yields that
\eqref{Equation: SI Weak Convergence 2} is equal to
\[\left\langle \sum_{k=1}^{2m}L_{\mc K_k}(Z),\sum_{k=1}^{2m}L_{\mc K_k}(Z)\right\rangle^{m}=\|L_{|\bs t|}(Z)\|_2^{2m},\]
as desired.

We now prove the weak convergence statement.
Since measurable rectangles are a convergence-determining class
for weak convergence (e.g., \cite[Page 18, Example 2.3]{Billingsley}),
it suffices to prove that, almost surely,
\[\lim_{\bs\ze\to0}\mr{si}^{\bs\ze}_{\bs t,2m,q,Z^{\bs x,\bs x}_{\bs t}}\big([0,u_1)\times\cdots\times[0,u_{2m})\big)
=\mr{si}_{|\bs t|,2m,q,Z^{\bs x,\bs x}_{\bs t}}\big([0,u_1)\times\cdots\times[0,u_{2m})\big)\]
for all $0\leq u_k\leq |\bs t|$, $1\leq k\leq 2m$.
By definition of $\mr{si}^{\bs\ze}_{\bs t,2m,q,Z^{\bs x,\bs x}_{\bs t}}$ and
$\mr{si}_{|\bs t|,2m,q,Z^{\bs x,\bs x}_{\bs t}}$  (i.e., \eqref{Equation: SI Measures})
and the limit \eqref{Equation: Weak Convergence of SI Measures Normalization},
this reduces to showing that, almost surely,
\begin{multline}
\label{Equation: SI Weak Convergence 3}
\lim_{\bs\ze\to0}\int_{[0,u_1)\times\cdots\times[0,u_n)}\prod_{\{\ell_1,\ell_2\}\in q}\rho_{\ze(s_{\ell_1}),\ze(s_{\ell_2})}\big(Z^{\bs x,\bs x}_{\bs t}(s_{\ell_1})-Z^{\bs x,\bs x}_{\bs t}(s_{\ell_2})\big)\d\bs s\\
=\prod_{\{\ell_1,\ell_2\}\in q}\left\langle L_{u_{\ell_1}}(Z^{\bs x,\bs x}_{\bs t}),L_{u_{\ell_2}}(Z^{\bs x,\bs x}_{\bs t})\right\rangle
\end{multline}
for all $0\leq u_k\leq |\bs t|$, $1\leq k\leq 2m$.
The proof of this fact is exactly the same
as that of \eqref{Equation: Weak Convergence of SI Measures Normalization},
except that we replace the intervals $\mc K_k$ and rectangles $\mc C_{k_1,\ldots,k_{2m}}$ therein by
\[\mc K^{(u)}_{k_\ell}=\big[t_0+\cdots+t_{k_\ell-1},t_0+\cdots+t_{k_\ell}\big)\cap[0,u_\ell),\qquad
\mc C^{(u)}_{k_1,\ldots,k_{2m}}=\mc K^{(u)}_{k_1}\times\cdots\times\mc K^{(u)}_{k_{2m}}.\]
(The fact that we consider uncountably many choices of $u_k$'s in \eqref{Equation: SI Weak Convergence 3}
is not a problem for the almost-sure convergence; the limit
\eqref{Equation: Weak Convergence of SI Measures Normalization} holds on the event where the local time
process is continuous, which is independent of $\bs t$ or the choices of $u_k$.)
\end{proof}

\begin{proof}[Proof of Lemma \ref{Lemma: Zero Sets of SI Measures}]
Given $1\leq\ell\leq 2m$ and $t\in[0,|\bs t|)$, let us denote
\[\mr{B}^0_{\ell,t}=\big\{(s_1,\ldots,s_{2m})\in[0,|\bs t|)^{2m}:s_\ell=t\big\}=[0,|\bs t|)^{\ell-1}\times\{t\}\times[0,|\bs t|)^{2m-\ell}.\]
Then, we have that
\[\mr{B}_{\bs t,2m}\subset\bigcup_{\substack{1\leq\ell\leq 2m\\1\leq k\leq n-1}}\mr{B}^0_{\ell,t_1+\cdots+t_k}.\]
Thus, by a union bound, it suffices to show that for any choice of $1\leq\ell\leq 2m$ and $t\in[0,|\bs t|)$,
one has
$\mr{si}_{|\bs t|,2m,q,Z^{\bs x,\bs x}_{\bs t}}\big(\mr{B}^0_{\ell,t}\big)=0$
almost surely.
By \eqref{Equation: SI Measures} and continuity from above of probability measures, one has
\[\mr{si}_{|\bs t|,2m,q,Z^{\bs x,\bs x}_{\bs t}}\big(\mr{B}^0_{\ell,t}\big)
=\lim_{\eps\to0}\frac{1}{\|L_t(Z)\|_2^2}\left\langle L_{[t,t+\eps)}(Z^{\bs x,\bs x}_{\bs t}),L_{[0,|\bs t|)}(Z^{\bs x,\bs x}_{\bs t})\right\rangle.\]
This converges to zero
on the probability-one event where $Z^{\bs x,\bs x}_{\bs t}$'s local time is continuous:
\eqref{Equation: Local Time Integral Definition} implies that $L_{[t,t+\eps)}(Z^{\bs x,\bs x}_{\bs t})\to0$ pointwise,
and we then get the result by dominated convergence.
\end{proof}

\begin{proof}[Proof of Lemma \ref{Lemma: Continuity of SI Measures Integrands}]
Referring back to \eqref{Equation: Tuples of Jumps}, we remark that the set-valued function $\bs s\mapsto\mc W^{\bs i}_{\bs z(|\bs t|,2m,\bs s)}$ remains constant
so long as we do not change the number of components of $\bs s$ that are inside each sub-interval of $[0,|\bs t|)$ the form
\[[t_0+\cdots+t_{k-1},t_0+\cdots+t_k).\]
In other words, $\bs s\mapsto\mc W^{\bs i}_{\bs z(|\bs t|,2m,\bs s)}$ is constant inside each connected component of the set $[0,|\bs t|)^{2m}\setminus\mr{B}_{\bs t,2m}$.
Thus, if we let $\mc K$ denote one such connected component and we restrict to $\bs s\in\mc K$,
it suffices to show that the map
\[\bs s\mapsto\mf C_{|\bs t|}\big(\pi_{\bs s}(q),\mc U_{(\bs s,\bs u)}\big)
F\big(\mc U_{(\bs s,\bs u)},Z^{\bs x,\bs x}_{\bs t}\big),\qquad\bs s\in\mc K\]
is almost-surely continuous and bounded on $\mc K$ for any fixed $\bs u\in\mc W^{\bs i}_{\bs z(|\bs t|,2m,\bs s)}$.

Once we fix the number of coordinates in $\bs s$ that are in each sub-interval of $[0,|\bs t|)$
and the random walk path $\bs u$, the quantity $\mf C_{|\bs t|}\big(\pi_{\bs s}(q),\mc U_{(\bs s,\bs u)}\big)$
becomes constant (this is because, by Definition \ref{Definition: Combinatorial Constant}, $\mf C_t(p,U)$ actually
only depends on the jumps $J_k$ and the matching $p$; it does not depend on the jump times $\tau_k$).
Thus, we only need to verify the almost-sure continuity and boundedness of $F\big(\mc U_{(\bs s,\bs u)},Z^{\bs x,\bs x}_{\bs t}\big)$
on $\mc K$.

If we examine the definition of $F$ in \eqref{Equation: Trace Moment Limit Functional F}, then we note that the indicator
function that fixes the endpoints of $\mc U_{(\bs s,\bs u)}$ does not depend on $\bs s$ once $\bs u$ is fixed.
Thus, we need only verify the almost-sure continuity and boundedness of
\[\bs s\mapsto\mr e^{-\int_0^{|\bs t|}V(\mc U_{(\bs s,\bs u)}(s),Z^{\bs x,\bs x}_{\bs t}(s))\d s+\frac{\si^2}{2}\|L_{|\bs t|}(\mc U_{(\bs s,\bs u)},Z^{\bs x,\bs x}_{\bs t})\|_\mu^2+\mf B_{|\bs t|}(\mc U_{(\bs s,\bs u)},Z^{\bs x,\bs x}_{\bs t})}\]
on $\mc K$.
That this is bounded is easily obtained from the assumption that $V$ is bounded below (Assumption \ref{Assumption: Potential}),
as well as the fact that we can obtain upper bounds for $\mf s^{(0,\ldots,0)}_{\bs t}(A)=\frac{\si^2}{2}\|L_{|\bs t|}(A)\|_\mu^2$
and $\mf B_{|\bs t|}(A)$ that only depend on $Z$
(e.g., \eqref{Equation: UI Boundary Bound} and \eqref{Equation: s_t bound 3}). As for the continuity,
recall the definitions of $z_k$ and $u^{(k)}_\ell$ (with $1\leq k\leq n$ and $0\leq\ell\leq z_k$) in
\eqref{Equation: Tuples of Jumps 0} and \eqref{Equation: Tuples of Jumps}, and let us henceforth
define
\[\mc S^{(k)}_{\ell}\subset[t_0+\cdots+t_{k-1},t_0+\cdots+t_k)\] as the interval of time that corresponds to the step $u^{(k)}_\ell$
in the continuous-time path $\mc U_{(\bs s,\bs u)}$, in the sense that
\[\mc U_{(\bs s,\bs u)}=\sum_{k=1}^n\sum_{\ell=0}^{z_k}u^{(k)}_\ell\mbf 1_{\mc S^{(k)}_\ell};\]
for instance, for $k=1$ we have
\[\mc S^{(1)}_0=\big[0,\fo{s_1}\big),~\mc S^{(1)}_1=\big[\fo{s_1},\fo{s_2}\big),\ldots,\mc S^{(1)}_{z_1-1}=\big[\fo{s_{z_1-1}},\fo{s_{z_1}}\big),~\mc S^{(1)}_{z_1}=\big[\fo{s_{z_1}},t_1\big),\]
and so on for $k\geq2$.
Then, by \eqref{Equation: Multivariate Boundary Local Time 1} and \eqref{Equation: Multivariate Local Time}, for any $(i,x)\in\mc A$ one has
\begin{align}
\label{Equation: Continuity of SI Functional 1}
L_{|\bs t|}^{(i,x)}(\mc U_{(\bs s,\bs u)},Z^{\bs x,\bs x}_{\bs t})=\sum_{k=1}^{n}\sum_{\ell\leq z_k:~u^{(k)}_\ell=i}L_{\mc S^{(k)}_\ell}^x(Z^{\bs x,\bs x}_{\bs t}),
\end{align}
and for any $c\in\partial I$ (in the two cases where $I$ has a boundary), one has
\begin{align}
\label{Equation: Continuity of SI Functional 2}
\mf L_{|\bs t|}^{(i,c)}(\mc U_{(\bs s,\bs u)},Z^{\bs x,\bs x}_{\bs t})=\sum_{k=1}^{n}\sum_{\ell\leq z_k:~u^{(k)}_\ell=i}\mf L^c_{\mc S^{(k)}_\ell}(Z^{\bs x,\bs x}_{\bs t}).
\end{align}
Given that we chose versions of
the local time processes $L^x_{[u,v)}(Z)$ and $\mc L^c_{[u,v)}(Z)$ that are jointly continuous
with respect to $u,v$, and $x$ in Definitions \ref{Definition: Combined Stochastic Processes} and \ref{Definition: Regular Local Time}, this implies that
the functionals \eqref{Equation: Multivariate Boundary Local Time 1} and \eqref{Equation: Multivariate Local Time}
are continuous with respect to $\bs s\in\mc K$ with probability one,
since changing $\bs s$'s components only changes the lengths of the intervals $\mc S^{(k)}_\ell$.
Since $x\mapsto L^x_{|\bs t|}(Z^{\bs x,\bs x}_{\bs t})$ is almost-surely compactly supported, the continuity of $\|L_{|\bs t|}(\mc U_{(\bs s,\bs u)},Z^{\bs x,\bs x}_{\bs t})\|_\mu^2$ and
$\mf B_{|\bs t|}(\mc U_{(\bs s,\bs u)},Z^{\bs x,\bs x}_{\bs t})$ follow immediately from this
(recall the definition of the latter in \eqref{Equation: Multivariate Boundary Local Time 2}). As for the integral involving $V$, we note that its continuity
follows from that of \eqref{Equation: Continuity of SI Functional 1} by writing
\[\int_0^{|\bs t|}V\big(\mc U_{(\bs s,\bs u)}(s),Z^{\bs x,\bs x}_{\bs t}(s)\big)\d s=\int_{\mc A}V(a)L^a_{|\bs t|}(\mc U_{(\bs s,\bs u)},Z^{\bs x,\bs x}_{\bs t})\d\mu(a),\]
noting once again that the above integral can be restricted to a compact set.
This completes the proof of Lemma \ref{Lemma: Continuity of SI Measures Integrands}, and thus also of Proposition \ref{Proposition: Trace Moments Limit}.
\end{proof}

\subsection{A Scalar Real-Valued Reduction for Propositions \ref{Proposition: Operator Form-Bound},
\ref{Proposition: Uniform Eigenvalue Bounds}, and \ref{Proposition: Eigenvalue Convergence}}
\label{Section: Main Proposition 5}

Propositions \ref{Proposition: Operator Form-Bound},
\ref{Proposition: Uniform Eigenvalue Bounds}, and \ref{Proposition: Eigenvalue Convergence} can all be reduced to similar results from \cite{GaudreauLamarreEJP} in the scalar setting with $\mbb F=\mbb R$.
This is because, by linearity/bilinearity, the form $\xi^{\eps,\ze}(f,g)$
evaluated in any functions $f,g:\mc A\to\mbb F$ can be written as a linear combination
of white noises in $\mbb R$ evaluated in functions from $I$ to $\mbb R$. In this section,
we make this observation more precise.

To see this, we first recall that we can write
\begin{align}
\label{Equation: Smoothed Noise Sum Decomposition}
\xi^{\eps,\ze}(f,g)=\sum_{i=1}^r\xi^{\eps}_{i,i}\big(f(i,\cdot),g(i,\cdot)\big)+\sum_{i\neq j}\xi^{\ze}_{i,j}\big(f(i,\cdot),g(j,\cdot)\big),\qquad\eps,\ze\geq0;
\end{align}
moreover, letting $(W_{i,j})_{i\leq i<j\leq r}$ be the Brownian motions introduced in Definition \ref{Definition: Noise Coupling},
we have that
\begin{align}
\label{Equation: Multivariate to Scalar for Convergence 1}
\xi^{\eps}_{i,i}\big(f(i,\cdot),g(i,\cdot)\big)=\si (W_{i,i}\star\bar\rho_\eps)'\big(f(i,\cdot),g(i,\cdot)\big)\qquad\text{for all }1\leq i\leq r,
\end{align}
and
\begin{align}
\label{Equation: Multivariate to Scalar for Convergence 2}
\xi^{\ze}_{i,j}\big(f(i,\cdot),g(j,\cdot)\big)=\upsilon (W_{i,j}\star\bar\rho_\ze)'\big(f(i,\cdot),g(j,\cdot)\big)\qquad\text{for all }1\leq i<j\leq r.
\end{align}
We note that $W_{i,i}$ are already standard Brownian motions in $\mbb R$; for $1\leq i<j\leq r$, we recall from Definition \ref{Definition: Noise Coupling} that we can expand \eqref{Equation: Multivariate to Scalar for Convergence 2} as
\begin{multline}
\label{Equation: Multivariate to Scalar for Convergence 3}
(W_{i,j}\star\bar\rho_\ze)'\big(f(i,\cdot),g(j,\cdot)\big)\\
=\begin{cases}
(W_{i,j;1}\star\bar\rho_\ze)'\big(f(i,\cdot),g(j,\cdot)\big)&\text{if }\mbb F=\mbb R,\\
\frac{1}{\sqrt2}\Big((W_{i,j;1}\star\bar\rho_\ze)'\big(f(i,\cdot),g(j,\cdot)\big)+(W_{i,j;\msf i}\star\bar\rho_\ze)'\big(f(i,\cdot),\msf ig(j,\cdot)\big)\Big)&\text{if }\mbb F=\mbb C,\\
\frac{1}{2}\Big(
\sum_{\ell\in\{1,\msf i,\msf j,\msf k\}}(W_{i,j;\ell}\star\bar\rho_\ze)'\big(f(i,\cdot),\ell g(j,\cdot)\big)\Big)&\text{if }\mbb F=\mbb H,
\end{cases}
\end{multline}
where $W_{i,j;\ell}$ are i.i.d. standard Brownian motions in $\mbb R$.

\eqref{Equation: Multivariate to Scalar for Convergence 1} and \eqref{Equation: Multivariate to Scalar for Convergence 3}
show that $\xi^{\eps,\ze}(f,g)$ can always be written as a linear combination of white noises in $\mbb R$
evaluated in scalar-valued functions from $I$ to $\mbb F$. We now finish this introductory section by showing
that we need only consider functions from $I$ to $\mbb R$.
When $\mbb F=\mbb R$, this observation
is of course entirely vacuous, as $f(i,x),g(j,x)\in\mbb R$ for all $x\in I$.
Otherwise,
pick an arbitrary basis element
\[\ell\in\begin{cases}
\{1,\msf i\}&\text{if }\mbb F=\mbb C,\\
\{1,\msf i,\msf j,\msf k\}&\text{if }\mbb F=\mbb H;
\end{cases}\]
define $h=\ell g$; and finally expand the components of $f$ as
\[f(i,x)=\begin{cases}
f_1(i,x)+f_2(i,x)\msf i&\text{if }\mbb F=\mbb C\\
f_1(i,x)+f_2(i,x)\msf i+f_3(i,x)\msf j+f_4(i,x)\msf k&\text{if }\mbb F=\mbb H
\end{cases}\]
and similarly for $\big(h_{m}(i,\cdot)\big)_{m\geq1}$.
To simplify the presentation, we use the notational shorthand $\mc N_\eps=(W\star\bar\rho_\eps)'$. Then,
by bilinearity and the complex product, one has
\begin{multline}
\label{Equation: Multivariate to Scalar for Convergence 4}
\mc N_\eps\big(f(i,\cdot),h(j,\cdot)\big)=
\mc N_\eps\big(f_1(i,\cdot),h_1(j,\cdot)\big)-\mc N_\eps\big(f_2(i,\cdot),h_2(j,\cdot)\big)\\
+\Big(\mc N_\eps\big(f_1(i,\cdot),h_2(j,\cdot)\big)+\mc N_\eps\big(f_2(i,\cdot),h_1(j,\cdot)\big)\Big)\msf i.
\end{multline}
when $\mbb F=\mbb C$,
and by bilinearity and the definition of quaternion product,
\begin{align}
\label{Equation: Multivariate to Scalar for Convergence 5}
&\mc N_\eps\big(f(i,\cdot),h(j,\cdot)\big)\\
\nonumber
&=\mc N_\eps\big(f_1(i,\cdot),h_1(j,\cdot)\big)-\mc N_\eps\big(f_2(i,\cdot),h_2(j,\cdot)\big)-\mc N_\eps\big(f_3(i,\cdot),h_3(j,\cdot)\big)-\mc N_\eps\big(f_2(i,\cdot),h_2(j,\cdot)\big)\\
\nonumber
&+\Big(\mc N_\eps\big(f_1(i,\cdot),h_2(j,\cdot)\big)+\mc N_\eps\big(f_2(i,\cdot),h_1(j,\cdot)\big)+\mc N_\eps\big(f_3(i,\cdot),h_4(j,\cdot)\big)-\mc N_\eps\big(f_4(i,\cdot),h_3(j,\cdot)\big)\Big)\msf i\\
\nonumber
&+\Big(\mc N_\eps\big(f_1(i,\cdot),h_3(j,\cdot)\big)-\mc N_\eps\big(f_2(i,\cdot),h_4(j,\cdot)\big)+\mc N_\eps\big(f_3(i,\cdot),h_1(j,\cdot)\big)+\mc N_\eps\big(f_4(i,\cdot),h_2(j,\cdot)\big)\big)\msf j\\
\nonumber
&+\Big(\mc N_\eps\big(f_1(i,\cdot),h_4(j,\cdot)\big)+\mc N_\eps\big(f_2(i,\cdot),h_3(j,\cdot)\big)-\mc N_\eps\big(f_3(i,\cdot),h_2(j,\cdot)\big)+\mc N_\eps\big(f_4(i,\cdot),h_1(j,\cdot)\big)\Big)\msf k
\end{align}
when $\mbb F=\mbb H$,
noting that $f_m(i,x),h_m(j,x)\in\mbb R$ for all $x\in I$ and $m\geq1$.
We are now ready to embark on the proofs of
Propositions \ref{Proposition: Operator Form-Bound},
\ref{Proposition: Uniform Eigenvalue Bounds}, and \ref{Proposition: Eigenvalue Convergence}.

\subsection{Proof of Propositions \ref{Proposition: Operator Form-Bound} and \ref{Proposition: Uniform Eigenvalue Bounds}}
\label{Section: Main Proposition 6}

\subsubsection{Proof of Proposition \ref{Proposition: Operator Form-Bound} Part 1: \eqref{Equation: +1 Norm}
and Equivalence}

For every $1\leq i\leq r$, let us denote the norm
\begin{align}
\label{Equation: L Star Direct Sum 1}
\|f\|_{\star;i}^2=\|f'\|_2^2+\|f\|_2^2+\|V_+(i,\cdot)^{1/2}f\|_2^2,\qquad f\in L^2(I,\mbb F).
\end{align}
By \cite[Proposition 3.2]{GaudreauLamarreEJP}, there exists constants $C_i>0$ large
enough so that
\[\|f\|_{+1;i}=\sqrt{\mc E_i(f,f)+(C_i+1)\|f\|_2}\]
is a norm on $D(\mc E_i)$, which is equivalent to $\|f\|_{\star;i}^2$
(in \cite[Proposition 3.2]{GaudreauLamarreEJP}, this result is only stated
for real-valued functions, but it can be trivially extended to $\mbb C$ and $\mbb H$
by linearity and noting that $L^2(I,\mbb C)\cong L^2(I,\mbb R)^{\oplus 2}$ and $L^2(I,\mbb H)\cong L^2(I,\mbb R)^{\oplus 4}$, the latter with $\prec\cdot,\cdot\succ$).
We then obtain the claim regarding \eqref{Equation: +1 Norm} in Proposition \ref{Proposition: Operator Form-Bound}
by noting that, for $f\in L^2(\mc A,\mbb F)$,
if we let $C=C_1+\cdots+C_i$, then
\begin{align}
\label{Equation: L Star Direct Sum 2}
\|f\|_\star^2=\sum_{i=1}^r\|f(i,\cdot)\|_{\star;i}^2
\qquad\text{and}\qquad\|f\|_{+1}^2=\sum_{i=1}^r\|f(i,\cdot)\|_{+1;i}^2.
\end{align}

\subsubsection{Proof of Proposition \ref{Proposition: Operator Form-Bound} Part 2: \eqref{Equation: Noise Form Bound}}

Suppose first that $\xi_{i,j}$ are white noises for all $1\leq i,j\leq r$.
By equivalence of
$\|\cdot\|_{+1}$ and $\|\cdot\|_{\star}$, it is enough to show that almost surely,
for every $\ka\in(0,1)$, there exists $\nu>0$
such that $|\xi(f,f)|\leq\ka\|f\|^2_{\star}+\nu\|f\|_\mu^2$ for every $f\in C_0^\infty(\mc A,\mbb F)$.
Toward this end, by
\eqref{Equation: Multivariate to Scalar for Convergence 1}--\eqref{Equation: Multivariate to Scalar for Convergence 3}
(with $\eps=\ze=0$)
and the triangle inequality,
it suffices to show the following: Let $W$ be a standard Brownian motion in $\mbb R$.
For any choice of $\ka\in(0,1)$ and
\[\ell\in\begin{cases}
\{1\}&\text{if }\mbb F=\mbb R\\
\{1,\msf i\}&\text{if }\mbb F=\mbb C\\
\{1,\msf i,\msf j,\msf k\}&\text{if }\mbb F=\mbb H,
\end{cases}\]
there exists $\nu>0$ such that
\begin{align}
\label{Equation: White Noise Form Bound in fi and fj}
\big|W'\big(f(i,\cdot),\ell f(j,\cdot)\big)\big|\leq\ka\|f\|_{\star}^2+\nu\|f\|_\mu^2
\end{align}
for every $f\in C_0^\infty(\mc A,\mbb F)$ and $1\leq i,j\leq r$.

Suppose first that we are in Case 1 or Case 2.
Consider the averaged process
\[\tilde W(x)=\int_x^{x+1}W(x)\d x,\qquad x\in\mbb R,\]
and write $W=\tilde W+(W-\tilde W)$. Then, by definition of $W'$
in \eqref{Equation: White Noise 0}
and an integration by parts, we have that
\begin{multline*}
W'\big(f(i,\cdot),\ell f(j,\cdot)\big)=f(i,0)^*\tilde W(0)\msf\ell f(j,0)+\big\langle f(i,\cdot),\tilde W'\ell f(j,\cdot)\big\rangle\\
+\big\langle f'(i,\cdot),(\tilde W-W)\ell f(j,\cdot)\big\rangle+\big\langle f(i,\cdot),(\tilde W-W)\ell f'(j,\cdot)\big\rangle.
\end{multline*}
By \cite[Lemma 3.1]{GaudreauLamarreEJP}, for every $\ka'>0$, there exists some $\nu'>0$
such that
\begin{align}
\label{Equation: Pointwise Bound in f}
|f(i,x)|^2\leq\ka'\|f'(i,\cdot)\|_2^2+\nu'\|f(i,\cdot)\|_2^2\qquad\text{for every }x\in I.
\end{align}
Thus, since $\tilde W(0)$ is finite and $xy\leq x^2+y^2$ for $x,y\geq0$, the latter of which implies
\begin{align}
\label{Equation: Pointwise Bound in f 2}
|f(i,0)^*\tilde W(0)\ell f(j,0)|\leq|f(i,0)|^2|\tilde W(0)|+|f(j,0)|^2|\tilde W(0)|,
\end{align}
in order to prove \eqref{Equation: White Noise Form Bound in fi and fj},
it suffices to show that almost surely, for every $\ka\in(0,1)$, there exists $\nu>0$ such that,
\begin{multline}
\label{Equation: White Noise Form Bound in fi and fj 2}
\Big|\big\langle f(i,\cdot),\tilde W'\ell f(j,\cdot)\big\rangle\Big|
+\Big|\big\langle f'(i,\cdot),(\tilde W-W)\ell f(j,\cdot)\big\rangle\Big|+\Big|\big\langle f(i,\cdot),(\tilde W-W)\ell f'(j,\cdot)\big\rangle\Big|\\
\leq\ka\|f\|_{\star}^2+\nu\|f\|_\mu^2.
\end{multline}
For this purpose,
by combining Jensen's inequality with the fact that $xy\leq\frac{\theta}{2}x^2+\frac{1}{2\theta}y^2$ for any $x,y\geq0$ and $\theta>0$,
\begin{multline}
\label{Equation: White Noise Form Bound in fi and fj 3}
\Big|\big\langle f(i,\cdot),\tilde W'\ell f(j,\cdot)\big\rangle\Big|
\leq\int_I|f(i,x)||\tilde W'(x)|^{1/2}|f(j,x)||\tilde W'(x)|^{1/2}\d x\\
\leq\frac12\int_I|f(i,x)|^2|\tilde W'(x)|\d x+\frac12\int_I|f(j,x)|^2|\tilde W'(x)|\d x;
\end{multline}
\begin{multline}
\label{Equation: White Noise Form Bound in fi and fj 4}
\Big|\big\langle f'(i,\cdot),(\tilde W-W)\ell f(j,\cdot)\big\rangle\Big|
\leq\int_I |f'(i,x)||W(x)-\tilde W(x)||f(j,x)|\d x\\
\leq\frac{\theta}{2}\|f'(i,\cdot)\|_2^2+\frac{1}{2\theta}\int_I|f(j,x)|^2\big(W(x)-\tilde W(x)\big)^2\d x;
\end{multline}
and similarly,
\begin{align}
\label{Equation: White Noise Form Bound in fi and fj 5}
\Big|\big\langle f'(i,\cdot),(\tilde W-W)\ell f(j,\cdot)\big\rangle\Big|
\leq\frac{\theta}{2}\|f'(j,\cdot)\|_2^2+\frac{1}{2\theta}\int_I|f(i,x)|^2\big(W(x)-\tilde W(x)\big)^2\d x.
\end{align}
By \cite[(4.2)]{GaudreauLamarreEJP}, there exists a finite random variable $C>0$ such that,
almost surely,
\[|\tilde W'(x)|,\big(W(x)-\tilde W(x)\big)^2\leq C\log(2+|x|)\]
for all $x\in I$. In particular, by the growth lower bound in Assumption \ref{Assumption: Potential},
this means that almost surely, for any $\theta>0$, there exists a $\nu_\theta>0$ such that for every $1\leq i\leq r$,
\[|\tilde W'(x)|,\big(W(x)-\tilde W(x)\big)^2\leq\theta^2 V(i,x)_++\nu_\theta\]
for all $x\in I$. If we plug this back into \eqref{Equation: White Noise Form Bound in fi and fj 3}--\eqref{Equation: White Noise Form Bound in fi and fj 5},
then we get
\begin{multline}
\label{Equation: White Noise Form Bound in fi and fj 6}
\Big|\big\langle f(i,\cdot),\tilde W'\ell f(j,\cdot)\big\rangle\Big|\\
\leq\tfrac{\theta^2}2\big(\|f(i,\cdot)V(i,\cdot)_+^{1/2}\|_2^2+\|f(j,\cdot)V(j,\cdot)_+^{1/2}\|_2^2\big)+\tfrac{\nu_\theta}2\big(\|f(i,\cdot)\|_2^2+\|f(j,\cdot)\|_2^2\big);
\end{multline}
\begin{align}
\label{Equation: White Noise Form Bound in fi and fj 7}
\Big|\big\langle f(i,\cdot),\tilde W'\ell f(j,\cdot)\big\rangle\Big|\leq\tfrac{\theta}{2}\|f'(i,\cdot)\|_2^2+\tfrac{\theta}2\|f(j,\cdot)V(j,\cdot)_+^{1/2}\|_2^2+\tfrac{\nu_\theta}{2\theta}\|f(j,\cdot)\|_2^2;
\end{align}
and similarly
\begin{align}
\label{Equation: White Noise Form Bound in fi and fj 8}
\Big|\big\langle f'(i,\cdot),(\tilde W-W)\ell f(j,\cdot)\big\rangle\Big|\leq\tfrac{\theta}{2}\|f'(j,\cdot)\|_2^2+\tfrac\theta2\|f(i,\cdot)V(i,\cdot)_+^{1/2}\|_2^2+\tfrac{\nu_\theta}{2\theta}\|f(i,\cdot)\|_2^2
\end{align}
Given that $\theta>0$ above can be taken arbitrarily small, and that 
\[\|f'(i,\cdot)\|_2^2,\|f(i,\cdot)V(i,\cdot)_+^{1/2}\|_2^2\leq\|f\|_\star^2\qquad\text{and}\qquad\|f(i,\cdot)\|_2^2\leq\|f\|_\mu^2\]
for all $1\leq i\leq r$, we obtain \eqref{Equation: White Noise Form Bound in fi and fj 2}
from \eqref{Equation: White Noise Form Bound in fi and fj 6}--\eqref{Equation: White Noise Form Bound in fi and fj 8}.

We now finish the proof of \eqref{Equation: Noise Form Bound}---and thus of Proposition \ref{Proposition: Operator Form-Bound} in the case of white noises---by establishing
\eqref{Equation: White Noise Form Bound in fi and fj} in Case 3. In that case, by definition of $W'$, we can write
\[\big|W'\big(f(i,\cdot),\ell f(j,\cdot)\big)\big|\leq|f(i,\th)||f(j,\th)||W(\th)|
+\big|\big\langle f'(i,\cdot),Wf(j,\cdot)\big\rangle\big|+\big|\big\langle f(i,\cdot),Wf'(j,\cdot)\big\rangle\big|.\]
With this in hand, we obtain \eqref{Equation: Noise Form Bound} by controlling $|f(i,\th)||f(j,\th)||W(\th)|$ using \eqref{Equation: Pointwise Bound in f}
and \eqref{Equation: Pointwise Bound in f 2}, and then noting that
\begin{multline}
\big|\big\langle f'(i,\cdot),W\ell f(j,\cdot)\big\rangle\big|\leq\sup_{y\in[0,\th]}|W(y)|\int_I |f'(i,x)||f(j,x)|\d x\\
\leq\sup_{y\in[0,\th]}|W(y)|\big(\tfrac\theta2\|f'(i,\cdot)\|_2^2+\tfrac1{2\theta}\|f(j,x)\|_2^2\big)
\end{multline}
for any choice of $\theta>0$ (a similar bound holds for $\langle f(i,\cdot),W\ell f'(j,\cdot)\rangle$).

As a final remark, we note that exactly the same argument applies if $\xi$ contain
regular noises, as per Definition \ref{Definition: Regular Noise}. For this,
we simply replace $W$ in the above argument by the integral $\Xi(x)=\int_0^x\mc X(x)\d x$,
where $\Xi$ is one of the regular noises used to define any of the $\xi_{i,j}$.
This case is also covered by \cite[(4.2)]{GaudreauLamarreEJP}
(in the terminology of \cite{GaudreauLamarreEJP}, the case of a regular noise
is called {\it bounded}, which comes from the fact that the noise's covariance
function is in $L^\infty(\mbb R)$; see \cite[Example 2.28]{GaudreauLamarreEJP}).
Thus, the proof of Proposition \ref{Proposition: Operator Form-Bound} is now complete.

\subsubsection{Proof of Proposition \ref{Proposition: Uniform Eigenvalue Bounds}}

We now conclude this subsection by proving Proposition \ref{Proposition: Uniform Eigenvalue Bounds}.
The reason why we prove this in the same subsection as Proposition \ref{Proposition: Operator Form-Bound}
is that its proof is very similar to \eqref{Equation: Noise Form Bound}. Indeed, by the variational characterization
of the eigenvalues in \eqref{Equation: Variational} and the equivalence of $\|\cdot\|_\star$ with
the norm induced by $H$'s quadratic form in \eqref{Equation: +1 Norm}, the statement of Proposition
\ref{Proposition: Uniform Eigenvalue Bounds} can be reduced to the following:
Almost surely, for every $\ka\in(0,1)$, there exists some $\nu>0$ such that for every
$\eps,\ze\in[0,1)$ and $f\in C_0^\infty(\mc A,\mbb F)$, one has
\begin{align}
\label{Equation: Uniform Form Bounds}
|\xi^{\eps,\ze}(f,f)|\leq\ka\|f\|^2_{\star}+\nu\|f\|_\mu^2.
\end{align}
In similar fashion to the reduction we obtained in \eqref{Equation: White Noise Form Bound in fi and fj},
by
\eqref{Equation: Multivariate to Scalar for Convergence 1}--\eqref{Equation: Multivariate to Scalar for Convergence 3},
this latter claim can be reduced to showing that
\begin{align}
\label{Equation: Uniform Noise Form Bound in fi and fj}
\big|(W\star\bar\rho_\eps)'\big(f(i,\cdot),\ell f(j,\cdot)\big)\big|\leq\ka\|f\|_{\star}^2+\nu\|f\|_\mu^2
\end{align}
for every $\eps>0$ (here $\eps$ plays the role of both $\eps$ and $\ze$) and $f\in C_0^\infty(\mc A,\mbb F)$.
By replicating the proof of \eqref{Equation: White Noise Form Bound in fi and fj} above verbatim,
for this it suffices to show that if we define
\[\tilde W_\eps(x)=\int_x^{x+1}W\star\bar\rho_\eps(x)\d x,\qquad x\in\mbb R,\]
then there exists a finite random variable $C>0$ such that, almost surely,
\[|\tilde W_\eps'(x)|,\big(W_\eps(x)-\tilde W_\eps(x)\big)^2\leq C\log(2+|x|)\]
for all $\eps\in[0,1)$, and moreover that
\[\sup_{\eps\in[0,1)}\sup_{x\in[0,\th]}|W\star\bar\rho_\eps(x)|<\infty.\]
These are proved in \cite[Lemma 5.18]{GaudreauLamarreEJP}.

\subsection{Proof of Proposition \ref{Proposition: Eigenvalue Convergence}}
\label{Section: Main Proposition 7}

\subsubsection{Part 1: Two Lemmas}

This proof relies on two technical lemmas:

\begin{lemma}\label{Lemma: L Star Compactness}
If $(f_n)_{n\in\mbb N}\subset D(\mc E)$ is such that $\sup_n\|f_n\|_\star<\infty$, then
there exists $f\in D(\mc E)$ and a subsequence $(n_i)_{i\in\mbb N}$ along which
\begin{enumerate}
\item $\displaystyle\lim_{i\to\infty}\|f_{n_i}-f\|_\mu=0;$
\item $\displaystyle\lim_{i\to\infty}\langle g,f'_{n_i}\rangle_\mu=\langle g,f'\rangle_\mu$ for every $g\in L^2(\mc A,\mbb F)$;
\item $\displaystyle\lim_{i\to\infty}f_{n_i}=f$ uniformly on compact sets; and
\item $\displaystyle\lim_{i\to\infty}\langle g,f_{n_i}\rangle_\star=\langle g,f\rangle_\star$
for every $g\in D(\mc E)$.
\end{enumerate}
\end{lemma}

\begin{lemma}
\label{Lemma: Generic Convergence}
Let $(\mc Q_{\eps})_{\eps\in[0,1)}$ be a collection of bilinear forms such that the following conditions hold:
\begin{enumerate}
\item There exists $\ka\in(0,1)$ and $\nu>0$ such that
\[|\mc Q_\eps(f,f)|\leq\ka\mc E(f,f)+\nu\|f\|_\mu^2\qquad\text{for every }\eps\in[0,1)\text{ and }f\in D(\mc E).\]
\item $(\mc Q_{\eps})_{\eps\in[0,1)}$ have a common form core $\mr{FC}$. For every $f,g\in\mr{FC}$, one has
\begin{align}
\label{Equation: Generic Convergence 1}
\lim_{\eps\to0}\mc Q_\eps(f,g)=\mc Q_0(f,g).
\end{align}
\item If $(\eps_n)_{n\in\mbb N}\subset(0,1)$ is a vanishing sequence, $\sup_{n}\|f_n\|_\star<\infty$, and $f_n\to f$ in the
sense of Lemma \ref{Lemma: L Star Compactness} (1)--(4), then for every $g\in\mr{FC}$, one has
\begin{align}
\label{Equation: Generic Convergence 2}
\lim_{n\to\infty}\mc Q_{\eps_n}(f_n,g)=\mc Q_0(f,g).
\end{align}
\end{enumerate}
If we let $T_\eps$ denote the unique self-adjoint operators with form $\mc E+\mc Q_\eps$, then
every vanishing sequence $(\eps_n)_{n\in\mbb N}\subset(0,1)$ has a subsequence $(\eps_{n_\ell})_{\ell\in\mbb N}$ along which
\begin{align}
\label{Equation: Generic Lemma Eigenvalue Convergence}
\lim_{\ell\to\infty}\la_k(T_{\eps_{n_\ell}})=\la_k(T_0)\qquad\text{for every }k\geq1.
\end{align}
\end{lemma}

Since the proof of these lemmas are mostly technical in nature and amount to trivial modifications
of results previously proved in \cite{GaudreauLamarreEJP}, we omit a full proof and opt instead
to provide precise references:

On the one hand,
Lemma \ref{Lemma: L Star Compactness}
follows from expanding $\|\cdot\|_\star^2=\sum_i\|\cdot\|_{\star;i}^2$
using \eqref{Equation: L Star Direct Sum 1} and \eqref{Equation: L Star Direct Sum 2},
and noting that \cite[Lemma 5.15]{GaudreauLamarreEJP} proved
the same statement as Lemma \ref{Lemma: Generic Convergence}
with $\|\cdot\|_{\star}$, $D(\mc E)$, and $\langle\cdot,\cdot\rangle_\mu$ / $\|\cdot\|_\mu$ / $L^2(\mc A,\mbb F)$
respectively replaced by $\|\cdot\|_{\star;i}$, $D(\mc E_i)$, and $\langle\cdot,\cdot\rangle$ / $\|\cdot\|_2$ / $L^2(I,\mbb F)$
(see also \cite[Fact 2.2]{RamirezRiderVirag}, \cite[Fact 2.2]{BloemendalVirag}, and \cite[Fact 3.2]{BloemendalVirag2}).

On the other hand,
Lemma \ref{Lemma: Generic Convergence} follows from a verbatim application
of the argument presented in \cite[Pages 33 and 34, starting from (5.31)]{GaudreauLamarreEJP}.
More specifically:
\begin{enumerate}
\item Lemma \ref{Lemma: L Star Compactness} plays the role of
\cite[Lemma 5.15]{GaudreauLamarreEJP};
\item the forms $\mc Q_\eps$ and the operators $T_\eps$ play the role of
$\mc E_\eps$ and $\hat H_\eps$ in \cite{GaudreauLamarreEJP};
\item Lemma \ref{Lemma: Generic Convergence}-(1) plays the role of
\cite[Lemmas 5.17 and 5.18 and Remark 5.19]{GaudreauLamarreEJP};
\item Lemma \ref{Lemma: Generic Convergence}-(2) plays the role of
\cite[Lemma 5.20, (5.29)]{GaudreauLamarreEJP}; and
\item Lemma \ref{Lemma: Generic Convergence}-(3) plays the role of
\cite[Lemma 5.20, (5.30)]{GaudreauLamarreEJP}.
\end{enumerate}

\subsubsection{Part 2: Application of Lemmas}

At this point, in order to prove Proposition \ref{Proposition: Eigenvalue Convergence},
it suffices to show that for every fixed $\ze>0$,
almost surely,
the sequence of forms $(\mc Q_\eps)_{\eps\in[0,1)}=(\xi^{\eps,\ze})_{\eps\in[0,1)}$
satisfies the conditions (1)--(3) in the statement of Lemma \ref{Lemma: Generic Convergence},
and similarly for $(\mc Q_\eps)_{\eps\in[0,1)}=(\xi^{0,\eps})_{\eps\in[0,1)}$
(where in this latter sequence $\eps$ now plays the role of the parameter $\ze$ in \eqref{Equation: Eigenvalue Convergence 2}).

For this purpose, we first note that the claim of Lemma \ref{Lemma: Generic Convergence}-(1)
for both sequences follows from \eqref{Equation: Uniform Form Bounds}
and the equivalence of the norms $\|\cdot\|_\star$ and $\|\cdot\|_{+1}$
(the latter of which was established in the proof of Proposition \ref{Proposition: Operator Form-Bound}
above). Next, regarding items (2) and (3), we have by Remark
\ref{Remark: Form Core} that all forms have a common core $\mr{FC}=\mr{FC}_1\oplus\cdots\oplus\mr{FC}_r$, where the $\mr{FC}_i$ are defined
in Remark \ref{Remark: Form Core for Hi}. If we combine this with
the decompositions
\eqref{Equation: Smoothed Noise Sum Decomposition},
\eqref{Equation: Multivariate to Scalar for Convergence 4},
and \eqref{Equation: Multivariate to Scalar for Convergence 5},
then the claims in \eqref{Equation: Generic Convergence 1} and 
\eqref{Equation: Generic Convergence 2} can be reduced to the following:
Let $W$ be a standard Brownian motion, and let $1\leq i\leq r$ be fixed.
For every real-valued $f,g\in\mr{FC}_i$, one has
\begin{align}
\label{Equation: Generic Convergence 1 - scalar}
\lim_{\eps\to0}(W\star\bar\rho_\eps)'(f,g)=W'(f,g).
\end{align}
Moreover, if $(\eps_n)_{n\in\mbb N}\subset(0,1)$ is a vanishing sequence,
the sequence of functions
$f_n\in L^2(I,\mbb R)$ is such that
$\sup_{n}\|f_n\|_{\star;i}<\infty$, and $f_n\to f$ in the
sense of Lemma \ref{Lemma: L Star Compactness} (1)--(4)
(with $\|\cdot\|_{\star}$ and $\langle\cdot,\cdot\rangle_\mu$ / $\|\cdot\|_\mu$ / $L^2(\mc A,\mbb F)$
respectively replaced by $\|\cdot\|_{\star;i}$ and $\langle\cdot,\cdot\rangle$ / $\|\cdot\|_2$ / $L^2(I,\mbb F)$), then for every real-valued $g\in\mr{FC}_i$, one has
\begin{align}
\label{Equation: Generic Convergence 2 - scalar}
\lim_{n\to\infty}(W\star\bar\rho_{\eps_n})'(f_n,g)=W'(f,g).
\end{align}
These are the statements proved in \cite[Lemma 5.20]{GaudreauLamarreEJP},
in the special case where, in the notation of \cite{GaudreauLamarreEJP}, $\Xi=W$
and $\xi=W'$.
The proof of Proposition \ref{Proposition: Eigenvalue Convergence} is thus complete.

\section{Proof of Theorem \ref{Theorem: Rigidity}: Trace Covariance Asymptotics}
\label{Section: Covariance Estimates}

\begin{notation}
Throughout this section, we use $C,c>0$
to denote constants that are independent of $t_1,t_2$,
whose exact value may change from one display to the next.
\end{notation}

\begin{remark}
In Cases 1 and 2, we assume that there exists
$\ka,\nu>0$ such that
\begin{align}
\label{Equation: Always Assume Linear Growth}
V(i,x)\geq\ka|x|-\nu\qquad\text{ for all $1\leq i\leq r$ and $x\in I$}.
\end{align}
In these two cases, the fact that $I$ is infinite makes this a nontrivial strengthening 
of Assumption \ref{Assumption: Potential}.
However, the assumption that $V$ is bounded below and the fact that $I$ is bounded
in Case 3 implies that \eqref{Equation: Always Assume Linear Growth} also holds in that case.
Thus, for convenience we assume that \eqref{Equation: Always Assume Linear Growth} holds
in all cases in this section.
\end{remark}

\subsection{Proof Outline}

The only technical input into the proof of Theorem \ref{Theorem: Rigidity}
is the following:

\begin{proposition}
\label{Proposition: Covariance Estimate}
Let the hypotheses of Theorem \ref{Theorem: Rigidity} hold,
and recall the definition of $\hat H^{0,\ze}$ for $\ze>0$ in \eqref{Equation: Double Approximation Operator}.
There exists $C>0$ such that for every $t_1,t_2\in(0,1]$,
\begin{align}
\label{Equation: Rigidity RHS}
\sup_{\ze>0}\left|\mbf{Cov}\Big[\mr{Tr}\big[\mr e^{-t_1\hat H^{0,\ze}}\big],\mr{Tr}\big[\mr e^{-t_2\hat H^{0,\ze}}\big]\Big]\right|\leq C\cdot
\begin{cases}
\displaystyle\frac{(t_1t_2)^{1/4}}{\sqrt{\max\{t_1,t_2\}}}&\text{Cases 1 and 2},
\vspace{10pt}\\
(t_1t_2)^{1/4}&\text{Case 3}.
\end{cases}
\end{align}
\end{proposition}

By \eqref{Equation: Trace Formula for 0,zeta}, for any $\ze>0$ we can write
\begin{multline*}
\mbf{Cov}\Big[\mr{Tr}\big[\mr e^{-t_1\hat H^{0,\ze}}\big],\mr{Tr}\big[\mr e^{-t_2\hat H^{0,\ze}}\big]\Big]
=\mbf E\left[\int_{\mc A}\hat K^{0,\ze}(t_1;a,a)\d\mu(a)\int_{\mc A}\hat K^{0,\ze}(t_2;a,a)\d\mu(a)\right]\\-\mbf E\left[\int_{\mc A}\hat K^{0,\ze}(t_1;a,a)\d\mu(a)\right]\mbf E\left[\int_{\mc A}\hat K^{0,\ze}(t_2;a,a)\d\mu(a)\right].
\end{multline*}
By combining \eqref{Equation: Trace Moments Limit 2.1} and Theorem \ref{Theorem: Trace Moment Formulas}, this converges
as $\ze\to0$ to
\[\mbf E\Big[\mr{Tr}\big[\mr e^{-t_1\hat H}\big]\mr{Tr}\big[\mr e^{-t_2\hat H}\big]\Big]-\mbf E\Big[\mr{Tr}\big[\mr e^{-t_1\hat H}\big]\Big]\mbf E\Big[\mr{Tr}\big[\mr e^{-t_2\hat H}\big]\Big]
=\mbf{Cov}\Big[\mr{Tr}\big[\mr e^{-t_1\hat H}\big],\mr{Tr}\big[\mr e^{-t_2\hat H}\big]\Big].\]
Thus, we conclude that
\begin{align}
\label{Equation: Rigidity RHS Limit}
\lim_{\ze\to0}\left|\mbf{Cov}\Big[\mr{Tr}\big[\mr e^{-t_1\hat H^{0,\ze}}\big],\mr{Tr}\big[\mr e^{-t_2\hat H^{0,\ze}}\big]\Big]\right|
=\left|\mbf{Cov}\Big[\mr{Tr}\big[\mr e^{-t_1\hat H}\big],\mr{Tr}\big[\mr e^{-t_2\hat H}\big]\Big]\right|.
\end{align}
If we combine Proposition \ref{Proposition: Covariance Estimate} with \eqref{Equation: Rigidity RHS Limit}, we obtain that
\begin{align}
\label{Equation: Rigidity RHS ze=0}
\left|\mbf{Cov}\Big[\mr{Tr}\big[\mr e^{-t_1\hat H}\big],\mr{Tr}\big[\mr e^{-t_2\hat H}\big]\Big]\right|\leq C\cdot
\begin{cases}
\displaystyle\frac{(t_1t_2)^{1/4}}{\sqrt{\max\{t_1,t_2\}}}&\text{Cases 1 and 2},
\vspace{10pt}\\
(t_1t_2)^{1/4}&\text{Case 3},
\end{cases}
\end{align}
whenever the hypotheses of Theorem \ref{Theorem: Rigidity} hold.
This obviously implies \eqref{Equation: Rigidity Sufficient Condition 1}
and \eqref{Equation: Rigidity Sufficient Condition 2}, hence the proof of
Theorem \ref{Theorem: Rigidity} is complete.

\begin{remark}
Given that we spend much of this paper proving exact formulas for the mixed moments
of $\hat H$ with white noise (as opposed to its smooth approximations $\hat H^{\eps,\ze}$), it is natural to
wonder why we prove \eqref{Equation: Rigidity RHS} instead of
establishing \eqref{Equation: Rigidity RHS ze=0} directly. The reason for this is entirely technical,
and lies in the fact that manipulating the conditioned process $\hat U$ is more difficult
than the random walk $U$. To get around this, on the way to establishing 
\eqref{Equation: Rigidity RHS}, we use several bounds that simplify the expression for the
covariance before taking $\ze\to0$, the latter of which allow to sidestep
$\hat U$ entirely.
\end{remark}

It now only remains to prove Proposition \ref{Proposition: Covariance Estimate}:

\subsection{Proof of Proposition \ref{Proposition: Covariance Estimate}}

\subsubsection{Step 1. Covariance Formula}

Our first step toward Proposition \ref{Proposition: Covariance Estimate}
is to use the trace moment formula in Proposition \ref{Proposition: Smooth Trace Moments}
to provide an expression for
\begin{multline}
\label{Equation: Covariance Formula 1}
\mbf{Cov}\Big[\mr{Tr}\big[\mr e^{-t_1\hat H^{0,\ze}}\big],\mr{Tr}\big[\mr e^{-t_2\hat H^{0,\ze}}\big]\Big]\\
=\mbf E\Big[\mr{Tr}\big[\mr e^{-t_1\hat H^{0,\ze}}\big]\mr{Tr}\big[\mr e^{-t_2\hat H^{0,\ze}}\big]\Big]
-\mbf E\Big[\mr{Tr}\big[\mr e^{-t_1\hat H^{0,\ze}}\big]\Big]\mbf E\Big[\mr{Tr}\big[\mr e^{-t_2\hat H^{0,\ze}}\big]\Big]
\end{multline}
that is amenable to computation. For the first term on the right-hand side of \eqref{Equation: Covariance Formula 1},
Proposition \ref{Proposition: Smooth Trace Moments} in the case $n=2$ yields:
\begin{multline}
\label{Equation: Covariance Formula 2}
\mbf E\Big[\mr{Tr}\big[\mr e^{-t_1\hat H^{0,\ze}}\big]\mr{Tr}\big[\mr e^{-t_2\hat H^{0,\ze}}\big]\Big]
=\int_{\mc A^2}\Pi_A(\bs t;\bs a,\bs a)\mbf E\Bigg[\mf m^{\bs\ze}_{\bs t}(A^{\bs a,\bs a}_{\bs t})\\
\cdot\mr e^{-\int_0^{|\bs t|}V(A^{\bs a,\bs a}_{\bs t}(s))\d s+\mf B_{|\bs t|}(A^{\bs a,\bs a}_{\bs t})+\frac{\si^2}{2}\|L_{[0,t_1)}(A^{\bs a,\bs a}_{\bs t})+L_{[t_1,t_1+t_2)}(A^{\bs a,\bs a}_{\bs t})\|_\mu^2}\Bigg]\d\mu^2(\bs a),
\end{multline}
where we denote $\bs t=(t_1,t_2)$ and $\bs\ze=(\ze,\ze)$.
In contrast, two applications of Proposition \ref{Proposition: Smooth Trace Moments} in the case $n=1$
yields
\begin{multline}
\label{Equation: Covariance Formula 3}
\mbf E\Big[\mr{Tr}\big[\mr e^{-t_1\hat H^{0,\ze}}\big]\Big]\mbf E\Big[\mr{Tr}\big[\mr e^{-t_2\hat H^{0,\ze}}\big]\Big]
=\prod_{k=1}^2\int_{\mc A}\Pi_A(t_k;a_k,a_k)\mbf E\Bigg[\mf m^{\ze}_{t_k}(A^{a_k,a_k}_{t_k})\\
\cdot\mr e^{-\int_0^{t_k}V(A^{a_k,a_k}_{t_k}(s))\d s+\mf B_{t_k}(A^{a_k,a_k}_{t_k})+\frac{\si^2}{2}\|L_{[0,t_k)}(A^{a_k,a_k}_{t_k})\|_\mu^2}\Bigg]\d\mu(a_k).
\end{multline}
In order to relate this to \eqref{Equation: Covariance Formula 3}, we introduce the following couplings:

\begin{notation}
\label{Notation: Z and Z bar coupling}
For any given $\bs a=(a_1,a_2)\in\mc A^2$, where $a_k=(i_k,x_k)$ for $k=1,2$,
we let $A^{a_1,a_1}_{t_1}=(U^{i_1,i_1}_{t_1},Z^{x_1,x_1}_{t_1})$ be as in
Definitions \ref{Definition: Combined Stochastic Processes}
and \ref{Definition: Combined Stochastic Processes 2}, and let
$\bar A^{a_2,a_2}_{t_2}=(\bar U^{i_2,i_2}_{t_2},\bar Z^{x_2,x_2}_{t_2})$
be a process that has the same distribution as $A_{t_2}^{a_2,a_2}$, but which we assume
is independent of $A^{a_1,a_1}_{t_1}$. Then we can couple these two processes with
the concatenated process $A^{\bs a,\bs a}_{\bs t}$ in \eqref{Equation: Covariance Formula 3}
in such a way that
\[A^{\bs a,\bs a}_{\bs t}(s)=\begin{cases}
A^{a_1,a_1}_{t_1}(s)&\text{if }0\leq s< t_1,
\vspace{5pt}\\
\bar A^{a_2,a_2}_{t_2}(s-t_1)&\text{if }t_1\leq s\leq t_1+t_2.
\end{cases}\]
\end{notation}

Under this coupling, the additivity of integrals and boundary local time yields
\[\int_0^{t_1}V\big(A^{a_1,a_1}_{t_1}(s)\big)\d s+\int_0^{t_2}V\big(\bar A^{a_2,a_2}_{t_2}(s)\big)\d s=\int_0^{|\bs t|}V\big(A^{\bs a,\bs a}_{\bs t}(s)\big)\d s,\]
\[\mf B_{t_1}(A^{a_1,a_1}_{t_1})+\mf B_{t_2}(\bar A^{a_2,a_2}_{t_2})=\mf B_{|\bs t|}(A^{\bs a,\bs a}_{\bs t}),\]
where $\bs a=(a_1,a_2)$. Moreover,
\[L_{[0,t_1)}(A^{a_1,a_1}_{t_1})=L_{[0,t_1)}(A^{\bs a,\bs a}_{\bs t});\qquad L_{[0,t_2)}(\bar A^{a_2,a_2}_{t_2})=L_{[t_1,t_1+t_2)}(A^{\bs a,\bs a}_{\bs t});\]
and $\prod_{k=1}^2\Pi_A(t_k;a_k,a_k)=\Pi(\bs t;\bs a,\bs a)$. Therefore, an application of Fubini's theorem yields that
\eqref{Equation: Covariance Formula 3} is equal to
\begin{multline}
\label{Equation: Covariance Formula 4}
\int_{\mc A^2}\Pi_A(\bs t;\bs a,\bs a)\mbf E\Bigg[\mf m^{\ze}_{t_1}(A^{a_1,a_1}_{t_1})\mf m^{\ze}_{t_2}(\bar A^{a_2,a_2}_{t_2})\\
\cdot\mr e^{-\int_0^{|\bs t|}V(A^{\bs a,\bs a}_{\bs t}(s))\d s+\mf B_{|\bs t|}(A^{\bs a,\bs a}_{\bs t})+\frac{\si^2}{2}(\|L_{[0,t_1)}(A^{\bs a,\bs a}_{\bs t})\|_\mu^2+\|L_{[t_1,t_1+t_2)}(A^{\bs a,\bs a}_{\bs t})\|_\mu^2)}\Bigg]\d\mu^2(\bs a).
\end{multline}
In summary, the only two differences between \eqref{Equation: Covariance Formula 2} and \eqref{Equation: Covariance Formula 4} are
\[\mf m^{\bs\ze}_{\bs t}(A^{\bs a,\bs a}_{\bs t})
\quad\text{vs.}\quad
\mf m^{\ze}_{t_1}(A^{a_1,a_1}_{t_1})\mf m^{\ze}_{t_2}(\bar A^{a_2,a_2}_{t_2})\]
and
\[\tfrac{\si^2}{2}\|L_{[0,t_1)}(A^{\bs a,\bs a}_{\bs t})+L_{[t_1,t_1+t_2)}(A^{\bs a,\bs a}_{\bs t})\|_\mu^2
\quad\text{vs.}\quad
\tfrac{\si^2}{2}(\|L_{[0,t_1)}(A^{\bs a,\bs a}_{\bs t})\|_\mu^2+\|L_{[t_1,t_1+t_2)}(A^{\bs a,\bs a}_{\bs t})\|_\mu^2).\]
In order to analyze the contributions of these differences
to the covariance one at a time,
we add and subtract the intermediary term 
\begin{multline*}
\int_{\mc A^2}\Pi_A(\bs t;\bs a,\bs a)\mbf E\Bigg[\mf m^{\bs\ze}_{\bs t}(A^{\bs a,\bs a}_{\bs t})\\
\cdot\mr e^{-\int_0^{|\bs t|}V(A^{\bs a,\bs a}_{\bs t}(s))\d s+\mf B_{|\bs t|}(A^{\bs a,\bs a}_{\bs t})+\frac{\si^2}{2}(\|L_{[0,t_1)}(A^{\bs a,\bs a}_{\bs t})\|_\mu^2+\|L_{[t_1,t_1+t_2)}(A^{\bs a,\bs a}_{\bs t})\|_\mu^2)}\Bigg]\d\mu^2(\bs a);
\end{multline*}
in doing so, the covariance in \eqref{Equation: Covariance Formula 1} (which we write as the difference between
\eqref{Equation: Covariance Formula 2} and \eqref{Equation: Covariance Formula 4} with the intermediary term
added and subtracted)
can be written as the sum of the following two terms:
\begin{multline}
\label{Equation: Covariance Formula Piece 1}
\int_{\mc A^2}\Pi_A(\bs t;\bs a,\bs a)\mbf E\Bigg[\mf m^{\bs\ze}_{\bs t}(A^{\bs a,\bs a}_{\bs t})\mr e^{-\int_0^{|\bs t|}V(A^{\bs a,\bs a}_{\bs t}(s))\d s+\mf B_{|\bs t|}(A^{\bs a,\bs a}_{\bs t})}\\
\cdot\left(\mr e^{\frac{\si^2}{2}\|L_{[0,t_1)}(A^{\bs a,\bs a}_{\bs t})+L_{[t_1,t_1+t_2)}(A^{\bs a,\bs a}_{\bs t})\|_\mu^2}-\mr e^{\frac{\si^2}{2}(\|L_{[0,t_1)}(A^{\bs a,\bs a}_{\bs t})\|_\mu^2+\|L_{[t_1,t_1+t_2)}(A^{\bs a,\bs a}_{\bs t})\|_\mu^2)}\right)\Bigg]\d\mu^2(\bs a);
\end{multline}
and
\begin{multline}
\label{Equation: Covariance Formula Piece 2}
\int_{\mc A^2}\Pi_A(\bs t;\bs a,\bs a)\mbf E\Bigg[\left(\mf m^{\bs\ze}_{\bs t}(A^{\bs a,\bs a}_{\bs t})-\mf m^{\ze}_{t_1}(A^{a_1,a_1}_{t_1})\mf m^{\ze}_{t_2}(\bar A^{a_2,a_2}_{t_2})\right)\\
\cdot\mr e^{-\int_0^{|\bs t|}V(A^{\bs a,\bs a}_{\bs t}(s))\d s+\mf B_{|\bs t|}(A^{\bs a,\bs a}_{\bs t})+\frac{\si^2}{2}(\|L_{[0,t_1)}(A^{\bs a,\bs a}_{\bs t})\|_\mu^2+\|L_{[t_1,t_1+t_2)}(A^{\bs a,\bs a}_{\bs t})\|_\mu^2)}\Bigg]\d\mu^2(\bs a).
\end{multline}
The proof of 
Proposition \ref{Proposition: Covariance Estimate} now relies on showing that for any choice of $\ze>0$,
\begin{align}
\label{Equation: Rigidity RHS Reduced to 2 Cases}
|\eqref{Equation: Covariance Formula Piece 1}|~,|\eqref{Equation: Covariance Formula Piece 2}|\leq\text{right-hand side of }\eqref{Equation: Rigidity RHS}.
\end{align}
We bound these terms one at a time.

\subsubsection{Step 2. Upper Bound for \eqref{Equation: Covariance Formula Piece 1}}

Noting that
\begin{multline*}
\|L_{[0,t_1)}(A^{\bs a,\bs a}_{\bs t})+L_{[t_1,t_1+t_2)}(A^{\bs a,\bs a}_{\bs t})\|_\mu^2\\=
\|L_{[0,t_1)}(A^{\bs a,\bs a}_{\bs t})\|_\mu^2+\|L_{[t_1,t_1+t_2)}(A^{\bs a,\bs a}_{\bs t})\|_\mu^2+2\big\langle L_{[0,t_1)}(A^{\bs a,\bs a}_{\bs t}),L_{[t_1,t_1+t_2)}(A^{\bs a,\bs a}_{\bs t})\big\rangle_\mu,
\end{multline*}
and then recalling the bound $|\mf m^{\bs\ze}_{\bs t}(A)|\leq\mr e^{(r-1)|\bs t|}\bar{\mf m}^{\bs\ze}_{\bs t}(A),$
where $\bar{\mf m}^{\bs\ze}_{\bs t}(A)$ was defined in \eqref{Equation: Vector Valued Self-Intersection 1 Without C},
for any $t_1,t_2\in(0,1]$ we can bound the absolute value of \eqref{Equation: Covariance Formula Piece 1} as follows:
\begin{multline}
\label{Equation: Covariance Formula Piece 1 - 1}
|\eqref{Equation: Covariance Formula Piece 1}|\leq C\int_{\mc A^2}\Pi_A(\bs t;\bs a,\bs a)\mbf E\Bigg[\bar{\mf m}^{\bs\ze}_{\bs t}(A^{\bs a,\bs a}_{\bs t})\mr e^{-\int_0^{|\bs t|}V(A^{\bs a,\bs a}_{\bs t}(s))\d s+\mf B_{|\bs t|}(A^{\bs a,\bs a}_{\bs t})}\\
\cdot\mr e^{\frac{\si^2}{2}(\|L_{[0,t_1)}(A^{\bs a,\bs a}_{\bs t})\|_\mu^2+\|L_{[t_1,t_1+t_2)}(A^{\bs a,\bs a}_{\bs t})\|_\mu^2)}\left(\mr e^{\si^2\langle L_{[0,t_1)}(A^{\bs a,\bs a}_{\bs t}),L_{[t_1,t_1+t_2)}(A^{\bs a,\bs a}_{\bs t})\rangle_\mu}-1\right)\Bigg]\d\mu^2(\bs a).
\end{multline}
Recall from \eqref{Equation: Multivariate Local Time} that
\[\sum_{i=1}^rL_{[0,t_1)}^{(i,x)}(A^{\bs a,\bs a}_{\bs t})=L_{t_1}^{x}(Z^{x_1,x_1}_{t_1})
\qquad\text{and}\qquad
\sum_{i=1}^rL_{[t_1,t_1+t_2)}^{(i,x)}(A^{\bs a,\bs a}_{\bs t})=L_{t_2}^{x}(\bar Z^{x_2,x_2}_{t_2}).\]
In particular, since local time is nonnegative,
\[\langle L_{[0,t_1)}(A^{\bs a,\bs a}_{\bs t}),L_{[t_1,t_1+t_2)}(A^{\bs a,\bs a}_{\bs t})\rangle_\mu\leq\langle L_{t_1}(Z^{x_1,x_1}_{t_1}),L_{t_2}(\bar Z^{x_2,x_2}_{t_2})\rangle.\]
Moreover, by \eqref{Equation: s_t bound 3.0} applied to the case where $f$ is $A^{\bs a,\bs a}_{\bs t}$'s local time,
\[\|L_{[0,t_1)}(A^{\bs a,\bs a}_{\bs t})\|_\mu^2\leq\|L_{t_1}(Z^{x_1,x_1}_{t_1})\|_2^2,
\quad\|L_{[t_1,t_1+t_2)}(A^{\bs a,\bs a}_{\bs t})\|_\mu^2\leq\|L_{t_2}(\bar Z^{x_2,x_2}_{t_2})\|_2^2.\]
Next, by \cite[Proposition A.1]{Prev2}, there exists some $C>0$ such that
\[\sup_{\bs x\in I^2}\Pi_Z(\bs t;\bs x,\bs x)\leq C(t_1t_2)^{-1/2},\qquad \text{for all }t_1,t_2\in(0,1].\]
If we combing all this with \eqref{Equation: UI Boundary Bound} and \eqref{Equation: Always Assume Linear Growth}, we get
from \eqref{Equation: Covariance Formula Piece 1 - 1} that
\begin{multline}
\label{Equation: Covariance Formula Piece 1 - 2}
|\eqref{Equation: Covariance Formula Piece 1}|\leq C(t_1t_2)^{-1/2}\int_{\mc A^2}\Pi_U(\bs t;\bs i,\bs i)\mbf E\Bigg[\bar{\mf m}^{\bs\ze}_{\bs t}(A^{\bs a,\bs a}_{\bs t})\mr e^{-\int_0^{|\bs t|} \ka|Z^{\bs x,\bs x}_{\bs t}(s)|\d s+\bar{\mf B_{|\bs t|}}(Z^{\bs x,\bs x}_{\bs t})}\\
\cdot\mr e^{\frac{\si^2}{2}(\|L_{t_1}(Z^{x_1,x_1}_{t_1})\|_2^2+\|L_{t_2}(\bar Z^{x_2,x_2}_{t_2})\|_2^2)}\left(\mr e^{\si^2\langle L_{t_1}(Z^{x_1,x_1}_{t_1}),L_{t_2}(\bar Z^{x_2,x_2}_{t_2})\rangle}-1\right)\Bigg]\d\mu^2(\bs a).
\end{multline}
At this point, we notice that the only term that depends on $U$ in \eqref{Equation: Covariance Formula Piece 1 - 2}
is $\bar{\mf m}^{\bs\ze}_{\bs t}(A^{\bs a,\bs a}_{\bs t})$. Thus, by replicating the analysis performed in
\eqref{Equation: Trace Moments UI 3}--\eqref{Equation: Trace Moments UI 6} (with $n=2$), if we first average
over $U$ only in \eqref{Equation: Covariance Formula Piece 1 - 2}, then we get from \eqref{Equation: Covariance Formula Piece 1 - 2}
the $\ze$-independent bound
\begin{multline}
\label{Equation: Covariance Formula Piece 1 - 3}
|\eqref{Equation: Covariance Formula Piece 1}|\leq C(t_1t_2)^{-1/2}\int_{I^2}\mbf E\Bigg[\mr e^{-\int_0^{|\bs t|} \ka|Z^{\bs x,\bs x}_{\bs t}(s)|\d s+\bar{\mf B_{|\bs t|}}(Z^{\bs x,\bs x}_{\bs t})}\\
\cdot\mr e^{\frac{2(r-1)^2\upsilon^2+\si^2}{2}(\|L_{t_1}(Z^{x_1,x_1}_{t_1})\|_2^2+\|L_{t_2}(\bar Z^{x_2,x_2}_{t_2})\|_2^2)}\left(\mr e^{\si^2\langle L_{t_1}(Z^{x_1,x_1}_{t_1}),L_{t_2}(\bar Z^{x_2,x_2}_{t_2})\rangle}-1\right)\Bigg]\d\bs x.
\end{multline}
Then, an application of H\"older's inequality in \eqref{Equation: Covariance Formula Piece 1 - 3} yields
\begin{multline}
\label{Equation: Covariance Formula Piece 1 - 4}
|\eqref{Equation: Covariance Formula Piece 1}|\leq C(t_1t_2)^{-1/2}\int_{I^2}\mbf E\left[\mr e^{-4\int_0^{|\bs t|} \ka|Z^{\bs x,\bs x}_{\bs t}(s)|\d s}\right]^{1/4}\mbf E\left[\mr e^{4\bar{\mf B_{|\bs t|}}(Z^{\bs x,\bs x}_{\bs t})}\right]^{1/4}\\
\cdot\mbf E\left[\mr e^{2(n(r-1)^2\upsilon^2+\si^2)(\|L_{t_1}(Z^{x_1,x_1}_{t_1})\|_2^2+\|L_{t_2}(\bar Z^{x_2,x_2}_{t_2})\|_2^2)}\right]^{1/4}\\
\cdot\mbf E\left[\left(\mr e^{\si^2\langle L_{t_1}(Z^{x_1,x_1}_{t_1}),L_{t_2}(\bar Z^{x_2,x_2}_{t_2})\rangle}-1\right)^4\right]^{1/4}\d\bs x.
\end{multline}

If we use exactly the same Brownian scaling arguments as in \cite[Lemma 4.6]{Prev2}, then we see that there exists a constant $C>0$ such that
\begin{align}
\label{Equation: SI Local Time Exponential Moment Bound}
\sup_{\bs t\in(0,1]^2,~\bs x\in I^2}\mbf E\left[\mr e^{4\bar{\mf B_{|\bs t|}}(Z^{\bs x,\bs x}_{\bs t})}\right]^{1/4}
\mbf E\left[\mr e^{2(n(r-1)^2\upsilon^2+\si^2)(\|L_{t_1}(Z^{x_1,x_1}_{t_1})\|_2^2+\|L_{t_2}(\bar Z^{x_2,x_2}_{t_2})\|_2^2)}\right]^{1/4}<C.
\end{align}
Thus, we get from \eqref{Equation: Covariance Formula Piece 1 - 4} that
\begin{multline}
\label{Equation: Covariance Formula Piece 1 - 5}
|\eqref{Equation: Covariance Formula Piece 1}|\leq C(t_1t_2)^{-1/2}\int_{I^2}\mbf E\left[\mr e^{-4\int_0^{|\bs t|} \ka|Z^{\bs x,\bs x}_{\bs t}(s)|\d s}\right]^{1/4}\\
\cdot\mbf E\left[\left(\mr e^{\si^2\langle L_{t_1}(Z^{x_1,x_1}_{t_1}),L_{t_2}(\bar Z^{x_2,x_2}_{t_2})\rangle}-1\right)^4\right]^{1/4}\d\bs x.
\end{multline}
At this point, in order to prove that \eqref{Equation: Covariance Formula Piece 1 - 5} is bounded by the right-hand side of \eqref{Equation: Rigidity RHS},
we separate our analysis for Cases 1 and 2 versus 3.

We begin with Case 3, as it is the simplest. In this case, we notice that since the interval $I=(0,\th)$ is bounded and $\mr e^{-4\int_0^{|\bs t|} \ka|Z^{\bs x,\bs x}_{\bs t}(s)|\d s}\leq1$,
\eqref{Equation: Covariance Formula Piece 1 - 5} yields
\begin{align}
\label{Equation: Covariance Formula Piece 1 - 6 pre 0}
|\eqref{Equation: Covariance Formula Piece 1}|\leq C(t_1t_2)^{-1/2}\sup_{\bs x\in I^2}\mbf E\left[\left(\mr e^{\si^2\langle L_{t_1}(Z^{x_1,x_1}_{t_1}),L_{t_2}(\bar Z^{x_2,x_2}_{t_2})\rangle}-1\right)^4\right]^{1/4}.
\end{align}
Then, if we apply the inequality $(\mr e^{x}-1)^4\leq (x\mr e^{|x|})^4$, followed by H\"older's inequality, we get the upper bound
\begin{align}
\label{Equation: Covariance Formula Piece 1 - 6 pre 1}
|\eqref{Equation: Covariance Formula Piece 1}|\leq C(t_1t_2)^{-1/2}\sup_{\bs x\in I^2}\prod_{k=1}^2\mbf E\left[\|L_{t_k}(Z^{x_k,x_k}_{t_k})\|_2^{8}\right]^{1/8}\mbf E\left[\mr e^{8\si^2\|L_{t_k}(Z^{x_k,x_k}_{t_k})\|_2}\right]^{1/8},
\end{align}
where the expectation factors into a product over $k=1,2$ because of the independence of $Z^{x_1,x_1}_{t_1}$
and $\bar Z^{x_2,x_2}_{t_2}$.
Thanks to \eqref{Equation: SI Local Time Exponential Moment Bound}, this further reduces to
\begin{align}
\label{Equation: Covariance Formula Piece 1 - 6}
|\eqref{Equation: Covariance Formula Piece 1}|\leq C(t_1t_2)^{-1/2}\sup_{\bs x\in I^2}\prod_{k=1}^2\mbf E\left[\|L_{t_k}(Z^{x_k,x_k}_{t_k})\|_2^{8}\right]^{1/8}.
\end{align}
By \cite[(4.31)]{Prev2} (in the notation of \cite{Prev2}, the parameter $\mf d$
is equal to $3/2$ in the case of white noise; see \cite[(2.16)]{Prev2}), we have the following estimate:
\begin{lemma}
\label{Lemma: SI Local Time Moment Scaling}
In Cases 1, 2, and 3, for any $k=1,2$ and $\theta\geq1$, there exists a constant $C$ such that
\[\sup_{x_k\in I}\mbf E\left[\|L_{t_k}(Z^{x_k,x_k}_{t_k})\|_2^{\theta}\right]^{1/\theta}\leq Ct_k^{3/4}\qquad\text{for all $t_k\in(0,1]$.}\]
\end{lemma}
The combination of this with \eqref{Equation: Covariance Formula Piece 1 - 6} implies that \eqref{Equation: Rigidity RHS Reduced to 2 Cases}
holds for \eqref{Equation: Covariance Formula Piece 1} in Case 3.

Consider now Cases 1 and 2.
If we replicate the bound \eqref{Equation: Potential Power Growth Triangle Inequality} with $\bar{\mf a}=1$
and apply \eqref{Equation: Exponential Moments of Brownian Maxima}, then we get from
\eqref{Equation: Covariance Formula Piece 1 - 5} that
\begin{align}
\label{Equation: Covariance Formula Piece 1 - 7}
|\eqref{Equation: Covariance Formula Piece 1}|\leq C(t_1t_2)^{-1/2}\int_{I^2}\mr e^{-t_1|x_1|-t_2|x_2|}\mbf E\left[\left(\mr e^{\si^2\langle L_{t_1}(Z^{x_1,x_1}_{t_1}),L_{t_2}(\bar Z^{x_2,x_2}_{t_2})\rangle}-1\right)^4\right]^{1/4}\d\bs x.
\end{align}
In order for the random variable inside the expectation in \eqref{Equation: Covariance Formula Piece 1 - 7}
to be nonzero, it is necessary that the supports of the independent Brownian bridges $Z^{x_1,x_1}_{t_1}$
and $\bar Z^{x_2,x_2}_{t_2}$ intersect. Thus, by introducing the indicator of this event in the expectation
and applying H\"older's inequality, we get from \eqref{Equation: Covariance Formula Piece 1 - 7} that
\begin{multline}
\label{Equation: Covariance Formula Piece 1 - 8}
|\eqref{Equation: Covariance Formula Piece 1}|\leq C(t_1t_2)^{-1/2}\int_{I^2}\mr e^{-t_1|x_1|-t_2|x_2|}
\mbf P\left[Z^{x_1,x_1}_{t_1}\text{ and }\bar Z^{x_2,x_2}_{t_2}\text{ intersect}\right]^{1/8}\\
\cdot\mbf E\left[\left(\mr e^{\si^2\langle L_{t_1}(Z^{x_1,x_1}_{t_1}),L_{t_2}(\bar Z^{x_2,x_2}_{t_2})\rangle}-1\right)^8\right]^{1/8}\d\bs x.
\end{multline}
By replicating the analysis in \eqref{Equation: Covariance Formula Piece 1 - 6 pre 0}--\eqref{Equation: Covariance Formula Piece 1 - 6}
and Lemma \ref{Lemma: SI Local Time Moment Scaling}, \eqref{Equation: Covariance Formula Piece 1 - 8}
implies that
\begin{align}
\label{Equation: Covariance Formula Piece 1 - 9}
|\eqref{Equation: Covariance Formula Piece 1}|\leq C(t_1t_2)^{1/4}\int_{I^2}\mr e^{-t_1|x_1|-t_2|x_2|}
\mbf P\left[Z^{x_1,x_1}_{t_1}\text{ and }\bar Z^{x_2,x_2}_{t_2}\text{ intersect}\right]^{1/8}\d\bs x.
\end{align}
By \cite[(4.14) and the two following displays]{RigiditySAO}, there exists $C,c>0$ such that
\[\mbf P\left[Z^{x_1,x_1}_{t_1}\text{ and }\bar Z^{x_2,x_2}_{t_2}\text{ intersect}\right]^{1/8}\leq C\left(\mr e^{-c(x_1-x_2)^2/t_1}+\mr e^{-c(x_1-x_2)^2/t_2}\right)\]
for all $t_1,t_2\in(0,1]$; hence
\eqref{Equation: Covariance Formula Piece 1 - 9} yields
\begin{align}
\label{Equation: Covariance Formula Piece 1 - 10}
|\eqref{Equation: Covariance Formula Piece 1}|\leq C(t_1t_2)^{1/4}\int_{\mbb R^2}\mr e^{-t_1|x_1|-t_2|x_2|}
\left(\mr e^{-c(x_1-x_2)^2/t_1}+\mr e^{-c(x_1-x_2)^2/t_2}\right)\d\bs x,
\end{align}
where we have replaced $I$ by $\mbb R$ in Case 2, which only increases the size of the integral.
Denote $t_*=\max\{t_1,t_2\}$. Then,
\eqref{Equation: Covariance Formula Piece 1 - 10} implies that
\begin{align}
\label{Equation: Covariance Formula Piece 1 - 11}
|\eqref{Equation: Covariance Formula Piece 1}|\leq C(t_1t_2)^{1/4}\int_{\mbb R^2}\mr e^{-t_*|x_1|}\mr e^{-c(x_1-x_2)^2/t_*}\d\bs x=C(t_1t_2)^{1/4}\cdot\frac{2 \sqrt{\pi }}{\sqrt{c t_*}};
\end{align}
hence \eqref{Equation: Rigidity RHS Reduced to 2 Cases}
holds for \eqref{Equation: Covariance Formula Piece 1} in Case 1 and 2 as well.

\subsubsection{Step 3. Upper Bound for \eqref{Equation: Covariance Formula Piece 2}}

Recall the coupling introduced in Definition \ref{Notation: Z and Z bar coupling}.
Given that $\bs\ze=(\ze,\ze)$, we have from \eqref{Equation: Vector Valued Self-Intersection 1} that
the difference
$\mf m^{\bs\ze}_{\bs t}(A^{\bs a,\bs a}_{\bs t})-\mf m^{\ze}_{t_1}(A^{a_1,a_1}_{t_1})\mf m^{\ze}_{t_2}(A^{a_2,a_2}_{t_2})$
can be written as
\begin{multline}
\label{Equation: Covariance Formula Piece 2 - 1}
\mr e^{(r-1)|\bs t|}\Bigg(\mbf 1_{\{N(|\bs t|)\text{ is even}\}}\sum_{p\in\mc P_{N(|\bs t|)}}\mf C_{|\bs t|}(p,U^{\bs i,\bs i}_{\bs t})\prod_{\{\ell_1,\ell_2\}\in p}\upsilon^2\rho_{\ze}\big(Z^{\bs x,\bs x}_{\bs t}(\tau_{\ell_1})-Z^{\bs x,\bs x}_{\bs t}(\tau_{\ell_2})\big)\\
-\mbf 1_{\{N(t_1),N(|\bs t|)-N(t_1)\text{ are even}\}}\sum_{p_1\in\mc P_{N(t_1)},~p_2\in\mc P_{N(|\bs t|)-N(t_1)}}\Bigg\{\cdots\\
\cdots\mf C_{t_1}(p_1,U^{i_1,i_1}_{t_1})\mf C_{t_2}(p_2,\bar U^{i_2,i_2}_{t_2})\prod_{\{\ell_1,\ell_2\}\in p_1\cup p_2}\upsilon^2\rho_{\ze}\big(Z^{\bs x,\bs x}_{\bs t}(\tau_{\ell_1})-Z^{\bs x,\bs x}_{\bs t}(\tau_{\ell_2})\big)\Bigg\}
\Bigg).
\end{multline}

Suppose that $N(t_1)$ and $N(|\bs t|)-N(t_1)$ are even, and
let $p\in\mc P_{N(|\bs t|)}$ be such that we can write $p=p_1\cup p_2$ for some
$p_1\in\mc P_{N(t_1)}$ and $p_2\in\mc P_{N(|\bs t|)-N(t_1)}$.
We claim that in such a case, one has
\begin{align}
\label{Equation: Covariance Formula Piece 2 - 2}
\mf C_{|\bs t|}(p,U^{\bs i,\bs i}_{\bs t})=\mf C_{|\bs t|}(p_1\cup p_2,U^{\bs i,\bs i}_{\bs t})=\mf C_{t_1}(p_1,U^{i_1,i_1}_{t_1})\mf C_{t_2}(p_2,\bar U^{i_2,i_2}_{t_2}).
\end{align}
When $\mbb F=\mbb R$ or $\mbb C$, this follows directly by Definition \ref{Definition: Combinatorial Constant}-(1) and -(2):
The matching $p=p_1\cup p_2$ only contains pairs $\{\ell_1,\ell_2\}$ such that $J_{\ell_1}=J_{\ell_2}$ and/or $J_{\ell_1}=J_{\ell_2}^*$
if and only if the same is true for both $p_1$ and $p_2$. When $\mbb F=\mbb H$,
by Definition \ref{Definition: Combinatorial Constant}-(3) we obtain \eqref{Equation: Covariance Formula Piece 2 - 2} if we additionally ensure that
\begin{align*}
\mf D_{|\bs t|}(p,U^{\bs i,\bs i}_{\bs t})=\mf D_{|\bs t|}(p_1\cup p_2,U^{\bs i,\bs i}_{\bs t})=\mf D_{t_1}(p_1,U^{i_1,i_1}_{t_1})\mf D_{t_2}(p_2,\bar U^{i_2,i_2}_{t_2})
\end{align*}
whenever $J_{\ell_1}=J_{\ell_2}$ or $J_{\ell_1}=J_{\ell_2}^*$ for all $\{\ell_1,\ell_2\}\in p$.
By \eqref{Equation: D constant for H}, this reduces to
\begin{multline}
\label{Equation: Covariance Formula Piece 2 - 3}
\sum_{m\in\mc B^{0,0}_{N(|\bs t|)},~m\text{ respects }(p,J)}(-1)^{f(m,p,J)}=\sum_{m\in\mc B^{0,0}_{N(|\bs t|)},~m\text{ respects }(p_1\cup p_2,J)}(-1)^{f(m,p_1\cup p_2,J)}\\
=\left(\sum_{\substack{m_1\in\mc B^{0,0}_{N(t_1)}\\m_1\text{ resp. }(p_1,J_1)}}(-1)^{f(m_1,p_1,J_1)}\right)\left(\sum_{\substack{m_2\in\mc B^{0,0}_{N(|\bs t|)-N(t_1)}\\m_2\text{ resp. }(p_2,\bar J_2)}}(-1)^{f(m_2,p_2,\bar J_2)}\right),
\end{multline}
where $J_1$ denotes the jumps in the path $U^{i_1,i_1}_{t_1}$, and $\bar J_2$ the jumps in the
path $\bar U^{i_2,i_2}_{t_2}$. If $m\in\mc B^{0,0}_{N(|\bs t|)}$ can be written as the concatenation of $m_1\in\mc B^{0,0}_{N(t_1)}$
and $m_2\in\mc B^{0,0}_{N(|\bs t|)-N(t_1)}$, then clearly $f(m_1,p_1,J_1)+f(m_2,p_2,\bar J_2)=f(m,p,J)$ since flips can only occur in pairs matched by $p=p_1\cup p_2$; moreover,
the fact that $m_1$ respects $(p_1,J_1)$ and $m_2$ respects $(p_2,\bar J_2)$ implies that $m$ respects $(p,J)$.
Thus, \eqref{Equation: Covariance Formula Piece 2 - 3} follows if we show that every
$m\in\mc B^{0,0}_{N(|\bs t|)}$ that respects $(p_1\cup p_2,J)$ can be written as a concatenation of
$m_1\in\mc B^{0,0}_{N(t_1)}$ which respects $(p_1,J_1)$ and
$m_2\in\mc B^{0,0}_{N(|\bs t|)-N(t_1)}$ which respects $(p_2,\bar J_2)$.
For this, we recall from Definition \ref{Definition: Combinatorial Constant}-(3.1) and -(3.2) that in order for $m$ to respect $(p_1\cup p_2,J)$,
the steps in $m$ that are paired up by $p$ must be taken from the following set of pairs of steps:
\[\Big\{\big\{(0,0),(1,1)\big\},\big\{(0,1),(1,0)\big\},\big\{(0,0),(0,0)\big\},\big\{(1,1),(1,1)\big\}\Big\}.\]
In the above, the only steps that move from $0$ to $1$ or vice versa are $(0,1)$ and $(1,0)$,
and these must be paired by $p$. In particular, any uninterrupted segment of $m$ that only contains pairs of steps
matched by $p$ must start and end at the same element.
Consequently, the first $N(t_1)$ steps in $m$, which are matched by $p_1$,
must form a path in $\mc B^{0,0}_{N(t_1)}$ that respects $(p_1,J_1)$,
and the following $N(|\bs t|)-N(t_1)$ steps in $m$, which are matched by $p_2$,
must form a path in $\mc B^{0,0}_{N(|\bs t|)-N(t_1)}$ that respects $(p_2,\bar J_2)$.
This then concludes the proof of \eqref{Equation: Covariance Formula Piece 2 - 2}.

Thanks to \eqref{Equation: Covariance Formula Piece 2 - 2},
we now see that every $p\in\mc P_{N(|\bs t|)}$ that we can write $p=p_1\cup p_2$ for some
$p_1\in\mc P_{N(t_1)}$ and $p_2\in\mc P_{N(|\bs t|)-N(t_1)}$
is cancelled out by the difference in
\eqref{Equation: Covariance Formula Piece 2 - 1}.
Thus, if we let
$\mc R_{m_1,m_2}$
denote the set of $p\in\mc P_{m_1+m_2}$ that cannot be written as
$p=p_1\cup p_2$ for some
$p_1\in\mc P_{m_1}$ and $p_2\in\mc P_{m_2}$, then 
\eqref{Equation: Covariance Formula Piece 2 - 1} simplifies to
\begin{multline*}
\mf m^{\bs\ze}_{\bs t}(A^{\bs a,\bs a}_{\bs t})-\mf m^{\ze}_{t_1}(A^{a_1,a_1}_{t_1})\mf m^{\ze}_{t_2}(A^{a_2,a_2}_{t_2})\\
=\mr e^{(r-1)|\bs t|}\mbf 1_{\{N(|\bs t|)\text{ is even}\}}\sum_{p\in\mc R_{N(t_1),N(|\bs t|)-N(t_1)}}\mf C_{|\bs t|}(p,U^{\bs i,\bs i}_{\bs t})\prod_{\{\ell_1,\ell_2\}\in p}\upsilon^2\rho_{\ze}\big(Z^{\bs x,\bs x}_{\bs t}(\tau_{\ell_1})-Z^{\bs x,\bs x}_{\bs t}(\tau_{\ell_2})\big).
\end{multline*}
At this point, if we combine \eqref{Equation: Combinatorial constant bound} with the fact that
$\mc R_{0,0}$ is empty (as $\varnothing=\varnothing\cup\varnothing$),
then we obtain the bound
\begin{multline}
\label{Equation: Covariance Formula Piece 2 - 4}
\big|\mf m^{\bs\ze}_{\bs t}(A^{\bs a,\bs a}_{\bs t})-\mf m^{\ze}_{t_1}(A^{a_1,a_1}_{t_1})\mf m^{\ze}_{t_2}(A^{a_2,a_2}_{t_2})\big|\\
\leq\mr e^{(r-1)|\bs t|}\mbf 1_{\{N(|\bs t|)\geq2\text{ is even}\}}\sum_{p\in\mc R_{N(t_1),N(|\bs t|)-N(t_1)}}\prod_{\{\ell_1,\ell_2\}\in p}\upsilon^2\rho_{\ze}\big(Z^{\bs x,\bs x}_{\bs t}(\tau_{\ell_1})-Z^{\bs x,\bs x}_{\bs t}(\tau_{\ell_2})\big).
\end{multline}
Thus, plugging this back into \eqref{Equation: Covariance Formula Piece 2} yields
\begin{multline}
\label{Equation: Covariance Formula Piece 2 - 5}
|\eqref{Equation: Covariance Formula Piece 2}|\leq
\int_{\mc A^2}\Pi_A(\bs t;\bs a,\bs a)\mbf E\Bigg[\mr e^{(r-1)|\bs t|}\mbf 1_{\{N(|\bs t|)\geq2\text{ is even}\}}\\
\cdot\sum_{p\in\mc R_{N(t_1),N(|\bs t|)-N(t_1)}}\prod_{\{\ell_1,\ell_2\}\in p}\upsilon^2\rho_{\ze}\big(Z^{\bs x,\bs x}_{\bs t}(\tau_{\ell_1})-Z^{\bs x,\bs x}_{\bs t}(\tau_{\ell_2})\big)\\
\cdot\mr e^{-\int_0^{|\bs t|}V(A^{\bs a,\bs a}_{\bs t}(s))\d s+\mf B_{|\bs t|}(A^{\bs a,\bs a}_{\bs t})+\frac{\si^2}{2}(\|L_{[0,t_1)}(A^{\bs a,\bs a}_{\bs t})\|_\mu^2+\|L_{[t_1,t_1+t_2)}(A^{\bs a,\bs a}_{\bs t})\|_\mu^2)}\Bigg]\d\mu^2(\bs a).
\end{multline}
Then, if we replicate the bounds on $V$, $\Pi_{Z}$, $\mf B_t(A)$, and the local time $L_{[u,v)}(A)$
used to arrive to \eqref{Equation: Covariance Formula Piece 1 - 2}---together with an application
of the tower property in the expectation with respect to $U$---we obtain from
\eqref{Equation: Covariance Formula Piece 2 - 5} that
\begin{multline}
\label{Equation: Covariance Formula Piece 2 - 6}
|\eqref{Equation: Covariance Formula Piece 2}|\leq
C(t_1t_2)^{-1/2}\int_{\mc A^2}\mbf E\Bigg[\cdots\\
\cdot\Pi_U(\bs t;\bs i,\bs i)\mbf E_{U^{\bs i,\bs i}_{\bs t}}\Bigg[\mbf 1_{\{N(|\bs t|)\geq2\text{ is even}\}}\sum_{p\in\mc R_{N(t_1),N(|\bs t|)-N(t_1)}}\prod_{\{\ell_1,\ell_2\}\in p}\upsilon^2\rho_{\ze}\big(Z^{\bs x,\bs x}_{\bs t}(\tau_{\ell_1})-Z^{\bs x,\bs x}_{\bs t}(\tau_{\ell_2})\big)\Bigg]\\
\cdot\mr e^{-\int_0^{|\bs t|} \ka|Z^{\bs x,\bs x}_{\bs t}(s)|\d s+\bar{\mf B_{|\bs t|}}(Z^{\bs x,\bs x}_{\bs t})+\frac{\si^2}{2}(\|L_{t_1}(Z^{x_1,x_1}_{t_1})\|_2^2+\|L_{t_2}(\bar Z^{x_2,x_2}_{t_2})\|_2^2)}\Bigg]\d\mu^2(\bs a).
\end{multline}

Using the fact that
\[\Pi_U(\bs t;\bs i,\bs i)\mbf E\big[F(U^{\bs i,\bs i}_{\bs t})\big]=\mbf E\big[F(U^{\bs i}_{\bs t})\mbf 1_{\{\forall k:~U^{\bs i}_{\bs t}(s)\to i_k\text{ as }s\to(t_0+\cdots+t_k)^+\}}\big]
\leq\mbf E\big[F(U^{\bs i}_{\bs t})\big]\]
for any nonnegative functional $F$, we note that
\begin{align}
\nonumber
\Pi_U(\bs t;\bs i,\bs i)\mbf E_{U^{\bs i,\bs i}_{\bs t}}\Bigg[\mbf 1_{\{N(|\bs t|)\geq2\text{ is even}\}}\sum_{p\in\mc R_{N(t_1),N(|\bs t|)-N(t_1)}}\prod_{\{\ell_1,\ell_2\}\in p}\upsilon^2\rho_{\ze}\big(Z^{\bs x,\bs x}_{\bs t}(\tau_{\ell_1})-Z^{\bs x,\bs x}_{\bs t}(\tau_{\ell_2})\big)\Bigg]\\
\label{Equation: Covariance Formula Piece 2 - 7}
\leq
\mbf E_{U^{\bs i}_{\bs t}}\Bigg[\mbf 1_{\{N(|\bs t|)\geq2\text{ is even}\}}\sum_{p\in\mc R_{N(t_1),N(|\bs t|)-N(t_1)}}\prod_{\{\ell_1,\ell_2\}\in p}\upsilon^2\rho_{\ze}\big(Z^{\bs x,\bs x}_{\bs t}(\tau_{\ell_1})-Z^{\bs x,\bs x}_{\bs t}(\tau_{\ell_2})\big)\Bigg].
\end{align}
Then, using the law of total expectation, and the facts that $N(|\bs t|)$ is a Poisson random variable
with parameter $(r-1)|\bs t|$ and that conditional on $N(t_1)=2m$,
the jump times $(\tau_1,\ldots,\tau_{2m})$ are the order statistics of i.i.d. uniform
random points on $[0,|\bs t|)^{2m}$, we can bound \eqref{Equation: Covariance Formula Piece 2 - 7} by:
\begin{multline}
\label{Equation: Covariance Formula Piece 2 - 8}
\eqref{Equation: Covariance Formula Piece 2 - 7}\leq\sum_{m=1}^{\infty}\Bigg(\int_{[0,|\bs t|)^{2m}_<}
\sum_{p\in\mc R_{m_1(\bs s),m_2(\bs s)}}\left\{\prod_{\{\ell_1,\ell_2\}\in p}\upsilon^2\rho_{\ze}\big(Z^{\bs x,\bs x}_{\bs t}(s_{\ell_1})-Z^{\bs x,\bs x}_{\bs t}(s_{\ell_2})\big)\right\}\\
\cdot\frac{(2m)!}{|\bs t|^{2m}}\d\bs s\Bigg)\frac{\big((r-1)|\bs t|\big)^{2m}}{(2m)!}\mr e^{-(r-1)|\bs t|},
\end{multline}
where we recall that $[0,|\bs t|)_<^{2m}$ denotes the tuples $\bs s=(s_1,\ldots,s_{2m})\in[0,|\bs t|)^{2m}$
such that $s_1<s_2<\cdots<s_{2m}$, and we use $m_1(\bs s)$ and $m_2(\bs s)$ respectively to denote
the number of components of $\bs s$ that land in $[0,t_1)$ and $[t_1,t_1+t_2)$
(i.e., $m_1(\bs s)$ is the largest integer $\ell$ such that $s_\ell< t_1$, and $2m=m_1(\bs s)+m_2(\bs s)$).
If $p\in\mc R_{m_1(\bs s),m_2(\bs s)}$, then this means that there exists at least one pair of indices $1\leq \ell_1<\ell_2\leq 2m$ such that $s_{\ell_1}\in[0,t_1)$,
$s_{\ell_2}\in[t_1,|\bs t|)$, and $\{\ell_1,\ell_2\}\in p$. The indicator function that the tuple $\bs s$ satisfies this is obviously bounded by the sum of indicators
\[\mc I^{t_1,t_2}(\bs s)=\sum_{u,v=1}^{2m}\mbf 1_{\{s_u\in[0,t_1)\}}\mbf 1_{\{s_v\in[t_1,t_1+t_2)\}}.\]
If we combine this upper bound in \eqref{Equation: Covariance Formula Piece 2 - 8} with the replacement of the sum over $p\in\mc R_{m_1(\bs s),m_2(\bs s)}$
to the sum over $p\in\mc P_{2m}$, then we get the further upper bound
\begin{multline}
\label{Equation: Covariance Formula Piece 2 - 9}
\eqref{Equation: Covariance Formula Piece 2 - 7}\leq\sum_{m=1}^{\infty}\Bigg(\int_{[0,|\bs t|)^{2m}_<}
\sum_{p\in\mc P_{2m}}\mc I^{t_1,t_2}(\bs s)\left\{\prod_{\{\ell_1,\ell_2\}\in p}\upsilon^2\rho_{\ze}\big(Z^{\bs x,\bs x}_{\bs t}(s_{\ell_1})-Z^{\bs x,\bs x}_{\bs t}(s_{\ell_2})\big)\right\}\\
\cdot\frac{(2m)!}{|\bs t|^{2m}}\d\bs s\Bigg)\frac{\big((r-1)|\bs t|\big)^{2m}}{(2m)!}\mr e^{-(r-1)|\bs t|}.
\end{multline}
Now, we note that the function being integrated in \eqref{Equation: Covariance Formula Piece 2 - 9} is invariant
with respect to permuting the entries of $\bs s$, hence we can rewrite \eqref{Equation: Covariance Formula Piece 2 - 9} by
removing the condition that the $s_i$'s be ordered as follows:
\begin{multline*}
\eqref{Equation: Covariance Formula Piece 2 - 7}\leq\sum_{m=1}^{\infty}\Bigg(\int_{[0,|\bs t|)^{2m}}
\sum_{p\in\mc P_{2m}}\mc I^{t_1,t_2}(\bs s)\left\{\prod_{\{\ell_1,\ell_2\}\in p}\upsilon^2\rho_{\ze}\big(Z^{\bs x,\bs x}_{\bs t}(s_{\ell_1})-Z^{\bs x,\bs x}_{\bs t}(s_{\ell_2})\big)\right\}\\
\cdot\frac{1}{|\bs t|^{2m}}\d\bs s\Bigg)\frac{\big((r-1)|\bs t|\big)^{2m}}{(2m)!}\mr e^{-(r-1)|\bs t|}.
\end{multline*}
Now, if we cancel out both appearances of $|\bs t|^{2m}$, pull out the sums over $p$ and with in $\mc I_{t_1,t_2}(\bs s)$,
and then apply a change of variables such that the variables $u$ and $v$ in $\mbf 1_{\{s_u\in[0,t_1)\}}\mbf 1_{\{s_v\in[t_1,t_1+t_2)\}}$
are always $u=1$ and $v=2$, then Tonelli's theorem yields
\begin{multline*}
\eqref{Equation: Covariance Formula Piece 2 - 7}\leq\sum_{m=1}^{\infty}(2m-1)!!(2m)(2m-1)\Bigg(\int_{[0,t_1]\times[t_1,t_1+t_2]}\upsilon^2\rho_{\ze}\big(Z^{\bs x,\bs x}_{\bs t}(s_1)-Z^{\bs x,\bs x}_{\bs t}(s_2)\big)\d\bs s\Bigg)\\
\Bigg(\int_{[0,|\bs t|)^2}\upsilon^2\rho_{\ze}\big(Z^{\bs x,\bs x}_{\bs t}(s_1)-Z^{\bs x,\bs x}_{\bs t}(s_2)\big)\d\bs s\Bigg)^{m-1}
\frac{\big((r-1)\big)^{2m}}{(2m)!}\mr e^{-(r-1)|\bs t|},
\end{multline*}
where the prefactors in front of the integral come from the fact that $|\mc P_{2m}|=(2m-1)!!$ and that there are $(2m)(2m-1)$ summands in
$\mc I_{t_1,t_2}(\bs s)$. Replicating the application of Young's convolution inequality in \eqref{Equation: Young's for zeta},
we get that for intervals $\mc K_1,\mc K_2\subset[0,\infty)$,
\[\int_{\mc K_1\times\mc K_2}\upsilon^2\rho_{\ze}\big(Z^{\bs x,\bs x}_{\bs t}(s_1)-Z^{\bs x,\bs x}_{\bs t}(s_2)\big)\d\bs s
\leq\upsilon^2\big\langle L_{\mc K_1}(Z^{\bs x,\bs x}_{\bs t}), L_{\mc K_2}(Z^{\bs x,\bs x}_{\bs t})\big\rangle.\]
Therefore, we get that \eqref{Equation: Covariance Formula Piece 2 - 7}$\cdot\mr e^{(r-1)|\bs t|}$ is bounded above by
\begin{multline*}
\frac{\big\langle L_{t_1}(Z^{x_1,x_1}_{t_1}), L_{t_2}(\bar Z^{x_2,x_2}_{t_2})\big\rangle}{\|L_{|\bs t|}(Z^{\bs x,\bs x}_{\bs t})\|_2^2}\sum_{m=1}^{\infty}\frac{(2m-1)!!(2m)(2m-1)}{(2m)!}\left(\frac{(r-1)^2\upsilon^2\|L_{|\bs t|}(Z^{\bs x,\bs x}_{\bs t})\|_2^2}{2}\right)^m\\
=(r-1)^2 \upsilon ^2\big\langle L_{t_1}(Z^{x_1,x_1}_{t_1}), L_{t_2}(\bar Z^{x_2,x_2}_{t_2})\big\rangle\left(\|L_{|\bs t|}(Z^{\bs x,\bs x}_{\bs t})\|_2^2 (r-1)^2 \upsilon ^2+1\right)
  \mr e^{\frac{1}{2}(r-1)^2 \upsilon ^2\|L_{|\bs t|}(Z^{\bs x,\bs x}_{\bs t})\|_2^2},
  \end{multline*}
where the second line follows from direct calculation.

If we now combine this with \eqref{Equation: Covariance Formula Piece 2 - 6} and \eqref{Equation: Covariance Formula Piece 2 - 7}, then we obtain
\begin{multline*}
|\eqref{Equation: Covariance Formula Piece 2}|\leq
C(t_1t_2)^{-1/2}\int_{I^2}\mbf E\Bigg[\cdots\\
\cdot\big\langle L_{t_1}(Z^{x_1,x_1}_{t_1}), L_{t_2}(\bar Z^{x_2,x_2}_{t_2})\big\rangle\left(\|L_{|\bs t|}(Z^{\bs x,\bs x}_{\bs t})\|_2^2 (r-1)^2 \upsilon ^2+1\right)
  \mr e^{\frac{1}{2}(r-1)^2 \upsilon ^2\|L_{|\bs t|}(Z^{\bs x,\bs x}_{\bs t})\|_2^2}\\
\cdot\mr e^{-\int_0^{|\bs t|} \ka|Z^{\bs x,\bs x}_{\bs t}(s)|\d s+\bar{\mf B_{|\bs t|}}(Z^{\bs x,\bs x}_{\bs t})+\frac{\si^2}{2}(\|L_{t_1}(Z^{x_1,x_1}_{t_1})\|_2^2+\|L_{t_2}(\bar Z^{x_2,x_2}_{t_2})\|_2^2)}\Bigg]\d\bs x.
\end{multline*}
At this point, we obtain
\eqref{Equation: Rigidity RHS Reduced to 2 Cases}
for
\eqref{Equation: Covariance Formula Piece 2}
by replicating essentially the same analysis that we performed in \eqref{Equation: Covariance Formula Piece 1 - 6 pre 0}:
The only difference is that we replace
\[\mr e^{\si^2\langle L_{t_1}(Z^{x_1,x_1}_{t_1}),L_{t_2}(\bar Z^{x_2,x_2}_{t_2})\rangle}-1\]
by
\[\big\langle L_{t_1}(Z^{x_1,x_1}_{t_1}), L_{t_2}(\bar Z^{x_2,x_2}_{t_2})\big\rangle\left(\|L_{|\bs t|}(Z^{\bs x,\bs x}_{\bs t})\|_2^2 (r-1)^2 \upsilon ^2+1\right),\]
but both of these terms are controlled in the same way. That is,
we introduce an indicator that the paths of $Z^{x_1,x_1}_{t_1}$ and $\bar Z^{x_2,x_2}_{t_2}$ must intersect,
and then use Cauchy-Schwarz with Lemma \ref{Lemma: SI Local Time Moment Scaling}.
With this in hand, the proof of Theorem \ref{Theorem: Rigidity}, and thus of all of our results,
is complete.

%

%

\appendix
\section{Illustrations}
\label{Appendix: Illustrations}

\begin{figure}[htbp]
\begin{center}
\includegraphics[width=\textwidth]{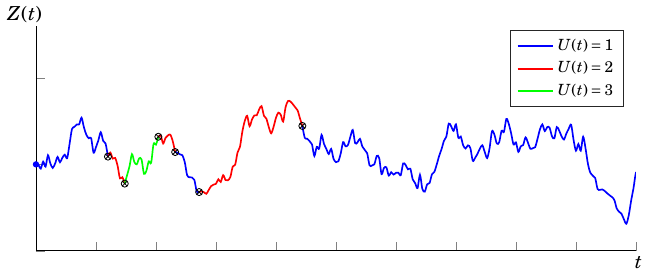}

\caption{This figure shows a realization $A=(U,Z)$.
The $y$-axis represents the output of $Z$, and 
the colors represent the value of the process $U$ (the jump times $\tau_k$ therein
are represented by the symbol $\otimes$).}
\label{Figure: Generic Path of A}
\end{center}
\end{figure}

\begin{figure}[htbp]
\begin{center}
\includegraphics[width=\textwidth]{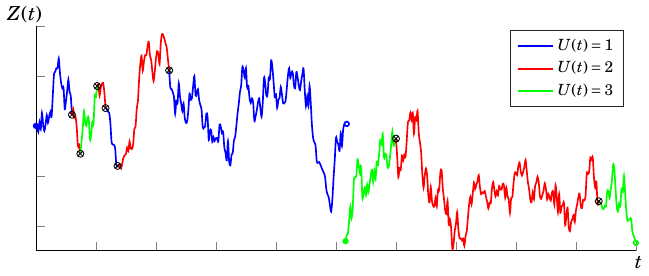}

\caption{This figure shows a realization of the c\`{a}dl\`{a}g concatenated path
$A^{\bs a,\bs b}_{\bs t}$ with $\bs a=\big((1,1),(3,0)\big)$, $\bs b=\big((1,1),(3,0)\big)$,
and $\bs t=(1,1)$ (i.e., the first path goes from $(i,x)=(1,1)$ to itself in time $1$,
and the second path goes from $(i,x)=(3,0)$ to itself in time $1$).}
\label{Figure: Concatenated Path of A}
\end{center}
\end{figure}

\begin{figure}[htbp]
\begin{center}
\includegraphics[width=\textwidth]{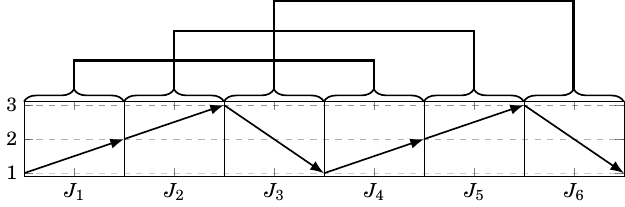}
\caption{In this figure, the arrows represent the first six jumps of $U$ by time $t$,
namely, $(1,2),(2,3),(3,1),(1,2),(2,3),(3,1)$. The partition above the
arrows serve to highlight which jumps were paired by the perfect matching $p=\big\{\{1,4\},\{2,5\},\{3,6\}\big\}$.
With this in hand, we conclude that this particular combination of $(p,U)$
is such that $\mf C_t(p,U)\neq0$ when $\mbb F=\mbb R,\mbb H$, since $J_1=J_4$, $J_2=J_5$,
and $J_3=J_6$. However, $\mf C_t(p,U)=0$ when $\mbb F=\mbb C$
because paired jumps are not reversed.}
\label{Figure: Walk 1 and 4}
\end{center}
\end{figure}

\begin{figure}[htbp]
\begin{center}
\includegraphics[width=\textwidth]{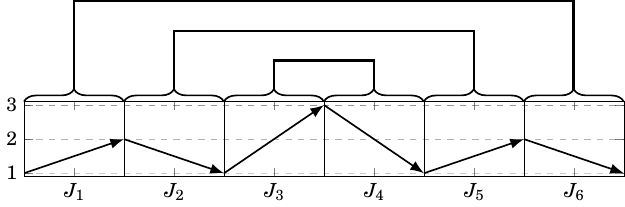}
\caption{In this figure, the six jumps by time $t$ are $(1,2),(2,1),(1,3),(3,1),(1,2),(2,1)$,
and $p=\big\{\{1,6\},\{2,5\},\{3,4\}\big\}$.
In this case, $\mf C_t(p,U)\neq0$ for all $\mbb F=\mbb R,\mbb C,\mbb H$, since every paired jumps
are reversed (i.e., $J_1=J_6^*$, $J_2=J_5^*$, and $J_3=J_4^*$).}
\label{Figure: Walk 2}
\end{center}
\end{figure}

\begin{figure}[htbp]
\begin{center}
\includegraphics[width=\textwidth]{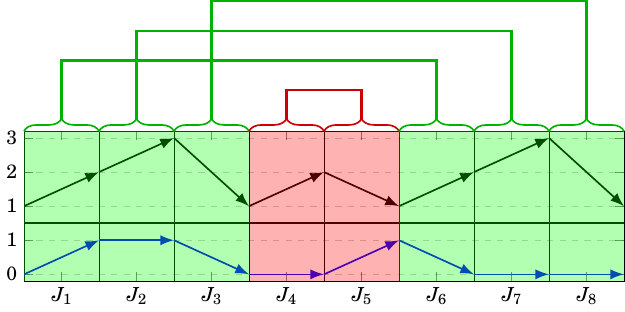}
\caption{In this figure, the black arrows are the steps $J_k$ until $N(t)$,
the blue arrows are the steps of a binary sequence $m\in\mc B^{0,0}_{N(t)}$,
and the partition above all that represents the pairs given by $p$. In order to help
decipher whether this choice of $m$ contributes to $\mf D_t(p,U)$,
the matched pairs of steps such that $J_{\ell_1}=J_{\ell_2}$ are colored in green,
and the matched pairs such that $J_{\ell_1}=J_{\ell_2}^*$ are colored in red.
In this case,
$m$ does not respect $(p,J)$, and thus it does not contribute to $\mf D_t(p,U)$.
(E.g., $J_4$ and $J_5$ are matched and in opposite directions, but the corresponding steps
in the binary sequence, $\big((0,0),(0,1)\big)$, are not among those listed in
Definition \ref{Definition: Combinatorial Constant}-(3.2))}
\label{Figure: Binary non-example}
\end{center}
\end{figure}

\begin{figure}[htbp]
\begin{center}
\includegraphics[width=\textwidth]{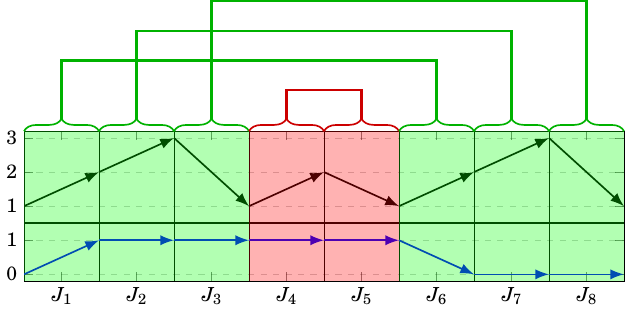}
\caption{In this figure, $m$ contributes to $\mf D_t(p,U)$,
since every matched pairs of jumps $J_{\ell_1},J_{\ell_2}$ are accompanied by steps
in the binary sequence that satisfy Definition \ref{Definition: Combinatorial Constant}-(3.1) and -(3.2).
Moreover, there is one flip that occurs in this sequence:
$J_1=J_6$ and $\big((m_0,m_1),(m_5,m_6)\big)=\big((0,1),(1,0)\big)$;
hence $f(m,p,J)=1$.}
\label{Figure: Binary 1}
\end{center}
\end{figure}

\begin{figure}[htbp]
\begin{center}
\includegraphics[width=\textwidth]{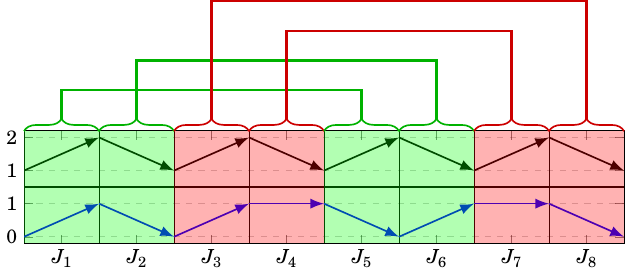}
\caption{This is another example where $m$ contributes
to $\mf D_t(p,U)$. Moreover, in this case we have $f(m,p,J)=2$;
the two flips occur at the pairs $J_1,J_5$ and $J_2,J_6$.}
\label{Figure: Binary 2}
\end{center}
\end{figure}

\begin{figure}[htbp]
\begin{center}
\includegraphics[width=\textwidth]{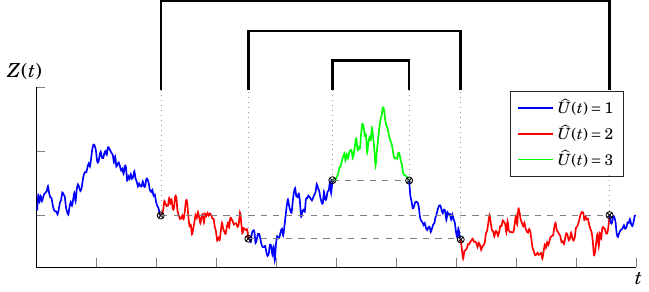}

\caption{In this figure, the colored path represents a realization of $(\hat U,Z)$
with $Z(0)=1=Z(1)$.
The symbols $\otimes$ along the path represent the jump times $\hat\tau_k$,
and the different colors represent the values of the process $\hat U(t)$.
The jumps $J_k$ and the random matching $\hat\pi(P)$ (the latter of which
is represented by the partition above $Z$'s path) are the same as $J_k$ and $p$ in Figure \ref{Figure: Walk 2}.
In this illustration, we observe that
$Z(\hat\tau_{\ell_1})=Z(\hat\tau_{\ell_2})$ for every matched pair $\{\ell_1,\ell_2\}\in \hat\pi(P)$, which is consistent with the fact that
the pairs $\hat\tau_{\ell_1}$ and $\hat\tau_{\ell_2}$ were sampled
according to $Z$'s self-intersection measure.}
\label{Figure: Condition on Jump Times and Values of X}
\end{center}
\end{figure}

\end{document}